\documentclass[12pt]{amsart}

\usepackage{dsfont}
\usepackage{amsxtra}
\usepackage{amsmath}
\usepackage{amscd}
\usepackage{amssymb}
\usepackage{amsfonts}
\usepackage[all]{xy}
\usepackage{mathrsfs}
\usepackage{amsthm}
\usepackage{color}
\usepackage{upgreek }
\usepackage{tikz}
\usetikzlibrary{matrix,arrows,decorations.pathmorphing,automata, matrix, positioning, calc, shapes.multipart}
\usepackage{tikz-cd}
\usepackage{enumerate}

\usepackage{ae}

\usepackage{enumitem}

\usepackage{blkarray}
\usepackage{multirow}
\usepackage{mathtools}
\usepackage{graphicx}

\numberwithin{itemcounter}{subsection}

\theoremstyle{plain}
\newtheorem*{conjecture}{Conjecture}
\newtheorem*{thmA}{Theorem A}
\newtheorem*{thmB}{Theorem B}
\newtheorem*{thmC}{Theorem C}

\newtheorem{theorem}{Theorem}[section]
\newtheorem{lemma}[theorem]{Lemma}
\newtheorem{definition-lemma}[theorem]{Definition-Lemma}

\newtheorem{proposition}[theorem]{Proposition}

\newtheorem{corollary}[theorem]{Corollary}
\theoremstyle{definition}
\newtheorem{definition}[theorem]{Definition}
\theoremstyle{remark}
\newtheorem{remark}[theorem]{Remark}

\numberwithin{equation}{section}

\def\bbC{\mathbb{C}}

\def\bbF{\mathbb{F}}

\def\bbN{\mathbb{N}}

\def\bbZ{\mathbb{Z}}

\def\scrC{\mathscr{C}}

\def\scrH{\mathscr{H}}

\def\scrI{{I}}

\def\scrL{\mathscr{L}}

\def\scrP{\mathscr{P}}

\def\scrS{\mathscr{S}}

\def\scrU{\mathscr{U}}
\def\scrQU{\mathscr{QU}}

\def\frakg{\mathfrak{g}}

\def\fraks{\mathfrak{s}}
\def\frakl{\mathfrak{l}}
\def\frakA{\mathfrak{A}}

\def\frakC{\mathfrak{C}}

\def\frakE{\mathfrak{E}}

\def\frakS{\mathfrak{S}}

\def\frakg{\mathfrak{g}}
\def\frakh{\mathfrak{h}}
\def\frakl{\mathfrak{l}}

\def\bfL{\mathbf{L}}

\def\bfB{\mathbf{B}}

\def\bfF{\mathbf{F}}
\def\bfG{\mathbf{G}}
\def\bfH{\mathbf{H}}

\def\bfN{\mathbf{N}}
\def\bfO{\mathbf{O}}
\def\bfP{\mathbf{P}}

\def\bfS{\mathbf{S}}
\def\bfT{\mathbf{T}}
\def\bfU{\mathbf{U}}
\def\bfV{\mathbf{V}}
\def\bfW{\mathbf{W}}

\def\bfY{\mathbf{Y}}

\def\bfP{\mathbf{P}}
\def\bfT{\mathbf{T}}
\def\bfV{\mathbf{V}}

\def\geqs{\geqslant}
\def\leqs{\leqslant}

\def\simto{\overset{\sim}\to}

\def\bfSp{{\mathbf{Sp}}}

\def\bfSO{{\mathbf{SO}}}
\def\O{\operatorname{O}\nolimits}
\def\SO{\operatorname{SO}\nolimits}
\def\Sp{\operatorname{Sp}\nolimits}

\def\Irr{{\operatorname{Irr}\nolimits}}

\def\sp{{\operatorname{sp}\nolimits}}

\def\e{{\operatorname{e}\nolimits}}
\def\op{{\operatorname{op}\nolimits}}

\def\GL{\operatorname{GL}\nolimits}

\def\diag{\operatorname{diag}\nolimits}

\def\Mat{{\operatorname{Mat}\nolimits}}

\def\H{\operatorname{H}\nolimits}

\def\P{\operatorname{P}\nolimits}
\def\Q{\operatorname{Q}\nolimits}
\def\R{\operatorname{R}\nolimits}
\def\X{\operatorname{X}\nolimits}

\def\Z{\operatorname{Z}\nolimits}

\def\cl{{\operatorname{cl}\nolimits}}

\def\rk{{\operatorname{rk}\nolimits}}

\def\rank{{\operatorname{rank}\nolimits}}
\def\min{{\operatorname{min}\nolimits}}

\def\Hom{\operatorname{Hom}\nolimits}

\def\End{\operatorname{End}\nolimits}
\def\Res{\operatorname{Res}\nolimits}
\def\res{\operatorname{res}\nolimits}
\def\Ind{\operatorname{Ind}\nolimits}

\def\Id{\operatorname{Id}\nolimits}
\def\id{\operatorname{id}\nolimits}

\def\Mod{\operatorname{-Mod}\nolimits}
\def\umod{\operatorname{-umod}\nolimits}
\def\qumod{\operatorname{-qumod}\nolimits}
\def\mod{\operatorname{-mod}\nolimits}

\def\idun{\mathbf{1}}

\newcommand{\cF}{{\mathcal{F}}}

\newcommand{\cC}{{\mathcal{C}}}

\newcommand{\SL}{{\operatorname{SL}}}

\let\la=\lambda

\renewcommand*{\H}{\mathrm{H}}
\newcommand*{\K}{\mathcal{K}}

\newcommand*{\comp}{\models}

\newcommand*{\tuple}[1]{\boldsymbol{#1}}

\newcounter{BKproof}

\stepcounter{BKproof}

\allowdisplaybreaks

\begin{document}
\title{Categorical actions and derived equivalences for finite odd-dimensional orthogonal groups}

\author{Pengcheng Li, Yanjun Liu and Jiping Zhang}
\address{School of Mathematical Sciences, Peking University, Beijing 100871, China}
\email{pcli17@pku.edu.cn}

\address{School of Mathematics and Statistics, Jiangxi Normal University, Nanchang, China}
\email{liuyanjun@pku.edu.cn}

\address{School of Mathematical Sciences, Peking University,
  Beijing 100871, China}
\email{jzhang@pku.edu.cn}

\thanks{The first and third authors gratefully acknowledge the support by National Key R\&D Program of China (Grant No. 2020YFE0204200) and NSFC (11631001 \& 11871083),
and the second author gratefully acknowledges
the support by an Alexander von Humboldt Fellowship for Experienced Researchers.
Also, the second author deeply thanks the support by
NSFC (12171211) and the Natural Science Foundation of Jiangxi Province (20192ACB21008).}



\begin{abstract} In this paper we prove that
Brou\'{e}'s abelian defect group conjecture holds true for the
finite odd-dimensional
orthogonal groups $\SO_{2n+1}(q)$ at linear primes with $q$ odd.
We first make use of the reduction theorem of Bonnaf\'{e}-Dat-Rouquier
to reduce the problem to isolated blocks.
Then we construct a categorical action of a Kac-Moody algebra
 on the category of quadratic unipotent representations of the  various groups $\SO_{2n+1}(q)$
 in non-defining characteristic, by generalizing
  the corresponding work of Dudas-Varagnolo-Vasserot for unipotent representations.
 This is one of the main ingredients of our work which may be of independent interest.
 To obtain derived equivalences
 of blocks and their Brauer correspondents, we define and investigate
 isolated RoCK blocks.
 Finally, we establish the desired derived equivalence based on
 the work of Chuang-Rouquier  that categorical actions provide
derived equivalences between  certain weight spaces.
\end{abstract}
\maketitle

\pagestyle{myheadings}
\markboth{Categorical actions and derived equivalences}{P. Li, Y. Liu and J. Zhang}


\section*{Introduction}\label{intro}

In the modular representation theory of finite groups,
there are some famous conjectures, one of which is
Brou\'{e}'s abelian defect group conjecture
put forward by Brou\'{e} \cite{broue1990} around 1988.

\begin{conjecture}
[Brou\'e's abelian defect group conjecture \cite{broue1990}] \label{conj:BADGC}
Let $G$ be a finite group, $\ell$ a prime and $(K, \mathcal O, k)$  a
splitting $\ell$-modular system for all subgroups of $G$.
Let $B$ be a block algebra of $\mathcal OG$ or $kG$ with a defect group $P$.
If $P$ is abelian, then $B$ is derived equivalent to its
Brauer correspondent in $N_G(P)$. \qed
\end{conjecture}

In 2008, Chuang and Rouquier \cite{CR} proved that Brou\'e's abelian defect group conjecture is true
in its stronger version of splendid derived equivalences for symmetric groups and for
general linear groups over finite fields in the non-defining characteristic case.
In \cite{DVV}, Dudas, Varagnolo and Vasserot proved that
Brou\'e's abelian defect group conjecture is true for
unipotent blocks of finite unitary groups $U_n(q)$ at linear primes.
Since the reduction theorem of
Bonnaf\'{e}-Dat-Rouquier \cite[Theorem 7.7]{BDR17} reduces Brou\'e's abelian defect group conjecture to unipotent blocks
in (and only in) groups of type $A$, they actually proved that
Brou\'e's abelian defect group conjecture is true for finite unitary groups $U_n(q)$ at linear primes.
In their another paper \cite{DVV2},
Dudas, Varagnolo and Vasserot proved that
Brou\'e's abelian defect group conjecture is true for
unipotent blocks
of finite classical groups $\SO_{2n+1}(q)$ and $\Sp_{2n}(q)$  at odd linear primes
with $q$ odd.

\smallskip

 Throughout this paper,
we always let $\ell$ be a prime and $(K, \mathcal O,k)$ be an $\ell$-modular system
such that both $K$ and $k$ contain all roots of unity.
The main purpose of this paper is to prove the following:
\begin{thmA}\label{thm:Main}
Brou\'{e}'s abelian defect group conjecture is true for
 all $\ell$-blocks of finite
 odd-dimensional orthogonal groups $\SO_{2n+1}(q)$  at linear primes $\ell$ with $q$ odd. \qed
\end{thmA}

The proof of Theorem A  depends on
 the theory of categorical action,
developed by Chuang and Rouquier (see \cite{CR,R08}).
 Let $G_n:=\SO_{2n+1}(q)$, and $R$ be  one of $K, \mathcal O$ or $k$.
Using the tower of inclusion of groups
$\cdots \subset \mathrm{G}_n \subset  \mathrm{G}_{n+1} \subset \cdots,$
one can form the abelian categories
$$\scrU_R:=\bigoplus_{n\in \bbN}RG_n\umod~\mbox{and}~\scrQU_R:=\bigoplus_{n\in \bbN}RG_n\qumod$$
of unipotent representations and quadratic unipotent representations
of the various  groups $G_n$, respectively. 

\smallskip

To prove Theorem A, we construct
 a representation datum $$(E,F,T,X;E',F',T',X';H,H')$$ on $\scrQU_R$ based on
the representation datum $(E,F,T,X)$  of Dudas, Varagnolo and Vasserot \cite{DVV2} on $\scrU_R.$
Here  $(E,F)$ and $(E',F')$
are adjoint pairs, where $F$ and $F'$ are modified from a Harish-Chandra induction
from $G_n$ to $G_{n+1}$ through the  embedding of specific Levi subgroups,
 and $E$ and $E'$ are from their adjoint restrictions.
The
natural transformations $X$ of $F$  and $X'$ of $F'$
have to be defined blockwisely. However, the natural transformations $T\in\End(F^2)$,
$T'\in\End((F')^2)$, $H\in \Hom(FF',F'F)$ and $H'\in \Hom(F'F,FF')$
can be globally constructed.
Interestingly, the functor arising from the spinor norm of $\SO_{2n+1}(q)$
implies the symmetry between the functors $F$ and $F'$ and
between the natural transformations
$X$ and $X'$, $T$ and $T'$, and $H$ and $H'$.
The above representation datum  essentially comes from the
representation datum on $\bigoplus_{n\in \bbN}R\rm{O}_{2n+1}(q)\mod$
whose construction is global and is also valid for groups $\Sp_{2n}(q)$ and $\O_{2n}^{\pm}(q)$.
See \S \ref{cha:repdatum} for more details.

\smallskip

In general, when evaluating at a cuspidal quadratic unipotent module $E_{t_+,t_-}$,
the eigenvalues of $X(E_{t_+,t_-})$  on $F(E_{t_+,t_-})$ and of
$X'(E_{t_+,t_-})$ on $F'(E_{t_+,t_-})$ are determined up to a sign \cite{L77}.
Using the compatibility of the Jordan decomposition of characters and Brauer trees,
we determine all of the signs under suitable modular systems, so that all eigenvalues involved are powers of $-q$ (see Theorem \ref{thm:HL-BC}).
Indeed, it turns out that there are explicit isomorphisms among
the  following algebras (see Theorems \ref{thm:connectwith-HL} and \ref{thm:F^nu}):
\begin{itemize}[leftmargin=8mm]
 \item the endomorphism algebra of the evaluation of a combinatorial sum of $F$s and $F'$s
at $E_{t_+,t_-}$,
\item some cyclotomic quiver Hecke algebra of disconnected quiver,
and
\item a generic algebra introduced by  Howlett and Lehrer to prove
 Comparison Theorem.
 \end{itemize}

 \smallskip
We shall prove that the generalized eigenvalues of $X$ on $F$ and $X'$ on $F'$ both belong to $q^{\bbZ}\cup -q^{\bbZ}$.
Corresponding to the disjoint union of quivers associated to them,
we can form a Kac-Moody algebra $\frakg$ isomorphic to  $\fraks\frakl_{\bbZ}^{\oplus 4}$  (resp. $(\widehat{\fraks\frakl}_{f})^{\oplus 4}$
or $\widehat{\fraks\frakl}_{f}^{\oplus 2}$)
in characteristic zero (resp. in the linear or unitary case),
where $f$ is the order of $q$ modulo $\ell$
(see Theorem \ref{thm:Linearprime} and \ref{thm:unitary}).
\smallskip

\begin{thmB}\label{thm:ginfinityB-intro}Let $R=K$ or $k.$ Let $\mathfrak{A}(\frakg)$ be
the Kac-Moody $2$-category of $\frakg$. Then the tuple $$(E,F,T,X;E',F',T',X';H,H')$$
 endows $\scrQU_R$
with a structure of $\frakA(\frakg)$-categorification,
which is isomorphic to a direct sum of
some minimal categorifications of $\frakA(\frakg).$
\end{thmB}

Let $E=\bigoplus_{i\in I}E_i$ and  $F=\bigoplus_{i\in I}F_i$ (resp. $E'=\bigoplus_{i'\in I'}E'_{i'}$ and $F'=\bigoplus_{i'\in I'}F'_{i'}$) be the decomposition of the functors into generalized
$i$-eigenspaces for $X$ (resp. generalized
$i'$-eigenspaces for $X'$). Then $\{[E_i], [F_i],[E'_{i'}], [F'_{i'}]\}_{i\in I,i'\in I'}$ act as the Chevalley generators of $\frakg$ on
the Grothendieck group $[\scrQU_R]$ of $\scrQU_R$ and many problems on $\scrQU_R$ have a Lie-theoretic
counterpart. By looking at the action of $[E_i],[F_i],[E'_{i'}],[F'_{i'}]$ on
the basis of $[\scrQU_R]$ formed by quadratic unipotent characters and their $\ell$-reduction, we prove that there is a natural $\frakg$-module isomorphism
$$\bigoplus_{t_+,t_-\in\bbN}(\bfF({\tuple\xi_{t_+}})\otimes\bfF({\tuple\xi_{t_-}}))  \simto [\scrQU_R],$$
where $\bfF({\tuple\xi_{t_+}})$ and $\bfF({\tuple\xi_{t_-}})$ are both level 2 Fock spaces
(see Theorem \ref{thm:ginfinityB}, \ref{thm:Linearprime} and \ref{thm:unitary}). Through
this isomorphism, the basis of quadratic unipotent characters (or their $\ell$-reduction) is sent to
the standard monomial basis.
This generalizes the  corresponding work of Dudas-Varagnolo-Vasserot \cite{DVV2} for unipotent representations.

\smallskip

Finally, we define the isolated RoCK blocks which
play an important role in  obtaining
derived equivalences between blocks with their Brauer correspondents
(see \S \ref{subsec:RouBlock}).

\begin{thmC} 
Let $G=\SO_{2n+1}(q)$ with $q$ odd.
Assume that $\ell$ is an odd linear prime with respect to $q$.
Then any isolated RoCK block of $\SO_{2n+1}(q)$ with abelian defect groups is derived equivalent to its Brauer
correspondent in $N_G(P)$.
\end{thmC}

In the linear prime case, by
the work of Chuang and Rouquier \cite{CR}
 that the categorical actions provide derived equivalences between certain weight spaces,
 we can finally get a proof of Theorem A.
 
\smallskip

Here we remark that our method depends on
the work of Fong-Srinivasan \cite{FS82,FS89} about
the block theory of classical groups with connected center.
The center of $\Sp_{2n}(q)$ and $\O_{2n}^{\pm}(q)$ are not connected,
and blocks of $\Sp_{2n}(q)$ and $\O_{2n}^{\pm}(q)$ are not completely parameterized
since the decompositions of their Lusztig inductions 
are not known in general. However, if the condition on $q$
in \cite{S08} can be removed, our method is excepted to  apply to
 blocks of the groups  $\Sp_{2n}(q)$ and $\O_{2n}^{\pm}(q)$.

\smallskip
The paper is organized as follows.
 In Section \ref{cha:2kamoody}, we collect basic concepts and notation.
 It includes quiver Hecke algebra,  Kac-Moody $2$-category, representation datum and
 integrable $2$-representation, etc.
In Section \ref{cha:typeA}, we recall the 2-representation theory of type $A$,
minimal categorical representation, and
the charged Fock space. 
Section  \ref{cha:modular}
consists of
the definition of Brauer homomorphism and correspondent,
some
standard facts in the representation theory of finite groups of Lie type
in non-defining characteristic, and representations of $\SO_{2n+1}(q)$.
Section \ref{cha:repdatum} is about an
 $\frakA(\mathfrak g)$-representation
 datum on $\bigoplus_{n\in \bbN}RG_n\mod$
for groups $G_n=\SO_{2n+1}(q)$, $\Sp_{2n}(q)$ and $\O_{2n}^{\pm}(q)$.
In Section \ref{cha:QU}, we prove Theorem B, by restricting
 the representation datum obtained in
Section \ref{cha:repdatum} to quadratic unipotent (i.e., isolated) blocks of $\SO_{2n+1}(q)$.
In the final section (Section \ref{cha:broue}), we first define isolated RoCK blocks
and then prove Theorems C and  A.

\section{2-Kac-Moody representations}\label{cha:2kamoody}

Here we collect some preliminaries about
quiver Hecke algebras,  Kac-Moody $2$-categories and their representations, mainly referring to \cite{R08}.
Throughout this section, $R$ will denote a noetherian commutative domain (with unit) in general.

\subsection{Algebras and categories\label{sec:rings-cat}}

Given an algebra $A$, we
denote by $A^\op$ the opposite algebra of $A$. Analogously,
by  $\scrC^\op$ we denote
the opposite category of a category $\scrC$.
An \emph{$R$-category} $\scrC$ is an additive category enriched over the tensor
category of $R$-modules. 
As usual, we write $EF$ for a composition of functors $E$ and $F$,
and $\psi\circ\phi$ for a composition of
morphisms of functors (or natural transformations) $\psi$ and $\phi$.
Also, we denote by $1_\scrC$ the identity functor on $\scrC$,
and by  $1_F$ or sometimes by $F$ the identity element in the endomorphism
ring $\End(F)$.
\smallskip

Let $\scrC$ be an abelian $R$-category.
We denote by $[\scrC]$ the complexified Grothendieck group of $\scrC$,
 by $[M]$ the isomorphism class of  an object $M$ of $\scrC$ in $[\scrC]$,
 and by $[F]$ the linear map on $[\scrC]$ induced by an exact endofunctor $F$ of $\scrC$.
\smallskip

If the Hom spaces of $\scrC$ are finitely generated over $R$, then
the category $\scrC$ is called Hom-finite. In this case, we set $\scrH(M)=
\End_\scrC(M)^\text{op}$ for an object $M\in\scrC$,
so that $\scrH(M)$ is an $R$-algebra which is finitely generated as an $R$-module.

\smallskip

Now let $A$ be an $R$-algebra that is free and finitely generated over $R$.
Let $\scrC=A\mod$. We write $\Irr(\scrC)$ or  $\Irr(A)$ for
the set of isomorphism classes of simple objects of $\scrC$.
If there is a ring homomorphism from $R$ to $S$ then
we have $SA=S\otimes_R A$ and the $S$-category $S\scrC=SA\mod$.
Given another $R$-category $\scrC'$ as above and an exact ($R$-linear)
functor
$F:\scrC\to\scrC'$, then $F$ is represented by a projective object $P\in\scrC$, i.e., $F= \Hom_{\scrC}(P,\bullet):\scrC\to \scrC'.$
We set $SF=\Hom_{S\scrC}(SP,\bullet):S\scrC\to S\scrC'$.

\smallskip

Finally, for a finite group $G$,
the group ring of $G$ over $R$ is denoted by $RG$.
 If  $R$ is not a field, an $RG$-module which is free as an $R$-module
 will be called an \emph{$RG$-lattice}.

\smallskip

\subsection{Quiver Hecke algebras }\label{sec:quiverHecke}
In this section, we recall the definition of quiver Hecke algebras.
\subsubsection{Cartan datum and Kac-Moody Lie algebra} \label{subsec:cartandatum}
 Let $I$ be an index set. A matrix
$A=(a_{st})_{s,t \in I}$ is called a {\em symmetrizable generalized
Cartan matrix} if it satisfies
\begin{itemize}[leftmargin=8mm]
\item $a_{st}$ are integers for each $s,t \in I$ and in particular $a_{ss}  = 2$,
\item  $a_{st} \leqs 0$ for all $s \neq t$, where $a_{st} = 0$ if and only if $a_{ts} = 0$, and
\item there are positive integers
$d_s, d_t$ such that $d_s a_{st} = d_t a_{ts}$
for all $s,t \in I$.
\end{itemize}

\smallskip

A \emph{Cartan datum} $(\X_\scrI,\X_\scrI^\vee,\langle\bullet,\bullet\rangle_\scrI,
\Pi, \Pi^\vee)$ associated with a symmetrizable generalized
Cartan matrix $A$ consists of
\begin{itemize}[leftmargin=8mm]
 \item  a free abelian group $\X_\scrI,$
 \item a free abelian group $\X_\scrI^\vee,$
 \item  a set of vectors
$\Pi=\{ \alpha_s\in \X_\scrI \mid \ s \in I\}$ called \emph{simple roots},

 \item a set of vectors $\Pi^{\vee}= \{ \alpha^\vee_s \in\X^\vee_\scrI\ | \ s \in I\}$
called \emph{simple coroots},
 \item  a perfect pairing $\langle \bullet,  \bullet \rangle_\scrI : \X_\scrI^\vee\times \X_\scrI
 \longrightarrow \bbZ$
\end{itemize}
 satisfying the following properties:
\begin{itemize}[leftmargin=8mm]
 \item  $\Pi^\vee$ are linearly independent in $\X_\scrI^\vee$,
 \item for each $s\in \scrI$ there exists a fundamental weight $\Lambda_s\in\X_\scrI$
 satisfying $\langle \alpha_t^\vee,\Lambda_s \rangle_\scrI = \delta_{st}$ for all $t\in\scrI$,
 \item  $\langle \alpha_t^\vee,\alpha_s\rangle_\scrI =~a_{st}$.
\end{itemize}

\smallskip

The \emph{Kac-Moody algebra} $\frakg_\scrI$ corresponding to this datum is the Lie algebra generated by
the Chevalley generators $e_s, f_s$ for $s \in \scrI$ and the Cartan algebra
$\frakh = \bbC\otimes \X_\scrI^{\vee}$.
The Lie algebra $\frakg_\scrI'$ is the derived
subalgebra $[\frakg_\scrI,\frakg_\scrI]$.
We write  $$\Q_\scrI = \bigoplus_{s\in\scrI} \bbZ \alpha_s,\quad
\Q_\scrI^\vee = \bigoplus_{s\in\scrI} \bbZ \alpha_s^\vee,
\quad\text{and}\,\, \P_\scrI = \bigoplus_{s\in\scrI} \bbZ \Lambda_s$$
for the root lattice, the coroot lattice,  and the weight lattice of $\frakg_\scrI$, respectively.
In addition, we set $\Q_{\scrI}^+=\bigoplus_{s\in \scrI}\bbN\alpha_s$.
When there is no risk of confusion, we simply write $X=X_I$, $\frakg=\frakg_\scrI$ and $\P=\P_{\scrI}$, etc.
\smallskip

A $\frakg$-module $V$ is called an \emph{integrable highest weight} $\frakg$-module
if $V$ satisfies
\begin{itemize}[leftmargin=8mm]
  \item $V = \bigoplus_{\omega \in \X} V_\omega$ is the weight space decomposition
  and $\dim V_\omega < \infty$ for each $\omega \in \X$,
  \item the actions of $e_s$ and $f_s$ on $V$ are locally nilpotent for each $s \in \scrI$,
  \item there exists a finite set $B \subset \X$ such that
  $\mathrm{wt}(V) \subset B + \sum_{s \in \scrI} \bbZ_{\leqslant 0} \alpha_s$.
\end{itemize}

Let $\X^+ = \{\omega \in \X\, \mid\, \langle \alpha_s^\vee,\omega\rangle \in \bbN \text{ for all } s \in \scrI\}$ be the set of \emph{integral dominant weights}.
Given $\Lambda \in  \X^+$, there exists a unique irreducible
highest weight module, denoted by $\bfL(\Lambda)$,  with highest weight $\Lambda$.
\begin{remark}\label{rk:sum} The direct sum of two Kac-Moody algebras can be constructed as follows.
Let $\mathcal{C}=(\X_\scrI,\X_\scrI^\vee,\langle\bullet,\bullet\rangle_\scrI,
\Pi, \Pi^\vee)$ and $\mathcal{C}'=(\X_{\scrI'},\X_{\scrI'}^\vee,\langle\bullet,\bullet\rangle_{\scrI'},
\Pi', \Pi'^\vee)$ be two Cartan data associated with symmetrizable generalized
Cartan matrices $A$ and $A'$, respectively.
 We can construct a new Cartan datum $(\X_{\K},\X_{\K}^\vee,\langle\bullet,\bullet\rangle_{\K},
\Pi_{\K}, \Pi_{\K}^\vee)$ associated with
the block-diagonal symmetrizable generalized
Cartan matrix $\begin{pmatrix}A&0\\ 0&A'\end{pmatrix}$,
where
\begin{itemize}[leftmargin=8mm]
 \item  $\K=\scrI\sqcup \scrI'$,
 \item $X_{\K}=\X_\scrI \oplus \X_{\scrI'}$ and
 $X^\vee_{\K}=\X^\vee_\scrI \oplus \X^\vee_{\scrI'}$,
 \item   $\Pi_{\K}=\Pi\times \{0\}\sqcup \{0\}\times\Pi'$
 and $\Pi_{\K}^\vee=\Pi^\vee\times \{0\}\sqcup \{0\}\times\Pi'^\vee$, and
 \item  $\langle v+v', w+w'\rangle_{\K}=\langle v,w\rangle_\scrI+\langle v',w'\rangle_{\scrI'}$
 for all $v\in \X_\scrI^\vee, $ $v'\in \X_{\scrI'}^\vee,$ $w\in \X_\scrI$ and  $w'\in \X_{\scrI'}.$
\end{itemize}
It is called the \emph{direct sum} of $\mathcal{C}$ and $\mathcal{C}'.$
The Kac-Moody Lie algebra corresponding to the direct sum of $\mathcal{C}$ and $\mathcal{C}'$
is 
exactly $\frakg_{I}\oplus\frakg_{I'}.$
\end{remark}

\subsubsection{Quiver Hecke algebra}\label{subsec:QHA}

Here we recall the definition of
quiver Hecke algebras, which are also called the Khovanov-Lauda-Rouquier algebras.
\smallskip

Let $(\X_\scrI,\X_\scrI^\vee,\langle\bullet,\bullet\rangle_\scrI,
\Pi, \Pi^\vee)$ be a Cartan datum associated with $A$ and $\frakg$ be the associated Kac-Moody Lie algebra.
For $n\in \Z_{\ge 0}$ and $\beta \in \Q_{\scrI}^{+}$ such that the height of $\beta$ is  $|\beta|=n$, we set
$$I^{\beta} = \{ \tuple k = (k_1, \dots, k_n) \in I^n |
\alpha_{k_1} + \cdots + \alpha_{k_n} = \beta \}.$$
Naturally, the symmetric group $\mathfrak{S}_{n} = \langle s_1, \ldots, s_{n-1} \rangle$ acts on $I^n$, where $s_a = (a, a+1)$.
We define a matrix $Q=(Q_{st}(u,v))_{s,t\in \scrI}$ with entries in $R[u,v]$ such that
$$Q_{st}(u,v)=\begin{cases}
0 & \text{ if } s=t, \\
\gamma_{st} &\text{ if }s\not=t \text{ and }a_{st}=0,\\
\gamma_{st}u^{-a_{st}}+\sum_{\substack{0\le p<-a_{st}\\ 0\le q<-a_{ts}}}
\beta_{st}^{pq} u^p v^q + \gamma_{ts}v^{-a_{ts}} & \text{ if }s\not=t
\text{ and }a_{st}\not=0.
\end{cases}$$
where $\gamma_{st}\in R^\times$  for $s,t \in I$
with $s \neq t$ such that
$\gamma_{st} = \gamma_{ts}$ if $a_{st}=0$, and  $\beta_{st}^{pq} \in R$ for $s,t \in I$, $0
\leqs p < -a_{st}$ and $0 \leqs q < -a_{ts}$
such that $\beta_{st}^{pq} = \beta_{ts}^{qp}$.
\begin{definition}[\cite{{KL09},{R08}}] \label{def:KLRalg}
The {\em quiver Hecke algebra}  $\H_\beta(Q)$ associated with $Q$ at $\beta$
 is the associative algebra over $R$
generated by $e(\tuple k)$ $({\tuple k}=(k_1,\dots,k_n) \in \scrI^{\beta})$, $x_a$ $(1 \le a \le n)$,
$\tau_b$ $(1 \le b \le n-1)$ satisfying the following defining
relations:
\begin{align*}
& e(\tuple k) e(\tuple k') = \delta_{\tuple k, \tuple k'} e(\tuple k), \ \
\sum_{\tuple k \in I^{\beta}}  e(\tuple k) = 1, \\
& x_{a} x_{b} = x_{b} x_{a}, \ \ x_{a} e(\tuple k) = e(\tuple k) x_{a}, \\
& \tau_{b} e(\tuple k) = e(s_{b}(\tuple k)) \tau_{b}, \ \ \tau_{a} \tau_{b} =
\tau_{b} \tau_{a} \ \ \text{if} \ |a-b|>1, \\
& \tau_{a}^2 e(\tuple k) = Q_{ k_{a},  k_{a+1}} (x_{a}, x_{a+1})
e(\tuple k), \\
& (\tau_{a} x_{b} - x_{s_a(b)} \tau_{a}) e(\tuple k) = \begin{cases}
-e(\tuple k) \ \ & \text{if} \ b=a,  k_{a} =  k_{a+1}, \\
e(\tuple k) \ \ & \text{if} \ b=a+1,  k_{a}=  k_{a+1}, \\
0 \ \ & \text{otherwise},
\end{cases} \\[.5ex]
& (\tau_{a+1} \tau_{a} \tau_{a+1}-\tau_{a} \tau_{a+1} \tau_{a}) e(\tuple k)\\
& =\begin{cases} \dfrac{Q_{ k_{a},  k_{a+1}}(x_{a},
x_{a+1}) - Q_{ k_{a+2},  k_{a+1}}(x_{a+2}, x_{a+1})}
{x_{a} - x_{a+2}}e(\tuple k) \ \ & \text{if} \
 k_{a} =  k_{a+2}, \\
0 \ \ & \text{otherwise}.
\end{cases}
\end{align*}

\end{definition}

We  define $$\bfH_n(Q) \coloneqq \bigoplus_{ |\beta|=n} \bfH_\beta(Q).$$
 Observe that $\bfH_\beta(Q) =\bfH_n(Q)e(\beta)$, 
 where $e(\beta): = \sum_{\tuple k \in I^{\beta}}e(\tuple k)$.
If $\scrI$ is finite, then the direct sum is finite and $\bfH_n(Q)$ is an  algebra
 with unit $\sum_{\tuple k \in I^n} e(\tuple k)$.
 If $n = 0$ we understand $\bfH_\beta(Q) = \bfH_0(Q) = R$.

 \begin{remark} \label{rmk:quiver}
 Assume that $\Gamma$ is a loop-free quiver with vertex set $\scrI$ and edge set $E$.
 For any $i, j \in I$,
   we write $\lvert i \to j\rvert$
 for the (finite) number of $a \in E$ such that the starting vertex $o(a)$ is $i$ and
 the ending vertex $t(a)$ is $j$.
 We define $i.i:=2$ and $i.j \coloneqq -(\lvert i \to j \rvert + \lvert i \leftarrow j\rvert)$
  for $i\neq j$.
  Since $o(a)\not=t(a)$ for any edge $a\in E$, this
  defines a symmetric Cartan matrix $(i.j)_{i,j\in\scrI}$
  and so a (derived) Kac-Moody algebra $\frakg_{\Gamma}$ over $\bbC$.
%

\smallskip

Let $u, v$ be two indeterminates over $R$.
For any $i, j \in I$, we define the polynomial $Q_{ij}(u, v) \in R[u, v]$ by
\begin{equation*}
\label{equation:def_Q}
Q_{ij}(u, v) \coloneqq \begin{cases}
(-1)^{\lvert i \to j \rvert} (u - v)^{-i.j} &\text{if } i \neq j,
\\
0 &\text{otherwise.}
\end{cases}
\end{equation*}
We call $Q:=(Q_{ij}(u, v))_{i,j\in I}$ the matrix associated with the quiver $\Gamma$,
and define
  $$\bfH_\beta(\Gamma) \coloneqq \bfH_\beta(Q).$$

 \end{remark}
\subsubsection{Cyclotomic quiver Hecke algebra}\label{subsec:CQHA}
Let $\Lambda \in \X_{\scrI}^{+}$ be an integral dominant weight.
For 
$s\in \scrI$, we choose a monic polynomial
$$f_s^{\Lambda}(u)=\sum_{k=0}^{\langle \alpha^\vee_{s},\Lambda \rangle}
c_{s;k}u^{\langle \alpha^\vee_{s},\Lambda \rangle-k} $$
of
degree $\langle \alpha^\vee_{s},\Lambda \rangle$ with $c_{s;k}\in
{R}$ and $c_{s;0}=1$.
For $1\le a\le n$, we write
$$f^{\Lambda}(x_a)= \sum_{\tuple k \in I^{n}} f_{k_a}^{\Lambda}(x_a) e(\tuple k)\in \bfH_{n}(Q).$$

\begin{definition}\label{Def:cqh}
The {\em cyclotomic quiver Hecke
algebra $\bfH_{\beta}^{\Lambda}(Q)$ at $\beta$} is defined to be the
quotient algebra
$$\bfH_{\beta}^{\Lambda}(Q) := \frac{\bfH_{\beta}(Q)}{\bfH_{\beta}(Q) f^\Lambda(x_1) \bfH_{\beta}(Q)}.$$
If $\beta=0$ we understand $\bfH_{\beta}^{\Lambda}(Q)=R$.
\end{definition}

Now let $\bfH_{\beta}^{\Lambda}(Q)$ be a cyclotomic quiver Hecke
algebra 
at $\beta$, and assume that $I$ is finite.
For each $n\ge0$, we define
\begin{equation*}
\bfH_n^{\Lambda}(Q) := \frac{\bfH_{n}(Q)}{\bfH_{n}(Q) f^\Lambda(x_1) \bfH_{n}(Q)}.
\end{equation*}
Then $\bfH_n^{\Lambda}(Q) \cong\bigoplus_{|\beta|=n}\bfH_{\beta}^\Lambda(Q)$ and $\bfH_{\beta}^{\Lambda}(Q) =\bfH_n^{\Lambda}(Q)e(\beta)$.
\begin{remark}
When considering the cyclotomic quiver Hecke algebra $\bfH^{\Lambda}_{\beta}(\Gamma)$
associated with a quiver $\Gamma$, we always assume $c_{s;k}=0$ for all $k> 0.$
\end{remark}

Given $\tuple k\in \scrI^n$, $\tuple k'\in \scrI^m$ we write $\tuple k\tuple k'\in \scrI^{n+m}$ for their concatenation.
For each $s\in I$,  the $R$-algebra embedding
$\iota_s: \bfH_{\beta}(Q)\hookrightarrow \bfH_{\beta+\alpha_s}(Q)$
given by
$$e(\tuple k),x_a,\tau_b\mapsto e(\tuple k s),x_a,\tau_b~\mbox{with}~
\tuple k\in I^\beta,1\leqs a\leqs n,1\leqs b\leqs n-1$$
induces an $R$-algebra homomorphism
$\iota_s: \bfH_{\beta}^\Lambda(Q)\to \bfH_{\beta+\alpha_s}^\Lambda(Q)$.

\smallskip

The induction and restriction functors form an adjoint pair $(F^{\Lambda}_s,E^{\Lambda}_s)$ with
$$\begin{array}{ccccl}
 F_{s}^{\Lambda}:&(\bfH_{\beta}^{\Lambda}(Q))\Mod & \longrightarrow & (\bfH_{\beta+ \alpha_s}^{\Lambda}(Q))\Mod & \mbox{and} \\ \vspace{2ex}
 E_{s}^{\Lambda}:&(\bfH_{\beta+\alpha_s}^{\Lambda}(Q))\Mod & \longrightarrow & (\bfH_{\beta}^{\Lambda}(Q))\Mod &
\end{array}
$$
given by
$$
\begin{array}{cccl} \label{eq:E_iLambda}
F_{s}^{\Lambda}(M) & = &  \bfH_{\beta+\alpha_s}^{\Lambda}(Q) e(\beta, s)
\otimes_{\bfH_{\beta}^{\Lambda}(Q)} M & \mbox{and} \\
E_{s}^{\Lambda}(N)& =&  e(\beta, s) N = e(\beta, s)\bfH_{\beta+
\alpha_s}^{\Lambda}(Q) \otimes_{\bfH^{\Lambda}_{\beta+\alpha_s}(Q)} N, \\
 \end{array}
$$
where $M \in\bfH_{\beta}^{\Lambda}(Q)\mod$, $N \in \bfH_{\beta+\alpha_s}^{\Lambda}(Q)\mod$ and
 $$e(\beta, s) = \sum_{\tuple k \in I^{\beta }}
e(\tuple k s)\in \bfH_{\beta+\alpha_s}(Q).$$

\smallskip

\subsection{Kac-Moody $2$-categories and their representations}
\label{sec:2-rep}

In this section, we recall from \cite{R08}
the definitions of Kac-Moody $2$-categories and their representations  including
integrable 2-representations.

\subsubsection{Kac-Moody $2$-category} \label{subsec:kacmoody2cat}
According to the work of Brundan \cite{B16},
the Kac-Moody 2-categories defined by Rouquier and by Khovanov and Lauda are equivalent.
Let $(\X_\scrI,\X_\scrI^\vee,\langle\bullet,\bullet\rangle_\scrI,
\Pi, \Pi^\vee)$ be a Cartan datum and $\frakg$ be the associated Kac-Moody Lie algebra.

\begin{definition}\label{def1} The {\em Kac-Moody $2$-category} $\mathfrak{A}(\mathfrak{g})$
is the strict additive $R$-linear $2$-category with

\begin{itemize}[leftmargin=8mm]
  \item  object set $\X_{\scrI}$,
  \item generating $1$-morphisms $E_s 1_\lambda:\lambda \rightarrow \lambda-\alpha_s~\mbox{and}~
F_s 1_\lambda:\lambda \rightarrow \lambda+\alpha_s$ for each $s \in I$ and $\lambda \in \X_\scrI$, and
  \item generating $2$-morphisms
$x_s:F_s 1_\lambda \rightarrow F_s 1_\lambda$, $\tau_{st}:F_s F_t 1_\lambda \rightarrow F_t F_s 1_\lambda$,
$\eta_s:1_\lambda \rightarrow F_s E_s 1_\lambda$~\mbox{and}~$\epsilon_s:E_s F_s 1_\lambda \rightarrow 1_\lambda$
\end{itemize}
satisfying
\begin{itemize}[leftmargin=8mm]
  \item[$(i)$] the {\em quiver Hecke relations}:

  \begin{enumerate}
\item
\label{en:half1}
$\tau_{st}\circ \tau_{ts}=Q_{st}(F_tx_s,x_tF_s)$,
\item
\label{en:half2}
$\tau_{tu}F_s\circ F_t \tau_{su}\circ \tau_{st}F_u-
F_u \tau_{st}\circ \tau_{su}F_t\circ F_s \tau_{tu}\\=
\begin{cases}
\frac{Q_{st}(x_sF_t,F_sx_t)F_s-F_sQ_{st}(F_tx_s,x_tF_s)}{x_sF_tF_s-F_sF_tx_s}F_s
 & \text{ if } s=u,\vspace{0.2cm}\\
0  & \text{ otherwise,}
\end{cases}$
\item
\label{en:half3}
$\tau_{st}\circ x_s F_t-F_t x_s\circ \tau_{st}=\delta_{st}$,
\item
\label{en:half4}
$\tau_{st}\circ F_sx_t-x_tF_s\circ \tau_{st}=-\delta_{st}$,
\end{enumerate}
\smallskip

\item[$(ii)$] the {\em right adjunction relations}:
\smallskip
\begin{enumerate}
\item[(5)]
$(\epsilon_s{E_s})\circ({E_s}\eta_s)={E_s}$,
\item[(6)] $({F_s}\epsilon_s)\circ(\eta_s{F_s})={F_s}$,
\end{enumerate}
\smallskip

\item[$(iii)$] the {\em inversion relations}, that is,
the following $2$-morphisms are isomorphisms:

\begin{itemize}[leftmargin=8mm]
\item[$\bullet$]  when $\langle\alpha_s^\vee,\lambda\rangle\ge 0$,
$$\rho_{s,\lambda}=
\sigma_{ss}+\sum_{i=0}^{\langle\alpha_s^\vee,\lambda\rangle-1}\epsilon_s\circ
(x_s^i F_s):
E_s F_s\idun_\lambda\to F_s E_s\idun_\lambda \oplus
\idun_\lambda^{\langle\alpha_s^\vee,\lambda\rangle}$$
\item[$\bullet$]  when $\langle\alpha_s^\vee,\lambda\rangle\le 0$,
$$\rho_{s,\lambda}=
\sigma_{ss}+\sum_{i=0}^{-1-\langle\alpha_s^\vee,\lambda\rangle}
(F_s x_s^i)\circ \eta_s:
E_s F_s\idun_\lambda\oplus
 \idun_\lambda^{-\langle\alpha_s^\vee,\lambda\rangle}
\to F_s E_s\idun_\lambda$$
\item[$\bullet$] $\sigma_{st}:E_sF_t\idun_\lambda\to F_tE_s\idun_\lambda$
 for all $s\not=t$ and all $\lambda$,
\end{itemize}
\end{itemize}
where\begin{equation*}\label{sigmast} \sigma_{st}=(F_t E_s\epsilon_t)\circ (F_t \tau_{ts}F_s)\circ
(\eta_t E_sF_t):E_s F_t\to F_t E_s.
\end{equation*}
\end{definition}

\smallskip

\subsubsection{{$\mathfrak{A}(\mathfrak{g})$-categorification}}\label{subsec:g-cat}
Let $\scrC$ be an abelian $R$-category.

\begin{definition} An
$\mathfrak{A}(\mathfrak{g})$-\emph{representation datum} on $\scrC$ is a tuple
$$(\{E_s\}_{s\in I},\{F_s\}_{s\in I},\{x_s\}_{s\in I},\{\tau_{s,t}\}_{s,t\in I})$$ where for each $s\in I$,
$E_s$ and $F_s$ are bi-adjoint functors $\scrC\to\scrC$,
$x_s\in\End(F_s)$ and $\tau_{st}\in\Hom(F_sF_t,F_tF_s)$
are endomorphisms of functors satisfying the quiver Hecke relations  (1)-(4) in Definition \ref{def1}.

\end{definition}

\begin{remark}
If the tuple $\big(\{E_s\}_{s\in I},\{F_s\}_{s\in I},\{x_s\}_{s\in I},\{\tau_{s,t}\}_{s,t\in I}\big)$ is an $\mathfrak{A}(\mathfrak{g})$-representation datum,  then for each $n\in\bbN$, the map
\begin{equation*}
\begin{aligned}
\phi_n\, :\bfH_{n}(Q)&\rightarrow (\bigoplus_{\tuple k,\tuple k'\in I^n} \Hom(F_{k_n}\dots
F_{k_1},F_{k'_n}\dots F_{k'_1}))^\op\\
e(\tuple k)&\mapsto {F_{k_n}\cdots F_{k_1}} \\
x_{a,\tuple k}:={x_ae(\tuple k)}&\mapsto F_{k_n}\cdots F_{k_{a+1}}x_{k_a}F_{k_{a-1}}\cdots F_{k_1}\\
\tau_{a,\tuple k}:={e(s_a(\tuple k))\tau_ae(\tuple k)}&\mapsto F_{k_n}\cdots F_{k_{a+2}}\tau_{k_{a+1},k_a}
F_{k_{a-1}}\cdots F_{k_1}
\end{aligned}
\end{equation*}
with ${\tuple k}=(k_1,\dots,k_n)$ and ${\tuple k}'=(k_1',\dots,k_n')$
is a well-defined $R$-algebra homomorphism.
\end{remark}

\begin{definition}A \emph{representation of $\mathfrak{A}(\mathfrak{g})$} on $\scrC$
is defined to be a strict $2$-functor from $\mathfrak{A}(\mathfrak{g})$ to the
strict $2$-category of the $R$-linear category $\scrC$.
This is equivalent to the data of
\begin{itemize}[leftmargin=8mm]
\item an $R$-linear category $\scrC=\bigoplus\limits_{\lambda\in \X_{\scrI}}\scrC_\lambda$,
\item an $\mathfrak{A}(\mathfrak{g})$-representation datum $\big(\{E_s\}_{s\in I},\{F_s\}_{s\in I},\{x_s\}_{s\in I},\{\tau_{s,t}\}_{s,t\in I}\big)$
\end{itemize}
such that the maps $\rho_{s,\lambda}$ and $\sigma_{st}$ ($s\not=t$) 
are isomorphisms.

\smallskip

An {\em integrable $2$-representation} of $\mathfrak{A}(\mathfrak{g})$
is a representation of $\mathfrak{A}(\mathfrak{g})$ such that
$E_s$ and $F_s$ are locally nilpotent for all $s\in I$.
\end{definition}
\begin{theorem}[\cite{R08}]\label{def:2kacmoodyrep}
 Assume given an $\mathfrak{A}(\mathfrak{g})$-representation datum
$$(\{E_s\}_{s\in I},\{F_s\}_{s\in I},\{x_s\}_{s\in I},\{\tau_{s,t}\}_{s,t\in I})$$ on $\scrC$ and a decomposition $\scrC=\bigoplus_{\omega\in\X_{\scrI}}\scrC_\omega$ such that
\begin{itemize}[leftmargin=8mm]
\item[$(a)$] the actions of $[E_s]$ and $[F_s]$ for $s\in \scrI$ endow $[\scrC]$ with
the structure of an integrable $\frakg$-module such that $[\scrC]_\omega=[\scrC_\omega]$, and
\item[$(b)$] $E_s(\scrC_\omega)\subset\scrC_{\omega+\alpha_s}$ and
$F_s(\scrC_\omega)\subset\scrC_{\omega-\alpha_s}$.
\end{itemize}
Then the datum above defines an (integrable) $\mathfrak{A}(\mathfrak{g})$-representation on $\scrC.$\qed
\end{theorem}

The tuple $$(\{E_s\}_{s\in I},\{F_s\}_{s\in I},\{x_s\}_{s\in I},\{\tau_{s,t}\}_{s,t\in I})$$
and the decomposition $\scrC=\bigoplus_{\omega\in \X}
\scrC_{\omega}$ is called an \emph{$\mathfrak{A}(\mathfrak{g})$-categorification} of the integrable $\frakg$-module
$[\scrC]$.

\begin{definition}\label{equ}\rm
Let $\scrC$ and $\scrC'$ be $\mathfrak{A}(\frakg)$-categorifications with the tuples
$$(\{E_s\}_{s\in I},\{F_s\}_{s\in I},\{x_s\}_{s\in I},\{\tau_{s,t}\}_{s,t\in I})
~\mbox{and}~(\{E'_s\}_{s\in I},\{F'_s\}_{s\in I},\{x'_s\}_{s\in I},\{\tau'_{s,t}\}_{s,t\in I}),$$ respectively.
A functor $\frakE:\scrC \rightarrow \scrC'$
is called
{\em strongly equivariant}
if there exist
isomorphisms of functors
$\Phi_s:
F_{s} \circ \frakE
\stackrel{\sim}{\rightarrow}
\frakE\circ F_{s}$ for all $s\in\scrI$
such that
\begin{itemize}[leftmargin=8mm]
\item[$(a)$]
the natural transformation
$E'_s \frakE \epsilon \circ E'_s \Phi_s E_s \circ \eta_s'
\frakE E_s:\frakE\circ E_s \rightarrow E'_s \circ \frakE $
is an isomorphism
\item[$(b)$]
$\Phi_s \circ x'_s \frakE  = \frakE x_s \circ \Phi_s\in \Hom(F'_s \circ \frakE,\frakE\circ F_s)$, and
\item[$(c)$]
$\Phi_t F_s \circ F'_t \Phi_s
\circ \tau_{st}' \frakE=\frakE \tau_{st} \circ \Phi_s F_t \circ F'_s \Phi_t\in \Hom(F'_sF'_t\circ \frakE, \frakE \circ F_tF_s)$.
\end{itemize}
If $\frakE$ is an equivalence of categories, then the axiom (a) holds
automatically. In this case, we call the functor $\frakE$ an {\em isomorphism} of $\mathfrak{A}(\frakg)$-representations.
\end{definition}

\subsubsection{Minimal categorical representations}
\label{subsec:mini}

Let $\Lambda \in \X_{\scrI}^{+}$ be a dominant integral weight, and
$\lambda=\Lambda-\beta$ for a given $\beta\in \Q_{\scrI}^+$.
We define $$\scrL(\Lambda):=\bigoplus_{\beta\in Q^+}\bfH_{\beta}^\Lambda(Q)\mod,
  \scrL(\Lambda)_{\lambda}:=\bfH_{\beta}^\Lambda(Q)\mod,$$
and the following data on
$\scrL(\Lambda)$ :
\begin{itemize}[leftmargin=8mm]
\item $E_s1_\lambda=E^{\Lambda}_s1_{\beta}$,
\item $F_s1_\lambda=F^{\Lambda}_s1_{\beta}$,
\item $x_s1_\lambda\in\Hom(F^{\Lambda}_s1_\beta,F^{\Lambda}_s1_\beta)$
is represented by the right multiplication by
$x_{n+1}$ on $\bfH_{\beta+\alpha_s}^\Lambda(Q)e(\beta,s),$
\item $\tau_{st}1_\lambda\in\Hom(F^{\Lambda}_sF^{\Lambda}_t1_\beta,F^{\Lambda}_tF^{\Lambda}_s1_\beta)$
is represented by the right multiplication by
$\tau_{n+1}$ on $\bfH_{\beta+\alpha_s+\alpha_t }^\Lambda(Q)e(\beta,ts)$ where $e(\beta,ts)=\sum_{\tuple k \in I^{\beta }}
e(\tuple k ts)$.
\end{itemize}

\begin{theorem}[\cite{KK12}, \cite{K12}]\label{thm:minimalcat}
  The endofunctors $E^{\Lambda}_s$ and $F^{\Lambda}_s$ of $\scrL(\Lambda)$ are bi-adjoint for all $s\in \scrI.$ The tuple $(\{E^{\Lambda}_s\}_{s\in I},\{F^{\Lambda}_s\}_{s\in I},\{x_s\}_{s\in I},\{\tau_{s,t}\}_{s,t\in I})$ and
   the decomposition $\scrL(\Lambda)=
  \bigoplus_{\omega\in \X} \scrL(\Lambda)_{\,\omega}$ form a minimal $\frakA(\frakg)$-categorical representation
  of $\bfL(\Lambda)$.\qed
\end{theorem}

\smallskip

\subsection{Derived equivalences}\label{sec:derivedequ}
Let $V$ be an integrable $\frakg$-module. For $i \in \scrI$, the simple reflection
$$s_i = \exp(-f_i) \exp(e_i) \exp(-f_i)$$ has an action on $V$. For each weight $\omega\in \X$,
it maps a weight space $V_\omega$ to $V_{s_i(\omega)}$ with
$s_i(\omega) = \omega - \langle \alpha_i^\vee,\omega \rangle \alpha_i$.
If $\scrC$ is a categorification of $V$, then it restricts to an $\fraks\frakl_2$-categorification in the sense
of Chuang-Rouquier. In particular, the simple objects are weight vectors for the categorical $\fraks\frakl_2$-action.
Thus, the theory of Chuang-Rouquier can be applied and
 \cite[Theorem. 6.6]{CR} implies that $s_i$
can be lifted to a derived equivalence $\Theta_i$ of $\scrC$.  (See also \cite[Section 1.7]{DVV2}.)

\begin{theorem}\label{thm:reflection}
Assume that $R$ is a field.
Let $$(\{E_s\}_{s\in I},\{F_s\}_{s\in I},\{x_s\}_{s\in I},\{\tau_{s,t}\}_{s,t\in I})$$
be an $\mathfrak{A}(\mathfrak{g})$-categorification on an abelian $R$-category $\scrC$, and
$i \in \scrI$. Then there exists a derived self-equivalence $\Theta_i$ of
$\scrC$ which restricts to derived equivalences
 $$\Theta_i \, : \, D^b(\scrC_\omega) \mathop{\longrightarrow}\limits^\sim
 D^b(\scrC_{s_i(\omega)})$$
for all weights $\omega \in \X$. Furthermore, $[\Theta_i] = s_i$ as a linear map
of $[\scrC]$.
\qed
\end{theorem}

\smallskip

\section{Categorical Representation theory of type A}\label{cha:typeA}

\smallskip
In this section, we focus on the categorical representations of affine Lie algebras of type A, following \cite{DVV2} and keeping their notation.
\subsection{Affine Lie algebras associated with a quiver}\label{sec:typeA}
Let $1\neq q \in R^\times$ and  $\scrI$ be a (possibly infinite) subset of $R^\times$.
To the pair $(\scrI,q)$ we  associate a (not necessarily connected)
quiver $\scrI(q)$ (also denoted by $\scrI$) as follows:
\begin{itemize}[leftmargin=8mm]
  \item  the vertex set is $\scrI$, and
  \item there is an arrow $i\to i\cdot q$
if and only if $i, i\cdot q\in \scrI$.
\end{itemize}
In this paper,  we will always assume that $(q^{\bbZ}\scrI(q))/q^{\bbZ}$ is finite.
Assume that $\scrI$ is stable under the multiplication by $q$ and $q^{-1}$.
If $q$ is a primitive $e$-th root of unity
then the quiver $\scrI(q)$ is cyclic of type $A_{e-1}^{(1)}$,
 while if  $q$ is not a root of unity then  $\scrI(q)$
is the disjoint union of quivers of type $A_\infty$.

\smallskip

The quiver $\scrI(q)$ defines a symmetric generalized Cartan matrix
$A = (a_{ij})_{i,j\in \scrI}$ with

$$
\begin{cases}
a_{ii}= 2 &   \\
a_{ij} =-1 & \text{if}~i \rightarrow j~\text{or}~j \rightarrow i\\
a_{ij}=0 & \text{ otherwise}.
\end{cases}
$$ To this Cartan matrix one can
associate the (derived) Kac-Moody algebra $\fraks\frakl_{\scrI}'=\widetilde{\fraks\frakl}_{\scrI}$ over $\bbC$, which
has Chevalley generators $e_i,f_i$ for $i\in \scrI$, subject to the usual
relations.
For $i\in\scrI$, let $\alpha_i,$ $\alpha^\vee_i$ be the simple root and coroot
corresponding to $e_i$ and let $\Lambda_i$ be the $i$-th fundamental weight. Recall that
 $\Q^\vee = \bigoplus\limits_{i\in \scrI} \bbZ \alpha_i^\vee$ and
$\P=\bigoplus\limits_{i\in \scrI}\bbZ \Lambda_i.$

When $\scrI = q^{\bbZ}$, there are two cases to consider.

If $q$ has finite order $e$, then $\scrI$ is isomorphic
  to a cyclic quiver of type $A_{e-1}^{(1)}$. We can form $\X^\vee = \Q^\vee
  \oplus \bbZ \partial$ and $\X = \P \oplus \bbZ \delta$ with
  $$ \langle \alpha^\vee_j,\Lambda_i\rangle = \delta_{ij},\quad \langle\partial,\Lambda_i \rangle = \langle \alpha^\vee_i,  \delta\rangle = 0, \quad\langle \partial,\delta \rangle = 1.$$
   The pairing is non-degenerate, and
  $\frakg_\scrI$ is isomorphic to the Kac-Moody algebra
  $$\widehat{\fraks\frakl}_e = \fraks\frakl_e(\bbC) \otimes \bbC[t,t^{-1}] \oplus
  \bbC c \oplus \bbC \partial.$$
  An explicit isomorphism sends $e_{q^i}$ (resp. $f_{q^i}$) to the matrix
  $E_{i,i+1} \otimes 1$ (resp. $E_{i+1,i} \otimes 1$) if $i \neq e$ and $e_1$
  (resp. $f_1$) to $E_{e,1}\otimes t$ (resp. $E_{1,e}\otimes t^{-1}$).
  Via this isomophism the central element $c$ corresponds to $\sum_{i \in \scrI}
  \alpha_i^\vee$,
  and the derived algebra $\frakg_\scrI$ corresponds
  to
  $$\widetilde{\fraks\frakl}_e=\fraks\frakl_e(\bbC) \otimes \bbC[t,t^{-1}] \oplus \bbC c.$$
    Let $\frakh = \bbC\otimes \X^\vee$ be the Cartan subalgebra of $\widehat{\fraks\frakl}_e,$ and let $\frakh^* =  \bbC\otimes\X$ be its dual.
The space $\frakh^*$ is equipped with a bilinear symmetric form defined by
$$ \textstyle{(\Lambda_{q^k} | \Lambda_{q^l}) = \min(k,l) - \frac{kl}{e},\quad (\Lambda_{q^k} | \delta) =1,\quad (\delta | \delta) =0} $$ for all $0\leqs k,l\leqs e-1.$
For $\Lambda \in \frakh^*$ we shall write $|\Lambda|^2$ for $(\Lambda | \Lambda).$
Then we have that $\alpha_i = 2\Lambda_i - \Lambda_{qi} - \Lambda_{q^{-1}i} + \delta_{i,1} \delta,$ where $\delta_{i,1} =1$ if $i=1$ and $\delta_{i,1} =0$ if otherwise.

 If $\scrI$ is infinite, then $\frakg_\scrI$ is isomorphic to
  $\mathfrak{sl}_\bbZ$, the Lie algebra of traceless matrices with finitely many
  non-zero entries. It will be sometimes useful to consider a completion of
$\frakg_\scrI$ denoted by $\overline{\frakg}_\scrI$, which has $\prod \bbC \alpha_i^\vee \simeq \bbC^\scrI$ as a Cartan subalgebra. This allows
to consider some infinite sums of the generators, such as $c = \sum \alpha_i^\vee$
which is a central element in $\overline{\frakg}_\scrI$. This will not affect
the representation theory of $\frakg_\scrI$ as we will  only  work with integrable
representations.

\smallskip

Let $R'$ be another commutative domain with unit, and $\theta : R \longrightarrow R'$ be
a ring homomorphism. Then there is a Lie algebra homorphism $\overline{\frakg}_{\theta(\scrI)}'
\longrightarrow \overline{\frakg}_{\scrI}'$ defined on the Chevalley generators by
$$e_i   \longmapsto \displaystyle\sum_{\theta(j) = i} e_j \qquad \text{and} \qquad
	 f_i \longmapsto \displaystyle\sum_{\theta(j) = i} f_j.$$

\subsection{$\fraks\frakl_{I}$-categorification} \label{sec:datumA}

\subsubsection{Representation datum of type A}\label{subsec:datumA}
\begin{definition} \label{def:repdataA} Let $\scrC$ be an abelian $R$-category and $q\in R^{\times}$.
A \emph{representation datum $($of type $A$$)$} on $\scrC$ with parameter $q\neq 1$ is a tuple $(E,F,X,T)$ where
$E$, $F$ are bi-adjoint functors $\scrC\to\scrC$,   $X\in\End(F)^\times$ and $T\in\End(F^2)$
are endomorphisms of functors satisfying the \emph{affine Hecke relations}:
\begin{itemize}[leftmargin=8mm]
  \item[(a)] $1_FT\circ T1_F\circ 1_FT=T1_F\circ 1_FT\circ T1_F$,

  \item[(b)] $(T+1_{F^2})\circ(T-q1_{F^2})=0$, and

  \item[(c)] $T\circ(1_FX)\circ T=qX1_F$.
\end{itemize}
\end{definition}

Definition \ref{def:repdataA} can be formulated in terms of actions of affine Hecke algebras.
In fact, for  a pair $(E,F)$ of bi-adjoint functors, $X\in\End(F)$ and  $T\in\End(F^2)$,
the tuple $(E,F,X,T)$ is a representation datum if and only if
for each $m\in\bbN$, the map
\begin{equation}\label{equ:heckealgebrahom}
\begin{array}{rcl} \phi_{F^m}\, :\, \bfH_{R,m}^q & \longrightarrow & \End(F^m) ^{\op}\\[4pt]
  X_k & \longmapsto &  {F^{m-k}} X{F^{k-1}}\\
  T_l & \longmapsto & {F^{m-l-1}}T{F^{l-1}}\\
  \end{array}
\end{equation}
is a well-defined $R$-algebra homomorphism.
Here $\bfH_{R,m}^q$ is the \emph{affine Hecke algebra} of type $A_{m-1}$ over $R$,
generated by $T_1,\ldots,T_{m-1}$, $X^{\pm 1}_1,\ldots,X^{\pm 1}_m$
subject to the following relations:
$$
\begin{array}{ll}
 (T_i+1)(T_i-q)=0, &   X_iX_j=X_jX_i,  \\
 T_iT_{i+1}T_i=T_{i+1}T_iT_{i+1},  & X_iX_i^{-1}=X_i^{-1}X_i=1, \\
 T_iT_j=T_jT_i \ \ \textrm{if}\ \ |i-j|>1,   &T_iX_{i}T_i=qX_{i+1}, \\
      X_iT_j=T_jX_i\ \ \textrm{if}\ \ i-j\neq 0,1.
\end{array}
$$
As is conventional, we understand $\bfH_{R,0}^q=R$.

\smallskip

\smallskip

Now let $R$ be a field, $\scrC$ be Hom-finite, and $I=I(q)$ be the quiver as defined in \S \ref{sec:typeA}.

\begin{definition}[\cite{R08}]\label{def:rep--datum}
An \emph{$\fraks\frakl_I$-representation} on $\scrC$ consists of a representation datum
$(E,F,X,T)$ on $\scrC$ and a decomposition $\scrC=\bigoplus_{\omega\in\X}\scrC_\omega$
satisfying
\begin{itemize}[leftmargin=8mm]

\item[(a)] $F=\bigoplus_{i\in \scrI} F_i$ and $E=\bigoplus_{i\in \scrI} E_i$ where
$F_i$ and $E_i$ are the generalized $i$-eigenspaces of $X$ respectively acting on $F$ and on $E$  for $i\in \scrI,$

\item[(b)] the actions of $[E_i]$ and $[F_i]$ for $i\in \scrI$ endow $[\scrC]$ with
a structure of integrable $\fraks\frakl_{I}$-module such that $[\scrC]_\omega=[\scrC_\omega]$,

\item[(c)] $E_i(\scrC_\omega)\subset\scrC_{\omega+\alpha_i}$ and
$F_i(\scrC_\omega)\subset\scrC_{\omega-\alpha_i}$.
\end{itemize}
\end{definition}

The tuple $(E, F, X, T)$ and the decomposition $\scrC=\bigoplus_{\omega\in \X}
\scrC_{\omega}$ is called an \emph{$\fraks\frakl_{\scrI}$-categorification} of the integrable $\fraks\frakl_{\scrI}$-module
$[\scrC]$.

\subsubsection{Brundan--Kleshchev--Rouquier equivalence}\label{subsec:BKR}
Now we write $\Gamma$ for the quiver $I(q)$ with vertex set $\scrI.$
Let $\bfH_m(\Gamma)\Mod_0$ be the full subcategory of the category of representations
of $\bfH_m(\Gamma)$ consisting of objects on which $x_{a,\tuple k}$ is locally
nilpotent for all $\tuple k$ and $1 \leqs a \leqs m$.
Let $\cC_\Gamma$ be the category of modules over
the  affine Hecke algebra $\bfH_{R,m}^q$
which are direct sums of their generalized eigenspaces for each $X_i$, with eigenvalues in $\scrI$.
For $M \in \cC_\Gamma$ and $\tuple k=(k_1,\ldots,k_m) \in \scrI^m$,
we denote by $M_{\tuple k}$ the generalized $\tuple k $-eigenspace:
\[
M_{\tuple k } := \{ x \in M \mid (X_a - k_a)^Nx = 0
\text{ for all $1 \le a \le m$ and } N \gg 0 \}.
\]

The following theorem is due to Brundan-Kleshchev~\cite{BK3} and Rouquier~\cite[Theorem~3.16]{R08}.

\begin{theorem}\cite{BK3,R08}
\label{thm:BKR}
There exists an equivalence of categories
\[
\cC_\Gamma \simto
\bfH_m(\Gamma)\Mod_0
\]
which associates to $M \in \cC_\Gamma$ the representation $V$ defined by
\begin{enumerate}
\item $V_{\tuple k} = M_{\tuple k }$  for all $\tuple k  \in \scrI^m$;
\item
$x_{a,\tuple k} := (k^{-1}_aX_a - 1)e(\tuple k)$ for all $\tuple k  \in \scrI^m$ and $1 \le a \le m$;
\item
\label{it:BKR-functor}
$\tau_{a,\tuple k}$ given by the formulas
\[
  \tau_{a,\tuple k} := \begin{cases}
  (k_a(qX_a-X_{a+1})^{-1}(T_a-q))e(\tuple k) &
    \text{if $k_a = k_{a+1}$,} \\
(q^{-1}k^{-1}_a((X_a-X_{a+1})T_a + (q-1)X_{a+1}))e(\tuple k) & \text{if $k_{a+1} = k_a +1$,}\\
(\frac{X_a-X_{a+1}}{qX_a-X_{a+1}}(T_a-q)+1)e(\tuple k) & \text{otherwise}
\end{cases}
\]
for all $\tuple k \in \scrI^m$ and $1 \leqs a < m$.
\end{enumerate}
\end{theorem}

\smallskip
Let $(E,F,X,T)$  be an $\fraks\frakl_I$-representation datum on $\scrC.$
Assume there are decompositions
$E=\bigoplus_{i\in I} E_i$ and
$F=\bigoplus_{i\in I} F_i$,
where $X-i$ is locally nilpotent on $E_i$ and $F_i$.
Let $I(q)$ be the quiver defined as before.
We put
\begin{equation*}\label{tran}
\begin{aligned}
&x_i=i^{-1}X-1~\text{(acting on $F_i$) and}\\
&\tau_{ij}= \begin{cases}
i(qF_iX-XF_j)^{-1}(T-q) & \text{ if }i=j,\\
q^{-1}i^{-1}(F_iX-XF_j)T+i^{-1}(1-q^{-1})XF_j & \text{ if }i=qj,\\
\frac{F_iX-XF_j}{qF_iX-XF_j}(T-q)+1 & \text{ otherwise}
\end{cases}
\end{aligned}
\end{equation*}
(restricted to $F_iF_j$).
Then the tuple
$(\{E_i\}_{i\in I},\{F_i\}_{i\in I},\{x_i\}_{i\in I},\{\tau_{i,j}\}_{i,j\in I})$
defines  an  $\mathfrak{A}(\mathfrak{sl}_{I})$-representation datum on $\scrC$.
\begin{theorem}[\cite{R08}]
\label{th:slcat2rep}
Given an $\fraks\frakl_{I}$-categorification on $\scrC$, the construction as above
gives rise to an integrable action of $\mathfrak{A}(\mathfrak{sl}_{I})$  on $\scrC$.

Conversely, an integrable action of $\mathfrak{A}(\mathfrak{sl}_{I})$
 on $\scrC$ gives rise to
an $\fraks\frakl_{I}$-categorification on $\scrC$.
\end{theorem}
\subsection{Partitions }\label{sec:partitions}

\subsubsection{Compositions, partitions and $l$-partitions}\label{subec:partitions}
A \emph{composition} $\lambda\comp_{l} n$
of an integer $n$ is an ordered tuple $\lambda=(\lambda_1,\dots,\lambda_l)$ of positive integers such
that  $\sum_{i=1}^l\lambda_i=n$.

\smallskip

A \emph{partition} of $n$ is a non-increasing sequence of non-negative integers
$\lambda = (\lambda_1 \geqslant \lambda_2 \geqslant \cdots)$ with $\sum_i \lambda_i=n$,
to which one
associates the so-called \emph{Young diagram}
$Y(\lambda)=\{(x,y)\in\bbZ_{>0}\times\bbZ_{>0}\,\mid\,y\leqslant \lambda_x\}.$
We write $\scrP=\bigsqcup_n\scrP_n$ for the set
of all partitions, where $\scrP_n$ is the set of partitions of $n$.
For $\lambda\in \scrP$, we denote by $|\lambda|$ the \emph{weight}
of $\lambda$ and by $\lambda^*$ the partition  conjugate to $\lambda$.

\smallskip

An \emph{$l$-partition} of $n$ is an $l$-tuple $\tuple\lambda=(\lambda^{1},\ldots,\lambda^{l})$ of partitions whose weights add up to $n$,
and its Young diagram is the set $Y(\tuple\lambda)=\bigsqcup_{p=1}^lY(\lambda^p)\times\{p\}.$
The integer $|\tuple\lambda|=\sum_p|\lambda^{p}|$ is called the weight of the $l$-partition.
We write $\scrP^l=\bigsqcup_n\scrP^l_n$ for the set of all $l$-partitions, where
$\scrP^l_n$ is the set of $l$-partitions of $n$.

\subsubsection{Residues and contents}\label{subsec:contents}
We fix $\tuple \xi = (\xi_1,\ldots,\xi_l) \in \scrI^l$, and assume that $\scrI$ is stable under the multiplication by
$q$ and $q^{-1}$.
Let $A\in Y(\tuple\lambda)$ be the box which lies in
the $i$-th row and $j$-th column of the diagram of
$\lambda^p$.
The \emph{$(\tuple \xi,q)$-shifted residue} of the node $A$ is the element of $\scrI$ given by
$q\text{-}\!\res^{\tuple \xi}(A)=q^{j-i}\xi_{p}.$  If $\tuple\lambda$, $\tuple\mu$ are $l$-partitions such
that $|\tuple\mu|=|\tuple\lambda|+1$ we write $q\text{-}\!\res^{\tuple\xi}(\tuple\mu-\tuple\lambda)=i$ if $Y(\tuple\mu)$ is obtained by
adding a node of $(\tuple\xi,q)$-shifted residue $i$ to $Y(\tuple \lambda)$.  We call the node a removable $i$-node of $\tuple \mu$ and an addable node of $\tuple \lambda$.
A \emph{charge} of the tuple $\tuple \xi= (\xi_1,\ldots,\xi_l)$ is an $l$-tuple of integers
$\tuple s = (s_1,\ldots,s_l)$ such that $\xi_p = q^{s_p}$ for all $p = 1,\ldots,l$.
Conversely, given $\scrI \subset R^\times$ and $q \in R^\times$, any
$l$-tuple of integers $\tuple s \in \bbZ^l$ defines a tuple $\tuple \xi = (q^{s_1},\ldots,q^{s_l})$ with
charge $\tuple s$.
The \emph{$\tuple s$-shifted content}
of the box $A=(i,j,p)$ is the integer $\text{cont}^{\tuple s}(A)=s_p+j-i$. Similarly, if $\tuple\lambda$, $\tuple\mu$ are $l$-partitions such
that $|\tuple\mu|=|\tuple\lambda|+1$ we write $\text{cont}^{\tuple s}(\tuple\mu-\tuple\lambda)=i$ if $Y(\tuple\mu)$ is obtained by
adding a node of $\tuple s$-shifted content  $i$ to $Y(\lambda)$.
It is related to the residue of $A$ by the formula $q\text{-}\!\res^{\tuple \xi}(A)=q^{\text{cont}^{\tuple s}(A)}$.
We will also write $p(A)=p$.
\subsubsection{$l$-cores and $l$-quotients}\label{subsec:l-core}
We start with the case $l=1$.

Let  $(\lambda,s)$ be a charged partition.
Then it is uniquely determined by
the set $\beta_s(\lambda)=\{\lambda_u+s+1-u\,\mid
\,u\geqslant 1\}$ of the  so-called \emph{$\beta$-numbers}.
For a positive integer $e$,  an \emph{$e$-hook} of $(\lambda,s)$ is a pair $(x,x+e)$ such
that $x+e\in \beta_s(\lambda)$ and $x\not\in \beta_s(\lambda)$, and
an \emph{$e$-core} of $(\lambda,s)$ which does not depend on $s$
is the charged partition obtained by recursively removing $e$-hooks
$(x,x+e)$
(i.e., replacing $x+e$ with $x$ in the set of $\beta$-numbers).

\smallskip

Next, we construct a bijection $\tau_l:\scrP\times\bbZ\to\scrP^l\times\bbZ^l$.
It takes the pair $(\lambda,s)$ to $(\tuple\lambda,\tuple s),$ where $\tuple\lambda=(\lambda^1,\dots,\lambda^l)$ is an
$l$-partition and $\tuple s=(s_1,\dots,s_l)$ is a $l$-tuple in
$\bbZ^l(s)=\{\tuple s\in\bbZ^l\,\mid\,s_1+\cdots+s_l=s\}.$
The bijection is uniquely determined by the relation
$\beta_s(\lambda)=\bigsqcup_{p=1}^l\big(p-l+l\beta_{s_p}(\lambda^p)\big).$

\smallskip

The bijection $\tau_l$ takes the pair $(\lambda,0)$ to $(\tuple\lambda^{[l]},\tuple\lambda_{[l]}),$
where $\tuple\lambda^{[l]}$ is the \emph{$l$-quotient} of $\lambda$ and $\tuple\lambda_{[l]}$ lies
in $\bbZ^l(0)$. Since $\lambda$ is an $l$-core if and only if $\tuple\lambda^{[l]}=\tuple \emptyset$,
this bijection identifies the set of $l$-cores and $\bbZ^l(0)$. We define the \emph{$l$-weight} $w_l(\lambda) := |\tuple\lambda^{[l]}|$ of the partition $\lambda$ to be the weight of its $l$-quotient.

\subsection{Minimal categorical representations and Fock spaces}
\label{subsec:minimalcat}
We now assume that $R$ is a field
and that $\scrI$ is finite. In particular, $v \in R^\times$ will be  a root of unity.

\smallskip

Let $m \geqslant 0$, $q \in R^\times$ and $\tuple \xi=(\xi_1,\ldots,\xi_l)\in (R^\times)^l$
be a fixed tuple.
The
\emph{cyclotomic Hecke algebra} $\bfH^{q;\,\tuple\xi}_{R,m}$ is the quotient of $\bfH_{R,m}^q$
by the two-sided ideal generated by $\prod_{i=1}^l(X_1-\xi_{i})$, where $\bfH_{R,m}^q$ is
 the affine Hecke algebra as defined in \S \ref{subsec:datumA}.
It was known that $\bfH^{q;\,\tuple\xi}_{R,m}$ is
isomorphic to a cyclotomic quiver Hecke algebra.
The latter can be constructed as follows:

\smallskip

Any finite dimensional $\bfH^{q;\,\tuple \xi}_{R,m}$-module $M$
is the direct sum of the weight subspaces
  $$M_{\tuple k}=\{v\in M\,\mid\,(X_r-k_r)^dv=0,\
  r\in[1,m],\, d\gg 0\},\quad\tuple k=(k_1,\dots,k_m)\in R^m.$$
Let  $\{e_{\tuple k}\,;\,\tuple k
\in R^m\}$  be  a system of orthogonal idempotents of $\bfH^{q;\,\tuple \xi}_{R,m}$
such that $e_{\tuple k} M=M_{\tuple k}$ for each $M$.
  The eigenvalues of $X_r$ are always of the form $\xi_i q^j$ for some $i \in \{1,\ldots,l\}$
and $j \in \bbZ$.
Write $\scrI = \bigcup \xi_i q^\bbZ$. Let $\Gamma$ be the quiver $\scrI(q)$ with vertex set $I$  as defined in
\S \ref{sec:typeA} and we can consider the corresponding Kac-Moody algebra $\frakg_\scrI$ and
its root lattice $\Q_\scrI$.
By Brundan--Kleshchev--Rouquier isomorphism \cite{BK3,R08},
we have $\bfH^{q;\,\tuple \xi}_{R,m}\cong\bfH_m^{\Lambda_{\tuple \xi}}(\Gamma)$
 where  $\Lambda_{\tuple \xi}=\sum_{i=1}^l\Lambda_{\xi_{i}}.$

 \smallskip

 Given $\beta\in \Q_\scrI^+$ with $|\beta|=m$, let
$e_{\beta}=\sum_{\tuple k\in \scrI^{\beta}} e_{\tuple k}$. The nonzero $e_\beta$'s are the primitive central
idempotents in $\bfH^{q;\,\tuple \xi}_{R,m}$.
Under the above isomorphism, the idempotent $e_{\tuple k}$ (resp. $e_{\beta}$) of $\bfH^{q;\,\tuple \xi}_{R,m}$ is mapped to the idempotent $e(\tuple k)$ (resp. $e(\beta)$) of $\bfH_m^{\Lambda_{\tuple \xi}}(\Gamma).$
So we get $e_\beta\bfH^{q;\,\tuple \xi}_{R,m}
\cong e(\beta)\bfH_m^{\Lambda_{\tuple \xi}}(\Gamma)=
\bfH_{\beta}^{\Lambda_{\tuple\xi}}(\Gamma).$
Now we have the following isomorphic abelian categories:
  $$\scrL(\Lambda_{\tuple \xi}) = \bigoplus_{m\in\bbN}\bfH^{q;\, \tuple \xi}_{R,m}\mod \cong \bigoplus_{\beta\in Q^+}\bfH_{\beta}^{\Lambda_{\tuple\xi}}(\Gamma)\mod$$
  and
  $$\scrL(\Lambda_{\tuple \xi})_{\Lambda_{\tuple \xi}-\beta}=e_{\beta}\bfH^{q;\,\tuple \xi}_{R,m}\mod\cong\bfH_{\beta}^{\Lambda_{\tuple\xi}}(\Gamma)\mod.$$
By \S \ref{subsec:mini}, it is a minimal $\frakA(\frakg)$-categorical representation
  of $\bfL(\Lambda_{\tuple\xi})$.

 \smallskip

Between $\bfH^{q;\,\tuple \xi}_{R,n}\mod$ and $\bfH^{q;\,\tuple \xi}_{R,m}\mod$,
there is a pair $(\Ind^{n}_m,\,\Res^{n}_m)$ of exact adjoint functors
from the embedding $\bfH^{q;\,\tuple\xi}_{R,m}\hookrightarrow \bfH^{q;\,\tuple\xi}_{R,n}$
that is
induced by the $R$-algebra embedding of the affine Hecke algebras $\bfH_{R,m}^v
\hookrightarrow \bfH_{R,n}^v$ given by $T_i\mapsto T_i$ and $X_j\mapsto X_j$,
where $m<n$.
They induce endofunctors $E$ and $F$ of $\scrL(\Lambda_{\tuple \xi})$ by
$E=\bigoplus_{m\in\bbN}
\Res_m^{m+1}~\mbox{and}~F=\bigoplus_{m\in\bbN}\Ind_m^{m+1}.$ The right multiplication on
$\bfH^{q;\,\tuple\xi}_{R,m+1}$ by $X_{m+1}$ (resp. $T_{m+1}$)
yields an endomorphism of the functor $\Ind_m^{m+1}$ (resp. $\Ind_m^{m+2}$).
If we set $X\in\End(F)$ and $T\in\End(F^2)$ by
$X=\bigoplus_{m}X_{m+1}~\mbox{and}~T=\bigoplus_{m}T_{m+1},$
then the tuple $(E,F,X,T)$ is an $\fraks\frakl_{I}$-representation datum on $\scrL(\Lambda_{\tuple \xi})$.

 \smallskip

However, we have
$E=\bigoplus_{i\in I} E^{\Lambda_{\tuple \xi}}_i$ and
$F=\bigoplus_{i\in I} F^{\Lambda_{\tuple \xi}}_i$ (see \S\ref{subsec:mini} for the definition),
where $X-i$ is locally nilpotent on $E^{\Lambda_{\tuple \xi}}_i$ and $F^{\Lambda_{\tuple \xi}}_i$.
Indeed, by Theorem \ref{thm:BKR}, $e(\tuple k)X_{n+1}= {k_{n+1}}e(\tuple k)(x_{\tuple k} + 1)$ and $x_{\tuple k}$ is a nilpotent element.

So we have $e(n,i)X_{n+1}=i(1+u)e(n,i),$
 where $u$ is a nilpotent element and $e(n,i) = \sum_{\tuple k \in I^{n}}e(\tuple k i)\in \bfH_{n+1}(Q).$
 Hence $X_{n+1}-i$ is locally nilpotent on
 $E^{\Lambda_{\tuple \xi}}_i$, 
  as stated. By the transformation in \eqref{tran}, the tuple
$(\{E^{\Lambda_{\tuple \xi}}_i\}_{i\in I},\{F^{\Lambda_{\tuple \xi}}_i\}_{i\in I},\{x_i\}_{i\in I},\{\tau_{i,j}\}_{i,j\in I})$ defines  an  $\mathfrak{A}(\mathfrak{sl}_{I})$-representation datum on $\scrL(\Lambda_{\tuple \xi}) $.

 \smallskip

Let $S(\tuple\lambda)^{q,\,\tuple \xi}_R$ be
the $\bfH^{q;\,\tuple \xi}_{R,m}$-module  defined as in \cite[\S 2.4.3]{RSVV} or \cite[\S 5.3]{GJ}.

When $R$ is a field of characteristic 0 which contains a primitive $l$-th root of unity,
  we have $$F^{\Lambda_{\tuple \xi}}_iS(\tuple\lambda)^{q,\,\tuple \xi}_R=\bigoplus\limits_{q\text{-}\!\res^{\tuple\xi}(\tuple\mu-\tuple\lambda)=i}S(\tuple\mu)^{q,\,\tuple\xi}_R,$$
$$E^{\Lambda_{\tuple \xi}}_iS(\tuple\mu)^{q,\,\tuple\xi}_R=\bigoplus\limits_{q\text{-}\!\res^{\tuple\xi}(\tuple\mu-\tuple\lambda)=i}S(\tuple\lambda)^{q,\,\tuple \xi}_R$$
(see \cite[\S 5.6]{GJ}).
\smallskip

Recall that a \emph{Fock space} $\bfF(\tuple\xi)$
 is the $\bbC$-vector
space with basis
$\{|\tuple\lambda,\tuple\xi\rangle_I\,|\,\tuple\lambda\in\scrP^l\}$  and the action of $e_i, f_i$ for all $i \in \scrI$ given by
$$
  f_i(|\tuple\lambda,\tuple \xi\rangle_{\scrI})=\sum\limits_{q\text{-}\!\res^{\tuple\xi}(\tuple\mu-\tuple\lambda)=i}|\tuple\mu,\tuple\xi\rangle_{\scrI},$$$$
  e_i(|\tuple\mu,\tuple\xi\rangle_{\scrI})=\sum\limits_{q\text{-}\!\res^{\tuple\xi}(\tuple\mu-\tuple\lambda)=i}|\tuple\lambda,\tuple\xi\rangle_{\scrI}.
$$

For brevity, we shall omit the subscript $\scrI$ if there is no confusion. The basis
$\{|\tuple\lambda,\tuple\xi\rangle\,|\,\tuple\lambda\in\scrP^l\}$ is
called  the
\emph{standard monomial basis}, where every element is a weight vector.

Hence the operators $[E^{\Lambda_{\tuple \xi}}_i], [ F^{\Lambda_{\tuple \xi}}_i]$ endow $[\scrL(\Lambda_{\tuple \xi})]$ with a structure of $\fraks\frakl_{\scrI}'$-module.
And the composition $[\scrL(\Lambda_{\tuple \xi})] \simto \bfL(\Lambda_{\tuple \xi}) \to \bfF(\tuple \xi)$ obtained as above  sends the class
of $S(\tuple\lambda)^{q,\,\tuple \xi}_R$ to the standard monomial $|\tuple\lambda,\tuple\xi\rangle.$
\smallskip

If $\scrI= A_\infty$ then $\bfF(\tuple \xi)=\bfL(\Lambda_{\tuple \xi})$.
In general,
the $\frakg'$-submodule of $\bfF(\tuple\xi)$ generated by $|\tuple\emptyset,\tuple\xi\rangle$ is
 isomorphic to $\bfL(\Lambda_{\tuple \xi})$.

For each $p = 1,\ldots, l$, let $\scrI_p$ be the subquiver of $\scrI$ corresponding
to the subset $q^\bbZ \xi_p$ of $\scrI$. We define a relation on $\{1,\ldots,l\}$ by
$i \sim j \iff \scrI_i = \scrI_j$.
Denote by $\Omega = \{1,\ldots,l\}/\sim$ the set of
equivalence classes with respect to this relation,
and by $\tuple \xi_p$ for $p \in \Omega$  the tuple
of $(\xi_{i_1},\ldots,\xi_{i_r})$ where $(i_1,\ldots,i_r)$ is the ordered set of elements
in $\scrI_p$.
The decomposition $\scrI = \bigsqcup_{p \in \Omega} \scrI_p$ yields a canonical
decomposition of Lie algebras
 $\frakg'_\scrI  = \bigoplus_{p \in \Omega} \frakg'_{\scrI_p}$.
The corresponding decomposition of the Fock space is given in the following proposition.

\begin{proposition}\label{prop:tensorfock}
The map $|\tuple\lambda,\tuple\xi\rangle_\scrI \longmapsto \otimes_{p \in \Omega} |\tuple\lambda^p,
\tuple\xi_p\rangle_{\scrI_p}$ yields an isomorphism of $\frakg'_I$-modules
$ \bfF(\tuple \xi)_\scrI \ \mathop{\longrightarrow}\limits^\sim\ \bigotimes_{p\in \Omega}\bfF({\tuple \xi_p})_{\scrI_p}.$
 \qed
\end{proposition}

\subsection{The $\frakg$-action on the Fock space}\label{sec:chargfock}

A \emph{charged Fock space} is a pair $\bfF(\tuple s)=(\bfF(\tuple\xi),\tuple s)$ such that $\tuple s\in\bbZ^l$ is a charge
of $\tuple \xi$, that is  $\tuple \xi = (q^{s_1},\ldots,q^{s_l})$.

\smallskip

 Throughout this section, we will always
assume that $\scrI$ is either of type $A_\infty$ or a cyclic quiver. For more general quivers we
can invoke Proposition \ref{prop:tensorfock} to reduce to that case.
\smallskip

The action of $\fraks\frakl_{\scrI}'$ on $\bfF(\tuple \xi)$ can be extended to an action of $\fraks\frakl_{\scrI}$
when $\tuple \xi$ admits a charge $\tuple s$. We describe this action in the case where $q$ has finite
order $e$. In that case $\scrI = q^{\bbZ}$ is isomorphic to the cyclic quiver  $A^{(1)}_{e-1}$ (see \S \ref{sec:typeA}).
For an $l$-partition $\tuple \lambda$ and $i\in\scrI$ we define
\begin{align*}
&N_i(\tuple\lambda |\tuple s,e) = \sharp\{ \text{ addable $i$-nodes of $\tuple\lambda$}\}
 - \sharp\{ \text{ removable $i$-nodes of $\tuple\lambda$} \}, \\
&M_i(\tuple\lambda |\tuple s,e) = \sharp\{ \text{  $i$-nodes of $\tuple\lambda$}\},
\end{align*}
and for $\tuple s =(s_1,\dots,s_l) \in \bbZ^l$ we set
$$
\Delta(\tuple s|e) =  \frac{1}{2} \sum_{b=1}^l|\Lambda_{q^{s_b}}|^2 +  \frac{1}{2} \sum_{b=1}^l ( \frac{s_b^2}{e} - s_b).
$$
Now we can state the following theorem.
\begin{theorem}
The following formulas define on $\bfF(\tuple \xi)$ the structure of an integrable $\fraks\frakl_{\scrI}$-module:
\begin{align*}
&f_i |\tuple \lambda,\tuple s\rangle  = \sum_{q\text{-}\!\res^{\tuple\xi}(\tuple \mu-\tuple \lambda) = i}  |\tuple \mu,\tuple s\rangle, \\
&e_i |\tuple \mu,\tuple s\rangle  = \sum_{q\text{-}\!\res^{\tuple\xi}(\tuple \mu-\tuple \lambda) = i}|\tuple \lambda,\tuple s\rangle,  \\
&\alpha^\vee_i |\tuple \lambda,\tuple s\rangle  = {N_i(\tuple \lambda |\tuple s,e)} |\tuple \lambda,\tuple s\rangle,  \\
& \partial |\tuple \lambda,\tuple s\rangle  =  -(\Delta(\tuple s|e) + M_{1}(\tuple \lambda |\tuple s,e)) |\tuple \lambda,\tuple s\rangle .
\end{align*}\qed
\end{theorem}

\smallskip

For this action the weight of a standard basis element is
\begin{equation*}\label{eq:weightlevel1}
  \mathrm{wt}(|\tuple\lambda,\tuple s\rangle)= \Lambda_{q^{s_1}}+\cdots+\Lambda_{q^{s_l}} -\Delta(\tuple s,e)\,\delta  -\sum_{i\in\scrI} M_i(\tuple \lambda|\tuple s,e) \alpha_i.
\end{equation*}

\smallskip

We now describe the action of the affine Weyl group of $\frakg$ on the weights of $\bfF(\tuple s)$.
For $i \in \scrI \smallsetminus\{1\}$, we denote by $\alpha_i^\cl = 2 \Lambda_i -
\Lambda_{q i}-\Lambda_{q^{-1}i}$ and $\Lambda_i^\cl=\Lambda_i-\Lambda_1$ the $i$-th simple root and fundamental weight of
$\fraks \frakl_{e}$. These (classical) simple roots span the lattice of the classical roots $\Q^\cl$.
The affine Weyl group of $\frakg$ is $W =
\mathfrak{S}_\scrI \ltimes \Q^\cl$. It is generated by  $\sigma_i$ ($0\leqs i\leqs e-1$)
which act linearly on $\frakh^*$ by
$\sigma_i(\Lambda)=\Lambda-(\alpha_i,\Lambda)\alpha_i.$
 See \cite{U} for more details.

\section{Representations of finite groups} \label{cha:modular}
This section refers to representations of finite groups,
and consists of three parts. The first part recalls
the definition of Brauer homomorphism and correspondent,
the second part collects some
standard facts in the representation theory of finite groups of Lie type
in non-defining characteristic, and the third part
lists representations of $\SO_{2n+1}(q)$,
 especially quadratic unipotent characters and blocks of $\SO_{2n+1}(q)$.

\subsection{Brauer homomorphism and correspondent}\label{sub:BrauerHandCor}
Recall that we have assumed that $\ell$ is a prime and $(K, \mathcal O,k)$ is an $\ell$-modular system
such that both $K$ and $k$ contain all roots of unity.
When $G$ is a finite reductive group,
 we always assume that $\ell$ is different from the defining characteristic.
Let $R$ be any commutative domain (with 1) and $U$ be a finite group.
If $|U|$ is invertible in $R$, then $$e_U:=|U|^{-1}\sum_{u\in U}u$$
is a well-defined idempotent in $RU$.

\subsubsection{Brauer homomorphism}\label{subsec:brauerhom}
We briefly describe the Brauer homomorphism here. Let $R=\mathcal O$ or $k.$
Let $G$ be a finite group and $M$ an $RG$-module.
For $H \leqs G$, let $M^H$ be the set of $H$-invariant elements of $M$, i.e.,
$M^H=\{x \in M\mid  xh= x \text { for all } h \in H\}$,
so that the trace map $\operatorname{Tr}_{H}^{G}:
M^H \rightarrow M^G$
is given by
$\operatorname{Tr}_{H}^{G}(x)=\sum_{t \in H \backslash G} x t$ for $x \in M^H$.

\smallskip

Let $P$ be an $\ell$-subgroup of $G$.
The \emph{Brauer quotient} of $M$ with respect to $P$,
 is defined to be
$$M(P)=M^P/(\sum_{Q<P}{\rm Tr}_Q^P(M^Q)+\mathcal{J}M^P)$$
where $\mathcal{J}$ is the unique maximal ideal of $R.$
Note that $M(P)$ is  a $kN_G(P)$-module, and that
if $M$ is a summand of a permutation module then
 $M(P) \neq 0$ if and only if there is a direct summand of $M$
with a vertex containing $P$.

\smallskip

The \emph{Brauer homomorphism} of $M$ with respect to $P$ is the
natural surjection $$Br_P^G:M^P\rightarrow M(P).$$
In the case where $M = RG$ with the $G$-conjugate action,
we have $RG(P) = kC_G(P)$, and the
Brauer homomorphism of $RG$ with respect to $P$ is indeed the
natural algebraic homomorphism
$$Br_P^G:(RG)^P \rightarrow kC_G(P)$$ given by
$$Br_P^G(\sum_{g\in G}\alpha_gg)=\sum_{g\in C_G(P)}\overline{\alpha_g}g.$$

\subsubsection{Brauer correspondent}\label{brauercorr}

Let $b$ be a  block idempotent of $RG$ with defect group $P$, and
$H$ be a subgroup of $G$ containing $N_G(P)$.
If $c$ is a block idempotent of $RH$ with defect group $P$
and $Br_P^G(b)=Br_P^H(c)$, then $RHc$ is called
the \emph{Brauer correspondent} of $RGb$ in $H$.
In the case that $H=N_G(P)$, the block $RHc$ will be simply called the
Brauer correspondent of $RGb$.

\subsection{Finite reductive groups}\label{subsec:reductivegroup}
In this section we introduce some notation about
algebraic groups and finite reductive groups. Also, we
introduce two constructions of finite classical groups.

\subsubsection{Algebraic groups} \label{subsec:reductivegroup}
Let $\bfG$ be a (possibly disconnected) reductive algebraic group  over $\overline\bbF_q$
and $F$ be a Frobenius map of $\bfG$. Then the group of fixed points
$G=\bfG^F$ of $\bfG$ under $F$
is called a finite reductive group (or a finite group of Lie type).

\smallskip

Always, if $\bfG,\mathbf{H},\ldots$ are algebraic groups, we
denote by $G,H,\ldots$ their corresponding finite groups.

\smallskip

Assume that $\bfG$ is a connected reductive algebraic group.
Let $\bfB$ be an $F$-stable Borel subgroup of $\bfG$,
$\bfT$ be an $F$-stable maximal torus of $\bfG$ with $\bfT \subset\bfB$,
and $\bfN$ be the normalizer of $\bfT$ in $\bfG$.
The groups $\bfB$, $\bfN$ form a reductive $BN$-pair of $\bfG$ with Weyl group
$\bfW=\bfW_\bfG=\bfN/\bfT$.
Since both $\bfB$ and $\bfN$ are $F$-stable and $\bfG$ is connected,
the finite groups $B$, $N$ form a split BN-pair of $G$ whose Weyl group
is $W=\bfW^F=N/T$.

\smallskip

In the case where $\bfG$ is not connected,
a torus of $\bfG$ is torus of $\bfG^{\circ}$, where $\bfG^{\circ}$ is
the connected component of $\bfG$ containing the identity.
Moreover, a Borel subgroup of $\bfG$ is a maximal connected solvable subgroup of $\bfG$,
and a parabolic subgroup of $\bfG$ to be a subgroup $\bfP$
of $\bfG$ such that $\bfG/\bfP$ is complete.
Also, the unipotent radical $\bfV$ of a parabolic
subgroup $\bfP$ is the unique maximal connected unipotent normal subgroup of $\bfP$,
and a Levi complement to $\bfV$ in $\bfP$ is a subgroup $\bfL$ of $\bfP$ such that $\bfP = \bfV\rtimes\bfL.$
Note that a closed subgroup $\bfP$ of $\bfG$ is a parabolic subgroup of $\bfG$ if and only if $\bfP^{\circ}$ is
a parabolic subgroup of $\bfG^{\circ}.$
We have $\bfP^{\circ} = \bfP\cap\bfG^{\circ}$, that
 $\bfV$ is also the unipotent radical of $\bfP^{\circ}$,
 and that a Levi complement to $\bfV$ in
$\bfP$ is a subgroup of the form $N_{\bfP}(\bfL_{\circ})$, where $\bfL_{\circ}$ is a Levi complement of $\bfV$ in
$\bfP^{\circ}$ (so that $\bfL^{\circ} = \bfL_{\circ}$ and $\bfP^{\circ} = \bfV\rtimes\bfL_{\circ}$).

\smallskip

\subsubsection{Construction of classical groups} \label{sec:construction}\hfill\\
For $n\geqs 1$,
let $J_n$  be the $n\times n$ matrix with entry $1$ in $(i,n-i+1)$ for $1\leqs i\leqs n$
and zero elsewhere, and denote $\widetilde J_{2n}=\left(\begin{array}{cc}
  0  & J_n \\
-J_n & 0   \\
\end{array}\right)$ and
$J'_{2n}=\left(\begin{array}{cccc}
 0 & 0 & 0 & J_{n-1} \\
 0 & 1 & 0 & 0 \\
 0 & 0 & -\delta & 0 \\
 J_{n-1} & 0 & 0 & 0 \\
\end{array}\right)$, where $\delta$ is a non-square in $\mathbb{F}_q$.

\smallskip

(1) The odd-dimensional full orthogonal group $\bfO_{2n+1}$
is
  $$\bfO_{2n+1} = \O_{2n+1}(\overline\bbF_q)= \{g \in \bfG\bfL_{2n+1} \, \mid
  \, {}^tg  J_{2n+1} g =   J_{2n+1}\},$$
whose connected component
containing the identity is the odd-dimensional orthogonal group
  $\bfSO_{2n+1} = \mathrm{SO}_{2n+1}(\overline\bbF_q) =\bfO_{2n+1}\cap \bfS\bfL_{2n+1}(\overline\bbF_q).$
  The corresponding finite groups of fixed points under $F$
  are  the \emph{finite full orthogonal group} $\O_{2n+1}(q)$ and
   the \emph{finite orthogonal group} $\SO_{2n+1}(q)$, respectively,
where $F$ is the Frobenius endomorphism of $\bfO_{2n+1}$ induced  by
  the standard Frobenius map $F_q$ on $\bfG\bfL_{2n+1}$
  raising all entries of a matrix to their $q$th powers.

\vspace{1ex}

(2)
The symplectic group $\bfSp_{2n}$ is
  $$ \bfSp_{2n} = \Sp_{2n}(\overline\bbF_q)= \{g \in \bfG\bfL_{2n} \, \mid \, {}^tg \widetilde J_{2n} g =  \widetilde J_{2n}\},$$
  and the \emph{finite symplectic group} $\Sp_{2n}(q)$ is  the finite group $(\bfSp_{2n})^F$ of fixed points
of $\bfSp_{2n}$ under the Frobenius endomorphism $F$ of $\bfSp_{2n}$
induced by the standard Frobenius map $F_q$ on $\bfG\bfL_{2n}$.

\vspace{1ex}

(3) The even-dimensional full orthogonal group $\bfO_{2n}$ is
  $$\bfO_{2n}= \mathrm{O}_{2n}(\overline\bbF_q) = \{g \in \bfG\bfL_{2n} \, \mid
  \, {}^tg  J_{2n} g =   J_{2n}\}.$$

\begin{itemize}[leftmargin=8mm]
  \item The \emph{finite full orthogonal group}  $\O^{+}_{2n}(q) $ is the group $(\bfO_{2n})^F$ of fixed points
of $\bfO_{2n}$ under the Frobenius endomorphism $F$ of $\bfO_{2n}$
   induced by the standard Frobenius map $F_q$ on $\bfG\bfL_{2n}$.
  \item The \emph{finite non-split full orthogonal group}  $\O^{-}_{2n}(q)$ is
 the group $(\bfO_{2n})^{F'}$ of fixed points  of $\bfO_{2n}$ under $F'$
  where $F':=F_q\circ \sigma$ and  $\sigma(g)=sgs^{-1}$
for $g\in \bfG\bfL_{2n}$ with  $s\in \bfG\bfL_{2n}$ the permutation matrix
interchanging the $n$th and $(n+1)$th rows.
\end{itemize}

\smallskip

For later use, here we recall another way of constructing the above groups.
Let $\mathbf{V}$ be an $m$-dimensional vector space over $\overline{\mathbb{F}}_q$ with basis $\{e_1,\dots,e_m\}$,
 and $V$ be the $\mathbb{F}_q$-vector space given by the $\mathbb{F}_q$-span of
$\{e_1,\dots,e_m\}$. Suppose that $\tilde{\beta}$ is a
 bilinear form on $\mathbf{V}$, and that $\beta$
 is a bilinear form on $V$ (which has the same expression as $\tilde{\beta}$).  We write
 \begin{align*}
I(\mathbf{V},\tilde{\beta})&=\{x\in \mathbf{GL}(\mathbf{V})\mid \tilde{\beta}(xu,xv)=\tilde{\beta}(u,v)
~\text{for~all}~u,v\in \mathbf{V}\},  \\
I(V,\beta)&=\{x\in {\rm GL}(V)\mid \beta(xu,xv)=\beta(u,v)
~\text{for~all}~u,v\in V\},
\end{align*}
$I_0(\mathbf{V},\tilde{\beta})=I(\mathbf{V},\tilde{\beta})\cap \mathbf{SL}(\mathbf{V})$ and
$I_0(V,\beta)=I(V,\beta)\cap {\rm SL}(V)$.
Then we have Table \ref{tb:classical}, where
$\bfO_{2n}^-=\mathrm{O}^-_{2n}(\overline\bbF_q) = \{g \in \bfG\bfL_{2n} \, \mid
  \, {}^tg  J'_{2n} g =   J'_{2n}\}$, $\bfO_{2n}^+=\bfO_{2n}\cong \bfO_{2n}^-$ and
 $\bfSO^{\pm}_{2n}=\bfO^{\pm}_{2n}\cap \SL_{2n}(\overline\bbF_q)$
 (see \cite[Section 1]{FS89} and \cite[Section 4]{hiskes2000}).

\begin{center}
\begin{table} \caption{Classical groups} \label{tb:classical} 
\begin{tabular}{cccccc}
  \hline
 $m$ &  $\tilde{\beta}$~\text{and}~$\beta$ & $I(\mathbf{V},\tilde{\beta})$ &  $I(V,\beta)$ & $I_0(\mathbf{V},\tilde{\beta})$  & $I_0(V,\beta)$\\  \hline
 $2n+1$ & $u_1v_{2n+1}+\dots+u_{2n+1}v_1$  & $\bfO_{2n+1}$ & $\O_{2n+1}(q)$  & $\bfSO_{2n+1}$  & $\SO_{2n+1}(q)$ \\
 $2n$   & $(u_1v_{2n}+\dots+u_nv_{n+1})-$   &   $\bfSp_{2n}$ & $\Sp_{2n}(q)$  & $\bfSp_{2n}$ & $\Sp_{2n}(q)$  \\
         & $(u_{n+1}v_n+\dots+u_{2n}v_1)$ &  & &&  \\
  $2n$   & $u_1v_{2n}+\dots+u_{2n}v_1$  &   $\bfO^+_{2n}$ & $\O^+_{2n}(q)$  &  $\bfSO^+_{2n}$ &  $\SO^+_{2n}(q)$  \\
   $2n$   &  $u_1v_{2n}+\dots+u_nv_n+$ &  $\bfO^-_{2n}$ & $\O^-_{2n}(q)$ & $\bfSO^-_{2n}$ & $\SO^-_{2n}(q)$\\
    & $\delta u_{n+1}v_{n+1}+\dots+u_{2n}v_1$ &  & &&  \\
  \hline
\end{tabular}
\end{table}
\end{center}

\subsection{Representations of finite reductive groups}\label{sub:repReductive}
Here we make a sketchy introduction of the representation theory of finite reductive groups
in non-defining characteristic, and refer to \cite{CE,DM} for more details.

\subsubsection{Lusztig induction and restriction} \label{subsec:lusztigind}

Assume that $\bfG$ is a reductive algebraic group.
Let $\bfL$ be an $F$-stable Levi subgroup of $\bfG$ contained in a parabolic subgroup
$\bfP$.
As usual, we write $R^{\bfG}_{\bfL\subset\bfP}$ and ${}^*\!R^{\bfG}_{\bfL\subset\bfP}$
for Lusztig induction and restriction, respectively.
We shall mainly handle the case of finite odd-dimensional orthogonal groups,
where the Mackey formula holds. Hence
Lusztig induction and restriction do not depend on the
choice of the parabolic subgroup $\bfP$, and we shall abbreviate
$R^{\bfG}_{\bfL}=R^{\bfG}_{\bfL\subset\bfP}$
and ${}^*\!R^{\bfG}_{\bfL}={}^*\!R^{\bfG}_{\bfL\subset\bfP}$
 (see \cite{BM11} for more details).

\smallskip

\subsubsection{Harish-Chandra series}\label{subsec:HCseries}

We now assume that the parabolic subgroup $\bfP\subseteq\bfG$ in
\S\ref{subsec:lusztigind} is $F$-stable. Then
the group $L$ is $G$-conjugate to a standard Levi subgroup of $G$,
and the Lusztig induction and restriction are just
the usual Harish-Chandra induction  $R_L^{G}$  and restriction ${}^*\!R^{G}_L$.
If we assume that $q$ is  invertible in $R,$
the functors $R_L^{G}$ and ${}^*\!R^{G}_L$ are represented by the
$(RG,RL)$-bimodule $RG\cdot e_U$ and the $(RL,RG)$-bimodule $e_U\cdot RG$,
 respectively, where $U=\bfU^F$  and
$\bfU\subset \bfP$ is the unipotent radical of $\bfP$.
Namely,
for all $M\in RL\mod$ and $N\in RG\mod$, we have
  $$R_L^{G}(M)=RG\cdot e_U\otimes_{RL}M \quad \text{and} \quad
  {}^*\!R_L^{G}(N)=e_U\cdot RG\otimes_{RG}N.$$

\smallskip
Now let $R=K$ or $k$. Let
 $\chi$ be a character of $RG$, $b$
a central primitive idempotent of $RG$ and
$c$  a central primitive idempotent of $RL$.
We denote by $\chi_b$
the character of $RG$ given by $x\mapsto \chi(xb)$
(see \cite[p. 320 (3.3)]{Nagao}), and by $R_{L,c}^{G,b}(\psi)$
the character $(R_L^G(\psi_c))_b$ of $RG$, so we have
the functor
 \begin{gather*}
R_{L,c}^{G,b}: {\rm Char}(RLc)\rightarrow {\rm Char}(RGb).
\end{gather*}

\smallskip

Let $(L,E)$  be a cuspidal pair.
The \emph{Harish-Chandra series} of $(L,E)$  is  the set $\Irr(R G,E)\subseteq\Irr(R G)$
 consisting of the constituents  of the head of $R^G_L(E)$.
 There is a
bijection
$$\Irr(R G,E)
\mathop{\longleftrightarrow}\limits^{1:1} \Irr(\scrH(RG,E))$$
given by
the functor $\Hom_{RG}(R^{G}_L(E),-)$,
where the $R$-algebra $$\scrH(RG,E)=\End_{RG}(R^{G}_L(E))^\op$$ is
the \emph{ramified Hecke algebra} of $(L,E)$.
 It is well known that Harish-Chandra series form a partition of $\Irr(RG)$.

\subsubsection{Jordan decomposition}\label{subsec:jordan}

Let $(\bfG^*,F^*)$ be in duality with $(\bfG,F)$, and let $G^*=\bfG^{*F^*}.$
Deligne and Lusztig decomposed $\Irr(G)$ into rational Lusztig series
$$\mathrm{Irr}(G)=\coprod {\mathcal E}(G,(s))$$
where $(s)$ runs over the set of $G^*$-conjugacy classes of semisimple elements of
$G^*$.
The {\em unipotent characters} of $G$ are those in $\mathcal{E}(G,(1))$.
 By Lusztig \cite{Lu84},
 there exists a bijection $\mathcal{L}_s:\mathcal{E}(G,(s))\to \mathcal{E}(C_{G^*}(s)^*,(1))$ satisfying
 $$\chi(1)=\frac{|G|_{p'}}{|C_{G^*}(s)|_{p'}}(\mathcal{L}_s(\chi)(1)).$$
In particular, if $C_{\bfG^*}(s)$ is a Levi subgroup
of $\bfG^*$, then the characters in ${\mathcal E}(G,(s))$ can be described
in terms of unipotent characters of $C_{G^*}(s)$.

\smallskip

In general, let $\bfL$ be an $F$-stable Levi subgroup of $\bfG$ with dual $\bfL^*\subset \bfG^*$ containing
$C_{\bfG^*}(t)$ with $t\in G^*$ semisimple.
Then the Jordan decomposition above can be constructed by Deligne-Lusztig induction
$$\mathcal{E}(L,(t))\xrightarrow{\sim}\mathcal{E}(G,(t)),\ \psi\mapsto\pm
R_{\bfL}^{\bfG}(\psi).$$
If $t\in Z(\bfL^*)$, then there is a bijection
$\mathcal{E}(L,(1))\xrightarrow{\sim}\mathcal{E}(L,(t)),\ \psi\mapsto
\eta\psi$
where $\eta$ is the one-dimensional character of $L$ corresponding to $t$. Thus we obtain
a bijection
$\mathcal{E}(L,(1))\xrightarrow{\sim}\mathcal{E}(G,(t)).$

\smallskip

\subsubsection{$\ell$-blocks and their basic sets}\label{subsec:block}

For a semisimple $\ell'$-element $s$ of $G^*$,
we denote by $\mathcal E_{\ell}(G, (s))$
the union of the Lusztig series $\mathcal E(G,(st))$,
where $t$ runs through semisimple $\ell$-elements of $G^*$ commuting with $s$.

\begin{definition}\label{def:eGs}
Let $s\in (\bfG^*)^{F^*}$ be a semisimple element of $\ell'$-order and define
$$e_{s}^{G}:=\sum e_\chi,$$
where $\chi$ runs over characters of
$\mathcal E_{\ell}(G,(s))$ and $e_{\chi}\in Z(KG)$ is the central idempotent corresponding to $\chi.$
\end{definition}

\begin{theorem}[\cite{BM}]\label{Thm:eGs}
Let $s\in (\bfG^*)^{F^*}$ be a semisimple element of $\ell'$-order.
Then $e_{s}^{G}\in Z(\mathcal O G)$ and hence $\mathcal O G e_{s}^{G}$ is  a sum of blocks. \qed
\end{theorem}

The above work of Brou\'{e}-Michel yields a decomposition
$$\mathcal O G\mod=\bigoplus_{(s)} \mathcal O Ge_s^{G}\mod$$ where $(s)$
runs over $G^*$-conjugacy classes of semisimple $\ell'$-elements of $G^*$.

\smallskip

\begin{theorem}[\cite{GH91, Ge93}]\label{basicset}
	Let $\ell$ be a prime good for $G$ and different from the defining characteristic of $G$.
	Assume that $\ell$ does not divide $(Z(\bfG)/Z^\circ (\bfG))_F$~(the largest quotient of $Z(\bfG)$ on which $F$ acts trivially).
	Let $s\in G^*$ be a semisimple $\ell'$-element.
	Then $\mathcal E(G,(s))$ form a basic set of $\mathcal E_\ell(G,(s))$.\qed
\end{theorem}

\subsubsection{Isolated blocks}
Recall that an element of $\mathbf{G}^*$ is called \textit{quasi-isolated} (resp. \textit{isolated})
 if its centralizer in $\bfG^*$ (resp. the connected component of its centralizer in  $\bfG^*$)
 is not contained in any proper Levi subgroup of $\mathbf{G}^*$.

 \begin{definition} Let $s$ be a quasi-isolated (resp. isolated) semisimple element of $G^*$
 of order prime to $\ell$. We say that a block $b$ of $G$ is \textit{quasi-isolated}
 (resp. \textit{isolated}) if it occurs in $\mathcal{O} G e_s^{G}$.
\end{definition}

\smallskip

Let $\bfL$ be an $F$-stable Levi subgroup of $\bfG$ with dual $\bfL^*$
such that $C_{\bfG^*}^\circ(s)\subset\bfL^*$, and that
$\bfL^*$ is minimal with respect to this property.
Let $\bfP$ be a parabolic subgroup of $\bfG$ with unipotent radical $\bfV$ and Levi complement
$\bfL$, and let
$$\bfY^{\bfG}_{\bfP}:=\{g\bfV\in \bfG/\bfV\,\,|\,\,g^{-1}F(g)\in \bfV F(\bfV)\}$$
be the Deligne-Lusztig
variety associated to $\bfP$.
Rickard and Rouquier  showed that there exists an object
$G\Gamma_c(\bfY^{\bfG}_{\bfP},\mathcal{O})$ in $\mathrm{Ho}^b(\mathcal{O} G\text{-} \mathrm{perm})$,
the bounded homotopy category of $\ell$-permutation $\mathcal{O} G$-modules.
Write $H_c^i(\bfY^{\bfG}_{\bfP},\mathcal{O})$ for its $i$th cohomology group and abbreviate $H_c^{\mathrm{dim}(\bfY^{\bfG}_{\bfP})}(\bfY^{\bfG}_{\bfP},\mathcal{O})$ by $H_c^{\mathrm{dim}}(\bfY^{\bfG}_{\bfP},\mathcal{O})$.

\smallskip

Bonnaf\'e, Dat and Rouquier showed that
the right action of $L$ on $H^{\dim}(\bfY^{\bfG}_{\bfP},\mathcal{O})e_s^{L}$ extends to
an action of $N=N_{G}(\bfL,e_s^{L})$ commuting with the action of $G$.

\begin{theorem}[\cite{BDR17, Ru20}]
\label{th:introequiv}
Assume $C_{\bfG^*}^\circ(s)\subset\bfL^*$ and that
$\bfL^*$ is minimal with respect to this property.

The right action of $L$ on $G\Gamma_c(\bfY^{\bfG}_{\bfP},\mathcal{O})e_s^{L}$ extends
to an action of $N$ and the resulting complex induces a splendid
Rickard equivalence between $\mathcal{O}\bfG^Fe_s^{G}$ and $\mathcal{O} Ne_s^{L}$.
The bimodule $H^{\dim}(\bfY^{\bfG}_{\bfP},\mathcal{O})e_s^{L}$ induces a Morita equivalence
between $\mathcal{O}\bfG^Fe_s^{G}$ and $\mathcal{O} Ne_s^{L}$.

The bijections between blocks of
$\mathcal{O}\bfG^Fe_s^{G}$ and $\mathcal{O} Ne_s^{L}$ induced by those equivalences
preserve the local structure, and in particular, preserve defect groups.\qed
\end{theorem}

\subsection{Representations of $\SO_{2n+1}(q)$}\label{sec:SO}

For later use we collect some facts about representations of $\SO_{2n+1}(q)$,
including (quadratic unipotent) characters and blocks of $\SO_{2n+1}(q)$.
In particular, characters in a cyclic quadratic unipotent (i.e., isolated)
block of $\SO_{2n+1}(q)$ will be highlighted.

\smallskip

Throughout this section we assume that $G=\SO_{2n+1}(q)$ with $q$ odd,
and that $\ell$ is an odd prime not dividing $q$.
We write $f$ (resp. $d$) for the order of $q$ (resp. $q^2$) in $k^\times$.
As usual, the prime $\ell$ is called \emph{linear} if $f$ is odd, and \emph{unitary} otherwise.
In the first case we have $f=d$, and in the second $f=2d$.
\smallskip

\subsubsection{Semisimple elements in $G^*=\Sp_{2n}(q)$}\label{subsec:semi}

To describe the characters of $G$, we first recall semisimple conjugacy classes of $G^*=\Sp_{2n}(q)$.
We denote by  $\Irr(\bbF_{q}[X])$ the set of all  monic irreducible polynomials over the field $\bbF_{q}$.
For each $\Delta$ in $\Irr(\bbF_{q}[X])\setminus \{X\}$, we let $\Delta^*$ be the polynomial in
$\Irr(\bbF_{q}[X])$ whose roots are the inverses of the roots of $\Delta$,
and denote
\begin{align*}
\cF_{0}&=\left\{ X-1,X+1 ~\right\},\\
\cF_{1}&=\left\{ \Delta\in\Irr(\bbF_{q}[X])\mid \Delta\notin \cF_0, \Delta\neq X,\Delta=\Delta^* ~\right\},\\
\cF_{2}&=\left\{~ \Delta\Delta^* ~|~ \Delta\in\Irr(\bbF_{q}[X])\setminus \cF_0, \Delta\neq X,\Delta\ne\Delta^* ~\right\},~\mbox{and}\\
\cF&=\cF_0\cup\cF_1\cup\cF_2.
\end{align*}
Furthermore, we denote by $d_\Gamma$
the degree of  a polynomial $\Gamma\in\cF$  and by
$$\delta_\Gamma:=
\left\{ \begin{array}{ll} d_\Gamma & \text{if}\ \Gamma\in\cF_0; \\
\frac{1}{2}d_\Gamma & \text{if}\ \Gamma\in\cF_1\cup \cF_2 \end{array} \right.$$
the so-called reduced degree of  $\Gamma$.
In addition, we define a sign $\varepsilon_\Gamma$ for $\Gamma\in\cF_1\cup \cF_2$ by
$$\varepsilon_\Gamma=
\left\{ \begin{array}{ll} -1 & \text{if}\ \Gamma\in\cF_1; \\
1 & \text{if}\ \Gamma\in\cF_2 . \end{array} \right.$$

\smallskip

For a semisimple element $s\in G^*$, we write $s=\prod_{\Gamma}s(\Gamma)$
for the primary decomposition of $s$,
where  $\Gamma$ is the minimal polynomial of $s(\Gamma)$.
Let $m_\Gamma(s)$ be the multiplicity of $\Gamma$ in $s(\Gamma)$, so that
$\Gamma$ is an elementary divisor of $s$ if $m_\Gamma(s)\ne 0$.

Semisimple conjugacy classes of $G^*$ are in bijection with the
functions given by
\begin{align*}
\cF &\rightarrow \mathbb{N}_0\\
\Gamma &\mapsto m_\Gamma(s)
\end{align*}
with $m_{X+1}(s)$ even and $\sum_{\Gamma\in\cF} m_\Gamma(s)d_\Gamma = 2n$.
In particular, an isolated semisimple element of $\Sp_{2n}(q)$ has $2n_{+}$ eigenvalues of $1$ and $2n_{-}$ eigenvalues of $-1$ where $n=n_++n_-.$
Furthermore, if $s=\prod\limits_{\Gamma\in \cF}s(\Gamma)$ is the primary decomposition of $s$, then
$$C_{G^*}(s)\cong \Sp_{2m_{X-1}(s)}(q)\times \Sp_{2m_{X+1}(s)}(q)\times(\prod_{\Gamma\in
\cF_1}U_{m_\Gamma(s)}(q^{d_\Gamma}))\times(\prod_{\Gamma\in
\cF_2}\GL_{m_\Gamma(s)}(q^{d_\Gamma})).
$$
Dually, we have
$${C_{G^*}(s)}^*\cong G_{m_{X-1}(s)}(q)\times G_{m_{X+1}(s)}(q)\times(\prod_{\Gamma\in
\cF_1}U_{m_\Gamma(s)}(q^{d_\Gamma}))\times(\prod_{\Gamma\in
\cF_2}\GL_{m_\Gamma(s)}(q^{d_\Gamma})).
$$
See  \cite{FS89} for more details.

\smallskip

\subsubsection{Charged symbol}\label{subsec:symbol}
A \emph{charged symbol} for
a bipartition $\tuple\mu = (\mu^1,\mu^2)$ with charge $\tuple s =(s_1,s_2)$
is a pair of charged $\beta$-sets $ \beta_{\tuple s}(\tuple\mu) = (\beta_{s_1}(\mu^1),\beta_{s_2}(\mu^2))$,
Write $\Theta = \beta_{\tuple s}(\tuple\mu)$ for brevity.
If $\beta_{s_1}(\mu^1) =: X =  \{x_1 > x_2 > \cdots\}$ and $\beta_{s_2}(\mu^2) =: Y
= \{y_1 > y_2 > \cdots\}$ we write
 $$ \Theta  = (X,Y) \, = \, \left( \begin{array}{ccccc}
 x_1 &  x_2  & \ldots  \\ y_1 & y_2 & \ldots  \\ \end{array} \right).$$
The components $X$ is called the first row of the symbol $\Theta$
and $Y$ the second one.
The \emph{defect} and \emph{rank} of $\Theta$ are respectively
$$D = s_1-s_2\ \ \ \ \mbox{and}\ \ \ \ |\tuple\mu| + \lfloor D^2/4 \rfloor.$$

\smallskip

A \emph{$d$-hook} of $\Theta$ is a pair of integers $(x,x+d)$ which is either
a $d$-hook of $X$ or a $d$-hook of $Y$. The charged symbol obtained by deleting $x+d$ from $X$ (resp. $Y$)
and replacing it by $x$ is said to be gotten from $\Theta$ by \emph{removing the $d$-hook $(x,x+d)$}.
A \emph{$d$-cohook} is a pair of integers $(x,x+d)$ such that $x+d\in X$ and $x\not\in Y$, or
$x+d\in Y$ and $x\not\in X$.
The charged symbol obtained by deleting $x+d$ from $X$ (resp. $Y$) and adding $x$ to $Y$ (resp. $X$) is said to be gotten from $\Theta$ by \emph{removing the $d$-cohook $(x,x+d)$}.
Removing recursively all $d$-hooks (resp. $d$-cohooks) from $\Theta$,
we obtain the \emph{$d$-core} (resp. \emph{$d$-cocore}) of $\Theta$.

\smallskip

By $\Theta^\dag = (Y,X)$ we denote the charged symbol of charge $(s_2, s_1)$ obtained
by swapping the two $\beta$-sets. The defect of $\Theta^\dag$ is $-D$ but the rank is the same. The operation
 shifting simultaneously the charged $\beta$-sets $X$ and $Y$ by an integer $m$
  preserves both defect and rank.

\smallskip

\emph{Symbols} are orbits of charged symbols
under the shift operator and the transformation $\Theta \mapsto \Theta^\dag$.
We write
$$ \{X,Y\} \, = \, \left\{ \begin{array}{ccccc}
 x_1 &  x_2  & \ldots  \\ y_1 & y_2 & \ldots  \\ \end{array} \right\}$$
 for the symbol associated with $(X,Y)$. The rank of
the symbol is the rank of any charged symbol in its class whereas its defect is
the absolute value of the defect of any representative.
Removing and adding $d$-hooks or $d$-cohooks
are well-defined operations on symbols.

\smallskip

For a positive integer $d$, a \emph{$2d$-abacus} of
 a symbol $\{X,Y\}$ is a $2d$-linear diagram
in which for integers $i\geqs 0$ and $0\leqs
j\leqs d-1$,  we put a bead on the $i$th row and $j$th runner if and only if $di+j\in
X$ and on the $i$th row and $(d+j)$th runner if and only if $di+j\in Y$.

\subsubsection{Characters of $\SO_{2n+1}(q)$}\label{subsec:char}
We have known that
$\mathrm{Irr}(G)=\coprod {\mathcal E}(G,(s))$,
where $(s)$ runs over the set of $G^*$-conjugacy classes of semisimple elements of
$G^*$.
Recall that unipotent characters of ${\rm \GL}_n(q)$ and ${\rm U}_n(q)$ are both labeled by
partitions of $n$ (for example $(n)$ corresponds to the trivial character of
both groups), and that the unipotent characters of $\Sp_{2n}(q)$ and $\SO_{2n+1}(q)$ are both labeled by
symbols with rank $n$ and odd defect (see \cite{DL,L77} for more details).

\smallskip

Given a semisimple element $s\in G^*$ with $s=\prod\limits_{\Gamma\in \cF}s(\Gamma)$,
we write $\Psi_\Gamma(s)$ for the set of partitions of $m_\Gamma(s)$ if $\Gamma\in\cF_1\cup\cF_2$
and for the set of symbols with rank $[\frac{m_\Gamma(s)}{2}]$ and odd defect
if $\Gamma\in\cF_0$, and denote
\begin{equation*}\label{def-par-sym}
\Psi(s)=\prod_{\Gamma\in\cF} \Psi_\Gamma(s).
\end{equation*}

The Jordan decomposition $\mathcal{L}_s$ in \S \ref{subsec:jordan}
 induces a bijection between $\mathcal{E}(G,(s))$ and $\Psi(s)$.
For $\mu\in\Psi(s)$, we denote by $\psi_\mu$ the unipotent character of $C_{G^*} (s)$ corresponding to $\mu$.
  So the characters in  $\mathcal{E}(G,(s))$ corresponding to $\psi_\mu$ can be simply denoted by $\chi_{s,\mu}$,
in which the (conjugacy class of) semisimple element $s$ is called the \emph{semisimple label}
and $\mu\in\Psi(s)$ is called the \emph{unipotent label} of $\chi_{s,\mu}$.

\subsubsection{Harish-Chandra induction of $\SO_{2n+1}(q)$}

Let $M$ be a split
$F$-stable Levi subgroup of $G$. Then $M^*$ is a split $F^*$-stable Levi subgroup
of $G^*$. Let $s$ be an $F^*$-stable semisimple element of $M^*$.
As shown by Fong-Srinivasan \cite{FS82,FS89},
the Jordan decomposition of characters
 commutes with Lusztig induction, and in particular Harish-Chandra induction.
 Hence the computation of Harish-Chandra induction $R_{M}^G$
 reduces to Harish-Chandra induction
 $R_{{C_{M^*(s)}}^*}^{{C_{G^*(s)}}^*}$ with unipotent characters.
By Comparison Theorem, it
 can be reduced to
 the inductions of Weyl groups
 which are controlled by  Littlewood-Richardson coefficients (see \cite[\S 3]{hiskes2000}).

\smallskip

\subsubsection{Quadratic unipotent characters of $\SO_{2n+1}(q)$}\label{subsec:quchar}

We write $G=G_n=\SO_{2n+1}(q)$.

\begin{definition} A character $\chi \in {\mathcal E}(G,(s))$ is called quadratic
unipotent if $s^2=1$.
\end{definition}

Such a character is called square-unipotent in \cite{W} (see also \cite{S08}),
and is exactly unipotent if $s=1$.
 Clearly, if $s^2=1$ then
all eigenvalues of $s$ are $1$ or $-1$.

\begin{proposition}[Lusztig \cite{L77}]\label{prop:cuspidalBC}
 Up to isomorphism, there is at most one cuspidal quadratic  unipotent module in ${\mathcal E}(G_r,(s))$ for $s\in G_r^*$ satisfying $s^2=1$, where
 $G_r=\SO_{2r+1}(q)$.

Furthermore, if $s$ has $2r_+$ eigenvalues of $1$ and $2r_-$ eigenvalues of $-1$ (with $r=r_++r_-$)
then there exists one cuspidal quadratic  unipotent module in ${\mathcal E}(G_r,(s))$
if and only if $r_+=t_+(t_++1)$ and $r_-=t_-(t_-+1)$ for some $t_+,t_-\in \bbN$,
in which case the module will be denoted by $E_{t_+, t_-}$.
\end{proposition}
\begin{proof} We have $C_{{G_r}^*}(s)^*\cong G_{r_+}\times G_{r_-}$
by \S\ref{subsec:semi}, and that the center of $G_r$ and $G_{r_+}\times G_{r_-}$ have the same $\bbF_q^{\times}$-rank.
By the Jordan decomposition of characters \cite[Theorem]{L77},
there is a bijection between $\mathcal{E}(G_r,(s))$ and $\mathcal(C_{{G_r}^*}(s)^*,(1))$,
mapping the cuspidal quadratic
unipotent  modules in $\mathcal{E}(G_r,(s))$ to the cuspidal unipotent  modules in $\mathcal(C_{{G_r}^*}(s)^*,(1))$.
Note that any cuspidal unipotent
module of $G_{r_+}\times G_{r_-}$ is the product of a cuspidal  unipotent module of $G_{r_+}$ and a cuspidal unipotent module of $G_{r_-}.$ However, there is at most one cuspidal  unipotent module in ${\mathcal E}(G_{r_\pm},(1))$ and there exists one cuspidal  unipotent module in ${\mathcal E}(G_{r_\pm},(1))$ if and only if $r_\pm=t_\pm(t_\pm+1)$ for some $t_\pm\in\bbN$.
Hence the proposition holds.
\end{proof}

When $t_-=0$, the module $E_{t_+, t_-}$ is exactly the cuspidal unipotent module $E_{t_+}$ considered in \cite{DVV2}.
We write $K_1$ for the trivial $K\bbF^\times_{q}$-module and
$K_\zeta$ for the $K\bbF^\times_{q}$-module
   affording  the unique  character $\zeta$ of $\bbF^\times_{q}$ of order 2 (i.e., Legendre symbol for $\bbF^\times_{q}$).
Since $\mathrm{GL}_n(q)$ does not have any
cuspidal quadratic unipotent character unless $n =1$ in which case
 $K_1$ and $K_\zeta$ are clearly
  the only two cuspidal quadratic unipotent modules of  $\mathrm{GL}_1(q)=\bbF_q^{\times}$,
it follows from Proposition \ref{prop:cuspidalBC} that
any cuspidal quadratic unipotent pair of $G_n:=\SO_{2n+1}(q)$ is conjugate to a pair of the form
$$(L_{r,1^m}\,,\, E_{t_+,t_-}\otimes K_1^{m_+}\otimes K_\zeta^{m_-}),$$
where $L_{r,1^m}\simeq G_r \times (\bbF^\times_{q})^m$ with $n=r+m$,
$m=m_++m_-$, $n=n_++n_-$ and $n_\pm=t_\pm(t_\pm+1)+m_\pm$
with $t_\pm,m_\pm\in\bbN$. According to
\S \ref{subsec:HCseries},
the irreducible characters lying in the
Harish-Chandra series above $ E_{t_+,t_-}\otimes K_1^{m_+}\otimes K_\zeta^{m_-}$
are in bijection with the irreducible representations of the ramified Hecke algebra
$\scrH(KG_n,E_{t_+,t_-}\otimes K_1^{m_+}\otimes K_\zeta^{m_-})$  of type $B_{m_+}\times B_{m_-}$
(see  \cite[\S 5,8]{L77}).

\smallskip

Recall that $G_0=\{1\}$. As in \cite{DVV2} for the category $\scrU_K:=\bigoplus_{n\in \bbN}KG_n\mod$
of unipotent modules over $K$, we call the category of quadratic unipotent modules over $K$ the category
  $$\scrQU_K=\scrQU_K^{\SO}=\bigoplus_{n\in\bbN}K G_n\qumod.$$
This category is abelian semisimple. According to the work of Lusztig \cite[\S 5,8]{L77},
we have a parametrization of the quadratic unipotent characters
$$\Irr(\scrQU_K)=
\{E_{\Theta_+,\Theta_-}\,|\,\Theta_\pm\in\scrS_{\mathrm{odd}}\},$$ where
$\scrS_{\mathrm{odd}}$ means  the set of symbols of odd defects.
Since the conjugacy classes of involutions of $\Sp_{2n}(q)$
are uniquely determined by the multiplicities of eigenvalues $\pm 1$,
 we have
$$\Irr(KG_n\qumod)=\coprod_{(n_+,n_-)\comp_{2}\,\,n}{\mathcal E}(G_n,(s_{n_+,n_-}))$$
where  $${\mathcal E}(G_n,(s_{n_+,n_-}))=
\{E_{\Theta_+,\Theta_-}\,|\,\Theta_\pm\in\scrS_{\mathrm{odd}}\,
\,\text{and}\,\,\rank(\Theta_\pm)=n_{\pm}\}$$ and
  $s_{n_+,n_-}\in \Sp_{2n}(q)$
  which satisfies $s_{n_+,n_-}^2=1$ and
  has $2n_+$ eigenvalues of $1$ and $2n_-$ eigenvalues of $-1$.
Here $s=s_{n_+,n_-}$  and $\Theta_+\times\Theta_-$  are called the semisimple label
    and the unipotent label of $E_{\Theta_+,\Theta_-}$, respectively,
    and we also denote its corresponding character by $\chi_{\Theta_+,\Theta_-}.$
Furthermore, ${\mathcal E}(G_n,(s_{n_+,n_-}))$ can be partitioned to Harish-Chandra series $${\mathcal E}(G_n,(s_{n_+,n_-}))=\coprod{\mathcal E}(G_n,(L_{r,1^m}, E_{t_+,t_-}\otimes K_1^{m_+}\otimes K_\zeta^{m_-}))$$
where the union runs over $t_\pm,m_\pm\in \bbN$ satisfying
$m=m_++m_-$, $r=r_++r_-$ and $n=r+m=n_++n_-$ with
$r_{\pm}=t_{\pm}(t_{\pm}+1)$
 and  $r_\pm+m_\pm=n_\pm$.

 \smallskip

 For the unipotent case (i.e., the case where $t_-=0$),
 it is shown in \cite{DVV2} that
 there is a canonical bijection
 \begin{align}
 \label{bijectionB}\Irr(KG_n,E_{t_+,0})\, \mathop{\longleftrightarrow}\limits^{1:1}\, \Irr(\bfH^{q\,;\,\tuple\xi_{t_+}}_{K,m}),
 \end{align}
mapping  the unipotent $KG_{r_+}$-module $E_{\Theta_{t_+}(\tuple \mu)}$  to $S(\tuple \mu)_K^{q\,;\,\tuple\xi_{t_+}}$.
Here 
$$\tuple\xi_{t_+}=((-q)^{t_+},(-q)^{-1-t_+}),$$ $\tuple \mu$ is a  bipartition $(\mu^1,\mu^2)$ of $m_+$, and 
$\Theta_{t_+}(\tuple \mu)$ is the symbol $\big\{\beta_{t_+}(\mu^1),\beta_{-t_+-1}(\mu^2)\big\}$
associated with  $\tuple \mu$ so that
the defect and the rank of $\Theta_{t_+}(\tuple \mu)$ are
$D(\Theta_{t_+}(\tuple \mu))=2t_++1$ and $\rk(\Theta_{t_+}(\tuple \mu))=m+r_+.$

\smallskip

 For the quadratic unipotent case, let  $\bfH^{q\,;\,\tuple\xi_{t_+}}_{K,m_+}\times\bfH^{q\,;\,\tuple\xi_{t_-}}_{K,m_-}$
be the Hecke algebra of type $B_{m_+}\times B_{m_-}$ with
\begin{equation}\label{parameter}
\tuple\xi_{t_+}=((-q)^{t_+},(-q)^{-1-t_+})~\mbox{and}~
\tuple\xi_{t_-}=((-q)^{t_-},(-q)^{-1-t_-}).
\end{equation}
By \cite[\S 5,8]{L77}, the algebra $\scrH(KG_n, E_{t_+,t_-}\otimes K_1^{m_+}\otimes K_\zeta^{m_-})$
is isomorphic to $\bfH^{q\,;\,\tuple\xi_{t_+}}_{K,m_+}\times\bfH^{q\,;\,\tuple\xi_{t_-}}_{K,m_-}$.
It will be clear in Theorem \ref{thm:HL-BC} that with the choice of $\tuple\xi_{t_+}$ and $\tuple\xi_{t_-}$,
the above isomorphism can be chosen to be
canonical.
With the induced bijection \begin{align}
\Irr(\bfH^{q\,;\,\tuple\xi_{t_+}}_{K,m_+}\times\bfH^{q\,;\,\tuple\xi_{t_-}}_{K,m_-})
\mathop{\longleftrightarrow}\limits^{1:1}\,
\Irr(\scrH(KG_n, E_{t_+,t_-}\otimes K_1^{m_+}\otimes K_\zeta^{m_-})),
 \end{align}
 we may define $E_{\Theta_{t_+}(\tuple\mu_+)\times\Theta_{t_-}(\tuple\mu_-)}$ to
be the quadratic unipotent $KG_r$-module corresponding to
$S(\mu)_K^{q\,;\,\tuple\xi_{t_+}}\otimes S(\mu)_K^{q\,;\,\tuple\xi_{t_-}}$.
Thus
 $${\mathcal E}(G_n,(L_{r,1^m}, E_{t_+,t_-}\otimes K_1^{m_+}\otimes K_\zeta^{m_-}))=
\{E_{\Theta_{t_+}(\tuple\mu_+),\Theta_{t_-}(\tuple\mu_-)}\,|\tuple\mu_{\pm}\in \scrP^2_{m_{\pm}}\}$$
where
$\scrP^2_{m_{\pm}}$ is the set of $2$-partitions of $m_{\pm}$.

\smallskip
Clearly, we have $E_t=E_{\Theta_t(\emptyset)}$ and $E_{t_+,t_-}=E_{\Theta_{t_+}(\emptyset),\Theta_{t_-}(\emptyset)}$.
  Note that $Q_t^\dag=Q_{-1-t}$, $\Theta_t(\mu)=\Theta_{-1-t}(\mu^\dag)$ and $t(t+1)$ is invariant under the map $t \longmapsto -t-1$.
We will usually work with symbols $\Theta_t(\mu)$ such that $t \geqslant 0$
and use the symmetries above to deal with those such that $t<0$.
With this in mind, we have $E_{0,t_-}=E_{-1,t_-}$ and $E_{t_+,0}=E_{t_+,-1}.$
\begin{remark}
The Lusztig series $\mathcal{E}(G_r,(s_{r_+,r_-}))$ is
 mapped to the unipotent series $\mathcal{E}(G_{r_+}\times G_{r_-},(1))$
 under the Jordan decomposition of characters, where
  the quadratic unipotent module $E_{\Theta_+,\Theta_-}\in \mathcal{E}(G_r,(s_{r_+,r_-}))$
  is mapped to the unipotent  module $E_{\Theta_+}\otimes E_{\Theta_-}$ of $G_{r_+}\times G_{r_-}$.
   In particular, the quadratic unipotent cuspidal
    module $E_{t_+,t_-}\in \mathcal{E}(G_r,(s_{r_+,r_-}))$ is mapped to the
    unipotent cuspidal module $E_{t_+}\otimes E_{t_-}$,
     where $E_{t_+}$ is the unique unipotent cuspidal module of $G_{r_+}$
     and $E_{t_-}$ is the unique unipotent cuspidal module of $G_{r_-}.$
     Since the Jordan decomposition of characters commutes with Harish-Chandra induction,
     the  pair $(L_{r,1^m}\,,\,E_{\Theta_+,\Theta_-}\otimes K_1^{m_+}\otimes K_\zeta^{m_-})$
     is mapped to
     $$(L_{r_+,1^{m_+}}\times L_{r_-,1^{m_-}}\,,\,  (E_{\Theta_+}\otimes K_1^{m_+})\otimes (E_{\Theta_-}\otimes K_\zeta^{m_-}))$$
      and there is a bijection between the irreducible constituents of
      $$R_{L_{r,1^m}}^{G_n}(E_{\Theta_+,\Theta_-}\otimes K_1^{m_+}\otimes K_\zeta^{m_-})~\mbox{and}~
          R_{L_{r_+,1^{m_+}}
          \times L_{r_-,1^{m_-}}}^{G_{n_+}\times G_{n_-}}((E_{\Theta_+}\otimes K_1^{m_+})
          \otimes (E_{\Theta_-}\otimes K_\zeta^{m_-})),$$ the latter of which is isomorphic to
          $R_{L_{r_+,1^{m_+}}}^{G_{n_+}}(E_{\Theta_+}\otimes K_1^{m_+})\otimes R_{L_{r_-,1^{m_-}}}^{G_{n_-}}(E_{\Theta_-}\otimes K_\zeta^{m_-}).$
\end{remark}

\subsubsection{Blocks of $\SO_{2n+1}(q)$}\label{subsec:SOblock}
The characters in a block of $G$ has been classified in \cite{FS89}.
We define $e_\Gamma$ to be  the multiplicative order of $q^{2}$ or $\varepsilon_\Gamma q^{\delta_\Gamma}$ modulo $\ell$ according as $\Gamma\in\cF_0$ or $\Gamma\in\cF_1\cup\cF_2$.
For a semisimple $\ell'$-element $s$ of $G^*$ and $\Gamma\in\cF$,
we define $\mathcal{R}_\Gamma(s)$ to be
 the set of $e_\Gamma$-cores of partitions in
$\Psi_\Gamma(s)$. 
Denote
$$\mathcal{R}(s)=\prod_{\Gamma\in\cF} \mathcal{R}_\Gamma(s).$$
By \cite[\S 10]{FS89},
there is a bijection $(s,\kappa)\mapsto
b_{s,\kappa}$ from the set of $G^*$-conjugacy classes of pairs $(s,\kappa)$,
with $s\in G^*$ a semisimple $\ell'$-element and $\kappa\in \mathcal{R}(s)$,
 to the set of blocks of $G$.
 The semisimple $\ell'$-element $s$ (up to $G^*$-conjugacy class)
 is called the \emph{semisimple label} of $b_{s,\kappa}$
and $\kappa\in \mathcal{R}(s)$ is called the \emph{unipotent label} of $b_{s,\kappa}$.

\smallskip

Let $D$ be a defect group  of $b_{s,\kappa}$, and  $D^*$ be a fixed
dual defect group of $D$ as defined in  \cite[\S12]{FS89}.
An irreducible character $\chi_{t,\lambda}$ of
$G$ lies in $b_{s,\kappa}$ if and only if the following
statements hold:
\begin{enumerate}[leftmargin=8mm]
\item The $\ell'$-part $t_{\ell'}$ of $t$ is conjugate to $s$ and the $\ell$-part $t_{\ell}$ of $t$ is conjugate to $x$ for some $x\in D^*$,
\item
 The $e_{\Gamma}$-core of $\lambda_{\Gamma}$  belongs to $\mathcal{R}_\Gamma(s)$
whenever $\Gamma$ is
an elementary divisor of $s$, and
\item the $e_\Gamma$-core of $\lambda_\Gamma$ is empty for all other $\Gamma$.
\end{enumerate}

\subsubsection{Quadratic unipotent blocks of $\SO_{2n+1}(q)$}\label{subsec:SOqublock}
Here we define quadratic unipotent or isolated blocks, which are same in the case of $\SO_{2n+1}(q).$ Notice that we have assumed that $\ell$ is odd.

\begin{definition}\label{def:lblock}
\begin{itemize}[leftmargin=8mm] Let $R$ be $\mathcal{O}$ or $k.$
  \item[$(1)$] An $\ell$-block of $RG$ is called \emph{quadratic unipotent (resp. isolated)} if it contains
a quadratic unipotent (resp. an isolated) character.
  \item[$(2)$] An indecomposable $RG$-module is called \emph{quadratic unipotent (resp. isolated)}
if it lies in a quadratic unipotent (resp. an isolated) block of $RG$.
\item[$(3)$] A  $RG$-module is called \emph{quadratic unipotent (resp. isolated)}
if it is a direct sum of indecomposable modules which lies in  quadratic unipotent (resp. isolated) blocks of $RG$.
\end{itemize}
\end{definition}

We denote by $$\scrQU_R=\scrQU^{\SO}_R=\bigoplus_{n\in\bbN}R G_n\qumod$$
the category of quadratic unipotent modules over $R$,
where $G_n=\SO_{2n+1}(q)$ and
$R G_n\qumod$ is the category of quadratic unipotent $RG_n$-modules.
  This is an abelian category which is not semisimple when $R=\mathcal O$ or $k$.

\smallskip

Let $s_{n_+,n_-}\in G_n^*$ as in \S\ref{subsec:quchar}.
Recall from Theorem \ref{Thm:eGs} that the set $\mathcal E_{\ell}(G_n,(s_{n_+,n_-}))$
is a union of $\ell$-blocks of $G_n$.
\begin{definition}\label{def:lblock}
Let  $e_\chi$ be
the central idempotent corresponding to $\chi\in \Irr(G_n)$.
\begin{itemize}[leftmargin=8mm]
\item[$(1)$]We define
$$e_{s_{n_+,n_-}}^{KG_n}:=\sum\limits_{\chi\in\mathcal E(G,(s_{n_+,n_-}))}e_\chi~~\ \ {\rm and}~~\ \ e_{s_{n_+,n_-}}^{\mathcal O G_n}:=\sum\limits_{\chi\in\mathcal E_{\ell}(G,(s_{n_+,n_-}))}e_\chi,$$
which are central idempotents of $K G_n$ and $\mathcal O G_n$, respectively (see \S \ref{Thm:eGs}).
\item[$(2)$]We define $e_{s_{n_+,n_-}}^{k G_n}$ to be the  central idempotent of $kG_n$ that is
the image of $e_{s_{n_+,n_-}}^{\mathcal O G_n}$ under the natural epimorphism $\mathcal O G_n\to kG_n$.
\end{itemize}
\end{definition}

Now, for $R\in \{K,\mathcal O, k\}$, we have
$$RG_n\qumod=\bigoplus_{n_++n_-=n}  R G_ne_{s_{n_+,n_-}}^{RG_n}\mod.$$
An application of Theorem \ref{basicset} shows that there is
a $\bbZ$-linear isomorphism $$d_{\mathcal{O}G_n} : [K G_ne_{s_{n_+,n_-}}^{KG_n}\mod]\simto[k G_ne_{s_{n_+,n_-}}^{kG_n}\mod],$$
 induced by the restrictions of the decomposition maps under the $\ell$-modular system.
 This yields  $\bbZ$-linear isomorphisms $$d_{\mathcal{O}G_n}: [KG_n\qumod]\simto[k G_n\qumod]\,\,\text{and}\,\,d_\scrQU : [\scrQU_K]\simto[\scrQU_k].$$

\begin{theorem} \label{thm:isoblock-FS} Let $b$ be a quadratic unipotent $\ell$-block,
i.e., an isolated $\ell$-block of $G$.
Then $b$ corresponds to a pair of symbols
$(\Delta_+,\Delta_-)$ and the quadratic unipotent characters in $b$ are of the form
$\chi_{\Theta_+, \Theta_-}$, where
 ${\Delta}_{\pm}$ are $d$-cores
 and ${\Theta}_{\pm}$ have $d$-cores (resp. $d$-cocores) ${\Delta}_{\pm}$
 if $\ell$ is linear (resp. unitary).

\smallskip
In both cases, the quadratic unipotent characters in $b$ are precisely the constituents of
$$R^{G}_L(
{\chi}_{{\Delta_+}\times{\Delta_-}}\times 1^{w_+} \times\zeta^{w_-}),$$ where
\begin{itemize}[leftmargin=8mm]
  \item[$\bullet$] $L$ is a Levi subgroup of the form $ G_{r}\times T_+ \times T_-$,
  \item[$\bullet$] $T_+$ (resp. $T_-$) is the product of $w_+$ (resp. $w_-$) copies of a torus
     of order that is $q^{d}-1$ for the linear case and is $q^{d}+1$
 for the unitary case,
 \item[$\bullet$] $1$ (resp. $\zeta$) is the trivial character (resp. the unique character of order 2) of
  a (torus) factor of $T_+$ (resp. $T_-$), and
 \item[$\bullet$] the character ${\chi}_{{\Delta_+}\times{\Delta_-}}$
 is in a block of defect $0$ of $G_{n-dw}$ for $w=w_++w_-$.
   \end{itemize}\qed
\end{theorem}

By Theorem \ref{thm:isoblock-FS}, isolated blocks $b$ of $G$
are parametrized by tuples $$(\Delta_+\times\Delta_-,w_+,w_-)$$ where $\Delta_{\pm}$ is
a $d$-core (resp. $d$-cocore) when
$f$ is odd (resp. $f$ is even) and $w_{\pm}\in\bbN$.
Consequently, we may label the block $b$ by $b_{\Delta_+\times\Delta_-,w_+,w_-}$,
so that the unipotent label of $b_{\Delta_+\times\Delta_-,w_+,w_-}$ is $\Delta_+\times\Delta_-$ and the semisimple label of $b_{\Delta_+\times\Delta_-,w_+,w_-}$ is $s_{n_+,n_-}$ where $n_\pm=\mathrm{rank}(\Delta_\pm)+dw_\pm.$
The vector $\tuple w=(w_+,w_-)$ will be called the degree vector of $b$.

\smallskip

The defect groups of $b$ are isomorphic to the Sylow $\ell$-subgroups of
$$(T_+\rtimes(\bbZ/(2d\bbZ)\wr \mathfrak{S}_{w_+}))\times (T_-\rtimes(\bbZ/(2d\bbZ)\wr \mathfrak{S}_{w_-}))$$
(see \cite{CE} for the case of unipotent blocks, and \cite{E08} for the case of quadratic unipotent blocks).
The isolated block $b_{\Delta_+\times\Delta_-,w_+,w_-}$
 has an abelian defect group if and only if $w_+<\ell$ and $w_-<\ell$.
 In particular, the isolated block $b_{\Delta_+\times\Delta_-,w_+,w_-}$
 is of defect zero if and only if $\tuple w=(w_+,w_-)=(0,0)$,
 and is cyclic if and only if $\tuple w=(w_+,w_-)=(1,0)$ or $(0,1).$

\smallskip
\subsubsection{Cyclic isolated blocks}\label{subsec:Brauertree}

According to the work of Fong-Srinivasan \cite{FS90},
the character $\chi_{\Theta_+\times\Theta_-}$ of $G$
belongs to an isolated $\ell$-block with cyclic defect groups
 if and only if one of the following two cases occurs:
\begin{itemize}[leftmargin=8mm]
\item[$(a)$] $\Theta_+$ has a unique $d$-hook and $\Theta_-$ has no $d$-hook if $\ell$ is linear,
or $\Theta_+$ has a unique $d$-cohook and $\Theta_-$ has no $d$-cohook if $\ell$ is unitary;
\item[$(b)$] $\Theta_-$ has a unique $d$-hook and $\Theta_+$ has no $d$-hook if $\ell$ is linear,
or $\Theta_-$ has a unique $d$-cohook and $\Theta_+$ has no $d$-cohook if $\ell$ is unitary.
\end{itemize}

Suppose now that  $\chi_{\Theta_+\times\Theta_-}$
belongs to an isolated cyclic $\ell$-block, say $B$.
Let $\chi_{{\rm exc}}$ be the sum of the non-quadratic unipotent characters in $B$,
i.e., exceptional characters. Recall that the Brauer tree of $B$
is defined as follows:

\begin{itemize}[leftmargin=8mm]
\item Its vertices are $\chi_{{\rm exc}}$ and the quadratic unipotent characters of $B$, and
\item there is an edge between two vertices if and only if their
$\ell$-reduction has a common irreducible modular character, in which case the edge is labeled by
 the modular character.
\end{itemize}
By \cite[(5A), (6A)]{FS90}, the Brauer tree of the $\ell$-block $B$ is of the form
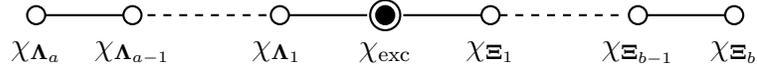
\begin{center}
\begin{tikzpicture}[scale=.4]
\draw[thick] (-1.7 cm,0) circle (.3cm);
\node [below] at (-1.7 cm,-.5cm) {$\chi_{\tuple\Lambda_a}$};
\draw[thick] (-1.4 cm,0) -- +(2.6 cm,0);
\draw[thick] (1.5 cm,0) circle (.3cm);
\node [below] at (1.5 cm,-.5cm) {$\chi_{\tuple\Lambda_{a-1}}$};
\draw[dashed,thick] (2 cm,0) -- +(4 cm,0);
\draw[thick] (6.4 cm,0) circle (.3cm);
\node [below] at (6.3 cm,-.5cm) {$\chi_{\tuple\Lambda_1}$};
\draw[thick] (6.7 cm,0) -- +(2.6 cm,0);
\draw[thick,fill=black] (9.9 cm,0) circle (.3cm);
\draw[thick] (9.9 cm,0) circle (.5cm);
\node [below] at (9.9 cm,-.55cm) {$\chi_{\text{exc}}$};
\draw[thick] (10.5 cm,0) -- +(2.6 cm,0);
\draw[thick] (13.4 cm,0) circle (.3cm);
\node [below] at (13.4 cm,-.5cm) {$\chi_{\tuple\Xi_1}$};
\draw[dashed,thick] (13.7 cm,0) -- +(4.2 cm,0);
\draw[thick] (18.3 cm,0) circle (.3cm);
\node [below] at (18.3 cm,-.5cm) {$\chi_{\tuple\Xi_{b-1}}$};
\draw[thick] (18.6 cm,0) -- +(2.6 cm,0);
\draw[thick] (21.5 cm,0) circle (.3cm);
\node [below] at (21.5 cm,-.5cm) {$\chi_{\tuple\Xi_b}$};
\end{tikzpicture}
\end{center}

For case (a), let $\{X,Y\}$ be the $d$-core (resp. $d$-cocore) of
$\Theta_+$ if $\ell$ is linear (resp. unitary).
\begin{itemize}[leftmargin=8mm]
  \item[(a1)] If $\ell$ is linear, then $\tuple\Lambda_k=\Lambda_{k,\,+}\times\Theta_-$
  and $\tuple\Xi_k=\Xi_{k,\,+}\times\Theta_-,$ where $\Lambda_{k,\,+}$ and $\Xi_{k,\,+}$
  are obtained by
  adding a $d$-hook to $X$ and $Y$,
  respectively, for each $k=1,\dots,d$. In this case,  $d=a=b$.
    \item[(a2)] If $\ell$ is unitary, then
  $\tuple\Lambda_k=\Lambda_{k,\,+}\times\Theta_-$ and
  $\tuple\Xi_k=\Xi_{k,\,+}\times\Theta_-,$  where
$\Lambda_{k,\,+}$ and $\Xi_{k,\,+}$ are obtained by adding a $d$-cohook to $Y$ and $X$,
respectively.
In this case,  $a=d+d_+$ and $b=d-d_+$, where $d_+$ is the defect of the symbol $\Theta_+$.
\end{itemize}
\smallskip

 Similarly, for case (b), let $\{X,Y\}$ be the $d$-core (resp. $d$-cocore) of
$\Theta_-$ if $\ell$ is linear (resp. unitary).
\begin{itemize}
  \item[(b1)] If $\ell$ is linear, then $\tuple\Lambda_k=\Theta_+\times\Lambda_{k,-}$
  and $\tuple\Xi_k=\Theta_+\times\Xi_{k,-},$
where $\Lambda_{k,-}$  and $\Xi_{k,-}$ are obtained by
  adding a $d$-hook to $X$ and $Y$, respectively, for each $k=1,\dots,d$. In this case,
  $d=a=b$.
  \item[(b2)] If $\ell$ is unitary,
  then $\tuple\Lambda_k=\Theta_+\times\Lambda_{k,-}$,
  $\tuple\Xi_k=\Theta_+\times\Xi_{k,-},$
where $\Lambda_{k,-}$ and $\Xi_{k,-}$
  are obtained by adding a $d$-cohook to $Y$ and $X$, 
  respectively.
In this case, $a=d+d_-$ and $b=d-d_-$, where $d_-$ is the defect of the symbol $\Theta_-$.
\end{itemize}

\subsection{Representations of $\O_{2n+1}(q)$}\label{sub:O2n+1}
For later use, here we introduce representations of $\O_{2n+1}(q)$.

Let
 $G=G_n=\SO_{2n+1}(q)$ and $\widetilde{G}=\widetilde{G}_n=\O_{2n+1}(q)$.
We have $\widetilde{G}\cong G\times \{\pm \id_{\widetilde{G}}\}$ and
so $R\widetilde{G}\cong RG\otimes_{R} R\{\pm \id_{\widetilde{G}}\}.$
In particular, each
primitive idempotent of $R\widetilde{G}$ has form  $e_{G}\otimes e_{\pm 1}$, where
$e_{G}$ is a primitive  idempotent of $RG$ and
$e_{1}$ (resp. $e_{-1}$) is the idempotent of $R\{\pm \id_{\widetilde{G}}\}$ corresponding to
the trivial character (resp. the sign character) of
$\{\pm \id_{\widetilde{G}}\}$.
Denote by $R_{1}$ and $R_{-1}$ the two irreducible modules of $\{\pm \id_{\widetilde{G}}\}$ corresponding to $e_{1}$ and $e_{-1}$, respectively.
Assume that $\ell$ is odd.
Then irreducible modules (resp. $\ell$-blocks) of $R\widetilde{G}$
are the tensor product of irreducible modules (resp. $\ell$-blocks) of
 $RG$ with $R_{1}$ or $R_{-1}$.
 A module $M\otimes R_{\pm 1}$ (resp. $\ell$-block $B\otimes R_{\pm 1}$) of $\widetilde{G}$ is called  quadratic unipotent if
the corresponding module $M$ (resp. $\ell$-block $B$) of $G$ is  quadratic unipotent.
So every quadratic unipotent module of $\widetilde{G}$ is of
form $E_{\Theta_+,\Theta_-}\otimes R_{\epsilon}$ with $\epsilon\in\{\pm 1\}$,
which will be simply denoted by $E_{\Theta+,\Theta_-,\epsilon}.$


\begin{definition}\label{def:lblock} Let $R$ be $\mathcal{O}$ or $k.$
\begin{itemize}[leftmargin=8mm]
  \item[$(1)$] An indecomposable $R\widetilde{G}$-module is called \emph{quadratic unipotent}
if it lies in a quadratic unipotent $\ell$-block of $R\widetilde{G}$.
\item[$(2)$] A  $R\widetilde{G}$-module is called \emph{quadratic unipotent}
if it is a direct sum of indecomposable modules which lies in  quadratic unipotent $\ell$-blocks of $R\widetilde{G}$.
\end{itemize}
\end{definition}

We denote by $$\scrQU^{\O}_R=\bigoplus_{n\in\bbN}R \widetilde{G}_n\qumod$$
the category of quadratic unipotent modules over $R$,
where
$R \widetilde{G}_n\qumod$ is the category of quadratic unipotent $R\widetilde{G}_n$-modules.
As in \S \ref{subsec:SOqublock}, we have
$\bbZ$-linear isomorphisms $d_\scrQU : [\scrQU^{\O}_K]\simto[\scrQU^{\O}_k]$
and $d_{\mathcal{O}\widetilde{G}_n}: [K\widetilde{G}_n\qumod]\simto[k \widetilde{G}_n\qumod]$.
Since the idempotents of $R\widetilde{G}_n$ which cover the idempotent $e_{s_{n_+,n_-}}^{RG_n}$
of $RG_n$ are exactly $e_{s_{n_+,n_-}}^{R\widetilde{G}_n,\,\epsilon}:=e_{s_{n_+,n_-}}^{RG_n}\otimes e_\epsilon$
for $\epsilon\in \{\pm 1\}$, we obtain
$$R\widetilde{G}_n\qumod=\bigoplus_{n_++n_-=n,\,\,\epsilon\in \{\pm 1\}}  R\widetilde{G}_ne_{s_{n_+,n_-}}^{R\widetilde{G}_n,\,\epsilon}\mod.$$


\section{$\mathfrak{A}(\fraks\frakl_{\K_q})$-Representation datum }\label{cha:repdatum}

We first
construct an $\mathfrak{A}(\fraks\frakl_{\K_q})$-representation datum on $\bigoplus_{n\in \bbN^*}RG_n\mod$
in this section,
where $G_n=\O_{2n+1}(q)$, $\Sp_{2n}(q)$ or $\O_{2n}^{\pm}(q)$
with $q$ odd, and $\bbN^*=\mathbb{Z}_{\geqslant 0}$  
except for the $\O_{2n}^{-}(q)$ case $\bbN^*=\mathbb{Z}_{>0}$  (Theorem \ref{thm:repdatumKq}).
The bi-adjoint functors used for it
can also be defined for $\bigoplus_{n\geqslant 0}R{\rm SO}_{2n+1}(q)\mod$,
however, the natural transformations defined in a similar way
fail to satisfy all relations needed.
Essentially inherited from the 
representation datum on $\bigoplus_{n\geqslant 0}R\rm{O}_{2n+1}(q)\mod$,
we can then blockwisely construct an $\mathfrak{A}(\fraks\frakl_{\K_q})$-representation datum on $\scrQU^{\SO}_R$
in \S \ref{sub:explicit-repdatumSO}.

\smallskip

Throughout this section we always assume that $q$ is a power of an odd prime
and $R$ is any commutative domain in which $q(q-1)$ is invertible.

\smallskip

\subsection{Idempotents and functors} \label{sub:ide-funct}
Throughout this section we 
let $\bfG_n$ be one of $\bfSO_{2n+1}$, $\mathbf{O}_{2n+1}$,
 $\bfSp_{2n}$ or $\mathbf{O}_{2n}$,
so that $G_n$ is one of  $\SO_{2n+1}(q)$, $\O_{2n+1}(q)$, $\Sp_{2n}(q)$ or $\O_{2n}^{\pm}(q)$.
As is conventional, we set $\bfG_0=G_0=\{1\}$,  except in the $\bfO_1$ case
$\bfG_0=G_0=\{\pm 1\}$.

\subsubsection{Embedding of Levi subgroups}\label{subsec:Levi}
We first define the embedding of some certain (possibly disconnected) Levi subgroups into $\bfG_{n}$.
Given $r,m, m_1,m_2,\dots,m_t\in\bbN$ such that $m=\sum_im_i$ and $n=r+m$,
we always embed the group $\bfG_{r}\times\bfG\bfL_{m_1}\times\dots\times
\bfG\bfL_{m_t}$  into $\bfG_{n}$ via the map:
\begin{align*}
\bfG_{r}\times\bfG\bfL_{m_1}\times\dots\times \bfG\bfL_{m_t} &\hookrightarrow
\bfG_{n}\\
B \times A_1\times\dots\times A_t&\mapsto {\rm diag}(A_t,\ldots,A_1,B,A_1',\ldots,A_t')
\end{align*}
where
$A_i'=J_{m_i}A_i^{-t}J_{m_i}$.
The above image, which is a Levi subgroup of $\bfG_{n}$, will be denoted by
$\bfL_{r,m_1,\dots,m_t}$, whereas the corresponding Levi subgroup of $G_{n}$
will be denoted by $L_{r,m_1,\dots,m_t}$.
In particular, the  finite groups corresponding to $\bfL_{r,m}$
and $\bfL_{r,1^m}$
  are $L_{r,m}
\simeq G_r \times \GL_m(q)$ and $L_{r,1^m}\simeq G_r \times \GL_1(q)^m$,
 respectively. We will abbreviate
$\bfL_r=\bfL_{r,1}$.
In addition, for $m<n$, we denote
$\iota_{n,m}$ by the embedding of $G_m$ into $G_n$
given by $g\mapsto {\rm diag}(\mathrm{id}_{n-m},g,\mathrm{id}_{n-m})$.

\subsubsection{Idempotents}\label{subsec:idempotent}
Here we define serveral idempotents to obtain bi-adjoint functors that are crucial to
construct representation data.

Let $\bfP_{r}$ be the standard parabolic subgroup of $\bfG_{r+1}$ containing $\bfL_r$
 and $P_r=\bfP_r^F$.
Let $\bfV_r$ be the unipotent radical of $\bfP_{r}$ with $V_r = \bfV_r^F$,
and $U_{r}\subset G_{r+1}$ be the subgroup given by
\begin{align}\label{U}
  U_{r}=V_{r}\rtimes\bbF^\times_{q}=
  \begin{pmatrix}*&*&\cdots&\cdots&*\\
	&1&&(0)&\vdots\\
	&&\ddots&&\vdots\\
	&(0)&&1&*\\
	&&&&*
  \end{pmatrix},
\end{align}
where we identify $\bbF^\times_{q}$ with $\{{\rm diag}(t,1,\ldots,1,t^{-1})\in U_{r}\mid t\in \bbF^\times_{q}\}.$
We define $$e_{r+1,r}:=e_{U_r}=\frac{1}{|U_r|}\sum_{u\in U_r}u,$$ so that
$e_{r+1,r}$  is a central idempotent of $RU_r$.

\vspace{1ex}

We shall find another idempotent of $RU_r$ for our purpose.
Since $\bbF^\times_{q}$ is a cyclic group of  order $q-1$,
we denote by $\zeta$  the unique character of order 2 (Legendre symbol for $\bbF^\times_{q}$)
and by $e_\zeta$ the associated central primitive idempotent of
 $R\bbF^\times_{q}$ given by $$e_\zeta:=\frac{1}{q-1}\sum_{g\in\bbF^\times_{q}}\zeta(g^{-1})g.$$
By the identification $\bbF^\times_{q}$ with $\{{\rm diag}(t,1,\ldots,1,t^{-1})\in U_{r}\mid t\in \bbF^\times_{q}\}$,
we may view $e_\zeta$ as an idempotent of $U_r$.
Now we define $$e'_{r+1,r}:=e_{\zeta} e_{V_r}.$$
The normality of $V_r$ in $U_r$ shows that
$e'_{r+1,r}$ is an idempotent of $RU_r$,
which does not depend on the choice of the
complements of $V_r$ in $U_r$
since all of them are conjugate in $U_r$.

\vspace{1ex}

For each $r<n$, we set $V_{n,r}=V_{n-1}\ltimes\dots\ltimes V_{r}$ and $U_{n,r}=U_{n-1}\ltimes\dots\ltimes U_{r}$.
Let $n=m+r$. We identify
${(\bbF^\times_{q})}^m=(\bbF^\times_{q})_1\times\dots\times(\bbF^\times_{q})_m$
with its image in $U_{n,r}$ under the embedding map
$$h_1\times\dots\times h_m\mapsto \mbox{diag}(h_m,\ldots,h_1, \text{Id}_{G_r}, h_1^{-1},\ldots,h_m^{-1}).$$

Given
   a character $\xi_1\times \cdots\times \xi_m$ of ${(\bbF^\times_{q})}^m=(\bbF^\times_{q})_1\times\dots\times(\bbF^\times_{q})_m,$
  we can
   define the associated central primitive idempotent  $e_{\xi_1\times \cdots\times \xi_m}$ of $R{(\bbF^\times_{q})}^m$ as above.
Viewing it as an idempotent of $U_{n,r}$ or $G_{n}$,
we define
$$\begin{array}{rl}
  e_{n,r}:=&e_{U_{n,r}}  = e_{n, n-1}e_{n-1,n-2}\cdots e_{r+1,r}~\mbox{and} \\
  e'_{n,r}:=& e'_{n, n-1}e'_{n-1,n-2}\cdots e'_{r+1,r},
\end{array}$$
so that $e_{n,r}=e_{1\times\cdots\times1}e_{V_{n,r}}$ is a central idempotent of $RU_{n,r}$,
and $e'_{n,r}=e_{\zeta\times\cdots\times\zeta}e_{V_{n,r}}$ is an idempotent of $RU_{n,r}$.


\subsubsection{Functors}\label{subsec:functor}

Since the modules $RG_{r+1}\cdot e_{{r+1,r}}$ and
  $RG_{r+1}\cdot e'_{{r+1,r}}$ are both left $RG_{r+1}$-projective
  and right $RG_{r}$-projective, by \cite[Corollary 9.2.4]{KZ98} we can define
the following bi-adjoint functors:
$$\begin{array}{rcrclcl}
 F_{r+1,r} &=&RG_{r+1}\cdot e_{{r+1,r}}\otimes_{RG_r} -    &:&\,RG_r\mod&\to &RG_{r+1}\mod, \\
   E_{r,r+1}&=&e_{{r+1,r}}\cdot RG_{r+1}\otimes_{RG_{r+1}}-   &:&\,RG_{r+1}\mod&\to& RG_r\mod,\\
  F'_{r+1,r}&=&RG_{r+1}\cdot e'_{{r+1,r}}\otimes_{RG_r}-     &:&\,RG_r\mod &\to& RG_{r+1}\mod,\\
E'_{r,r+1}  &=&e'_{{r+1,r}}\cdot RG_{r+1}\otimes_{RG_{r+1}}-   &:&\,RG_{r+1}\mod&\to & RG_r\mod.
\end{array}
$$

\begin{remark}\label{rem:functors} Here we would like to give some remarks for the above functors.
\begin{itemize}[leftmargin=8mm]
\item[$(1)$] The functors $F_{r+1,r}$ and $E_{r,r+1}$ are the same as those defined by
             Dudas-Varagnolo-Vasserot \cite{DVV2}.
  \item[$(2)$] The functors $F_{r+1,r}$ and $F'_{r+1,r}$ are understood to be functors of
  $$RG_r\mod\to RL_{r,1}\mod\cong R(G_r\times \bbF^\times_{q})\mod\to RG_{r+1}\mod.$$
  Therefore, for $M\in RG_r\mod$, we have
   $$F_{r+1,r}(M)\cong R_{L_{r,1}}^{G_{r+1}}(M\otimes R_1)~\mbox{and}~F'_{r+1,r}(M)\cong R_{L_{r,1}}^{G_{r+1}}(M\otimes R_\zeta),$$
   where $R_1$ and $R_\zeta$ are the $R\bbF^\times_{q}$-modules
   affording the trivial character  of $\bbF^\times_{q}$ and
   the Legendre symbol $\zeta$, respectively.
   \item[$(3)$]  In the case where $R=K$
   and $\mathbf{G}_n=\bfSO_{2n+1}$,
    the functor $F_{n+1,n}$ maps Lusztig series $\mathcal{E}(G_n,s_{n_+,n_-})$ to
    Lusztig series $\mathcal{E}(G_{n+1},s_{n_++1,n_-})$ and the functor $F_{n+1,n}'$
    maps Lusztig series $\mathcal{E}(G_n,s_{n_+,n_-})$ to
    Lusztig series $\mathcal{E}(G_{n+1},(s_{n_+,n_-+1})),$ i.e.,
    $$\ \ \ \ \ \ [F_{n+1,n}](\mathcal{E}(G_n,(s_{n_+,n_-})))\subset\bbN\mathcal{E}(G_{n+1},(s_{n_++1,n_-})) ~\mbox{and}~$$ $$[F_{n+1,n}'](\mathcal{E}(G_n,(s_{n_+,n_-})))\subset\bbN\mathcal{E}(G_{n+1},(s_{n_+,n_-+1}))$$
  (see \S \ref{subsec:quchar} for the meaning of the notation).
   Now let $E_{\Theta_+,\Theta_-}$ be as defined in \S\ref{subsec:quchar}.
Then we have $$F_{n+1,n}(E_{\Theta_+\times\Theta_-})\cong R_{L_{n,1}}^{G_{n+1}}
(E_{\Theta_+\times\Theta_-}\otimes K_1)\cong\bigoplus E_{\Lambda_+\times\Theta_-}$$
where the summand runs over all
$\Lambda_+$ such that $\Lambda_+$ is obtained from $\Theta_+$ by adding a $1$-hook,
and
$$F_{n+1,n}'(E_{\Theta_+\times\Theta_-})\cong R_{L_{n,1}}^{G_{n+1}}
(E_{\Theta_+\times\Theta_-}\otimes K_\zeta)\cong\bigoplus
E_{\Theta_+\times\Lambda_-}$$ where the summand runs over all $\Lambda_-$ such that $\Lambda_-$ is obtained from $\Theta_-$ by adding a $1$-hook.
\end{itemize}
\end{remark}

\subsection{Reflections and natural transformations}\label{sec:repdatumN} 
Here we define natural transformations of the functors in \S \ref{subsec:functor} via
the lifts of simple reflections.

\subsubsection{Lifts of simple reflections} \label{sub:lift-refl}
Let $\mathbf{G}_n$ be one of  $\bfSO_{2n+1}$, $\mathbf{O}_{2n+1}$, $\bfSp_{2n}$ or $\bfO_{2n}$.
Let $\bfT$ be the split maximal torus consisting of diagonal matrices of the connected component $\mathbf{G}_n^\circ$
of $\mathbf{G}_n$.  Except for $\mathbf{O}_{2n+1}$,
the  group $W_n:=N_{\mathbf{G}_{n}}(\bfT)/{\bfT}$ is a Weyl group of
type $B_n.$ 
Here we do not consider
$\bfSO_{2n}$ since they have a Weyl group of type $D_n$.
The numbering of the simple reflections of $W_n$ will be taken
the following convention:
\begin{align}\label{B}
\begin{split}
\begin{tikzpicture}[scale=.4]
\draw[thick] (0 cm,0) circle (.3cm);
\node [below] at (0 cm,-.5cm) {$s_1$};
\draw[thick] (0.3 cm,-0.15cm) -- +(1.9 cm,0);
\draw[thick] (0.3 cm,0.15cm) -- +(1.9 cm,0);
\draw[thick] (2.5 cm,0) circle (.3cm);
\node [below] at (2.5 cm,-.5cm) {$s_2$};
\draw[thick] (2.8 cm,0) -- +(1.9 cm,0);
\draw[thick] (5 cm,0) circle (.3cm);
\draw[thick] (5.3 cm,0) -- +(1.9 cm,0);
\draw[thick] (7.5 cm,0) circle (.3cm);
\draw[dashed,thick] (7.8 cm,0) -- +(4.4 cm,0);
\draw[thick] (12.5 cm,0) circle (.3cm);
\node [below] at (12.5 cm,-.5cm) {$s_{n-1}$};
\draw[thick] (12.8 cm,0) -- +(1.9 cm,0);
\draw[thick] (15 cm,0) circle (.3cm);
\node [below] at (14.8 cm,-.5cm) {$s_n$};
\end{tikzpicture}
\end{split}
\end{align}

When $\bfG_n=\bfSO_{2n+1}$, we have $\bfT = \{ \mathrm{diag}(x_n,
\ldots,x_1,1,x_1^{-1},\ldots,x_n^{-1})\mid x_i\in \overline{\bbF}_q^{\times}$ for all $1\leqs i\leqs n\}$.
For $i > 1$, the action of $s_i$ on $\mathrm{diag}(x_n,\ldots,x_1,1,x_1^{-1},\ldots,x_n^{-1})$
interchanges $x_{i-1}$ and $x_i$, whereas $s_1$ interchanges $x_1$ and $x_1^{-1}$.
So the simple reflection $s_i (i>1)$ can be lifted to $N_{\bfG_n}(\bfT)$ as the signed permutation matrix
$\dot{s}_i=\diag\big(\id_{n-i},\left(\begin{array}{cc} 0&1 \\ -1& 0 \\ \end{array}\right),\id_{\bfG_{i-2}},\left(\begin{array}{cc} 0&  -1 \\1& 0 \\ \end{array}\right),\id_{n-i}\big)$ and $s_1$ as the signed permutation matrix
\begin{equation*}
\dot{s}_1=\ \begin{blockarray}{(ccc|ccc|ccc)}
         \BAmulticolumn{3}{c|}{\multirow{2}{*}{$-\Id_{n-1}$}}&&&&&&\\
         &&&&&&&&&\\
         \cline{1-9}
        &&&0&0&-1&&&&\\
        &&&0&-1&0&&&&\\
        &&&-1&0&0&&&&\\
        \cline{1-9}
        &&&&&&\BAmulticolumn{3}{c}{\multirow{2}{*}{$-\Id_{n-1}$}}\\
        &&&&&&&&&\\
    \end{blockarray}\,\, .
\end{equation*}

When $\bfG_n=\bfO_{2n+1},$  we write
$\widehat{\bfT} = \{{\rm diag}(x_n,
\ldots,x_1,\pm1,x_1^{-1},\ldots,x_n^{-1})\mid x_i\in \overline{\bbF}_q^{\times}~\text{for~all}~1\leqs i\leqs n\}$ and $W_n= N_{\bfG_n}(\widehat{\bfT})/\widehat{\bfT}.$
It is easy to see that  $\widehat{\bfT}=C_{\bfG_n}(\bfT)$ and $W_n\cong N_{\bfG^{\circ}_{n}}(\bfT)/\bfT$ is also a Weyl group of
type $B_n$ on which $F$ acts trivially.
The action of $s_i$ on $\bfT$ for each $i > 1$ is
the same as in the $\bfSO_{2n+1}$ case,
and we lift them in the same way as above. However, the simple reflection $s_1$
is lifted to be the signed permutation matrix
\begin{equation*}
\dot{s}_1=\ \begin{blockarray}{(ccc|ccc|ccc)}
         \BAmulticolumn{3}{c|}{\multirow{2}{*}{$\Id_{n-1}$}}&&&&&&\\
         &&&&&&&&&\\
         \cline{1-9}
        &&&0&0&-1&&&&\\
        &&&0&1&0&&&&\\
        &&&-1&0&0&&&&\\
        \cline{1-9}
        &&&&&&\BAmulticolumn{3}{c}{\multirow{2}{*}{$\Id_{n-1}$}}\\
        &&&&&&&&&\\
    \end{blockarray}\,\, .
\end{equation*}

When $\bfG_n=\bfSp_{2n}$, we have $\bfT = \{ \mathrm{diag}(x_n,
\ldots,x_1,x_1^{-1},\ldots,x_n^{-1})\mid x_i\in \overline{\bbF}_q^{\times}$ for all $1\leqs i\leqs n\}$.
Similarly,
the simple reflection $s_i (i>1)$ can be lifted to $N_{\bfG_n}(\bfT)$ as the signed permutation matrix
$\dot{s}_i=\ \diag(\id_{n-i},\left(\begin{array}{cc} 0&1 \\ -1& 0 \\ \end{array}\right),\id_{\bfG_{i-2}},\left(\begin{array}{cc} 0&  -1 \\1& 0 \\ \end{array}\right),\id_{n-i})$ and $s_1$ as the signed permutation matrix
\begin{equation*}
\dot{s}_1=\ \begin{blockarray}{(cccc|cc|cccc)}
         \BAmulticolumn{4}{c|}{\multirow{2}{*}{$\Id_{n-1}$}}&&&&&\\
         &&&&&&&&&\\
         \cline{1-10}
        &&&&0&-1&&&&\\
        &&&&1&0&&&&\\
        \cline{1-10}
        &&&&&&\BAmulticolumn{3}{c}{\multirow{2}{*}{$\Id_{n-1}$}}\\
        &&&&&&&&&
    \end{blockarray}\ .
\end{equation*}

When $\bfG_{n}=\bfO_{2n}$, we have $\bfT = \{ \mathrm{diag}(x_n,
\ldots,x_1,x_1^{-1},\ldots,x_n^{-1})\mid x_i\in \overline{\bbF}_q^{\times}$ for all $1\leqs i\leqs n\}$.
%
For $i \neq 1$, the simple
reflection $s_i$ can be lifted to $N_{\bfG_n}(\bfT)$ as the signed permutation matrix
$\dot{s}_i=\ \diag(\id_{n-i},\left(\begin{array}{cc} 0&1 \\ -1& 0 \\ \end{array}\right),\id_{\bfG_{i-2}},\left(\begin{array}{cc} 0&  -1 \\1& 0 \\ \end{array}\right),\id_{n-i})$ and $s_1$ as the signed permutation matrix
\begin{equation*}
\dot{s}_1=\ \begin{blockarray}{(cccc|cc|cccc)}
         \BAmulticolumn{4}{c|}{\multirow{2}{*}{$\Id_{n-1}$}}&&&&&\\
         &&&&&&&&&\\
         \cline{1-10}
        &&&&0&-1&&&&\\
        &&&&-1&0&&&&\\
        \cline{1-10}
        &&&&&&\BAmulticolumn{3}{c}{\multirow{2}{*}{$\Id_{n-1}$}}\\
        &&&&&&&&&
    \end{blockarray}\ .
\end{equation*}

For the next section,
we now define
$t_1:=s_1$ and
 $$t_{k}:=s_{k}s_{k-1} \cdots s_2 t_1 s_2^{-1} \cdots s^{-1}_{k-1} s^{-1}_{k}\,\,\text{for}\,\, 2\leqs k\leqs n.$$
Correspondingly, we define their lifts by $\dot{t}_1=\dot{s}_1$ and  $$\dot{t}_{k}:=\dot{s}_{k}\dot{s}_{k-1} \cdots \dot{s}_2 \dot{t}_1 \dot{s}_2^{-1} \cdots \dot{s}^{-1}_{k-1} \dot{s}^{-1}_{k}\,\,\text{for}\,\, 2\leqs k\leqs n.$$
 It is easy to check
\begin{equation}\label{equ:braid}
\begin{aligned}
&\dot{s}_i\dot{s}_j=\dot{s}_j\dot{s}_i,\,\text{for all }\,\,j\neq i\pm1,\text{and}\,\,i,j\geqs 1\\
&\dot{s}_{i}\dot{s}_{i+1}\dot{s}_{i}=\dot{s}_{i+1}\dot{s}_{i}\dot{s}_{i+1},\,\text{for all }\,\,2\leqs i\leqs n-1,\\
&\dot{t}_i\dot{t}_j=\dot{t}_j\dot{t}_i,\,\text{for all }\,\,1\leqs i,j\leqs n,\\
&\dot{t}_i\dot{s}_j=\begin{cases}
\dot{s}_i\dot{t}_{i-1} & \text{ if }j=i, \\
\dot{s}_{i+1}\dot{t}_{i+1} & \text{ if }j=i+1, \\
\dot{s}_j\dot{t}_i & \text{ otherwise.}
\end{cases}
\end{aligned}
\end{equation}

\subsubsection{Natural transformations} \label{sub:naturaltransf}

Throughout this section we let $G_n=\O_{2n+1}(q)$, $\Sp_{2n}(q)$ and
$\O_{2n}^{+}(q)$ for $n\geqs 0$ or let  $G_n=\O_{2n}^{-}(q)$
for $n\geqs 1$. Notice that $G_n$ has the form $I(V,\beta)$ in Table \ref{tb:classical}
for each case.
\smallskip

An endomorphism of the functor $F_{r+1,r}$  (resp. $F_{r+1,r}'$) can be represented by an $(RG_{r+1},RG_r)$-bimodule
endomorphism of $RG_{r+1} e_{{r+1,r}}$ (resp. $RG_{r+1} e'_{{r+1,r}}$), or equivalently, by (right multiplication by) an element of $e_{{r+1,r}}
RG_{r+1} e_{{r+1,r}}$ (resp. $e_{{r+1,r}}' RG_{r+1} e_{{r+1,r}}'$ ) centralizing $RG_r$.
Let $\dot{t}_{r+1}$ and $\dot{s}_{r+2}$ be as defined in \S \ref{sub:lift-refl}.
It is easy to see that both $\dot{t}_{r+1}\in G_{r+1}$ and  $\dot{s}_{r+2}\in G_{r+2}$ centralize $G_r$.
In fact,
$\dot{t}_{r+1}$ 
 is exactly the matrix
  $$
  \left(\begin{array}{ccc} & & -1 \\ & \Id_{G_r} & \\ -1 \\ \end{array}\right),
 \quad \left(\begin{array}{ccc} & & -1 \\ & \Id_{G_r} & \\ 1 \\ \end{array}\right)
  \quad\text{or} \quad \left(\begin{array}{ccc} & & -1 \\ & \Id_{G_r} & \\ -1 \\ \end{array}\right)$$
when $G_n=\O_{2n+1}(q), \Sp_{2n}(q)$ or  $\O_{2n}^{\pm}(q)$, respectively.

Thus, the elements
\begin{align}
& X_{r+1,r}= q^r e_{{r+1,r}} \dot{t}_{r+1} \, e_{{r+1,r}},\qquad
T_{r+2,r}= q  e_{{r+2,r}} \dot{s}_{r+2}\,e_{{r+2,r}}
\end{align}
define  natural transformations of the functors $F_{r+1,r}$ and $F_{r+2,r+1}F_{r+1,r}$, respectively.
Similarly, the elements
\begin{align}
& X'_{r+1,r}= q^r e'_{{r+1,r}} \dot{t}_{r+1} \, e'_{{r+1,r}},\qquad
T'_{r+2,r}=q e'_{{r+2,r}} \dot{s}_{r+2}\,e'_{{r+2,r}}
\end{align}
define natural transformations of the functors $F'_{r+1,r}$ and $F'_{r+2,r+1}F'_{r+1,r}$,  respectively.

\smallskip
Observe that the functors $F_{r+2,r+1} F'_{r+1,r}$ and $F'_{r+2,r+1} F_{r+1,r}$
 are
represented by the $(RG_{r+2},RG_{r})$-bimodules
 $RG_{r+2} e_{r+2,r+1}e'_{r+1,r}$ and $RG_{r+2}e'_{r+2,r+1}e_{r+1,r}$
 with idempotents $e_{r+2,r+1}e'_{r+1,r}=e_{\zeta\times1}e_{V_{r+2,r}}$
 and $e'_{r+2,r+1}e_{r+1,r}=e_{1\times \zeta }e_{V_{r+2,r}}$, respectively.
In addition,
we define two natural transformations of
 $$\Hom(F_{r+2,r+1}F'_{r+1,r},F'_{r+2,r+1}F_{r+1,r})~{\rm and}~\Hom(F'_{r+2,r+1}F_{r+1,r},F_{r+2,r+1}F'_{r+1,r})$$
 which are respectively:
\begin{align}
&H_{r+2,r}=qe_{{r+2,r+1}}e'_{r+1,r} \dot{s}_{r+2}\,e'_{{r+2,r+1}}e_{r+1,r},\
H_{r+2,r}'=e'_{{r+2,r+1}}e_{r+1,r} \dot{s}^{-1}_{r+2}\,e_{{r+2,r+1}}e'_{r+1,r}.
\end{align}

\subsubsection{Relations of the natural transformations}\label{subsec:repdatumN}
Throughout this section, let $G_n=\O_{2n+1}(q)$, $\Sp_{2n}(q)$ or $\O_{2n}^{\pm}(q)$.
We set
$$E=\bigoplus_{r\in \bbN^*}E_{r,r+1},\quad F=\bigoplus_{r\in \bbN^*}F_{r+1,r},\quad
X=\bigoplus_{r\in \bbN^*}X_{r+1,r},\quad T=\bigoplus_{r\in \bbN^*}T_{r+2,r},$$
$$ E'=\bigoplus_{r\in \bbN^*}E'_{r,r+1},\quad F'=\bigoplus_{r\in \bbN^*}F'_{r+1,r},
\quad X'=\bigoplus_{r\in \bbN^*}X'_{r+1,r},
\quad
T'=\bigoplus_{r\in \bbN^*}T'_{r+2,r},$$
and
$$H=\bigoplus_{r\in \bbN^*}H_{r+2,r},\qquad
H'=\bigoplus_{r\in \bbN^*}H'_{r+2,r},$$
where $\bbN^*=\bbN$
except for the $\O_{2n}^{-}(q)$ case $\bbN^*=\bbN\backslash \{0\}$.
\begin{proposition}\label{prop:12relations}
The endomorphisms $X\in\End(F)$, $T\in\End(F^2)$, $X'\in\End(F')$, $T'\in\End((F')^2)$, $H\in \Hom(FF',F'F)$ and $H'\in \Hom(F'F,FF')$ satisfy the following relations:
\begin{itemize}[leftmargin=8mm]

  \item[$\mathrm{(a)}$] $(1_FT)\circ (T1_F)\circ (1_FT)=(T1_F)\circ (1_FT)\circ (T1_F)$,
  \item[$\mathrm{(b)}$] $(T+1_{F^2})\circ(T-q1_{F^2})=0$,
  \item[$\mathrm{(c)}$] $T\circ(1_FX)\circ T=qX1_F,$
  \item[$\mathrm{(d)}$] $(1_{F'}T')\circ (T'1_{F'})\circ (1_{F'}T')=(T'1_F')\circ (1_{F'}T')\circ (T'1_{F'})$,
  \item[$\mathrm{(e)}$] $(T'+1_{(F')^2})\circ(T'-q1_{(F')^2})=0$,
  \item[$\mathrm{(f)}$] $T'\circ(1_{F'}X')\circ T'=qX'1_{F'}$,
  \item[$\mathrm{(g)}$] $(H'1_{F})\circ (1_{F'}T)\circ (H1_{F})=(1_{F}H)\circ(T1_{F'})\circ(1_{F}H')$,
  \item[$\mathrm{(h)}$] $(H1_{F'})\circ (1_{F}T')\circ (H'1_{F'})=(1_{F'}H')\circ(T'1_{F})\circ(1_{F'}H)$,
  \item[$\mathrm{(i)}$] $HH'=1_{F'F}$,
  \item[$\mathrm{(j)}$] $H'H=1_{FF'}$,
  \item[$\mathrm{(k)}$] $H\circ (1_{F}X')=(X'1_{F})\circ H,$

  \item[$\mathrm{(l)}$] $H'\circ (1_{F'}X)=(X1_{F'})\circ H'$.
\end{itemize}
\end{proposition}

\begin{proof} Parts (a)-(c) have been shown in \cite[Proposition 6.1]{DVV2}.

\vspace{1ex}

For (d), note that $1_{F'_{r+3,r+2}}T'_{r+2,r}$ and $T'_{r+3,r+1}1_{F'_{r+1,r}}$
are given by right multiplication by the elements
$qe'_{r+3,r}\dot{s}_{r+2}e'_{r+3,r}$ and $qe'_{r+3,r}\dot{s}_{r+3}e'_{r+3,r}$ on the $(RG_{r+3}, RG_r)$-bimodule $RG_{r+3}e'_{r+3,r}$, respectively.
So, in order to show (d), it is equivalent to show the equality
$$e'_{r+3,r}\dot{s}_{r+2}e'_{r+3,r}\dot{s}_{r+3}e'_{r+3,r}\dot{s}_{r+2}e'_{r+3,r}=
e'_{r+3,r}\dot{s}_{r+3}e'_{r+3,r}\dot{s}_{r+2}e'_{r+3,r}\dot{s}_{r+3}e'_{r+3,r} $$
or
$$e_{\zeta^{\times 3}}e_{V_{r+3,r}}\dot{s}_{r+2}e_{V_{r+3,r}}\dot{s}_{r+3}e_{V_{r+3,r}}\dot{s}_{r+2}e_{V_{r+3,r}}
=e_{\zeta^{\times 3}}
e_{V_{r+3,r}}\dot{s}_{r+3}e_{V_{r+3,r}}\dot{s}_{r+2}e_{V_{r+3,r}}\dot{s}_{r+3}e_{V_{r+3,r}}$$
since  $e'_{r+3,r}=e_{\zeta^{\times 3}}e_{V_{r+3,r}}$ and the idempotent
$e_{\zeta^{\times 3}}$ is invariant under the conjugate action of $s_{r+2}$ and $s_{r+3}.$
(Here $\zeta^{\times3}:=\zeta\times \zeta\times\zeta$.)

\smallskip

However, the sets $$V_{r+3,r}\dot{s}_{r+2}V_{r+3,r}\dot{s}_{r+3}V_{r+3,r}\dot{s}_{r+2}V_{r+3,r}$$
and $$V_{r+3,r}\dot{s}_{r+3}V_{r+3,r}\dot{s}_{r+2}V_{r+3,r}\dot{s}_{r+3}V_{r+3,r}$$ are both equal to
$V_{r+3,r}\dot{s}_{r+2}\dot{s}_{r+3}\dot{s}_{r+2}V_{r+3,r}$ by the braid relations,
which proves (d).

\vspace{2ex}

For (e), let $s=\diag(-1,-1,\id_{G_r},-1,-1)$. It is easy to see that $\dot{s}_{r+2}=\dot{s}^{-1}_{r+2}\cdot s$ and
$e_{\zeta\times\zeta}\cdot s=\zeta(-1)^2e_{\zeta\times\zeta}=
e_{\zeta\times\zeta}.$ So we have
$$T'_{r+2,r}= q e'_{{r+2,r}} \dot{s}_{r+2}\,e'_{{r+2,r}}=q e'_{{r+2,r}} \dot{s}^{-1}_{r+2}\,e'_{{r+2,r}}.$$
 Hence
 \begin{align*}
  (T'_{r+2,r})^2&=q^2 e'_{{r+2,r}}\dot{s}_{r+2}e'_{{r+2,r}}\dot{s}^{-1}_{r+2}e'_{{r+2,r}}\\
  &=q^2 e'_{{r+2,r}}\,e_V\,e'_{{r+2,r}},\\
  &=q^2 e_{\zeta\times \zeta}e_{V_{r+2,r}}\,e_V\,e_{V_{r+2,r}},
\end{align*}
where $V$  is the subgroup of $G_{r+2}$ consisting of the matrices with diagonal
entries equal to 1 and off-diagonal entries equal to zero, except for the entries
$(2,1)$ and $(2r+4,2r+3)$ ($(2r+5,2r+4)$ in the $\O_{2n+1}(q)$ case),
i.e.,
  $$V=\begin{pmatrix}1&&&&\\ *&1&&&\\&&\id_{G_r}&&\\&&&1&\\&&&*&1\end{pmatrix}.$$

Observe that the group $V=u_{-\alpha}(\bbF_{q})$ is the root subgroup of $G_{r+2}$ associated
with some negative root $-\alpha$, and
the corresponding simple reflexion $s_\alpha$
is exactly $s_{r+2}$.  Let $U_\alpha$ be the (finite) root
subgroup of unitriangular matrices in the copy of $\GL_2$ in $G_{r+2}$ associated
with $\alpha$. We have $u_{\alpha}(t)u_{-\alpha}(-t^{-1})u_{\alpha}(t)=n_{\alpha}(t)$
in $\GL_2$ for $t\neq 0$, where
 $n_\alpha(t)=\left(\begin{array}{cc} 0&t  \\ -{t^{-1}} &0 \\ \end{array}\right)=\diag(t,t^{-1})\dot{s}_{\alpha}$
 and $\dot{s}_{\alpha}=\left(\begin{array}{cc} 0&1  \\ -1 &0 \\ \end{array}\right).$

Hence we obtain  $$U_{\alpha}u_{-\alpha}(-t^{-1})U_{\alpha}=
U_{\alpha}n_{\alpha}(t)U_{\alpha}=\diag(t,t^{-1})U_{\alpha}\dot{s}_{\alpha}U_{\alpha}.$$
Note that $e_{\zeta\times \zeta} \diag(t,t^{-1})=\zeta(t)\zeta(t^{-1})e_{\zeta\times \zeta}=e_{\zeta\times \zeta}$
for all $t\in \bbF_q^{\times}$
with the identification of $(\bbF_q^{\times})^2$ and the standard maximal torus of $\GL_2(q).$
Hence $$e_{\zeta\times\zeta}e_{U_{\alpha}}e_{U_{-\alpha}}e_{U_{\alpha}}=
\frac{1}{q}e_{\zeta\times\zeta}e_{U_{\alpha}}+\frac{q-1}{q}e_{\zeta\times\zeta
}e_{U_{\alpha}}\dot{s}_{\alpha}e_{U_{\alpha}}.$$

Now, let $B_\alpha$ be the (finite) Borel subgroup of upper triangular matrices
in the copy of $\GL_2$ in $G_{r+2}$ associated with $\alpha.$ Then
the image of $B_\alpha$ through the embedding $\GL_2\subset G_{r+2}$
lies in $U_{r+2,r}$. As a consequence, $$q^2 e_{\zeta\times \zeta}e_{V_{r+2,r}}\,e_V\,e_{V_{r+2,r}}=(q-1)qe_{\zeta\times \zeta}e_{V_{r+2,r}}\dot{s}_{r+2}e_{V_{r+2,r}}+qe_{\zeta\times \zeta}e_{V_{r+2,r}}.$$Thus,
$$(T'_{r+2,r})^2=(q-1)\,T'_{r+2,r}+q\,e'_{{r+2,r}},$$ as desired.

\vspace{2ex}

For (f), we similarly argue as for \cite[Proposition 4.1 (c)]{DVV}. 
The image of $\dot{t}_{r+1}\in G_{r+1}$ under the embedding map
$\iota_{r+2,r+1}$ defined in \S \ref{subsec:Levi} is exactly
the element of $G_{r+2}$ with the same notation $\dot{t}_{r+1}$.
(This is not true in the $\SO_{2n+1}(q)$ case, which is the reason to exclude this case.)
Hence
$$
\begin{array}{rclr}
  T'\circ(1_{F'}X')\circ T' &=&q^{r+2} e'_{r+2,r}\dot{s}_{r+2}e'_{r+2,r}\iota_{r+2,r+1}(\dot{t}_{r+1})e'_{r+2,r} \dot{s}^{-1}_{r+2}e'_{r+2,r} & \\
   &=&q^{r+2} e'_{{r+2,r}}e_V\dot{s}_{r+2}\dot{t}_{r+1} \dot{s}^{-1}_{r+2}\,e_Ve'_{r+2,r} & \ \ ~~(\text{in}~G_{r+2})\\
  &=& q^{r+2} e_{\zeta\times\zeta}e_{V_{r+2,r}}e_V\dot{t}_{r+2}\,e_Ve_{V_{r+2,r}} &
 \end{array}
 $$

Let $V'=^{t_{r+2} }\!V\subset V_{r+2,r}$ so that
$$V'=\begin{pmatrix}1&&&*&\\ &1&&&*\\&&\id_{G_r}&&\\&&&1&\\&&&&1\end{pmatrix},$$
where the off-diagonal entries in $V'$ which are possibly
non-zero are only the entries $(2,2r+4)$ and $(1,2r+3)$ ($(2,2r+5)$ and $(1,2r+4)$ in the $\O_{2n+1}(q)$ case).
By Chevalley's commutator formula, we obtain $[V',V]\subset V_{r+2,r}$, which implies
$$e_V\dot{t}_{r+2}\,e_Ve_{V_{r+2,r}}
=\dot{t}_{r+2}\,e_V e_{V_{r+2,r}}=e_{V'}\dot{t}_{r+2}\,e_{V_{r+2,r}}.$$
Thus, by moving $e_V$ to the left, we deduce that
$$T'_{r+2,r}X'_{r+1,r}T'_{r+2,r}=q^{r+2} e_{\zeta\times\zeta}e_{V_{r+2,r}}\dot{t}_{r+2}\,e_{V_{r+2,r}}=qX'_{r+2,r+1},$$
from which (f) follows.

\vspace{2ex}

For (g), we first mention the following presentation:\\
$\begin{array}{lrcl}
\ \ \ \ (1) &  1_{F'_{r+3,r+2}}T_{r+2,r} &\leftrightarrow& e'_{r+3,r+2}e_{r+2,r}\dot{s}_{r+2}e'_{r+3,r+2}e_{r+2,r}, \\
\ \ \ \ (2) &  T_{r+3,r+1}1_{F'_{r+1,r}} &\leftrightarrow& e_{r+3,r+1}e'_{r+1,r}\dot{s}_{r+3}e_{r+3,r+1}e'_{r+1,r},\\
\ \ \ \ (3) &   H_{r+3,r+1}1_{F_{r+1,r}}  &\leftrightarrow& e_{r+3,r+2}e'_{r+2,r+1}e_{r+1,r}\dot{s}_{r+3}e'_{r+3,r+2}e_{r+2,r}, \\
\ \ \ \ (4) &   1_{F_{r+3,r+2}}H_{r+1,r}  &\leftrightarrow& e_{r+3,r+1}e'_{r+1,r}\dot{s}_{r+2}e_{r+3,r+2}e'_{r+2,r+1}e_{r+1,r}, \\
\ \ \ \ (5) &   H'_{r+3,r+1}1_{F_{r+1,r}} &\leftrightarrow& e'_{r+3,r+2}e_{r+2,r}\dot{s}_{r+3}e_{r+3,r+2}e'_{r+2,r+1}e_{r+1,r},~\mbox{and} \\
\ \ \ \ (6) &   1_{F_{r+3,r+2}}H'_{r+1,r} &\leftrightarrow& e_{r+3,r+2}e'_{r+2,r+1}e_{r+1,r}\dot{s}_{r+2}e_{r+3,r+1}e'_{r+1,r}.
\end{array}
$\\
So, in order to prove (g), it is equivalent to show the equality
\begin{align*}
&e_{r+3,r+2}e'_{r+2,r+1}e_{r+1,r}\dot{s}_{r+3}e'_{r+3,r+2}e_{r+2,r}\dot{s}_{r+2}e'_{r+3,r+2}e_{r+2,r}\dot{s}_{r+3}e_{r+3,r+2}e'_{r+2,r+1}e_{r+1,r}\\
  =&e_{r+3,r+2}e'_{r+2,r+1}e_{r+1,r}\dot{s}_{r+2}e_{r+3,r+1}e'_{r+1,r}\dot{s}_{r+3}e_{r+3,r+1}e'_{r+1,r}\dot{s}_{r+2}e_{r+3,r+2}e'_{r+2,r+1}e_{r+1,r}.
 \end{align*}
Since $e_{r+3,r+2}e'_{r+2,r+1}e_{r+1,r}=e_{1\times\zeta\times 1}e_{V_{r+3,r}}$, $e'_{r+3,r+2}e_{r+2,r}=e_{1\times 1\times \zeta}e_{V_{r+3,r}}$ and $$e_{r+3,r+1}e'_{r+1,r}=e_{ \zeta\times 1\times 1}e_{V_{r+3,r}},$$
it suffices to show that
\begin{align*}
&e_{ 1\times \zeta \times 1}e_{V_{r+3,r}}\dot{s}_{r+3}e_{ 1\times 1\times \zeta }
     e_{V_{r+3,r}}\dot{s}_{r+2}e_{ 1\times 1\times\zeta }e_{V_{r+3,r}}\dot{s}_{r+3}e_{ 1\times \zeta \times 1}e_{V_{r+3,r}}\\
  =&e_{ 1\times \zeta \times 1}e_{V_{r+3,r}}\dot{s}_{r+2}e_{\zeta\times 1\times 1 }
  e_{V_{r+3,r}}\dot{s}_{r+3}e_{ \zeta\times 1\times 1 }e_{V_{r+3,r}}\dot{s}_{r+2}e_{ 1\times \zeta \times 1}e_{V_{r+3,r}},
 \end{align*}
namely,
\begin{align*}
&e_{1\times\zeta\times 1}e_{V_{r+3,r}}\dot{s}_{r+2}e_{V_{r+3,r}}\dot{s}_{r+3}e_{V_{r+3,r}}\dot{s}_{r+2}e_{V_{r+3,r}}\\
  =&e_{1\times\zeta\times 1}
e_{V_{r+3,r}}\dot{s}_{r+3}e_{V_{r+3,r}}\dot{s}_{r+2}e_{V_{r+3,r}}\dot{s}_{r+3}e_{V_{r+3,r}}.
 \end{align*}
However, the above equality is true by the end of proof in (d),
and thus (g) holds.

\vspace{1ex}

The proof of (h) is similar to that of (g).

\vspace{1ex}

To prove (i), we first compute
\begin{align*}
H_{r+2,r}H'_{r+2,r}&= qe'_{{r+2,r+1}}e_{r+1,r}\dot{s}^{-1}_{r+2}e_{{r+2,r+1}}e'_{r+1,r}\dot{s}_{r+2}e'_{{r+2,r+1}}e_{r+1,r}\\
  &= qe_{1\times\zeta}e_{V_{r+2,r}}\dot{s}^{-1}_{r+2}e_{\zeta\times 1}e_{V_{r+2,r}}\dot{s}_{r+2}e_{1\times\zeta}e_{V_{r+2,r}}\\
  &= qe_{1\times\zeta}e_{V_{r+2,r}}\dot{s}^{-1}_{r+2}e_{V_{r+2,r}}\dot{s}_{r+2}e_{V_{r+2,r}}\\
  &= qe_{1\times\zeta}e_{V_{r+2,r}}\,e_V\,e_{V_{r+2,r}}.
\end{align*}

As in the proof of (e), we identify $(\bbF_q^{\times})^2$ with the standard maximal torus of $\GL_2(q)$, so that
$$e_{1\times\zeta} \diag(t,t^{-1})=\begin{cases}
e_{1\times\zeta}&\text{if $t$ is a square in $\bbF_{q}^{\times}$},\\
-e_{1\times\zeta}&\text{if $t$ is not a square in $\bbF_{q}^{\times}$}.
\end{cases}$$
We have $$U_{\alpha}u_{-\alpha}(-t^{-1})U_{\alpha}=U_{\alpha}n_{\alpha}(t)U_{\alpha}=
\diag(t,t^{-1})U_{\alpha}\dot{s}_{\alpha}U_{\alpha},$$
and
 \begin{align*}e_{U_{\alpha}}e_{U_{-\alpha}}e_{U_{\alpha}}&=
\sum\limits_{t\in \bbF_q^{\times}}(\frac{1}{q}e_{U_{\alpha}}
u_\alpha(-t^{-1})e_{U_{\alpha}})+\frac{1}{q}e_{U_{\alpha}}\\
&=\sum\limits_{t\in \bbF_q^{\times}}(\frac{1}{q}\diag(t,t^{-1})e_{U_{\alpha}}\dot{s}_\alpha e_{U_{\alpha}})+\frac{1}{q}e_{U_{\alpha}}.
\end{align*}
Since there are $\frac{q-1}{2}$'s squares and $\frac{q-1}{2}$'s non-squares in $\bbF_{q}^{\times}$,
it follows that
\begin{align*}
e_{1\times\zeta}e_{U_{\alpha}}e_{U_{-\alpha}}e_{U_{\alpha}}&=
e_{1\times \zeta}\big(\sum\limits_{t\in \bbF_q^{\times}}(\frac{1}{q}\diag(t,t^{-1})e_{U_{\alpha}}\dot{s}_\alpha e_{U_{\alpha}})+\frac{1}{q}e_{U_{\alpha}}\big)\\
&=\frac{1}{q}e_{1\times\zeta}e_{U_{\alpha}}.
\end{align*}
Now, as in the proof of (e), we deduce that
$$H_{r+2,r}H'_{r+2,r}=e_{1\times \zeta}e_{V_{r+2,r}}=e'_{{r+2,r+1}}e_{r+1,r}=1_{F'F},$$
which proves (i).

\vspace{1ex}

The proof of (j) is similar to that of (i).

\vspace{1ex}

Now we prove (k). By (i), it is equivalent to show that $H\circ (1_{F}X')\circ H'=(X'1_{F})$.
The left side of the equality is equal to
$$q^{r+1} e'_{{r+2,r+1}}e_{r+1,r}\dot{s}_{r+2} e_{{r+2,r+1}}e'_{r+1,r}\iota_{r+2,r+1}(\dot{t}_{r+1})  e_{{r+2,r+1}}e'_{r+1,r} \dot{s}^{-1}_{r+2}\,e'_{{r+2,r+1}}e_{r+1,r},$$
 which simplifies to $$q^{r+1} e_{1\times \zeta}e_{V_{r+2,r}}e_V\dot{t}_{r+2}\,e_Ve_{V_{r+2,r}}.$$
As did in the proof of (f), we have
$$e_{V_{r+2,r}}e_V\dot{t}_{r+2}\,e_Ve_{V_{r+2,r}}=e_{V_{r+2,r}}\dot{t}_{r+2}e_{V_{r+2,r}}.$$
Thus,
$$\begin{array}{rcl}
    H_{r+2,r}(1_{F_{r+2,r+1}}X'_{r+1,r})H'_{r+2,r} & = & q^{r+1}e_{1\times\zeta}e_{V_{r+2,r}}\dot{t}_{r+2}e_{V_{r+2,r}} \\
     & = & X'_{r+2,r+1}1_{F_{r+1,r}},
  \end{array}
$$
and so (k) holds.

\vspace{1ex}

Finally, (l) holds by a similar argument of (k).
\end{proof}

\begin{remark}\label{rem:scalar}  We give two remarks about the tuple $(E,F,X,T;E',F',X',T';H, H')$.
\begin{itemize}[leftmargin=8mm]
  \item[$(a)$]  Multiplying $X$ and $X'$ by constant numbers $a$ and $b$ respectively,
all relations in Proposition \ref{prop:12relations}
remain to be true.
  \item[$(b)$] When separately considering $(E,F,X,T)$ and $(E',F',X',T')$,
  Proposition \ref{prop:12relations} $(a)$-$(c)$  and $(d)$-$(f)$
  show that both of them are representation data of type $A$.
\end{itemize}
\end{remark}

\vspace{2ex}

In the following, we shall show that relations in Proposition
\ref{prop:12relations} are enough for
the quiver Hecke relations required for
an $\mathfrak{A}(\fraks\frakl_{\K_q})$-representation datum
on $\bigoplus_{n\in \bbN^*}RG_n\mod$.

\smallskip
Let $\scrI$ (resp. $\scrI'$) be the set of the generalized eigenvalues of $X$ on $F$ (resp. $X'$ on $F'$).
 We have the decompositions
$$E=\bigoplus_{i\in I} E_i,\ \ \ \
F=\bigoplus_{i\in I} F_i,\ \ \ \
E'=\bigoplus_{i'\in I'} E'_{i'},\ \ \ \
F'=\bigoplus_{i'\in I'} F'_{i'}$$
such that $X-i$ is locally nilpotent on $E_i$ and $F_i$ and $X'-i'$ is locally nilpotent on $E'_{i'}$ and $F'_{i'}$,
 respectively. We put
\begin{itemize}[leftmargin=8mm]
  \item $x_i=i^{-1}X-1$ (acting on $F_i$), $x_{i'}={i'}^{-1}X'-1$ (acting on $F'_{i'}$),
\item $$\tau_{ij}= \begin{cases}
i(qF_iX-XF_j)^{-1}(T-q) & \text{ if }i=j, \\
q^{-1}i^{-1}(F_iX-XF_j)T+i^{-1}(1-q^{-1})XF_j & \text{ if }i=qj, \\
\frac{F_iX-XF_j}{qF_iX-XF_j}(T-q)+1 & \text{ otherwise}
\end{cases}$$
(restricted to $F_iF_j$),
\item $$\tau_{i'j'}= \begin{cases}
{i'}(qF'_{i'}X'-X'F'_{j'})^{-1}(T'-q) & \text{ if }i'=j', \\
q^{-1}i'^{-1}(F'_{i'}X'-X'F'_{j'})T'+i'^{-1}(1-q^{-1})X'F'_{j'} & \text{ if }i'=qj', \\
\frac{F'_{i'}X'-X'F'_{j'}}{qF'_{i'}X'-X'F'_{j'}}(T'-q)+1 & \text{ otherwise}
\end{cases}$$
(restricted to $F'_{i'}F'_{j'}$),
\item $\tau_{ij'}= H$
(restricted to $F_{i}F'_{j'}$), and
$\tau_{i'j}= H'$
(restricted to $F'_{i'}F_{j}$).
\end{itemize}

\smallskip

 Let $I(q)$ and $I'(q)$ be the quivers as defined in \S\ref{subsec:datumA},
 and write
$\K_q:=I(q)\sqcup I'(q)$ for the quiver that is the disjoint union of $I(q)$ and $I'(q)$.
Notice that there is no arrow between $I(q)$ and $I'(q)$. In particular,
$\K_q$ is not connected. Since $\K_q$ is of type $A$, we may denote by
$\fraks\frakl_{\K_q}=\fraks\frakl_{I}\oplus\fraks\frakl_{I'}$ the corresponding Kac-Moody algebra which is a direct sum of $\fraks\frakl_{I}$ and $\fraks\frakl_{I'}$ (see Remark \ref{rk:sum}).
Let $Q:=(Q_{st})_{s,t\in \K_q}$ be the matrix associated with $\fraks\frakl_{\K_q}$
as in Remark \ref{rmk:quiver}. Specifically,  for $s,t\in \K_q$,
the entry $Q_{st}$ of $Q$ is
\[
Q_{st} (u,v)=
\begin{cases}
  0 & \text{if $s = t$,} \\
  v-u & \text{if $s  \rightarrow t$,} \\
  u-v & \text{if $t  \rightarrow s$,}\\
   1 & \text{otherwise}
\end{cases}
\]
in terms of the entry $a_{st}$ of the Cartan matrix $A=(a_{st})_{s,t\in \K_q}$:
\[
a_{st}=
\begin{cases}
  2 & \text{if $s = t$,} \\
  -1 & \text{if there is an arrow between $s$ and $t$,} \\
  0 & \text{otherwise}.
\end{cases}
\]

For the simplification of expression, in the following result
we write $F_{i'}$ and $E_{i'}$ for $F'_{i'}$ and $E'_{i'}$, respectively,
whenever $i'\in I'$.  Hence, we may use
the above notation $s,t$ for uniformly indexing vertices of
$\K_q$.

\begin{theorem} \label{thm:repdatumKq}
The natural transformations $$x_s:F_s\to F_s \text{ and }
\tau_{st}:F_sF_t\to F_tF_s \text{ for }s,t\in \K_q$$
satisfy the following quiver Hecke relations:
\begin{itemize}[leftmargin=8mm]
\item[$(1)$]
\label{en:half11}
$\tau_{st}\circ \tau_{ts}=Q_{st}(F_tx_s,x_tF_s)$,
\item[$(2)$]
\label{en:half22}
$\tau_{tu}F_s\circ F_t \tau_{su}\circ \tau_{st}F_u-
F_u \tau_{st}\circ \tau_{su}F_t\circ F_s \tau_{tu}=\\
\begin{cases}
\frac{Q_{st}(x_sF_t,F_sx_t)F_s-F_sQ_{st}(F_tx_s,x_tF_s)}{x_sF_tF_s-F_sF_tx_s}F_s
 & \text{ if } s=u\vspace{0.2cm}\\
0  & \text{ otherwise,}
\end{cases}$
\item[$(3)$]
\label{en:half33}
$\tau_{st}\circ x_s F_t-F_t x_s\circ \tau_{st}=\delta_{st}$, and
\item[$(4)$]
\label{en:half44}
$\tau_{st}\circ F_sx_t-x_tF_s\circ \tau_{st}=-\delta_{st}.$
\end{itemize}
That is, the tuple $$\big(\{E_s\}_{s\in \K_q},\{F_s\}_{s\in\K_q},\{x_s\}_{s\in \K_q},\{\tau_{s,t}\}_{s,t\in \K_q}\big)$$
is an $\mathfrak{A}(\fraks\frakl_{\K_q})$-representation datum
on $\bigoplus_{n\in \bbN^*}RG_n\mod$.
\end{theorem}

\begin{proof} To check the relations (1)-(4), according to the definition
of $x_s$ and $\tau_{s,t}$ for $s,t\in \K_q$ we need to consider the following four cases
$$s,t\in I(q),\ \  s,t\in I'(q),\ \  s\in I(q)~\mbox{and}~t\in I'(q),\ \ \mbox{and}~s\in I'(q)~\mbox{and}~t\in I(q).$$
For the first case, it has been introduced in \S \ref{subsec:BKR}, and
the theorem holds by Proposition \ref{prop:12relations} (a)-(c)
or by a theorem of Rouquier \cite[Theorem~5.26]{R08}.
For the other cases, the theorem similarly holds.
Indeed, the relations about $H$ and $H'$ are needed for the latter two cases.
\end{proof}

With Theorem \ref{thm:repdatumKq},
 we shall also call the tuple $(E,F,X,T;E',F',X',T';H,H')$
 a representation datum (of $\mathfrak{A}(\fraks\frakl_{\K_q})$) on $\bigoplus_{n\in \bbN^*}RG_n\mod$.
\smallskip

\begin{remark}
When $q\equiv 1~ (\mathrm{mod}~4)$, $-1$ is a square in $\bbF_{q}$ and $\zeta(-1)=1$.
In this case,  we may
similarly obtain an $\mathfrak{A}(\fraks\frakl_{\K_q})$-representation
datum on $\bigoplus_{n\geqslant 0}R\rm{SO}_{2n+1}(q)\mod$, as we do in this subsection.
However, when $q\equiv -1~(\mathrm{mod}~4)$, the sign $\zeta(-1)$ causes extra difficulty and
we do not know whether or not such a representation datum
can also be constructed.
Fortunately,
there is an $\mathfrak{A}(\fraks\frakl_{\K_q})$-representation
datum on $\scrQU^{\SO}_R$ from that on $\bigoplus_{n\geqslant 0}R{\rm O}_{2n+1}(q)\mod$, which
will be constructed in \S \ref{sub:explicit-repdatumSO}.
\end{remark}

\smallskip

\subsection{Spinor norm and determinant of $\O_{2n+1}(q)$}\label{sub:spin+det}
In this section, we investigate functors arising from the spinor norm and the
determinant of $\O_{2n+1}(q)$ and their properties. As a consequence,
we show that the natural transformations $X$ and $X'$ defined in \S \ref{sec:repdatumN}
have the same generalized eigenvalues,
and that if $i$ is a  generalized eigenvalue of $X$ on $F$ (resp. $X'$ on $F'$) then
$-i$ is also a  generalized eigenvalue of $X$ on $F$ (resp. $X'$ on $F'$).

\smallskip

Throughout this section we let $G_n=\SO_{2n+1}(q)$ and
$\widetilde{G}_n=\O_{2n+1}(q)$ with $q$ odd,
and  keep all notation in \S  \ref{sub:ide-funct}-\ref{sec:repdatumN} for $\widetilde{G}_n$
if there is no specification.

\subsubsection{The functor of spinor norms}\label{subsec:spinor}
As in Table \ref{tb:classical},
let $V$ be the underlying orthogonal space of $\widetilde{G}_n$ over the field $\mathbb{F}_q$
 with the bilinear form $\beta$.  Then its quadratic form is
$\varphi(x_1,x_2,\ldots, x_{2n+1})=x_{1}x_{2n+1}+x_2x_{2n}+\cdots +1/2x_{n+1}^2$,
 satisfying $\beta(x,y)=\varphi(x+y)-\varphi(x)-\varphi(y).$
 \smallskip

For an anisotropic vector $v \in V$,
the reflection $\tau_v$  orthogonal to $v$ is defined as:
$$x \mapsto x - 2v\frac{\beta(v,x)}{\beta(v,v)}.$$
We may define a map that sends any element of $\widetilde{G}_n$ arising as reflection orthogonal to some anisotropic
vector $v \in V$ to the value $\beta(v,v)=2\varphi(v)$ modulo ${\mathbb{F}_q^{\times}}^2$. It
extends to a well-defined and unique homomorphism  $\theta_n:\widetilde{G}_n\rightarrow {\mathbb{F}_q^{\times}}/{\mathbb{F}_q^{\times}}^2$,
called  the spinor norm of $\widetilde{G}_n$.

\smallskip

The composition  of the spinor norm
with the Legendre symbol $\zeta:\bbF_q^{\times}\to \{\pm1\}$
affords a linear character of $\widetilde{G}_n$ of order 2,
which will be denoted by $sp_n$. Namely, $sp_n=\zeta\circ \theta_n$ and we also call it
the spinor norm of $\widetilde{G}_n$.
The restriction of $sp_n$ to $G_n$  will also be denoted by $sp_n$. It is exactly the
 (unique) non-trivial linear character
 $\hat z\in\mathcal{E}(G_n,z)$
 corresponding to $1_{G_n}\in\mathcal{E}(G_n^*,1)$ under the Jordan decomposition,
 where $z=-I_{2n}\in Z(G_n^*)=\langle z \rangle$. See  \cite[(8.20)]{CE04} for
 the definition of $\hat z$.
By \cite[Prop.~1.3.1(ii)]{DM}, for
$\chi_{s,\mu}\in \mathcal{E}(G_n,(s))$
we have
	$\hat z\cdot\chi_{s,\mu}=sp_n\cdot\chi_{s,\mu}=\chi_{-s,\mu}$
up to $G_n^*$-conjugacy class of
the pair $(s,\mu)$.
Let $R_{sp_n}$ be the 1-dimensional $R\widetilde{G}_n$-module affording $sp_n$,
whose restriction to $RG_n$ will also be denoted by $R_{sp_n}$ (since it will always be clear which group is considered).
If $E_{\Theta_+,\Theta_-}$ is a quadratic unipotent module of $\SO_{2n+1}$ over $K$,
then by  \cite[Propsition 4.8]{W},
$$K_{sp_n}\otimes E_{\Theta_+,\Theta_-}\cong E_{\Theta_-,\Theta_+},$$ yielding
an involution permutation on quadratic unipotent modules.
In particular, if $E_{t_+,t_-}$ is a cuspidal module, then
 $K_{sp_n}\otimes E_{t_+,t_-}\cong E_{t_-,t_+}$.

\begin{lemma}\label{Lem:spinor}
Let $\dot{t}_{n+1}$ and $\dot{s}_{n+2}$ for $n\geqs 0$ be as defined in \S \ref{sub:lift-refl}.
Then
\begin{itemize}[leftmargin=8mm]
  \item[$(1)$]  $sp_{n+1}(\diag(\lambda,1,\dots,1,\lambda^{-1}))=\zeta(\lambda),$
  \item[$(2)$]  $sp_{n+1}(u)=1$ for any unipotent element $u$,
  \item[$(3)$]  $sp_{n}(-\id_{\widetilde{G}_n})=\zeta(-1)^n$,
  \item[$(4)$]  $sp_{n+1}(\dot{t}_{n+1})=\zeta(2)$ for the $\widetilde{G}_{n+1}$ case,
  $sp_{n+1}(\dot{t}_{n+1})=\zeta(-2)\zeta(-1)^{n}$ for the $G_{n+1}$ case, and $sp_{n+2}(\dot{s}_{n+2})=1$ for both cases.
  \end{itemize}
\end{lemma}
\begin{proof}

(1) Denote by $e_i$ the $i$-th standard basis of $V=\bbF_q^{2n+3}$. Let $u_1=e_1+e_{2n+3}$ and $u_2= e_1+\lambda e_{2n+3}$.  It is easy to see that
 $\varphi(u_1)=1$ and $\varphi(u_2)=\lambda.$
Also,
$$\tau_{u_1}=\left(\begin{array}{ccccc} & & -1\\ &\id_{\widetilde{G}_n} &\\-1&&\\ \end{array}\right)\text{ and} \,\,\,\tau_{u_2}=\left(\begin{array}{ccccc} & & -\lambda^{-1}\\ &\id_{\widetilde{G}_n} &\\-\lambda&&\\ \end{array}\right).$$
It follows that  $\diag(\lambda,1,\dots,1,\lambda^{-1})=\tau_{u_1}\tau_{u_2}$, and so $$sp_{n+1}(\diag(\lambda,1,\dots,1,\lambda^{-1}))=\zeta(4\varphi(u_1)\varphi(u_2))=\zeta(\lambda).$$

\smallskip

(2) The oddness of $q$ implies that the order of a unipotent element is odd. Hence $sp_{n+1}(u)=1$ since $sp_{n+1}$ has order $2$.

\smallskip
(3) Let $v=-e_{n+1}$ so that $\varphi(v)=1/2$ and $\tau_v= \diag(\id_n,-1,\id_n).$ We have $sp(\tau_v)=1$.
Note that  $-\id_{G_n}$ can be written as a
product of $\diag(\id_n,-1,\id_n)$ and $d_k=\diag(\id_{k-1},-1,\id_{2n-2k+1},-1,\id_{k-1})$ for all $1\leqs k\leqs n$.
Since $sp_n(d_k)=\zeta(-1)$ for all $1\leqs k\leqs n$,
the assertion follows.

\smallskip

(4) Since $\dot{t}_{n+1}=\dot{s}_{n+1}\cdots\dot{s}_{2} \dot{t}_1 \dot{s}^{-1}_{2}\cdots \dot{s}^{-1}_{n+1}$, we have $sp_{n+1}(\dot{t}_{n+1})=sp_{n+1}(\dot{t}_1).$
For the $\widetilde{G}_{n+1}$ case, let $w_1=e_{n+1}+e_{n+3}$. Then
$$ \tau_{w_1}=\left(\begin{array}{ccccc} \id_n& & & & \\ &&& -1&\\&&1&&\\&-1&&& \\ &&&&\id_{n} \\ \end{array}\right),$$
and so $\dot{t}_1=\tau_{w_1}$.
 A direct computation shows that $\varphi(w_1)=1.$
Hence $sp_{n+1}(\dot{t}_{n+1})=sp_{n+1}(\dot{t}_{1})=\zeta(2).$

\smallskip
Similarly, for the $G_{n+1}$ case, let $w_2=-e_{n+1}-e_{n+3}$. Then
$$ \tau_{w_2}=\left(\begin{array}{ccccc} \id_n& & & & \\ &&& 1&\\&&1&&\\&1&&& \\ &&&&\id_{n} \\ \end{array}\right),$$
 and so $\dot{t}_1=\tau_{w_2}\cdot (-\id_{\widetilde{G}_{n+1}})$.
Also, a direct computation shows that $\varphi(w_2)=-1.$
Hence $sp_{n+1}(\dot{t}_{n+1})=sp_{n+1}(\dot{t}_{1})=sp_{n+1}
(\tau_{w_2})sp_{n+1}(-\id_{\widetilde{G}_{n+1}})=\zeta(-2)\zeta(-1)^{n}.$

\smallskip

Finally, recall from the proof of Proposition \ref{prop:12relations} (e) that $\dot{s}_{n+2}$
can be written as a product of unipotent elements, that is,  $\dot{s}_{n+2}=u_{-\alpha}(1)u_{\alpha}(-1)u_{-\alpha}(1)$.
By (2), we get $sp_{n+2}(\dot{s}_{n+2})=1.$

\end{proof}

Now we let $\widehat{G}_n$ denote $G_n$ or $\widetilde{G}_n$ and
define a functor $$\mathrm{Spin}_n:R\widehat{G}_n\mod\to R\widehat{G}_n\mod$$
by $$M\mapsto R_{sp_n}\otimes_R M.$$
Let  $\mathfrak{I}$ be a
sub-$(R\widehat{G}_n,R\widehat{G}_n)$-module generated by $sp_n(g)g\otimes1-1\otimes g$, where $g\in \widehat{G}_n.$
Then the functor $\mathrm{Spin}_n$
is represented by the $(R\widehat{G}_n,R\widehat{G}_n)$-bimodule
$R\widehat{G}_n\otimes_R R\widehat{G}_n/\mathfrak{I}$,
which is isomorphic to the $(R\widehat{G}_n,R\widehat{G}_n)$-bimodule $R\widehat{G}_n\cdot\mathfrak{sp}_n$ (with $\mathfrak{sp}_n$
formal)  
defined by $$g( x\cdot\mathfrak{sp}_n)h=gx(sp_n(h)h)\cdot\mathfrak{sp}_n\,\,\,\,\,\text{for all}\,\,g,h\in\widehat{G}_n\,\,\text{and}\,\, x\in R\widehat{G}_n.$$

\smallskip

Let $\mathrm{Spin}:=\bigoplus\limits_{n\geqs 0}\mathrm{Spin}_n:\bigoplus_{n\geqslant 0}R\widehat{G}_{n}\mod\to \bigoplus_{n\geqslant 0}R\widehat{G}_{n}\mod$.
It is easy to see that $\mathrm{Spin}^2\cong \Id.$
We define the maps $\Phi_n:F'_{n+1,n}\circ \mathrm{Spin}_n\to \mathrm{Spin}_{n+1}\circ F_{n+1,n}$ and
$\Phi'_n:F_{n+1,n}\circ \mathrm{Spin}_n\to \mathrm{Spin}_{n+1}\circ F'_{n+1,n}$
 via
$$\begin{array}{rcl}
R\widehat{G}_{n+1}e'_{n+1,n}\otimes_{R\widehat{G}_n}R\widehat{G}_n\cdot\mathfrak{sp}_n
&\to& R\widehat{G}_{n+1}\cdot\mathfrak{sp}_{n+1}\otimes_{R\widehat{G}_{n+1}}R\widehat{G}_{n+1}
e_{n+1,n}\\
ge'_{n+1,n}\otimes h\cdot\mathfrak{sp}_n &\mapsto & g\cdot\mathfrak{sp}_{n+1}\otimes sp_n(h)h e_{n+1,n}
\end{array}
$$
and
$$\begin{array}{rcl}
R\widehat{G}_{n+1}e_{n+1,n}\otimes_{R\widehat{G}_n}R\widehat{G}_n\cdot\mathfrak{sp}_n &\to& R\widehat{G}_{n+1}\cdot\mathfrak{sp}_{n+1}\otimes_{R\widehat{G}_{n+1}}R\widehat{G}_{n+1}
e'_{n+1,n}\\
ge_{n+1,n}\otimes h\cdot\mathfrak{sp}_n&\mapsto& g\cdot\mathfrak{sp}_{n+1}\otimes sp_n(h)h e'_{n+1,n},
\end{array}
$$ respectively,
for all $g\in \widehat{G}_{n+1}$ and $h\in \widehat{G}_n$.

%

\begin{lemma}\label{lem:Phi&Phi'}
 Let $\Phi=\bigoplus_{n\in\bbN}\Phi_n$ and  $\Phi'=\bigoplus_{n\in\bbN}\Phi'_n.$
Then  $\Phi$ (resp. $\Phi'$) is an isomorphism between the functors $F'\circ \mathrm{Spin}$ and $\mathrm{Spin}\circ F$ (resp. $F\circ \mathrm{Spin}$ and $\mathrm{Spin}\circ F'$).
\end{lemma}
\begin{proof}
The arguments for the maps $\Phi$ and $\Phi'$ are similar.
So we shall only prove the case of $\Phi$, and it suffices to show that $\Phi_n$ is
an $(R\widehat{G}_{n+1},R\widehat{G}_n)$-bimodule isomorphism for  each $n\in\bbN$.
We first prove that the map $\Phi_n$ is well-defined. In fact, it suffices to prove that
\begin{equation}\label{Phi:1}
\Phi_n(g k^{-1}e'_{n+1,n}\otimes kh\cdot\mathfrak{sp}_n)=\Phi_n(g e'_{n+1,n}\otimes h\cdot\mathfrak{sp}_n)
\end{equation}
and \begin{equation}\label{Phi:2}
\Phi_n( gue'_{n+1,n}\otimes h\cdot\mathfrak{sp}_n)=\Phi_n( sp_{n+1}(u) g e'_{n+1,n}\otimes h\cdot\mathfrak{sp}_n)
\end{equation}
for all $k\in \widehat{G}_n$ and $u\in U_{n}.$ This is because
the restriction of $sp_{n+1}$
to $U_n$ is equal to  $\mathrm{Inf}_{\bbF^{\times}_q}^{U_n}(\zeta)$ by Lemma \ref{Lem:spinor},
and so $e'_{n+1,n}=\frac{1}{|U_{n}|}\sum\limits_{u\in U_{n}}sp_{n+1}(u)u$ (since $U_n=V_n\rtimes\bbF^{\times}_q$).

\smallskip

For Equality (\ref{Phi:1}), we have
$$\begin{array}{rclr}
 \Phi_n(g k^{-1}e'_{n+1,n}\otimes kh\cdot\mathfrak{sp}_n) &=& gk^{-1}\cdot\mathfrak{sp}_{n+1}\otimes sp_n(kh)kh e_{n+1,n}  \\
  &=& g\cdot\mathfrak{sp}_{n+1}\otimes (sp_n(k^{-1})k^{-1})sp_n(kh)kh e_{n+1,n}\\
  &=& g\cdot\mathfrak{sp}_{n+1}\otimes sp_n(h)h e_{n+1,n}\\
  &=& \Phi_n(g e'_{n+1,n}\otimes h\cdot\mathfrak{sp}_n),
\end{array}
$$
where 
the second equality holds since
$$k^{-1}\cdot\mathfrak{sp}_{n+1}=\mathfrak{sp}_{n+1}\cdot sp_{n+1}(k^{-1})k^{-1}=\mathfrak{sp}_{n+1}\cdot sp_{n}(k^{-1})k^{-1}.$$
For Equality (\ref{Phi:2}), we
have
$$\begin{array}{rclr}
\Phi_n( gue'_{n+1,n}\otimes h\cdot\mathfrak{sp}_n)&=&
 gu\cdot\mathfrak{sp}_{n+1}\otimes sp_n(h)h e_{n+1,n} &  \\
  &=&g\cdot\mathfrak{sp}_{n+1} (sp_{n+1}(u)u)\otimes sp_n(h)h e_{n+1,n} &\\
  &=&sp_{n+1}(u)g\cdot\mathfrak{sp}_{n+1}\otimes sp_n(h)h u e_{n+1,n} &\\
  &=&sp_{n+1}(u)g\cdot\mathfrak{sp}_{n+1}\otimes sp_{n}(h)h e_{n+1,n}& \\
  &=&\Phi_n(sp_{n+1}(u) g e'_{n+1,n}\otimes h\cdot\mathfrak{sp}_n). &
\end{array}
$$

Now we prove that the map $\Phi_n$ is an $(R\widehat{G}_{n+1},R\widehat{G}_n)$-bimodule map.
Obviously,  $\Phi_n$ is a left $R\widehat{G}_{n+1}$-module map.
The conclusion that $\Phi_n$ is also a right $ R\widehat{G}_n$-module map follows from
the following:
$$\begin{array}{rcl}
\Phi_n(g e'_{n+1,n}\otimes h\cdot\mathfrak{sp}_n k)&=&\Phi_n(g e'_{n+1,n}\otimes sp_n(k)hk\cdot\mathfrak{sp}_n)\\
&=&g\cdot\mathfrak{sp}_{n+1}\otimes sp_n(k)sp(hk)hk e_{n+1,n}\\
&=& g\cdot\mathfrak{sp}_{n+1}\otimes sp_n(h)hk e_{n+1,n}\\
&=&\Phi_n(g e'_{n+1,n}\otimes h\cdot\mathfrak{sp}_n)k
\end{array}
$$
for any $k\in \widehat{G}_n$.

\smallskip

Finally, we prove that the map $\Phi_n$ is an isomorphism.
The surjection of $\Phi_n$ is clear since every element
$g\cdot\mathfrak{sp}_{n+1}\otimes h e_{n+1,n}$ has a  pre-image
$ge_{n+1,n}\otimes sp_n(h)h\cdot\mathfrak{sp}_{n}.$
It is easy to see that $R\widehat{G}_{n+1}e'_{n+1,n}\otimes_{R\widehat{G}_n}R\widehat{G}_n\cdot\mathfrak{sp}_n$ and
$R\widehat{G}_{n+1}\cdot\mathfrak{sp}_{n+1}\otimes_{R\widehat{G}_{n+1}}R\widehat{G}_{n+1}e_{n+1,n}$
are both $R$-free with the same rank equal to that of the
bimodule $R\widehat{G}_{n+1,n}e_{n+1,n}$ (also $R\widehat{G}_{n+1,n}e'_{n+1,n}$).
Since $\Phi_n$ is an $R$-free map, we get that $\Phi_n$ is an isomorphism.
\end{proof}

Now we investigate the influence of the functor $\mathrm{Spin}$ on
the representation datum on $\bigoplus_{n\geqslant 0}R\widetilde{G}_n\mod$ constructed in \S \ref{sec:repdatumN}.
The similar investigation of that on $\scrQU^{\SO}_R$ will be done in \S \ref{sub:explicit-repdatumSO}.
%

\begin{proposition}\label{prop:spinor} 
Let $(E,F,X,T;E',F',X',T';H,H')$ be
the representation datum on $\bigoplus_{n\geqslant 0}R\widetilde{G}_{n}\mod$ constructed in \S \ref{sec:repdatumN}.
\begin{itemize}[leftmargin=8mm]
 \item[$(a)$] If $\Phi$ and $\Phi'$ are as defined before Lemma \ref{lem:Phi&Phi'}, then the following  diagrams commute:
 \begin{itemize}
 \item[$(a1)$]
$$\xymatrix{F'\circ \mathrm{Spin}\ar[r]_{\Phi}^{\sim}\ar[d]^{X'1_{\rm {Spin}}} &\mathrm{Spin}\circ F \ar[d]^{\beta1_{\rm {Spin}} X}\\F'\circ
\mathrm{Spin}\ar[r]_{\Phi}^{\sim}& \mathrm{Spin}\circ F
}$$
\item[$(a2)$]
$$\xymatrix{F\circ \mathrm{Spin}\ar[r]_{\Phi'}^{\sim}\ar[d]^{X1_{\rm {Spin}}} &\mathrm{Spin}\circ F' \ar[d]^{\beta1_{\rm {Spin}} X'}\\F\circ
\mathrm{Spin}\ar[r]_{\Phi'}^{\sim}& \mathrm{Spin}\circ F'
}$$
\end{itemize}
where $\beta=\zeta(2).$
\item[$(b)$] Let $\Phi_1=(\Phi1_{F'})\circ(1_{F'}\Phi)$ (resp. $\Phi_1'=(\Phi'1_{F})\circ(1_{F}\Phi')$) be
the isomorphism between the functors $(F')^2\circ \mathrm{Spin}$ and $\mathrm{Spin}\circ F^2$
(resp. $F^2\circ \mathrm{Spin}$ and $\mathrm{Spin}\circ (F')^2$). Then the following  diagrams commute:
 \begin{itemize}
 \item[$(b1)$]
$$\xymatrix{F'\circ F'\circ \mathrm{Spin}\ar[r]_{\Phi_1}^{\sim}\ar[d]^{T'1_{\rm {Spin}}} &\mathrm{Spin}\circ F\circ F \ar[d]^{1_{\rm {Spin}} T}\\F'\circ F'\circ
\mathrm{Spin}\ar[r]_{\Phi_1}^{\sim}& \mathrm{Spin}\circ F\circ F
}$$
 \item[$(b2)$]
  $$\xymatrix{F\circ F\circ \mathrm{Spin}\ar[r]_{\Phi_1'}^{\sim}\ar[d]^{T1_{\rm {Spin}}} &\mathrm{Spin}\circ F'\circ F' \ar[d]^{1_{\rm {Spin}} T'}\\F\circ F\circ
\mathrm{Spin}\ar[r]_{\Phi_1'}^{\sim}& \mathrm{Spin}\circ F'\circ F'
}$$
\end{itemize}
\item[$(c)$] Let $\Phi_2=(\Phi1_{F})\circ(1_{F'}\Phi')$ (resp. $\Phi_2'=(\Phi'1_{F'})\circ(1_{F}\Phi)$)
 be the isomorphism between the functors $F'\circ F\circ \mathrm{Spin}$ and $\mathrm{Spin}\circ F\circ F'$
 (resp.  $F\circ F'\circ \mathrm{Spin}$ and $\mathrm{Spin}\circ F'\circ F$).
 Then the following  diagrams commute:
  \begin{itemize}
 \item[$(c1)$]
$$\xymatrix{F'\circ F\circ \mathrm{Spin}\ar[r]_{\Phi_2}^{\sim}\ar[d]^{H1_{\rm {Spin}}} &\mathrm{Spin}\circ F\circ F' \ar[d]^{q\cdot1_{\rm {Spin}} H'}\\F\circ F'\circ
\mathrm{Spin}\ar[r]_{\Phi'_2}^{\sim}& \mathrm{Spin}\circ  F'\circ F}$$
 \item[$(c2)$]
$$\xymatrix{F\circ F'\circ  \mathrm{Spin}\ar[r]_{\Phi_2'}^{\sim}\ar[d]^{H'1_{\rm {Spin}}} &\mathrm{Spin}\circ  F'\circ F \ar[d]^{q^{-1}\cdot1_{\rm {Spin}} H}\\F'\circ F\circ
\mathrm{Spin}\ar[r]_{\Phi_2}^{\sim}& \mathrm{Spin}\circ  F\circ F'
}$$
\end{itemize}
\end{itemize}
\end{proposition}

\begin{proof} Write $M=R\widetilde{G}_{n+1}e_{n+1,n}$ and $M'=R\widetilde{G}_{n+1}e'_{n+1,n}$.
\smallskip

For the diagram (a1), we note that
$$(1_{\rm {Spin}_{n+1}}X_{n+1,n})\circ\Phi_n:F'_{n+1,n}\circ \mathrm{Spin}_n\to \mathrm{Spin}_{n+1}\circ F_{n+1,n}$$
 is represented by the composition of the $(R\widetilde{G}_{n+1},R\widetilde{G}_n)$-bimodule maps
 $$
 \begin{array}{ccccc}
M'\otimes_{R\widetilde{G}_n}R\widetilde{G}_n\cdot\mathfrak{sp}_n
&\rightarrow & R\widetilde{G}_{n+1}\cdot\mathfrak{sp}_{n+1}\otimes_{R\widetilde{G}_{n+1}}M
&\rightarrow & R\widetilde{G}_{n+1}\cdot\mathfrak{sp}_{n+1}\otimes_{R\widetilde{G}_{n+1}}M\\
ge'_{n+1,n}\otimes h\cdot\mathfrak{sp}_n&\mapsto& g\cdot\mathfrak{sp}_{n+1}\otimes sp_n(h)h e_{n+1,n}
&\mapsto& g\cdot\mathfrak{sp}_{n+1}\otimes sp_n(h)h X_{n+1,n},
 \end{array}
$$
for all $g\in \widetilde{G}_{n+1}$ and $h\in \widetilde{G}_n.$
Also, we note that
$$\Phi_{n}\circ (X'_{n+1,n}1_{\rm {Spin}_n}):F'_{n+1,n} \circ \mathrm{Spin}_n\to \mathrm{Spin}_{n+1}\circ F_{n+1,n}$$
 is represented by the composition of the $(R\widetilde{G}_{n+1},R\widetilde{G}_n)$-bimodule maps
$$
 \begin{array}{cccl}
\noindent M'\otimes_{R\widetilde{G}_n}R\widetilde{G}_n\cdot\mathfrak{sp}_n
&\to&
M'\otimes_{R\widetilde{G}_n}R\widetilde{G}_n\cdot\mathfrak{sp}_n
&\to
 R\widetilde{G}_{n+1}\cdot\mathfrak{sp}_{n+1}\otimes_{R\widetilde{G}_{n+1}}M\\
ge'_{n+1,n}\otimes h\cdot\mathfrak{sp}_n &\mapsto & gX'_{n+1,n}\otimes h\cdot\mathfrak{sp}_n
&\mapsto  gX'_{n+1,n}\cdot\mathfrak{sp}_{n+1}\otimes sp_n(h)h e_{n+1,n}.
 \end{array}
$$
Lemma \ref{Lem:spinor} yields that
 $$e'_{n+1,n}\cdot\mathfrak{sp}_{n+1}=\mathfrak{sp}_{n+1}\cdot ( \frac{1}{|U_{n}|}\sum\limits_{u\in U_{n}}sp_{n+1}(u)^2u) =\mathfrak{sp}_{n+1}\cdot e_{n+1,n}$$ and $$\dot{t}_{n+1}\cdot\mathfrak{sp}_{n+1}
=\mathfrak{sp}_{n+1}\cdot sp_{n+1}(\dot{t}_{n+1})\dot{t}_{n+1}=\zeta(2)\mathfrak{sp}_{n+1}\cdot\dot{t}_{n+1}.$$
Since $X'_{n+1,n}=e'_{n+1,n}\dot{t}_{n+1}e'_{n+1,n}$ by definition,
we have $X'_{n+1,n}\cdot\mathfrak{sp}_{n+1}=\zeta(2)\mathfrak{sp}_{n+1}\cdot X_{n+1,n}$.
It follows that
$$gX'_{n+1,n}\cdot\mathfrak{sp}_{n+1}\otimes sp_n(h)h e_{n+1,n}=\zeta(2)g\cdot\mathfrak{sp}_{n+1}\otimes sp_n(h)h X_{n+1,n}.$$
Hence
 $(1_{\rm {Spin_{n+1}}} X_{n+1,n})\circ\Phi_n=\Phi_n\circ (\zeta(2)X'_{n+1,n}1_{\rm {Spin}_n})$,
 and so $(1_{\rm {Spin}} X)\circ\Phi=\Phi\circ (\zeta(2)X'1_{\rm {Spin}})$, i.e.,
the diagram (a1) commutes.

\smallskip

 The commutativity of the diagram (a2) follows by
 $X_{n+1,n}\cdot\mathfrak{sp}_{n+1}=\zeta(2)\mathfrak{sp}_{n+1}\cdot X'_{n+1,n}$ and
 the facts  that
 $$(1_{\rm {Spin_{n+1}}}X'_{n+1,n})\circ\Phi_n':F_{n+1,n}
   \circ \mathrm{Spin}_n\to \mathrm{Spin}_{n+1}\circ  F'_{n+1,n}$$
   is  represented by  the composition of the $(R\widetilde{G}_{n+1},R\widetilde{G}_n)$-bimodule maps
     $$
 \begin{array}{cccl}
M\otimes_{R\widetilde{G}_n}R\widetilde{G}_n\cdot\mathfrak{sp}_n
&\to& R\widetilde{G}_{n+1}\cdot\mathfrak{sp}_{n+1}\otimes_{R\widetilde{G}_{n+1}}M'
&\to R\widetilde{G}_{n+1}\cdot\mathfrak{sp}_{n+1}\otimes_{R\widetilde{G}_{n+1}}M'\\
ge_{n+1,n}\otimes h\cdot\mathfrak{sp}_n&\mapsto &
 g\cdot\mathfrak{sp}_{n+1}\otimes sp_n(h)h e'_{n+1,n}
 &\mapsto  g\cdot\mathfrak{sp}_{n+1}\otimes sp_n(h)h X'_{n+1,n},
 \end{array}
$$
  and that
 $$(X_{n+1,n}1_{\rm {Spin}_n})\circ\Phi_n': F_{n+1,n}\circ \mathrm{Spin}_n\to \mathrm{Spin}_{n+1}\circ F'_{n+1,n}$$
 is  represented by  the composition of the $(R\widetilde{G}_{n+1},R\widetilde{G}_n)$-bimodule maps
     $$
 \begin{array}{ccccc}
 M\otimes_{R\widetilde{G}_n}R\widetilde{G}_n\cdot\mathfrak{sp}_n
&\to&
M\otimes_{R\widetilde{G}_n}R\widetilde{G}_n\cdot\mathfrak{sp}_n
&\to& R\widetilde{G}_{n+1}\cdot\mathfrak{sp}_{n+1}\otimes_{R\widetilde{G}_{n+1}}M'\\
 ge_{n+1,n}\otimes h\cdot\mathfrak{sp}_n &\mapsto & gX_{n+1,n}\otimes h\cdot\mathfrak{sp}_n
 &\mapsto & gX_{n+1,n}\cdot\mathfrak{sp}_{n+1}\otimes sp_n(h)h e'_{n+1,n}.
 \end{array}
$$

\smallskip

For the diagram (b1), we write  $N=R\widetilde{G}_{n+2}e_{n+2,n}$ and $N'=R\widetilde{G}_{n+2}e'_{n+2,n}$.
By definition,
the restriction  of $\Phi_1$ to   $F_{n+2}\circ F_{n+1}\circ \rm{Spin}_n$
is represented by the $(RG_{n+2},RG_n)$-bimodule map
$$
\begin{array}{rcl}
  N'\otimes_{R\widetilde{G}_n}R\widetilde{G}_n\cdot
  \mathfrak{sp}_n &\to&
   R\widetilde{G}_{n+2}\cdot\mathfrak{sp}_{n+2}\otimes_{R\widetilde{G}_{n+2}}
   N  \\
  ge'_{n+2,n}\otimes h\cdot\mathfrak{sp}_n&\mapsto& g\cdot\mathfrak{sp}_{n+2}\otimes sp_n(h)h e_{n+2,n}
\end{array}
$$
for all $g\in \widetilde{G}_{n+2}$ and $h\in \widetilde{G}_n$.
So the restriction  of $ (T'1_{\rm {Spin}})\circ\Phi_1$ to $F_{n+2}\circ F_{n+1}\circ {\rm Spin}_n$
is represented by the composition of the $(R\widetilde{G}_{n+2},R\widetilde{G}_n)$-bimodule maps
$$
\begin{array}{ccccc}
N'\otimes_{R\widetilde{G}_n}R\widetilde{G}_n\cdot
  \mathfrak{sp}_n & \to &
   R\widetilde{G}_{n+2}\cdot\mathfrak{sp}_{n+2}\otimes_{R\widetilde{G}_{n+2}}N
   & \to &
   R\widetilde{G}_{n+2}\cdot\mathfrak{sp}_{n+2}\otimes_{R\widetilde{G}_{n+2}}N\\
ge'_{n+2,n}\otimes h\cdot\mathfrak{sp}_n& \mapsto  &
g\cdot\mathfrak{sp}_{n+2}\otimes sp_n(h)h e_{n+2,n}& \mapsto  &
g\cdot\mathfrak{sp}_{n+2}\otimes sp_n(h)h T_{n+2,n},
\end{array}
$$
and the restriction of
$\Phi_1\circ(1_{\rm {Spin}} T')$ to $F_{n+2}\circ F_{n+1}\circ {\rm Spin}_n$ is
 represented by the composition of the $(R\widetilde{G}_{n+2},R\widetilde{G}_n)$-bimodule maps
$$
\begin{array}{ccccc}
N'\otimes_{R\widetilde{G}_n}R\widetilde{G}_n\cdot
  \mathfrak{sp}_n & \to &
   N'\otimes_{R\widetilde{G}_n}R\widetilde{G}_n\cdot
  \mathfrak{sp}_n & \to &
   R\widetilde{G}_{n+2}\cdot\mathfrak{sp}_{n+2}\otimes_{R\widetilde{G}_{n+2}} N\\
ge'_{n+2,n}\otimes h\cdot\mathfrak{sp}_n& \mapsto  & gT'_{n+2,n}\otimes h\cdot\mathfrak{sp}_n& \mapsto  & gT'_{n+2,n}\cdot\mathfrak{sp}_{n+2}\otimes sp_n(h)h e_{n+2,n}.
\end{array}
$$
It is easy to see the two compositions are equal
since $T'_{n+1,n}\cdot\mathfrak{sp}_{n+1}=\mathfrak{sp}_{n+1}\cdot T_{n+2,n}.$
Hence $(1_{\rm {Spin}} T)\circ\Psi=\Psi\circ (T'1_{\rm {Spin}}),$ i.e., the diagram (b1) commutes.

\smallskip

The proofs for the other diagrams are similar and left to the reader.
\end{proof}

 \begin{remark}
All relations in Proposition
 \ref{prop:12relations} are preserved if we preserve $X, T, T'$, replace $X' $ by $\zeta(2)e'_{r+1,r} \dot{t}_{r+1}\,e'_{r+1,r},$ $H $ by $\sqrt{q}e_{{r+2,r+1}}e'_{r+1,r} s_{r+2}\,e'_{{r+2,r+1}}e_{r+1,r}$
 and $H' $ by $\sqrt{q}e'_{{r+2,r+1}}e_{r+1,r} s_{r+2}\,e_{{r+2,r+1}}e'_{r+1,r}$.
 In this case, the functor $\mathrm{Spin}$  interchanges $(E,F,X,T)$ and $(E',F',X',T')$,  and  interchanges $H$ and $H'$.
 However, we prefer to avoid the appearance of $\sqrt{q}$ in our construction of representation datum.
  \end{remark}

 \subsubsection{Determinant functor of  $\O_{2n+1}(q)$}\label{subsec:det}
Let $\mathrm{det}_n:\widetilde{G}_n\to \{\pm 1\}$ be the determinant homomorphism of $\widetilde{G}_n$
and $R_{\mathrm{det}_n}$ be the corresponding 1-dimensional module.
Similar to the functor ${\rm Spin}_n$, we define a functor $$\mathrm{Det}_n:R\widetilde{G}_n\mod\to R\widetilde{G}_n\mod$$
by $$M\mapsto R_{{\rm {det}_n}}\otimes_{R} M.$$
Let $\mathfrak{J}$ be a sub-$(R\widetilde{G}_n,R\widetilde{G}_n)$-module
 generated by $\mathrm{det}(g)g\otimes1-1\otimes g$, where $g\in \widetilde{G}_n.$ Then the functor $\mathrm{Det}_n$
is represented by the $(R\widetilde{G}_n,R\widetilde{G}_n)$-bimodule $R\widetilde{G}_n\otimes_R R\widetilde{G}_n/\mathfrak{J}$,
which is
isomorphic to the $(R\widetilde{G}_n,R\widetilde{G}_n)$-bimodule $R\widetilde{G}_n\cdot\mathfrak{det}_n$ (with $\mathfrak{det}_n$ formal)
defined by $$g( x\cdot\mathfrak{det}_n)h=gx({\rm det}_n(h)h)\cdot\mathfrak{det}_n
\,\,\,\,\,\text{for all }\,\,g,h\in\widetilde{G}_n\,\,\text{and}\,\, x\in R\widetilde{G}_n.$$

Let $$\mathrm{Det}:=\bigoplus\limits_{n\geqs 0}\mathrm{Det}_n:
\bigoplus\limits_{n\geqslant 0}R\widetilde{G}_{n}\mod\to \bigoplus\limits_{n\geqslant 0}R\widetilde{G}_{n}\mod.$$
Then $\mathrm{Det}^2\cong\Id$.
We define the maps $\Psi_n:F_{n+1,n}\circ \mathrm{Det}_n\to \mathrm{Det}_{n+1}\circ F_{n+1,n}$ and
$\Psi'_n:F'_{n+1,n}\circ \mathrm{Det}_n\to \mathrm{Det}_{n+1}\circ F'_{n+1,n}$
 via
$$\begin{array}{rcl}
R\widetilde{G}_{n+1}e_{n+1,n}\otimes_{R\widetilde{G}_n}R\widetilde{G}_n\cdot\mathfrak{det}_n
&\to& R\widetilde{G}_{n+1}\cdot\mathfrak{det}_{n+1}\otimes_{R\widetilde{G}_{n+1}}R\widetilde{G}_{n+1}
e_{n+1,n}\\
ge_{n+1,n}\otimes h\cdot\mathfrak{det}_n &\mapsto & g\cdot\mathfrak{det}_{n+1}\otimes {\rm det}_n(h)h e_{n+1,n}
\end{array}
$$
and
$$\begin{array}{rcl}
R\widetilde{G}_{n+1}e'_{n+1,n}\otimes_{R\widetilde{G}_n}R\widetilde{G}_n\cdot\mathfrak{det}_n &\to& R\widetilde{G}_{n+1}\cdot\mathfrak{det}_{n+1}\otimes_{R\widetilde{G}_{n+1}}R\widetilde{G}_{n+1}
e'_{n+1,n}\\
ge'_{n+1,n}\otimes h\cdot\mathfrak{det}_n&\mapsto& g\cdot\mathfrak{det}_{n+1}\otimes {\rm det}_n(h)h e'_{n+1,n},
\end{array}
$$ respectively,
for all $g\in \widetilde{G}_{n+1}$ and $h\in \widetilde{G}_n$.
Also, we have the following analogous results of Lemma \ref{lem:Phi&Phi'} and Proposition \ref{prop:spinor} for $\mathrm{Det}$,
whose proofs are left to the reader.


\begin{proposition}\label{prop:det}
 Let $\Psi=\bigoplus_{n\in\bbN}\Psi_n$ and  $\Psi'=\bigoplus_{n\in\bbN}\Psi'_n.$
Then $\Psi$ (resp. $\Psi'$) is an isomorphism between the functors $F\circ \mathrm{Det}$ and $\mathrm{Det}\circ F$ (resp. $F'\circ \mathrm{Det}$ and $\mathrm{Det}\circ F'$).
Also, we have the following statements.

 \begin{itemize}[leftmargin=8mm]
 \item[$(a)$] The isomorphisms $\Psi$ and $\Psi'$ satisfy the following commutative diagrams:
    \begin{itemize}
      \item[$(a1)$] $$\xymatrix{F\circ  \mathrm{Det}\ar[r]_{\Psi}^{\sim}\ar[d]^{X1_{\rm {Det}}} &
\mathrm{Det}\circ  F \ar[d]^{-1_{\rm {Det}} X}\\F\circ
\mathrm{Det}\ar[r]_{\Psi}^{\sim}& \mathrm{Det}\circ  F
}$$
      \item[$(a2)$] $$\xymatrix{F'\circ  \mathrm{Det}\ar[r]_{\Psi'}^{\sim}\ar[d]^{X'1_{\rm {Det}}} &\mathrm{Det}\circ F' \ar[d]^{-1_{\rm {Det}} X'}\\F'\circ
\mathrm{Det}\ar[r]_{\Psi'}^{\sim}& \mathrm{Det}\circ  F'
}$$
   \end{itemize}
 \item[$(b)$]
The isomorphisms $\Psi_1=(\Psi1_{F})\circ(1_{F}\Psi)$ between the functors $F^2\circ \mathrm{Det}$ and $\mathrm{Det}\circ F^2$
and $\Psi_1'=(\Psi'1_{F'})\circ(1_{F'}\Psi')$ between
 functors $(F')^2\circ \mathrm{Det}$ and $\mathrm{Det}\circ (F')^2$
 satisfy the following commutative diagrams:
     \begin{itemize}
      \item[$(b1)$]
$$\xymatrix{F\circ F\circ  \mathrm{Det}\ar[r]_{\Psi_1}^{\sim}\ar[d]^{T1_{\rm {Det}}} &\mathrm{Det}\circ  F\circ F \ar[d]^{1_{\rm {Det}} T}\\F\circ F\circ
\mathrm{Det}\ar[r]_{\Psi_1}^{\sim}& \mathrm{Det}\circ  F\circ F
}$$
   \item[$(b2)$]
$$\xymatrix{F'\circ F'\circ  \mathrm{Det}\ar[r]_{\Psi_1'}^{\sim}\ar[d]^{T'1_{\rm {Det}}} &\mathrm{Det}\circ  F'\circ F' \ar[d]^{1_{\rm {Det}} T'}\\F'\circ F'\circ
\mathrm{Det}\ar[r]_{\Psi1'}^{\sim}& \mathrm{Det}\circ  F'\circ F'
}$$
   \end{itemize}
   \item[$(c)$]
The isomorphisms $\Psi_2=(\Psi'1_{F})\circ(1_{F'}\Psi)$ between the functors $F'\circ F\circ \mathrm{Det}$
and $\mathrm{Det}\circ F'\circ F$  and
 $\Psi_2'=(\Psi1_{F'})\circ(1_{F}\Psi')$ between the functors
 $F\circ F'\circ \mathrm{Det}$ and $\mathrm{Det}\circ F\circ F'$ satisfy the following commutative diagrams:
      \begin{itemize}
      \item[$(c1)$]
$$\xymatrix{F'\circ F\circ  \mathrm{Det}\ar[r]_{\Psi_2}^{\sim}\ar[d]^{H1_{\rm {Det}}} &\mathrm{Det}\circ  F'\circ F \ar[d]^{1_{\rm {Det}} H}\\F\circ F'\circ
\mathrm{Det}\ar[r]_{\Psi'_2}^{\sim}& \mathrm{Det}\circ  F\circ F'
}$$
\item[$(c2)$]  $$\xymatrix{F\circ F'\circ  \mathrm{Det}\ar[r]_{\Psi_2'}^{\sim}\ar[d]^{H'1_{\rm {Det}}} &\mathrm{Det}\circ F\circ F' \ar[d]^{1_{\rm {Det}} H'}\\F'\circ F\circ
\mathrm{Det}\ar[r]_{\Psi_2}^{\sim}& \mathrm{Det}\circ  F'\circ F
}$$
   \end{itemize}
\end{itemize}

\end{proposition}

Proposition \ref{prop:det} has a straightforward consequence.

\subsection{$\mathfrak{A}(\fraks\frakl_{\K_q})$-Representation datum on $\scrQU_{R}^{\SO}$}
In this section, we shall show that $\scrQU_{R}^{\SO}$ is Morita equivalent
to two $F$- and $F'$-stable sub-categories $\scrQU^{\O,\,(+)}_R$ and $\scrQU^{\O,\,(-)}_R$ of $\scrQU_{R}^{\O}$
through the restriction functor.
%
This leads to a representation datum on $\scrQU_{R}^{\SO}$ which will be
explicitly constructed in \S \ref{sub:explicit-repdatumSO}.

Throughout this section, we keep to use the notation $\widetilde{G}_{n}=\O_{2n+1}(q)$ and $G_n=\SO_{2n+1}(q)$.
When simultaneously considering functors and natural transformations of
$\bigoplus_{n\geqslant 0}RG_n\mod$ and $\bigoplus_{n\geqslant 0}R\widetilde{G}_n\mod$,
we distinguish their notation by the right superscripts $G$ and $\widetilde{G}$ if necessary.
For instance, the notation
$F^G$ and $F^{\widetilde{G}}$ will denote the functors $F$ defined in \S \ref{subsec:functor}
for $\bigoplus_{n\geqslant 0}RG_n\mod$ and $\bigoplus_{n\geqslant 0}R\widetilde{G}_n\mod$, respectively.

\subsubsection{Functors between ${\rm O}_{2n+1}(q)\mod$ and ${\rm SO}_{2n+1}(q)\mod$
} \label{sub:resOtoSO}
Here we define
the functors arising from the natural embedding and
 projection between $\widetilde{G}_n$ and $G_{n}$,
 and investigate their influence on the previous functors $F,F', \rm{Spin}$ and $\rm{Det}$.

\smallskip

Let $\iota_n:G_{n}\hookrightarrow \widetilde{G}_{n}\cong G_{n}\times\{\pm \id_{\widetilde{G}_n}\}$
be the natural embedding $g\mapsto g\times \id_{\widetilde{G}_n}$
and $\pi_n: \widetilde{G}_{n}\cong G_{n}\times\{\pm \id_{\widetilde{G}_n}\}\hookrightarrow G_{n}$
be the projection $g\times(\pm \id_{\widetilde{G}_n})\mapsto g.$
Then $\iota_n$ and $\pi_n$ induce the restriction functor
$\iota_n^{*}=\mathrm{Res}^{\widetilde{G_n}}_{G_n}:R\widetilde{G}_n\mod\to
RG_{n}\mod$
and the inflation functor
$\pi_n^{*}=\mathrm{Inf}^{\widetilde{G}_n}_{G_n}:RG_{n}\mod\to R\widetilde{G}_n\mod$,
respectively.
Clearly, $\pi_n\circ \iota_n=\id$, and so
$\mathrm{Res}^{\widetilde{G}_n}_{G_n}\circ \mathrm{Inf}^{\widetilde{G}_n}_{G_n}\cong\Id_{G_n}$.
Hence we may view $RG_{n}\mod$ as a full
sub-category of $R\widetilde{G}_n\mod.$  Now we define  $${\rm Res}^{\widetilde{G}}_G:=\bigoplus\limits_{n\geqslant 0}{\rm Res}^{\widetilde{G}_n}_{G_n}:\bigoplus_{n\geqslant 0}R\widetilde{G}_{n}\mod\to\bigoplus_{n\geqslant 0}RG_{n}\mod.$$

\smallskip

Observe that $\mathrm{Res}^{\widetilde{G}_n}_{G_n}$ is represented by the $(RG_n,R\widetilde{G}_n)$-bimodule $_{RG_n}{R\widetilde{G}_n}_{R\widetilde{G}_n}$
and that
$\mathrm{Inf}^{\widetilde{G}_n}_{G_n}$ is represented by the $(R\widetilde{G}_n,RG_n )$-bimodule
${R\widetilde{G}_n \sigma_{n,\,+}}$, where $\sigma_{n,\,+}$
is the (central) idempotent $(\id_{\widetilde{G}_n}+c_n)/2$ of $R\widetilde{G}_n$ with $c_n:=-\id_{\widetilde{G}_n}\in Z(\widetilde{G}_n)$.

\smallskip

Now let $\sigma_{n,-}=(\id_{\widetilde{G}_n}-c_n)/2$ be the (central) idempotent of $R\widetilde{G}_n$,
so that $\id_{\widetilde{G}_n}=\sigma_{n,\,+}+\sigma_{n,-}$ and $\sigma_{n,\,+}\sigma_{n,-}=\sigma_{n,-}\sigma_{n,\,+}=0.$
Let $\mathrm{Inf}^{\widetilde{G}_n,{-1}}_{G_n}:RG_{n}\mod\to R\widetilde{G}_n\mod$
be the functor sending an $RG_{n}$-module $M$ to
the $R\widetilde{G}_n$-module $R\widetilde{G}_n \sigma_{n,-}\otimes_{R G_n}M.$
We also have $\mathrm{Res}^{\widetilde{G}_n}_{G_n}\circ \mathrm{Inf}^{\widetilde{G}_n,{-1}}_{G_n}\cong \Id_{G_n},$
and may also view $RG_n\mod$ as a full subcategory of
$R\widetilde{G}_n\mod$ through $\mathrm{Inf}^{\widetilde{G}_n,{-1}}_{G_n}$.
Here we mention that unlike $\mathrm{Inf}^{\widetilde{G}_n}_{G_n}$,
the functor $\mathrm{Inf}^{\widetilde{G}_n,{-1}}_{G_n}$
can not be induced by any group homomorphism between $\widetilde{G}_n$ and $G_n$.

\smallskip

With the embedding map $\iota_{n+1,n}:\widetilde{G}_n\to\widetilde{G}_{n+1},g\mapsto\diag(1,g,1)$, 
we define the natural transformations $\Theta_n: F^{G}_{n+1,n}\circ\mathrm{Res}^{\widetilde{G}_n}_{G_n}
\rightarrow\mathrm{Res}^{\widetilde{G}_{n+1}}_{G_{n+1}}\circ F^{\widetilde{G}}_{n+1,n}$ and
$\Theta_n':(F'_{n+1,n})^{G}\circ\mathrm{Res}^{\widetilde{G}_{n}}_{G_n}
\rightarrow\mathrm{Res}^{\widetilde{G}_{n+1}}_{G_n}\circ (F'_{n+1,n})^{\widetilde{G}}$
 via the $(RG_{n+1},R\widetilde{G}_n)$-bimodule maps
 $$\begin{array}{rrcl}
 RG_{n+1}e_{n+1,n}\otimes_{RG_n}R\widetilde{G}_n &\to& R\widetilde{G}_{n+1}\otimes_{R\widetilde{G}_{n+1}} R\widetilde{G}_{n+1}e_{n+1,n}\cong
 R\widetilde{G}_{n+1}e_{n+1,n} \\
 ge_{n+1,n}\otimes h &\mapsto&  g\iota_{n+1,n}(h)e_{n+1,n}
\end{array}
$$
and
$$
\begin{array}{rrcl}
RG_{n+1}e'_{n+1,n}\otimes_{RG_n}R\widetilde{G}_n &\to& R\widetilde{G}_{n+1}\otimes_{R\widetilde{G}_{n+1}} R\widetilde{G}_{n+1}e'_{n+1,n}\cong R\widetilde{G}_{n+1}e'_{n+1,n} \\
ge'_{n+1,n}\otimes h &\mapsto& g\iota_{n+1,n}(h)e'_{n+1,n},
\end{array}
$$
respectively, for all $g\in G_{n+1}$ and $h\in \widetilde{G}_n$.
Finally, we set
$$\Theta:=\bigoplus\limits_{n\geqslant 0}\Theta_n\ \ \ {\rm and}~ ~ ~\Theta':=\bigoplus\limits_{n\geqslant 0}\Theta'_n.$$

%

\begin{theorem}\label{Thm:iso-O-SO}
We have the following natural isomorphisms of functors:
\begin{itemize}[leftmargin=8mm]
  \item[$(a)$] $\Theta: F^{G}\circ\mathrm{Res}^{\widetilde{G}}_{G}
\xrightarrow{\sim}\mathrm{Res}^{\widetilde{G}}_{G}\circ F^{\widetilde{G}}$ and
$\Theta':(F')^{G}\circ\mathrm{Res}^{\widetilde{G}}_{G}
\xrightarrow{\sim}\mathrm{Res}^{\widetilde{G}}_{G}\circ (F')^{\widetilde{G}}$;
  \item[$(b)$] $\mathrm{Inf}^{\widetilde{G}_{n+1},\,\epsilon}_{G_{n+1}}\circ F^{G}_{n+1,n}\cong F^{\widetilde{G}}_{n+1,n}\circ \mathrm{Inf}^{\widetilde{G}_n,\,\epsilon}_{G_n}$  and
$(F')^{\widetilde{G}}_{n+1,n}\circ \mathrm{Inf}^{\widetilde{G}_n,\,\epsilon}_{G_n} \cong
\mathrm{Inf}^{\widetilde{G}_{n+1},\,\epsilon'}_{G_{n+1}}\circ (F')^{G}_{n+1,n}$, where $\epsilon'=\epsilon\zeta(-1)$;
\item[$(c)$]$\mathrm{Spin}_n\circ\mathrm{Res}^{\widetilde{G}_{n}}_{G_{n}}
\cong\mathrm{Res}^{\widetilde{G}_{n}}_{G_{n}}\circ \mathrm{Spin}_n$ and
$\mathrm{Res}^{\widetilde{G}_{n}}_{G_{n}}
\cong \mathrm{Res}^{\widetilde{G}_{n}}_{G_{n}}\circ \mathrm{Det}_n.$
\end{itemize}
\end{theorem}
\begin{proof} The proof is routine, and is left to the reader.
\end{proof}

\subsubsection{Central elements as  natural transformations of identity functor} \label{sub:id-det-spin-ff'}
Let $C_n$ be the natural transformation of
identity functor $\Id_{\widetilde{G}_n}=R\widetilde{G}_n\otimes_{R\widetilde{G}_n}-$
 of $R\widetilde{G}_{n}\mod$ given by  right multiplication by $c_n= -\id_{\widetilde{G}_{n}}\in Z(\widetilde{G}_{n})$.
The following result shows
the influence of the natural transformation $C:=\bigoplus_{n\in\bbN} C_n$
on the previous functors $F,F', \rm{Spin}$ and $\rm{Det}$.

\begin{proposition}\label{prop:id-det-spin-ff'}
\begin{itemize}[leftmargin=8mm]
  \item[$(a)$]
   The natural isomorphisms $\Xi$
   between the functors $F\circ \Id_{\widetilde{G}}$ and $\Id_{\widetilde{G}}\circ F$,
   and $\Xi'$ between the functors $F'\circ \Id_{\widetilde{G}}$ and $\Id_{\widetilde{G}}\circ F'$
   satisfy the following commutative diagrams:
     \begin{itemize}
      \item[$(a1)$]
  $$\xymatrix{F\circ  \Id_{\widetilde{G}}\ar[r]_{\Xi}^{\sim}\ar[d]^{1_{F} C} &\Id_{\widetilde{G}}\circ F \ar[d]^{C1_{F}}\\F\circ
\Id_{\widetilde{G}_n}\ar[r]_{\Xi}^{\sim}& \Id_{\widetilde{G}}\circ  F
}$$
  \item[$(a2)$]
$$\xymatrix{F'\circ  \Id_{\widetilde{G}}\ar[r]_{\Xi'}^{\sim}\ar[d]^{1_{F'} C} &\Id_{\widetilde{G}}\circ  F' \ar[d]^{\zeta(-1)\cdot C1_{F'}}\\F'\circ
\Id_{\widetilde{G}}\ar[r]_{\Xi'}^{\sim}& \Id_{\widetilde{G}} \circ F'
}$$
    \end{itemize}

\item[$(b)$] The natural isomorphism $\Pi_n$
  between the functors $\mathrm{Spin}_{n}\circ \Id_{\widetilde{G}_n}$ and $\Id_{\widetilde{G}_n}\circ \mathrm{Spin_n}$
   satisfies the following commutative diagram:
     $$\xymatrix{\mathrm{Spin}_{n}\circ  \Id_{\widetilde{G}_n}\ar[r]_{\Pi_n}^{\sim}\ar[d]^{1_{\rm{Spin}} C_{n}} &\Id_{\widetilde{G}_{n}}\circ  \mathrm{Spin}_{n} \ar[d]^{sp_n(c_n)\cdot C_{n}1_{\rm{Spin}}}\\\mathrm{Spin}_{n}\circ
\Id_{\widetilde{G}_n}\ar[r]_{\Pi_n}^{\sim}& \Id_{\widetilde{G}_{n}}\circ  \mathrm{Spin}_{n}
}$$
\item[$(c)$] The  natural isomorphism $\Delta_n$
between the functors $\mathrm{Det}_n\circ \Id_{\widetilde{G}_n}$ and $\Id_{\widetilde{G}_n}\circ \mathrm{Det}_n$
satisfies the following commutative diagram:
$$\xymatrix{\mathrm{Det}_{n}\circ  \Id_{\widetilde{G}_n}\ar[r]_{\Delta_n}^{\sim}\ar[d]^{1_{\rm{Det}} C_{n}} &\Id_{\widetilde{G}_{n}}\circ  \mathrm{Det}_{n} \ar[d]^{{\rm det}_n(c_n)\cdot C_{n}1_{\rm{Det}}}\\\mathrm{Det}_{n}\circ
\Id_{\widetilde{G}_n}\ar[r]_{\Delta_n}^{\sim}& \Id_{\widetilde{G}_{n}}\circ  \mathrm{Det}_{n}
}$$
\end{itemize}
\end{proposition}
\begin{proof} Here we only prove the diagram $(a2)$, since the proofs of the others are similar.
To do this, recall that $\iota_n:G_n\hookrightarrow \widetilde{G}_n$ is the natural embedding
$g\mapsto g\times \id_{\widetilde{G}_n}$.
We have that $\Xi'\circ (1_{F'} C_n)$ is represented
by the $(R\widetilde{G}_{n+1},R\widetilde{G}_n)$-bimodule
endomorphism of $R\widetilde{G}_{n+1}e'_{n+1,n}$ via
$ge'_{n+1,n}\mapsto ge'_{n+1,n}\iota_n(c_{n})$, and that
$(\zeta(-1)\cdot C_{n+1}1_{F'})\circ \Xi'$ is represented by
the $(R\widetilde{G}_{n+1},R\widetilde{G}_n)$-bimodule endomorphism of $R\widetilde{G}_{n+1}e'_{n+1,n}$
via $ge'_{n+1,n}\mapsto \zeta(-1) c_{n+1}ge'_{n+1,n}$ for all $g\in \widetilde{G}_{n+1}.$
Since $c_{n+1}=\mathrm{diag}(-1,\id_{\widetilde{G}_n},-1)\cdot \iota_n(c_n)$
and $\mathrm{diag}(-1,\id_{\widetilde{G}_n},-1) e'_{n+1,n}=\zeta(-1)e'_{n+1,n}$,
it follows that the endomorphisms $1_{F'} C_{n}$ and $\zeta(-1)\cdot C_{n+1}1_{F'}$
on $R\widetilde{G}_{n+1}e'_{n+1,n}$ are same. Hence the diagram $(a2)$ is commutative.
\end{proof}
Recall from \S \ref{sub:O2n+1} that
$\scrQU_{R}^{\O}=\bigoplus\limits_{n\in\bbN} R\widetilde{G}_{n}\qumod,$
where
$$R\widetilde{G}_{n}\qumod=\bigoplus\limits_{n_++n_-=n,\,\epsilon\in\{\pm 1\}}R\widetilde{G}_{n}e^{R\widetilde{G}_{n},\,\epsilon}_{s_{n_+,n_-}}\mod$$
with $\epsilon\in\{\pm\}.$
We denote $$\scrQU^{\O,\,\epsilon}_R=\bigoplus\limits_{n_++n_-=n,\,\,n\geqslant 0}  R\widetilde{G}_ne_{s_{n_+,n_-}}^{R\widetilde{G}_n,\,\epsilon}\mod.$$ Notice that the action of $C$ on any module in the category $\scrQU^{\O,\,\epsilon}_R$
is exactly by $\epsilon\cdot\id$.

\begin{corollary}\label{cor:sign} We have
\begin{itemize}[leftmargin=8mm]
    \item[$(a)$]
 $[F]([\scrQU^{\O,\,\epsilon}_K])\subset [\scrQU^{\O,\,\epsilon}_K]$ and $[F']([\scrQU^{\O,\,\epsilon}_K])\subset [\scrQU^{\O,\,\zeta(-1)\epsilon}_K]$;
 \item [$(b)$]
 $[\mathrm{Spin}_n]([K\widetilde{G}_{n}e^{K\widetilde{G}_{n},\,\epsilon}_{s_{n_+,n_-}}\mod])\subset [K\widetilde{G}_{n}e^{K\widetilde{G}_{n},\,\zeta(-1)^{n}\epsilon}_{s_{n_-,n_+}}\mod]$
and $[\mathrm{Det}]([\scrQU^{\O,\,\epsilon}_K])\subset [\scrQU^{\O,\,-\epsilon}_K];$
\item[$(c)$]$[\mathrm{Res}^{\widetilde{G}_n}_{G_n}]([\scrQU^{O,\varepsilon}_K])= [\scrQU^{\SO}_K]$ and
$[\mathrm{Inf}^{\widetilde{G}_n,\,\varepsilon}_{G_n}]([\scrQU^{SO}_K])= [\scrQU^{\O,\,\varepsilon}_K].$
 \end{itemize}
\end{corollary}

\begin{proof}
Clearly, the modules $E_{\Theta+,\Theta_-,+}$ and $E_{\Theta+,\Theta_-,-}$ can
be distinguished by 
the values of their corresponding characters at the central element $c_n=-\id_{\widetilde{G}_n}.$ Also, we have ${\rm det}_n(c_n)=-1$ and $sp_n(c_n)=\zeta(-1)^{n}$ by Lemma \ref{Lem:spinor} $(4).$
Thus $(a)$ and $(b)$ follow from Proposition \ref{prop:id-det-spin-ff'}.
For $(c)$, it directly follows by the definition of the restriction and inflation functors.
\end{proof}

Theorem \ref{Thm:iso-O-SO} and Corollary \ref{cor:sign}
have the following consequence.
\begin{corollary}\label{cor:spindet}
Let $E_{\Theta_+,\Theta_-,\epsilon}=E_{\Theta_+,\Theta_-}\otimes K_\epsilon$,
where $E_{\Theta_+,\Theta_-}$ is the quadratic unipotent character
of $G_n$ as in \S \ref{subsec:quchar}.
Then $\mathrm{Det}_n (E_{\Theta_+,\Theta_-,\epsilon})\cong E_{\Theta_+,\Theta_-,-\epsilon}$ and
 $\mathrm{Spin}_n (E_{\Theta_+,\Theta_-,\epsilon})\cong E_{\Theta_-,\Theta_+,\zeta(-1)^n\epsilon}.$

In particular, $\mathrm{Det}_n( E_{t_+,t_-,\,\epsilon})\cong E_{t_+,t_-,-\epsilon}$ and $\mathrm{Spin}_n( E_{t_+,t_-,\,\epsilon})\cong E_{t_-,t_+,\epsilon}.$\qed
\end{corollary}
 Although $\scrQU^{\O,\,+}_R$ and $\scrQU^{\O,\,-}_R$ are Morita equivalent to $\scrQU^{\SO}_R$
 via the restriction functor,
 Corollary \ref{cor:sign} shows that neither
$\scrQU^{\O,\,+}_R$ nor $\scrQU^{\O,\,-}_R$ is $F'$-stable
when $\zeta(-1)=-1$.
In the next subsection, we shall find  subcategories of $\scrQU_{R}^{\O}$ which are both $F$- and $F'$-stable in general
and Morita equivalent to $\scrQU^{\SO}_R$.

\subsubsection{Categories Morita equivalent to $\scrQU^{\SO}_R$}\label{ssub:Morita+-}

We define two subcategories of $\scrQU_{R}^{\O}$:
$$\scrQU_{R}^{\O,\,(+)}:=\bigoplus\limits_{n\in\bbN}\scrQU_{R}^{\widetilde{G}_{n},\,(+)}~{\rm where}~
\scrQU_{R}^{\widetilde{G}_{n},\,(+)}:=
\bigoplus\limits_{n_++n_-=n}R\widetilde{G}_{n}e^{R\widetilde{G}_{n},\,\zeta(-1)^{n_-}}_{s_{n_+,n_-}}\mod,$$
$$\scrQU_{R}^{\O,\, (-)}:=\bigoplus\limits_{n\in\bbN}\scrQU_{R}^{\widetilde{G}_{n},\,(-)}
~{\rm where}~ \scrQU_{R}^{\widetilde{G}_{n},\,(-)}:=
\bigoplus\limits_{n_++n_-=n}R\widetilde{G}_{n}e^{R\widetilde{G}_{n},\,-\zeta(-1)^{n_-}}_{s_{n_+,n_-}}\mod.$$
Then
$\scrQU_{R}^{\O}
={\scrQU}_{R}^{\O,\,(+)}\bigoplus {\scrQU}_{R}^{\O,\,(-)}$.
Clearly, if $\zeta(-1)=1$ then $\scrQU_{R}^{\O,\,(+)}=\scrQU_{R}^{\O,\,+}$ and $\scrQU_{R}^{\O,\,(-)}=\scrQU_{R}^{\O,\,-}$.

\begin{lemma}\label{Lem:stable} The following statements hold.
\begin{itemize}[leftmargin=8mm]
  \item[$(a)$] The subcategories ${\scrQU}_{R}^{\O,\,(+)}$ and ${\scrQU}_{R}^{\O,\,(-)}$
  are both $F^{\widetilde{G}}$- and $(F')^{\widetilde{G}}$-stable.
  \item[$(b)$] The subcategories ${\scrQU}_{R}^{\O,\,(+)}$ and ${\scrQU}_{R}^{\O,\,(-)}$ are
${\rm Spin}$-stable.
\item[$(c)$] The functor ${\rm Det}$ interchanges ${\scrQU}_{R}^{\O,\,(+)}$ and ${\scrQU}_{R}^{\O,\,(-)}$.
\end{itemize}
\end{lemma}

\begin{proof}
Let $E_{\Theta_+,\Theta_-}\in K{G}_{n}e^{K{G}_{n}}_{s_{n_+,n_-}}\mod$
 be the irreducible module as defined   in \S \ref{subsec:quchar},
 and $E_{\Theta_+,\Theta_-,\, \varepsilon}\in K\widetilde{G}_{n}e^{K\widetilde{G}_{n},\,\epsilon}_{s_{n_+,n_-}}\mod$
 be the irreducible module as defined  in \S \ref{sub:O2n+1}.
  By Theorem \ref{Thm:iso-O-SO} (a),
 the restriction of the module
 $F^{\widetilde{G}}(E_{\Theta_+,\Theta_-,\, \varepsilon})$
  (resp. $(F')^{\widetilde{G}}(E_{\Theta_+,\Theta_-,\, \varepsilon})$ ) to $KG_{n+1}$
  is isomorphic to the module $F^G(E_{\Theta_+,\Theta_-})$ (resp. $(F')^G(E_{\Theta_+,\Theta_-})$).
However, Remark \ref{rem:functors} (3) says that
$$F^G(E_{\Theta_+,\Theta_-})\in KG_{n+1}e^{KG_{n+1}}_{s_{n_++1,n_-}}\mod~\text{and}~
(F')^G(E_{\Theta_+,\Theta_-})\in KG_{n+1}e^{KG_{n+1}}_{s_{n_+,n_-+1}}\mod.$$
Therefore, by Corollary \ref{cor:sign} (a) and (c), we have
that
$$F^{\widetilde{G}}(E_{\Theta_+,\Theta_-,\, \varepsilon})\in K\widetilde{G}_{n}e^{K\widetilde{G}_{n},\,\epsilon}_{s_{n_++1,n_-}}\mod
~\text{and}~
(F')^{\widetilde{G}}(E_{\Theta_+,\Theta_-,\, \varepsilon})\in K\widetilde{G}_{n}e^{K\widetilde{G}_{n},\,\zeta(-1)\epsilon}_{s_{n_+,n_-+1}}\mod.$$
Hence
   $$[F^{\widetilde{G}}]([K\widetilde{G}_{n}e^{K\widetilde{G}_{n},\,\epsilon}_{s_{n_+,n_-}}\mod])\subset [K\widetilde{G}_{n+1}e^{K\widetilde{G}_{n+1},\,\epsilon}_{s_{n_++1,n_-}}\mod]$$
   and
$$[(F')^{\widetilde{G}}]([K\widetilde{G}_{n}e^{K\widetilde{G}_{n},\,\epsilon}_{s_{n_+,n_-}}\mod])
     \subset [K\widetilde{G}_{n+1}e^{K\widetilde{G}_{n+1},\,
     \zeta(-1)\epsilon}_{s_{n_+,n_-+1}}\mod].$$
Now (a) follows by the
$\bbZ$-linear isomorphisms $d_\scrQU : [\scrQU^{\O}_K]\simto[\scrQU^{\O}_k]$ (see \S \ref{sub:O2n+1}).

\smallskip

For (b), we note by Corollary \ref{cor:spindet} that
 $$\mathrm{Spin}_n (E_{\Theta_+,\Theta_-,\,\zeta(-1)^{n_-}\epsilon})
 \cong E_{\Theta_-,\Theta_+,\,\zeta(-1)^{n}\zeta(-1)^{n_-}\epsilon}\cong E_{\Theta_-,\Theta_+,\,\zeta(-1)^{n_+}\epsilon},$$
 By Theorem \ref{Thm:iso-O-SO} (c),
 we obtain
$\mathrm{Spin}_n (E_{\Theta_+,\Theta_-,\,\zeta(-1)^{n_-}\epsilon})\cong E_{\Theta_-,\Theta_+,\zeta(-1)^{n_+}\epsilon}.$
  Hence
  $[\mathrm{Spin}]([\scrQU^{\O,\,(\epsilon)}_K])\subset [\scrQU^{\O,\,(\epsilon)}_K]$,
  and so (b) follows.

\smallskip

Finally, again by Corollary \ref{cor:spindet},
we have
 $$\mathrm{Det}_n (E_{\Theta_+,\Theta_-,\,\zeta(-1)^{n_-}\epsilon})\cong E_{\Theta_-,\Theta_+,\,-\zeta(-1)^{n_-}\epsilon},$$
  and so $[\mathrm{Det}]([\scrQU^{\O,\,(\epsilon)}_K])\subset [\scrQU^{\O,\,(-\epsilon)}_K]$.
  Thus $(c)$ holds.
\end{proof}

\smallskip


For $\epsilon=\pm$, we let $\mathrm{Res}^{\scrQU,\,(\epsilon)}_{G_n}$
be the restriction of the functor $\mathrm{Res}^{\widetilde{G}_{n}}_{G_n}$ on ${\scrQU_{R}^{\widetilde{G}_{n},\,(\epsilon)}}$,
and denote
$$\mathrm{Res}^{\scrQU,\,(\epsilon)}:=\bigoplus\limits_{n\in\bbN}\mathrm{Res}^{\scrQU,\,(\epsilon)}_{G_n}.$$

%
%
%

\begin{proposition}\label{prop:res}
The  functor $\mathrm{Res}^{\scrQU,\,(+)}$ induces a Morita equivalence between
$\scrQU_{R}^{\O,\,(+)}$ and $\scrQU_{R}^{\SO}$.

Similarly, the functor  $\mathrm{Res}^{\scrQU,\,(-)}$
induces a Morita equivalence between
$\scrQU_{R}^{\O,\,(-)}$ and $\scrQU_{R}^{\SO}$ as well.
%
\end{proposition}
\begin{proof} It suffices to show the proposition  in the case where $R=\mathcal O.$ It directly follows from \cite[Theorem 9.18]{CE04} and the fact that
the functor $\mathrm{Res}^{\scrQU,\,(+)}_{G_n}$ (resp. $\mathrm{Res}^{\scrQU,\,(-)}_{G_{n}}$)
 induces a bijection between $\Irr(\scrQU_{R}^{\widetilde{G}_{n},\,(+)}\otimes_{R}K)$ (resp. $\Irr(\scrQU_{R}^{\widetilde{G}_{n},\,(-)}\otimes_{R}K)$) and $\Irr((RG_{n}\qumod)\otimes_{R}K)$.
\end{proof}

In general,  Morita equivalences may not be enough to
pass the representation datum on one category to another.
However, as we will see in the next subsection, the functor $\mathrm{Res}^{\scrQU,\,(+)}$ does pass the
$\mathfrak{A}(\fraks\frakl_{\K_q})$-representation datum on $\scrQU_{R}^{\O,\,(+)}$
to $\scrQU^{\SO}_R$.

\subsubsection{Construction of an
$\mathfrak{A}(\fraks\frakl_{\K_q})$-representation datum on $\scrQU^{\SO}_R$}
\label{sub:explicit-repdatumSO}
In this section
we construct an $\mathfrak{A}(\fraks\frakl_{\K_q})$-representation datum of $\scrQU^{\SO}_R$ where $R\in \{K,\mathcal O,k\}$,
based on the previous preparation.
Recall that for $R\in \{K,\mathcal O,k\}$,
$$\scrQU_R^{\SO}=\bigoplus\limits_{n\in\bbN}RG_n\qumod,~{\rm where}~
RG_n\qumod=\bigoplus\limits_{n_++n_-=n}RG_ne^{RG_n}_{s_{n_+,n_-}}\mod.$$
We write
$e_{\scrQU_R,n}=\sum\limits_{n_++n_-=n}e^{RG_n}_{s_{n_+,n_-}}$ for the identity idempotent of $RG_n\qumod$.

\smallskip

First, we restrict
the functors $F=\bigoplus_{n\in\bbN}F_{n+1,n}$ and $F'=\bigoplus_{n\in\bbN}F'_{n+1,n}$ defined
in \S \ref{subsec:functor} to $\scrQU_R^{\SO}$, and decompose them into blockwise forms.
We have known that the functor $F_{n+1,n}$ is represented by the $(RG_{n+1},RG_n)$-bimodule
 $RG_{n+1}e_{n+1,n}$.
  So the restriction of $F_{n+1,n}$ on $RG_n\qumod$
is represented by the $(RG_{n+1},RG_n)$-bimodule  $RG_{n+1}e_{n+1,n}e_{\scrQU,n}$.
Similarly, the restriction of $F'_{n+1,n}$ on $RG_n\qumod$
is represented by the $(RG_{n+1},RG_n)$-bimodule  $RG_{n+1}e'_{n+1,n}e_{\scrQU,n}$.
We start with the structure of these two bimodules.
\smallskip

\begin{lemma}\label{Lem:bimoduleisom}We have the following  $(RG_{n+1},RG_n)$-bimodule isomorphisms:
\begin{itemize}[leftmargin=8mm]
 \item[$(a)$] $RG_{n+1}e_{n+1,n}e_{\scrQU,n}\cong e_{\scrQU_{\mathcal O},n+1}\mathcal O G_{n+1}e_{n+1,n}e_{\scrQU_{\mathcal O},n},$ and
\item[$(b)$]  $RG_{n+1}e'_{n+1,n}e_{\scrQU,n}\cong e_{\scrQU_{\mathcal O},n+1}\mathcal O G_{n+1}e'_{n+1,n}e_{\scrQU_{\mathcal O},n}.$
\end{itemize}
\end{lemma}
\begin{proof} For brevity, we only give the proof of (a), and leave the proof of (b) to the reader.
First, we have  $$KG_{n+1}e_{n+1,n}e_{\scrQU_K,n}=e_{\scrQU_K,n+1}KG_{n+1}e_{n+1,n}e_{\scrQU_K,n}$$
since $[F][KG_n\qumod]\subset [KG_{n+1}\qumod].$
Now, let $R=k.$ Recall that Lusztig induction preserves the quadratic unipotent modules by Remark \ref{rem:functors}(3).
Since the map $d_{\mathcal{O}G_n}: [KG_n\qumod]\rightarrow [k G_n\qumod]$ is an isomorphism
by \S\ref{subsec:SOqublock}, we have
$$[F][kG_n\qumod]=[F][KG_n\qumod]\subset [KG_{n+1}\qumod]=[kG_{n+1}\qumod].$$
Hence $kG_{n+1}e_{n+1,n}e_{\scrQU_k,n}\cong e_{\scrQU_k,n+1}kG_{n+1}e_{n+1,n}e_{\scrQU_k,n}$
as $(kG_{n+1},kG_n)$-bimodules.
Finally, the isomorphism for the case where $R=\mathcal O$ is established by
Nakayama's lemma and the above isomorphism over $k$.
\end{proof}

\begin{lemma} \label{Lem:decomposition}
\begin{itemize}[leftmargin=8mm]
\item[$(a)$]
We have the  $(RG_{n+1},RG_n)$-bimodule decomposition
$$ RG_{n+1}e_{n+1,n}e_{\scrQU_R,n}\cong
\bigoplus\limits_{n_++n_-=n}
RG_{n+1}e_{n+1,n}e^{RG_{n}}_{s_{n_+,n_-}}.
$$
Moreover, for each direct summand, we have the $(RG_{n+1},RG_n)$-bimodule isomorphism   $$
RG_{n+1}e_{n+1,n}e^{RG_{n}}_{s_{n_+,n_-}}\cong
 e^{RG_{n+1}}_{s_{n_++1,n_-}}RG_{n+1}e_{n+1,n}e^{RG_{n}}_{s_{n_+,n_-}}.$$

\item[$(b)$] We have the  $(RG_{n+1},RG_n)$-bimodule decomposition
$$
RG_{n+1}e'_{n+1,n}e_{\scrQU_R,n}\cong
\bigoplus\limits_{n_++n_-=n}
 RG_{n+1}e'_{n+1,n}e^{RG_{n}}_{s_{n_+,n_-}}.
$$
 Moreover, for each direct summand, we have the $(RG_{n+1},RG_n)$-bimodule isomorphism   $$
RG_{n+1}e'_{n+1,n}e^{RG_{n}}_{s_{n_+,n_-}}\cong e^{RG_{n+1}}_{s_{n_+,n_-+1}}
RG_{n+1}e'_{n+1,n}e^{RG_{n}}_{s_{n_+,n_-}}.$$
\end{itemize}
\end{lemma}
\begin{proof}  The decomposition in (a) follows from the fact that
 $e_{\scrQU_R,n}=\sum\limits_{n_++n_-=n}e^{RG_n}_{s_{n_+,n_-}}$ is a central idempotent decomposition in $RG_n$.
For the latter conclusion of (a), we note by Remark \ref{th:introequiv} that
$$\begin{array}{rcccl}
    [F][ke^{kG_{n}}_{s_{n_+,n_-}}\mod]&=&[F][KG_ne^{KG_{n}}_{s_{n_+,n_-}}\mod]& \subset& [KG_{n+1}e^{KG_{n+1}}_{s_{n_++1,n_-}}\mod] \\
     & &  & =& [kG_{n+1}e^{kG_{n+1}}_{s_{n_++1,n_-}}\mod].
  \end{array}
$$
Hence $$
kG_{n+1}e_{n+1,n}e^{kG_{n}}_{s_{n_+,n_-}}\cong e^{kG_{n+1}}_{s_{n_++1,n_-}}
kG_{n+1}e_{n+1,n}e^{kG_{n}}_{s_{n_+,n_-}}$$ and
 $e^{kG_{n+1}}_{s_{m_+,m_-}}kG_{n+1}e_{n+1,n}e^{kG_{n}}_{s_{n_+,n_-}}=0$
unless $m_+=n_++1$ and $m_-=n_-$.
By a similar argument as for  Lemma \ref{Lem:bimoduleisom},
we indeed have the  $(RG_{n+1},RG_n)$-bimodule decomposition
$$
 RG_{n+1}e_{n+1,n}e_{\scrQU_R,n}\cong
\bigoplus\limits_{n_++n_-=n}e^{RG_{n+1}}_{s_{n_++1,n_-}}
RG_{n+1}e_{n+1,n}e^{RG_{n}}_{s_{n_+,n_-}}.
$$

\smallskip

By using a similar argument as for (a) with
$$\begin{array}{rcccl}
[F'][ke^{kG_{n}}_{s_{n_+,n_-}}\mod]&=&[ F'][KG_ne^{KG_{n}}_{s_{n_+,n_-}}\mod] &\subset & [KG_{n+1}e^{KG_{n+1}}_{s_{n_+,n_-+1}}\mod]\\
& & &=& [kG_{n+1}e^{kG_{n+1}}_{s_{n_+,n_-+1}}\mod],
\end{array}
$$
we conclude that (b) holds.
\end{proof}

Let
$F_{n_+,n_-}^{n_++1,n_-}$ be the functor from $RG_n\qumod$ to $RG_{n+1}\qumod$
defined by
tensoring with the $(RG_{n+1},RG_n)$-bimodule $RG_{n+1}e_{n+1,n}e^{RG_{n}}_{s_{n_+,n_-}}$,
which is isomorphic to $e^{RG_{n+1}}_{s_{n_++1,n_-}}RG_{n+1}e_{n+1,n}e^{RG_{n}}_{s_{n_+,n_-}}.$
Then $F_{n_+,n_-}^{n_++1,n_-}$ can be viewed as a functor from
$RG_ne^{RG_n}_{s_{n_+,n_-}}\mod$ to $RG_{n+1}e^{RG_{n+1}}_{s_{n_++1,n_-}}\mod.$
By the decomposition in Lemma \ref{Lem:decomposition} (a), the
restriction of $F_{n+1,n}$ on $RG\qumod$ has the decomposition
$$F_{n+1,n}\cong \bigoplus\limits_{n_++n_-=n}F_{n_+,n_-}^{n_++1,n_-}.$$
\smallskip

 Analogously, we define the functor $(F')^{n_+,n_-+1}_{n_+,n_-}:RG_{n}\qumod\to RG_{n+1}\qumod$
given by tensoring with the $(RG_{n+1},RG_n)$-bimodule $
RG_{n+1}e'_{n+1,n}e^{RG_{n}}_{s_{n_+,n_-}}$, isomorphic to
$e^{RG_{n+1}}_{s_{n_+,n_-+1}}RG_{n+1}e'_{n+1,n}e^{RG_{n}}_{s_{n_+,n_-}}.$
By the decomposition in Lemma \ref{Lem:decomposition} (b),
  the restriction of $F'_{n+1,n}$ on $RG\qumod$ has the decomposition
 $$F'_{n+1,n}\cong \bigoplus\limits_{n_++n_-=n}(F')_{n_+,n_-}^{n_+,n_-+1}.$$

\smallskip

Next, we shall define the natural transformations $X_{n+1,n}$ and $X'_{n+1,n}$
of $F_{n+1,n}$ and $F'_{n+1,n}$, respectively.
 With the above decompositions,
 it suffices to define them on each direct summand $F_{n_+,n_-}^{n_++1,n_-}$ and ${F'}_{n_+,n_-}^{n_+,n_-+1}$.

\smallskip

As in \S \ref{sub:naturaltransf}, an endomorphism of 
the functor $F_{n+1,n}$ on $RG\qumod$
can be represented by an $(RG_{n+1},RG_n)$-bimodule
endomorphism of $RG_{n+1} e_{{n+1,n}}e_{\scrQU_{\mathcal O},n}$, i.e.,
by (right multiplication by) an element of $e_{{n+1,n}}e_{\scrQU_{\mathcal O},n}
RG_{n+1} e_{{n+1,n}}e_{\scrQU_{\mathcal O},n}$  centralizing $RG_n$.
Also, an endomorphism of
$F_{n_+,n_-}^{n_++1,n_-}$
can be represented by an $(RG_{n+1},RG_n)$-bimodule
endomorphism of $RG_{n+1}e_{n+1,n}e^{RG_{n}}_{s_{n_+,n_-}}$,
i.e.,  by (right multiplication by) an element of
$e_{n+1,n}e^{RG_{n}}_{s_{n_+,n_-}}RG_{n+1} e_{n+1,n}e^{RG_{n}}_{s_{n_+,n_-}}$
 centralizing $RG_n$.
Recall that we have defined $\dot{t}_{n+1}\in G_{n+1}$ and $\dot{s}_{n+2}\in G_{n+2}$ in \S \ref{sub:lift-refl} for the $\SO_{2n+1}(q)$ case.
In fact, we have
 $$\dot{t}_{n+1}= \left(\begin{array}{ccc} & &-1\\ &-\Id_{G_n}&\\-1&&\\ \end{array}\right)
  $$
and that
both $\dot{t}_{n+1}$ and $\dot{s}_{n+2}$ centralize $G_n$.
We can thus define the natural transformation
$$X_{n_+,n_-}^{n_++1,n_-}= \zeta(-1)^{n_-}q^{n}e_{n+1,n}e^{RG_n}_{s_{n_+,n_-}}\dot{t}_{n+1}e_{n+1,n}
e^{RG_n}_{s_{n_+,n_-}}$$
of  $F_{n_+,n_-}^{n_++1,n_-}$
and  the natural transformation
$$X_{n+1,n}=\sum \limits_{n_++n_-=n} X_{n_+,n_-}^{n_++1,n_-}$$
 of  $F_{n+1,n}$.

  \smallskip

  Analogously,
  we can define the natural transformation
$$(X')_{n_+,n_-}^{n_+,n_-+1}=\zeta(2)\zeta(-1)^{n_-}q^{n}e^{RG_n}_{s_{n_+,n_-}}
e'_{n+1,n}\dot{t}_{n+1}e'_{n+1,n}e^{RG_n}_{s_{n_+,n_-}}$$
of  ${F'}_{n_+,n_-}^{n_+,n_-+1}$
and  the natural transformation
$$X'_{n+1,n}=\sum \limits_{n_++n_-=n} (X')_{n_+,n_-}^{n_+,n_-+1}$$
 of  $F'_{n+1,n}$.
\begin{remark}
Note that right multiplication by the element $ e_{n+1,n}\,\dot{t}_{n+1}e_{n+1,n}
$ is equivalent to right multiplication by the element
$e_{n+1,n}e_{\scrQU_{\mathcal O},n}\,\dot{t}_{n+1}e_{n+1,n}e_{\scrQU_{\mathcal O},n}$
 on the $(RG_{n+1},RG_n)$-bimodule $RG_{n+1}e_{n+1,n}e_{\scrQU_{\mathcal O},n}$.
 In this case, the natural transformations
 represented by them will be viewed as same.
\end{remark}

\smallskip

As in \S \ref{sub:naturaltransf}, we similarly define
$$
T_{n+2,n}=qe_{{n+2,n}} \dot{s}_{n+2}\,e_{{n+2,n}},\,\,\,\,\,T'_{n+2,n}=qe'_{{n+2,n}} \dot{s}_{n+2}\,e'_{{n+2,n}},$$
and $H_{n+2,n}$ and $H_{n+2,n}'$ to be
$$qe_{{n+2,n+1}}e'_{n+1,n} \dot{s}_{n+2}\,e'_{{n+2,n+1}}e_{n+1,n}~{\rm and}~
e'_{{n+2,n+1}}e_{n+1,n} (\dot{s}_{n+2})^{-1}\,e_{{n+2,n+1}}e'_{n+1,n},$$
respectively.
More precisely, when restricting to the category $\scrQU^{\SO}$, we have
$$T_{n+2,n}=qe_{{n+2,n}}e_{\scrQU_R,n} \dot{s}_{n+2}\,e_{{n+2,n}}e_{\scrQU_R,n},\,\,\,\,\,T'_{n+2,n}=qe'_{{n+2,n}}e_{\scrQU_R,n} \dot{s}_{n+2}\,e'_{{n+2,n}}e_{\scrQU_R,n},$$
and that $H_{n+2,n}$ and $H_{n+2,n}'$ are respectively:
$$q e_{{n+2,n+1}}e'_{n+1,n}e_{\scrQU_R,n} \dot{s}_{n+2}\,e'_{{n+2,n+1}}e_{n+1,n}e_{\scrQU_R,n}~{\rm and}~$$$$
e'_{{n+2,n+1}}e_{n+1,n} e_{\scrQU_R,n}(\dot{s}_{n+2})^{-1}\,e_{{n+2,n+1}}e'_{n+1,n}e_{\scrQU_R,n}.$$
\smallskip

Now, we set$$E=\bigoplus_{n\geqslant 0}E_{n,n+1},\quad F=\bigoplus_{n\geqslant 0}F_{n+1,n},\quad
X=\bigoplus_{n\geqslant 0}X_{n+1,n},\quad T=\bigoplus_{n\geqslant 0}T_{n+2,n},$$
$$ E'=\bigoplus_{n\geqslant 0}E'_{n,n+1},\quad F'=\bigoplus_{n\geqslant 0}F'_{n+1,n},
\quad X'=\bigoplus_{n\geqslant 0}X'_{n+1,n},
\quad
T'=\bigoplus_{n\geqslant 0}T'_{n+2,n},$$
and
$$H=\bigoplus_{n\geqslant 0}{H_{n+2,n}},\qquad
H'=\bigoplus_{n\geqslant 0}{H'_{n+2,n}}.$$

 In the following, we show that
  the above definition of $X,X',T,T',H$ and $H'$ is compatible with
  the $\mathfrak{A}(\fraks\frakl_{\K_q})$-representation
 datum on $\scrQU^{\O,\,(+)}$ under the restriction functor defined in \S \ref{sub:resOtoSO}.
 In this way, we shall essentially obtain an $\mathfrak{A}(\fraks\frakl_{\K_q})$-representation
 datum on $\scrQU^{\SO}_R$ from that on  $\scrQU^{\O,\,(+)}$.
Recall that
the right superscripts $G$ and $\widetilde{G}$
are used to distinguish functors and natural transformations of $RG\mod$ and $R\widetilde{G}\mod$, respectively.

\begin{theorem}\label{thm:SO-isos}
\begin{itemize}[leftmargin=8mm]
\item[$(a)$] Let $\Theta$ and $\Theta'$ be as defined in \S \ref{sub:resOtoSO},
and $\mathrm{Res}^{\scrQU,\,(+)}$ be as defined in \S \ref{ssub:Morita+-}. Then the following diagrams commute:
\begin{itemize}
  \item[$(a1)$] $$\xymatrix{F^{G}\circ \mathrm{Res}^{\scrQU,\,(+)} \ar[r]_{\Theta}^{\sim}\ar[d]^{X^{G}1_{\rm{Res}}} &\mathrm{Res}^{\scrQU,\,(+)}\circ  F^{\widetilde{G}}\ar[d]^{1_{\rm{Res}} X^{\widetilde{G}}}\\ F^{G}\circ \mathrm{Res}^{\scrQU,\,(+)}
 \ar[r]_{\Theta}^{\sim}& \mathrm{Res}^{\scrQU,\,(+)}\circ F^{\widetilde{G}}}$$
  \item[$(a2)$]  $$\xymatrix{(F')^{G}\circ \mathrm{Res}^{\scrQU,\,(+)} \ar[r]_{\Theta'}^
 {\sim}\ar[d]^{(X')^{G}1_{\rm{Res}}} &\mathrm{Res}^{\scrQU,\,(+)}\circ (F')^{\widetilde{G}} \ar[d]^{\zeta(2)\cdot 1_{\rm{Res}} (X')^{\widetilde{G}}}\\(F')^{G}\circ \mathrm{Res}^{\scrQU,\,(+)} \ar[r]_{\Theta'}^{\sim}& \mathrm{Res}^{\scrQU,\,(+)} \circ  (F')^{\widetilde{G}}
 }$$
\end{itemize}

\item[$(b)$] The isomorphism $\Theta_1=(\Theta1_{F})\circ(1_{F}\Theta)$ between the functors $(F')^2\circ \mathrm{Res}^{\widetilde{G}}_{G}$ and $\mathrm{Res}^{\widetilde{G}}_{G}\circ F^2$  and the isomorphism $\Theta_1'=(\Theta'1_{F'})\circ(1_{F'}\Theta')$ between the functors $F^2\circ \mathrm{Res}^{\widetilde{G}}_{G}$ and $\mathrm{Res}^{\widetilde{G}}_{G}\circ (F')^2$
    satisfying the following commutative diagrams:
\begin{itemize}
  \item[$(b1)$] $$\xymatrix{F^{G}\circ F^{G}\circ  \mathrm{Res}^{\scrQU,\,(+)} \ar[r]_{\Theta_1}^{\sim}\ar[d]^{T^G1_{\rm{Res}}} &\mathrm{Res}^{\scrQU,\,(+)}\circ   F^{\widetilde{G}}\circ F^{\widetilde{G}} \ar[d]^{1_{\rm{Res}} T^{\widetilde{G}}}\\F^{G}\circ F^{G}\circ
\mathrm{Res}^{\scrQU,\,(+)} \ar[r]_{\Theta_1}^{\sim}& \mathrm{Res}^{\scrQU,\,(+)}\circ F^{\widetilde{G}}\circ F^{\widetilde{G}}
}$$
\item[$(b2)$]
$$\xymatrix{(F')^{G}\circ (F')^{G}\circ \mathrm{Res}^{\scrQU,\,(+)} \ar[r]_{\Theta_1'}^{\sim}\ar[d]^{(T')^G1_{\rm{Res}}} &\mathrm{Res}^{\scrQU,\,(+)}\circ  (F')^{\widetilde{G}}\circ (F')^{\widetilde{G}} \ar[d]^{1_{\rm{Res}} (T')^{\widetilde{G}}}\\(F')^{G}\circ (F')^{G}\circ
\mathrm{Res}^{\scrQU,\,(+)} \ar[r]_{\Theta_1'}^{\sim}& \mathrm{Res}^{\scrQU,\,(+)}\circ   (F')^{\widetilde{G}}\circ (F')^{\widetilde{G}}
}$$
\end{itemize}

\item[$(c)$]
The isomorphism $\Theta_2=(\Theta'1_{F})\circ(1_{F'}\Theta)$ between the functors $F'\circ F\circ \mathrm{Res}^{\widetilde{G}}_{G}$ and $\mathrm{Res}^{\widetilde{G}}_{G}\circ F'\circ F$  and the isomorphism $\Theta_2'=(\Theta1_{F'})\circ(1_{F}\Theta')$
between the
 functors $F\circ F'\circ \mathrm{Res}^{\widetilde{G}}_{G}$ and $\mathrm{Res}^{\widetilde{G}}_{G}\circ F\circ F'$
 satisfying the following commutative diagrams:
\begin{itemize}
  \item[$(c1)$]
$$\xymatrix{(F')^{G}\circ F^{G}\circ  \mathrm{Res}^{\scrQU,\,(+)} \ar[r]_{\Theta_2}^{\sim}\ar[d]^{H^G1_{\rm{Res}}} &\mathrm{Res}^{\scrQU,\,(+)}\circ  (F')^{\widetilde{G}}\circ F^{\widetilde{G}}\ar[d]^{1_{\rm{Res}} H^{\widetilde{G}}}\\F^{G}\circ (F')^{G}\circ
\mathrm{Res}^{\scrQU,\,(+)} \ar[r]_{\Theta'_2}^{\sim}& \mathrm{Res}^{\scrQU,\,(+)}\circ  F^{\widetilde{G}}\circ (F')^{\widetilde{G}}
}$$
  \item[$(c2)$]
$$\xymatrix{F^{G}\circ (F')^{G}\circ  \mathrm{Res}^{\scrQU,\,(+)} \ar[r]_{\Theta_2'}^{\sim}\ar[d]^{(H')^G1_{\rm{Res}}} &\mathrm{Res}^{\scrQU,\,(+)}\circ  F^{\widetilde{G}}\circ (F')^{\widetilde{G}} \ar[d]^{1_{\rm{Res}} (H')^{\widetilde{G}}}\\(F')^{G}\circ F^{G}\circ
\mathrm{Res}^{\scrQU,\,(+)} \ar[r]_{\Theta_2}^{\sim}& \mathrm{Res}^{\scrQU,\,(+)}\circ (F')^{\widetilde{G}}\circ F^{\widetilde{G}}
}$$
\end{itemize}
\end{itemize}
\end{theorem}

\begin{proof}
It suffices to prove the diagrams for each subcategory $R\widetilde{G}_ne^{R\widetilde{G}_n,\,\zeta(-1)^{n_-}}_{s_{n_+,n_-}}\mod$,
on which the action of $c_n=-\id_{\widetilde{G}_n}$ is exactly by  multiplication by $\zeta(-1)^{n_-}.$

\smallskip

For (a1), we note that
$\Theta_n\circ ((X_{n+1,n})^{G}1_{\rm {Res_n}})$ is  represented by the composition of the $(RG_{n+1},R\widetilde{G}_n)$-bimodule maps
$$\begin{array}{rcccl}
RG_{n+1}e_{n+1,n}\otimes_{RG_n}R\widetilde{G}_n &\to& RG_{n+1}e_{n+1,n}\otimes_{RG_n}R\widetilde{G}_n &\to& R\widetilde{G}_{n+1}e_{n+1,n}\\
ge_{n+1,n}\otimes h &\mapsto& g (X_{n+1,n})^G\otimes h &\mapsto& g\iota_{n+1,n}(h) (X_{n+1,n})^G,
\end{array}
$$
for all  $g\in G_{n+1}$ and
 $h\in \widetilde{G}_n$.
Also,  $(1_{\rm {Res_{n+1}}}(X_{n+1,n})^{\widetilde{G}})\circ\Theta_n$ is represented by the composition of the $(RG_{n+1},R\widetilde{G}_n)$-bimodule maps
$$\begin{array}{rcccl}
RG_{n+1}e_{n+1,n}\otimes_{RG_n}R\widetilde{G}_n &\to& R\widetilde{G}_{n+1}e_{n+1,n} &\to& R\widetilde{G}_{n+1}e_{n+1,n} \\
ge_{n+1,n}\otimes h &\mapsto& g\iota_{n+1,n}(h)e_{n+1,n} &\mapsto& g\iota_{n+1,n}(h) (X_{n+1,n})^{\widetilde{G}}.
\end{array}
$$
It is easy to see that $(\dot{t}_{n+1})^{\widetilde{G}}=(\dot{t}_{n+1})^{{G}}\cdot\diag(1,-\id_{\widetilde{G}_n},1)$.
Hence
$$\begin{array}{rcl}
 (X_{n+1,n})^{\widetilde{G}} &=& q^n e_{n+1,n}(\dot{t}_{n+1})^{{G}}\cdot\diag(1,-\id_{\widetilde{G}_n},1)e_{n+1,n} \\
  &=& q^n e_{n+1,n}(\dot{t}_{n+1})^{{G}}e_{n+1,n}\cdot \iota_{n+1,n}(c_n),
\end{array}
$$
where the first equality holds by definition, the second by the fact  $\diag(1,-\id_{\widetilde{G}_n},1)=\iota_{n+1,n}(c_n)$.

Note that when restricting on $R\widetilde{G}_ne^{R\widetilde{G}_n,\,\zeta(-1)^{n_-}}_{s_{n_+,n_-}}\mod$,
 right multiplication by the element
$(X_{n+1,n})^G$ on $(F_{n+1,n})^G$ is equivalent  to
right multiplication
by the element $(X_{n_+,n_-}^{n_++1,n_-})^G=\zeta(-1)^{n_-}q^{n}e_{n+1,n}e^{RG_n}_{s_{n_+,n_-}}
(\dot{t}_{n+1})^Ge_{n+1,n}
e^{RG_n}_{s_{n_+,n_-}}.$
It follows that when restricting on $R\widetilde{G}_ne^{R\widetilde{G}_n,\,\zeta(-1)^{n_-}}_{s_{n_+,n_-}}\mod$,
$(X_{n+1,n})^{\widetilde{G}}$ is equal to $(X_{n+1,n})^{G}$ since the action of $c_n=-\id_{\widetilde{G}_n}$ is exactly by  multiplication by $\zeta(-1)^{n_-}.$ Hence  $(1_{{\rm Res}_{n+1}}X_{n+1,n}^{\widetilde{G}})\circ\Theta_n=\Theta_n\circ (X_{n+1,n}^{G}1_{\rm {Res}_n}),$ and so $(1_{\rm {Res}}X^{\widetilde{G}})\circ\Theta=\Theta\circ (X^{G}1_{\rm {Res}}).$
\smallskip

For (a2),  we have that $\Theta_n'\circ ((X'_{n+1,n})^{G}1_{\rm {Res}_n})$ is represented by the composition of the $(RG_{n+1},R\widetilde{G}_n)$-bimodule maps
$$\begin{array}{rcccl}
RG_{n+1}e'_{n+1,n}\otimes_{RG_n}R\widetilde{G}_n &\to& RG_{n+1}e'_{n+1,n}\otimes_{RG_n}R\widetilde{G}_n &\to& R\widetilde{G}_{n+1}e'_{n+1,n}\\
ge'_{n+1,n}\otimes h &\mapsto& g (X'_{n+1,n})^G\otimes h &\mapsto& g\iota_{n+1,n}(h) (X'_{n+1,n})^G,
\end{array}
$$
Also, $(1_{\rm {Res}_{n+1}}(X_{n+1,n}')^{\widetilde{G}})\circ\Theta_n'$ is represented by the composition of the $(RG_{n+1},R\widetilde{G}_n)$-bimodule maps
$$\begin{array}{rcccl}
RG_{n+1}e'_{n+1,n}\otimes_{RG_n}R\widetilde{G}_n &\to& R\widetilde{G}_{n+1}e'_{n+1,n} &\to& R\widetilde{G}_{n+1}e'_{n+1,n}\\
ge'_{n+1,n}\otimes h &\mapsto& g\iota_{n+1,n}(h)e'_{n+1,n} &\mapsto& g\iota_{n+1,n}(h) (X'_{n+1,n})^{\widetilde{G}}.
\end{array}
$$
Hence
$$\begin{array}{rcl}
 (X'_{n+1,n})^{\widetilde{G}} &=& q^n e'_{n+1,n}(\dot{t}_{n+1})^{{G}}\cdot\diag(1,-\id_{\widetilde{G}_n},1)e'_{n+1,n} \\
  &=&q^n e'_{n+1,n}(\dot{t}_{n+1})^{{G}}e'_{n+1,n}\cdot \iota_{n+1,n}(c_n),
\end{array}
$$
where 
the second  is by the fact that
$\diag(1,-\id_{\widetilde{G}_n},1)=\iota_{n+1,n}(c_n).$
\smallskip

Note that when restricting on $R\widetilde{G}_ne^{R\widetilde{G}_n,\,\zeta(-1)^{n_-}}_{s_{n_+,n_-}}\mod$, right multiplication by the element
$(X'_{n+1,n})^G$ on $(F'_{n+1,n})^G$ is equivalent  to right multiplication by the element $((X')_{n_+,n_-}^{n_++1,n_-})^G=\zeta(2)\zeta(-1)^{n_-}q^{n}e'_{n+1,n}e^{RG_n}_{s_{n_+,n_-}}
(\dot{t}_{n+1})^Ge'_{n+1,n}
e^{RG_n}_{s_{n_+,n_-}}.$
Hence when restricting on $R\widetilde{G}_ne^{R\widetilde{G}_n,\,\zeta(-1)^{n_-}}_{s_{n_+,n_-}}\mod$,
$\zeta(2)(X'_{n+1,n})^{\widetilde{G}}$ is equal to $(X'_{n+1,n})^{G},$
 and so $(\zeta(2)\cdot1_{\rm {Res}}(X')^{\widetilde{G}})\circ\Theta'=\Theta'\circ ((X')^{G}1_{\rm {Res}}).$

 \smallskip

The proofs of the other diagrams are similar but much easier since the natural transformations involved are both globally defined.
We leave them to the reader.
\end{proof}

\begin{proposition}\label{prop:12rel-SO}
The natural transformations
$X,X',T,T',H$ and $H'$ satisfy all relations
in Proposition \ref{prop:12relations}.
\end{proposition}
\begin{proof} We have known that the natural transformations
$X^{\widetilde{G}}, (X')^{\widetilde{G}}, T^{\widetilde{G}}, (T')^{\widetilde{G}},H^{\widetilde{G}}$, $(H')^{\widetilde{G}}$
satisfy all relations (a)-(l) in Proposition \ref{prop:12relations}.
This remains true after replacing $(X')^{\widetilde{G}}$ by $\zeta(2)(X')^{\widetilde{G}}$ by Remark \ref{rem:scalar} (a).
The commutative diagrams in Theorem \ref{thm:SO-isos} correspond
the natural transformations
$(X^{\widetilde{G}}, \zeta(2) (X')^{\widetilde{G}},T^{\widetilde{G}}, (T')^{\widetilde{G}},
H^{\widetilde{G}},(H')^{\widetilde{G}})$ of $\scrQU^{\O,+}_R$ to
the natural transformations $(X,X', T,T',H,H')$ of $\scrQU^{\SO}_R$, respectively.
Hence the proposition follows.
\end{proof}

\begin{remark}One may directly check the trueness of Proposition \ref{prop:12rel-SO} despite of the
complexity of checking each of the relations within blocks.
\end{remark}

\smallskip

Based on Proposition \ref{prop:12rel-SO},
we shall obtain a representation datum on $\scrQU^{\SO}$.
For brevity, we  use the same notation $E$ and $E'$ to denote the restriction of
the functors $E$ and $E'$ of $RG\mod$
to $RG\qumod$,  respectively.
Now we define $I,I',\K_q,x_s$ and $\tau_{st}$ in the same way as in \S \ref{sec:repdatumN}
and have the decompositions
$$E=\bigoplus_{i\in I} E_i,\ \ \ \
F=\bigoplus_{i\in I} F_i,\ \ \ \
E'=\bigoplus_{i'\in I'} E'_{i'},\ \ \ \
F'=\bigoplus_{i'\in I'} F'_{i'}$$
where $X-i$ is locally nilpotent on $E_i$ and $F_i$ and $X'-i'$ is locally nilpotent on $E'_{i'}$ and $F'_{i'}$,
 respectively.
Repeating the argument of Theorem \ref{thm:repdatumKq}, we have
\begin{theorem} \label{thm:repdatumSO} The natural transformations $x_s$ and
$\tau_{st}$  satisfy the quiver Hecke relations for all $s,t\in \K_q$.
That is, the tuple $$\big(\{E_s\}_{s\in \K_q},\{F_s\}_{s\in\K_q},\{x_s\}_{s\in \K_q},\{\tau_{st}\}_{s,t\in \K_q}\big)$$
is an $\mathfrak{A}(\fraks\frakl_{\K_q})$-representation datum
on $\scrQU^{\SO}$.\qed
\end{theorem}

With Theorem \ref{thm:repdatumSO},
 we shall also call the tuple $(E,F,X,T;E',F',X',T';H,H')$
 a representation datum (of $\mathfrak{A}(\fraks\frakl_{\K_q})$) on $\scrQU^{\SO}$.

\smallskip

In \S \ref{subsec:spinor}, we have investigated the relationship between
the generalized eigenvalues of $X^{\widetilde{G}}$ on $F^{\widetilde{G}}$ and $(X')^{\widetilde{G}}$ on $(F')^{\widetilde{G}}$.
Now we can also investigate the the relationship between the generalized eigenvalues of $X$ on $F$ and $X'$ on $F'$.

\begin{proposition} \label{prop:spin-SO}  %
Let $\Phi$ and $\Phi'$ be the $(RG_{n+1},RG_n)$-bimodule isomorphisms in Lemma \ref{lem:Phi&Phi'}.
Then the commutative diagrams in Proposition \ref{prop:spinor}
also hold true for $\scrQU^{\SO}$ with $\beta=1$.
\end{proposition}
\begin{proof}
Note that $\scrQU^{\O,\,(+)}_R$ is stable under the functor $\widetilde{\mathrm{Spin}}$
by Lemma \ref{Lem:stable}.
By Propositions \ref{prop:spinor} and \ref{prop:det},
we see that the functor $\widetilde{\mathrm{Spin}}$ swaps the following pairs $(F^{\widetilde{G}},(F')^{\widetilde{G}})$, $(X^{\widetilde{G}}, \zeta(2)(X')^{\widetilde{G}})$, $(T^{\widetilde{G}},(T')^{\widetilde{G}})$ and $(H^{\widetilde{G}},q(H')^{\widetilde{G}}).$
Also,  Theorem \ref{Thm:iso-O-SO} $(c)$ implies that
$$\mathrm{Res}^{\scrQU,\,(+)}_{G_n}\circ\rm{Spin}\cong \rm{Spin}\circ \mathrm{Res}^{\scrQU,\,(+)}_{G_n}.$$
Hence the proposition holds by Theorem \ref{thm:SO-isos}.
 \end{proof}

Finally,
we have the following Corollary.

\begin{corollary} \label{cor:I=I'SO}
Let $\scrI$ and $\scrI'$ be the sets of
the generalized eigenvalues of $X$ on $F$ and of $X'$ on $F'$, respectively.
Then $\scrI= \scrI'$ and we have the isomorphism
$\mathrm{Spin}\circ F_i\circ \mathrm{Spin}\cong F'_i$ for each $i\in \scrI. $ \qed
\end{corollary}

By Corollary \ref{cor:I=I'SO},
we may roughly say that $\scrQU_R^{\SO}$
has an ``$\mathfrak{A}(\fraks\frakl_{I}\bigoplus\fraks\frakl_{I})\rtimes \bbZ/{2\bbZ}$''-representation datum.

\medskip

\section{The categories of quadratic unipotent modules $\scrQU^{\SO}_K$ and $\scrQU^{\SO}_k$\label{sec:cat-uK-and-uk}}\label{cha:QU}

%

We shall mainly focus on $\scrQU^{\SO}_R$ instead of $\scrQU^{\O}_R$,
and simply write $\scrQU_R$ for $\scrQU^{\SO}_R$, where $R\in \{K,\mathcal{O},k\}$.
The main purposes of this section can be listed as follows:

\begin{itemize}[leftmargin=8mm]
  \item  Describe the structure of the endomorphism algebra of a functor that is
  denoted by $\bigoplus_{\tuple\nu\in \mathbb{J}_{\tuple m}}F^{\tuple \nu}$
  with $F^{+}:=F$ and $F^{-}:=F'$ as
   defined in \S \ref{sub:explicit-repdatumSO}
   (see \S \ref{sec:End-F+F-}).
  \item Show the connection of $\End(\bigoplus_{\tuple\nu\in \mathbb{J}_{\tuple m}}F^{\tuple \nu}(E_{t_+,t_-}))$
with the Howlett-Lehrer algebra (i.e., the algebra considered by Howlett-Lehrer), where
 $E_{t_+,t_-}$ is a  quadratic unipotent cuspidal module of $G_r=\SO_{2r+1}(q)$. Also,
 we shall show that they have the structure of a cyclotomic quiver Hecke algebra (see \S \ref{subsec:QUK}).
  \item Determine the structure of the ramified Hecke algebra related to  $E_{t_+,t_-}$. The main ingredient
  is to compute the eigenvalues of $X$ on $F$  and of $X'$ on $F'$.
  \item Describe a categorical action on the category $\scrQU_K$ in terms of
        its representation datum constructed in \S \ref{sub:explicit-repdatumSO}
          (see \S \ref{subsec:QUK}).
  \item For positive characteristic,
  we categorify $\scrQU_k$ in the cases of linear and unitary primes.
\end{itemize}

Recall that $\ell \nmid q$, and  $f$ (resp. $d$) is the order of $q$ (resp. $q^2$) in $k^\times$.
We will always assume that both $q$ and $\ell$ are odd, and that $f > 1$.
In particular $q(q-1) \in \mathcal{O}^\times$ and we can apply the previous constructions
with $R$ being any ring among $(K,\mathcal{O},k)$.

\smallskip

\subsection{Endomorphism algebras of functors}\label{sec:End-F+F-}

Here we describe the structure of the endomorphism algebra of a functor
that is the direct sum of compositions of some functors $F$s and $F'$s of $\scrQU_R$ defined in
\S \ref{sub:explicit-repdatumSO}.

\subsubsection{Sign vectors}\label{subsec:signvectors}
To deal with different ways of composition of a number of the functors $F$s  and $F'$s,
it is convenient to make use of sign vectors. Let $\mathbb{J} = \{+,-\}$.
For a $2$-composition $\tuple m = (m_+,m_-)\comp_2 m$,
we define the following set of sign vectors:
 $$\mathbb{J}_{\tuple m}=\mathbb{J}_{ m_+,m_-} = \{ \tuple\nu = (\nu_1, \ldots, \nu_m) \in  \mathbb{J}^m |
\sharp\{\nu_i|\nu_i=+\}=m_+, \sharp\{\nu_i|\nu_i=-\}=m_- \}.$$
Clearly, the cardinality $|\mathbb{J}_{\tuple m}|$ of $\mathbb{J}_{\tuple m}$ is $\frac{m!}{m_+ !  m_-!}$.

 \smallskip

Naturally, the symmetric group $\mathfrak{S}_m$
transitively acts on $\mathbb{J}_{\tuple m}$ via
 $$w \cdot \tuple \nu=w \cdot (\nu_1, \ldots, \nu_m) \coloneqq (\nu_{w^{-1}(1)}, \ldots, \nu_{w^{-1}(n)}),$$
where $w\in \mathfrak{S}_m$ and $\tuple \nu=(\nu_1, \ldots, \nu_m)\in \mathbb{J}_{\tuple m}$.
The stabilizer of $$\tuple\nu_0^{\tuple m} \coloneqq (+, \dots, +,  -, \dots, -)\in \mathbb{J}_{\tuple m}$$
is exactly $\mathfrak{S}_{m_+,m_-}= \mathfrak{S}_{m_+}\times \mathfrak{S}_{m_-}$ where $\mathfrak{S}_{m_+}$ is the symmetric group on the set $\{ 1, \dots, m_+\}$ and $\mathfrak{S}_{m_-}$ is the symmetric group on the set $\{m_++ 1, \dots, m\}$.
Hence there is a
 natural bijection $$\tuple\nu=x^{-1}(\tuple \nu_0^{\tuple m}) \mapsto \mathfrak{S}_{m_+,m_-}x:= C_{{\tuple\nu}}$$
 between the set $\mathbb{J}_{\tuple m}$ and
 the set of right cosets $\mathfrak{S}_{m_+,m_-}\backslash\mathfrak{S}_m$
 of $\mathfrak{S}_{m_+,m_-}$ in $\mathfrak{S}_m$.
Each coset $C_{{\tuple\nu}}$ has a unique element of minimal length,
which will be denoted by
$\pi_{\tuple\nu}$.
By \cite[Proposition 6.7]{ROS17},
the length of $\pi_{\tuple\nu}$ is equal to $$\sharp\{(i,j)|1\leqs i<j\leqs m\,,\,\nu_{j}=+ \,\,\text{and}\,\,\nu_{i}=- \}.$$
We notice that the element $\pi_{\tuple \nu}$ is the unique minimal length element of $\mathfrak{S}_m$ such that
\begin{equation}
\label{equation:pisigmalambda_sort}
\pi_{\tuple\nu} \cdot \tuple\nu=\tuple \nu_0^{\tuple m}.
\end{equation}
When $\tuple m$ is fixed, we abbreviate $\tuple \nu_0^{\tuple m}$ to
$\tuple\nu_0$.

\smallskip

\subsubsection{Structure of $\End(\bigoplus_{\tuple\nu\in \mathbb{J}_{\tuple m}}F^{\tuple \nu})$}\label{ssub:Composi-F-NT}
Let $\tuple\nu=(\nu_1, \ldots, \nu_m)$ be a sign vector of $\mathbb{J}^m$.
Associated with $\tuple\nu$, we define a functor $$F^{\tuple\nu}:=F^{\nu_m}F^{\nu_{m-1}}\cdots F^{\nu_2}F^{\nu_1},$$
 where $F^+:=F$ and $F^-:=F'.$
In particular, if $\tuple m=(m_+,m_-)\comp_2 m$ then
$$F^{\tuple \nu_0^{\tuple m}}=(F')^{m_-}F^{m_+}.$$

Our main purpose of this section is to show that
 the endomorphism algebra
 $\End(\bigoplus_{\tuple\nu\in \mathbb{J}_{\tuple m}}F^{\tuple \nu})$
 is  a matrix algebra. Namely, we have

\begin{theorem} \label{thm:End-F+F-}
 $\End(\bigoplus_{\tuple\nu\in \mathbb{J}_{\tuple m}}F^{\tuple \nu})\cong \mathrm{Mat}_{|\mathbb{J}_{\tuple m}|}(\End(F^{\tuple\nu_0^{\tuple m}})).$
\end{theorem}

We shall construct an explicit isomorphism needed for Theorem \ref{thm:End-F+F-},
by a series of lemmas.
For $1\leqs a\leqs b\leqs m$, we write
 $$\tuple\nu_{[a,b]}=(\nu_a,\nu_{a+1},\ldots,\nu_{b-1},\nu_b)\in \mathbb{J}^{b-a+1}.$$
Then we define natural transformations
$X_{a,\tuple \nu}:F^{\tuple \nu}\to F^{\tuple \nu}$  for $1\leqs a\leqs m$
 by
$$
\begin{array}{ccll}
X_{a,\tuple\nu} &=& \begin{cases}F^{\tuple\nu_{[a+1,m]}}X F^{\tuple\nu_{[1, a-1]}}&
    \text{ if } \nu_a=+,\\
            F^{\tuple\nu_{[a+1,m]}}X' F^{\tuple\nu_{[1, a-1]}}&
    \text{ if } \nu_a=-,\\
    \end{cases}
    \end{array}
$$
and $T_{b,\tuple \nu}:F^{\tuple \nu}\to F^{s_b(\tuple \nu)}$ for $1\leqs b\leqs m-1$ by
$$
\begin{array}{ccll}
T_{b,\tuple\nu}  &=&\begin{cases}F^{\tuple\nu_{[b+2, m]}}T F^{\tuple\nu_{[1, b-1]}}&
    \text{ if } \nu_b=\nu_{b+1}=+,\\
F^{\tuple\nu_{[b+2,m]}}T' F^{\tuple\nu_{[1, b-1]}}&
    \text{ if } \nu_b=\nu_{b+1}=-,\\
F^{\tuple\nu_{[b+2,m]}}H F^{\tuple\nu_{[1, b-1]}}&
    \text{ if } \nu_b=-\,\, \text{and }\,\,\nu_{b+1}=+,\\
    F^{\tuple\nu_{[b+2,m]}}qH' F^{\tuple\nu_{[1, b-1]}}&
    \text{ if } \nu_b=+\,\, \text{and }\,\,\nu_{b+1}=-,\\
    \end{cases}
\end{array}
$$
where $s_i=(i,i+1)\in \mathfrak{S}_m$.
It is clear that all $X_{a,\tuple\nu}$ are commutative for a fixed sign vector $\tuple \nu$.

\smallskip

We start to show that there is a well-defined natural transformation from $F^{\tuple\nu}$ to $F^{w(\tuple \nu)}$
for  $w\in \mathfrak{S}_m$.

\begin{lemma}\label{lem:braid}
Let $\tuple \nu \in \mathbb{J}^m$.
If $s_{a_r} \cdots s_{a_1}$ and $s_{b_r} \cdots s_{b_1}$
are two reduced expressions of $w\in \mathfrak{S}_m$, then
$$T_{a_r,s_{a_{r-1}}\cdots s_{a_1}\tuple\nu} \circ\cdots \circ T_{a_2,s_{a_1}\tuple\nu}\circ T_{a_1,\tuple \nu}
 = T_{b_r,s_{b_{r-1}\cdots s_{b_1}\tuple\nu}}\circ \cdots \circ T_{b_2, s_{b_1}\tuple\nu}\circ T_{b_1,\tuple \nu}.$$
Hence it defines a natural transformation from $F^{\tuple\nu}$ to $F^{w(\tuple \nu)}$,
 which will be denoted by  $T_{w,\tuple\nu}$.
\end{lemma}
\begin{proof}
By Matsumoto's theorem,
 it suffices to check that every braid relation in $s_{a_r} \cdots s_{a_1}$ also occurs in
$$T_{a_r,s_{a_{r-1}}\cdots s_{a_1}\tuple\nu}\circ \cdots \circ T_{a_2,s_{a_1}\tuple\nu}\circ T_{a_1,\tuple \nu}.$$
Let $1\leqs t\leqs r$ and $\tuple\nu'=s_{a_{t-1}}\cdots s_{a_1}\tuple\nu.$
If $|a_{t-1}-a_t|>1$, then
$s_{a_{t+1}}s_{a_t}=s_{a_t}s_{a_{t+1}}$ and
$$\begin{array}{rcl}
  T_{a_{t+1},s_{a_{t}}\tuple\nu'} \circ T_{a_{t},\tuple\nu'}
   & = &F^{{s_t(\tuple\nu')}_{[a_{t+1}+1, m]}}Y F^{{s_t(\tuple\nu')}_{[1,a_{t+1}-1]}}\circ
                 F^{\tuple\nu'_{[a_t+2, m]}} Z F^{\tuple\nu'_{[1, a_t-1]}}\\
    & = & F^{\tuple\nu'_{[a_{t+1}+2, m]}}Y F^{\tuple\nu'_{[ a_{t-1}+1,a_{t}-1]}} Z F^{\tuple\nu'_{[1, a_t-1]}} \\
  & = & T_{a_{t},s_{a_{t+1}}\tuple\nu'} \circ T_{a_{t+1},\tuple\nu'}, \\
\end{array}
$$
where $Z,Y\in \{T, T', H, qH'\}$.
So the relations for braids of length $2$ can be checked in this way.
By  Propositions \ref{prop:12rel-SO} and \ref{prop:12relations} (a), (d), and (g)-(j), the relations for braids of length $3$ can also be checked,
 finishing the proof.
\end{proof}

\smallskip

Now we fix a $2$-composition $\tuple m\comp_2 m$. From Lemma \ref{lem:pinu} to Lemma \ref{lem:T_nu-T},
we show some properties of the natural transformation $T_{\pi_{\tuple\nu},\tuple\nu}$, establishing
an isomorphism between the functors $F^{\tuple\nu}$ and $F^{\tuple\nu_0^{\tuple m}}$
with ${\tuple \nu}\in \mathbb{J}_{\tuple m}$.

\begin{lemma}\label{lem:pinu}
If $\tuple \nu\in \mathbb{J}_{\tuple m}$ and $\pi_{\tuple\nu}=s_{a_r}\cdots s_{a_1}$ is a reduced expression, then the terms $T_{a_t,s_{a_{t-1}}\cdot\cdots s_{a_1}\tuple\nu}$'s appearing in
$T_{\pi_{\tuple\nu},\tuple\nu}=T_{a_r,s_{a_{r-1}}\cdot\cdots s_{a_1}\tuple\nu}\circ
\cdots\circ T_{a_2, s_{a_1}\tuple\nu}\circ T_{a_1,\tuple \nu}$
for $1\leqs t\leqs r$ are all of the form
 $F^{\tuple\nu_{[a_{t}+2,m]}}H F^{\tuple\nu_{[1,a_{t}-1]}}.$ As a consequence,
 $T_{\pi_{\tuple\nu},\tuple\nu}$ is invertible
 and  $T_{\pi_{\tuple\nu},\tuple\nu}^{-1}=q^{-l(\pi_{\tuple\nu})}T_{\pi_{\tuple\nu}^{-1},\tuple\nu_0}.$
\end{lemma}
\begin{proof}
Observe that for all $t \in \{1, \dots, r\},$ we have
$$s_{a_t} \cdot (w_{t-1} \cdot \tuple\nu) \neq w_{t-1} \cdot \tuple\nu,$$
where
$w_{t-1} \coloneqq s_{a_{t-1}} \cdots s_{a_1}$ (with $w_{t} = 1$ if $t = 1$).
This is true because if $$s_{a_t} \cdot (w_{t-1} \cdot \tuple\nu) = w_{t-1} \cdot \tuple\nu,$$
then  $\pi'_{\tuple\nu} \cdot \tuple\nu =\pi_{\tuple\nu} \cdot \tuple\nu = \tuple\nu_0$,
where $\pi'_{\tuple \nu} \coloneqq s_{a_r} \cdots s_{a_{t+1}} s_{a_{t-1}} \cdots s_{a_1}$.
However, this contradicts the minimal length of
$\pi_{\tuple\nu}$ such that $\pi_{\tuple\nu} \cdot \tuple\nu = \tuple\nu_0$,
since the length of $\pi'_{\tuple \nu}$ is strictly shorter than that of $\pi_{\tuple\nu}$.

Now we must have $(w_{t-1} \cdot \tuple\nu)_a=-$ and $(w_{t-1} \cdot \tuple\nu)_{a+1}=+$,
  since otherwise
  the length of $s_{a_{t}}w_{t-1}$ is less than that of $w_{t-1}$, contradicting that
  $\pi_{\nu}=s_{a_r}\cdots s_{a_1}$ is a reduced expression.
  By Propositions \ref{prop:12rel-SO} and \ref{prop:12relations} (i) and (j), $H$ is invertible and its inverse is $H'$, and
  so the lemma holds.
\end{proof}

\begin{lemma}\label{Lem:change}
 $T_{\pi_{\tuple \nu},\tuple\nu}X_{\pi_{\tuple \nu}^{-1}(a),\tuple\nu}=X_{a,\tuple\nu_0}T_{\pi_{\tuple \nu},\tuple\nu},$ i.e., $X_{\pi_{\tuple \nu}^{-1}(a),\tuple\nu}T_{\pi_{\tuple \nu}^{-1},\tuple\nu_0}=T_{\pi_{\tuple \nu}^{-1},\tuple\nu_0}X_{a,\tuple\nu_0},$
 where $\tuple \nu\in \mathbb{J}_{\tuple m}$.
\end{lemma}
\begin{proof}
Let $\pi_{\tuple \nu}=s_{a_r}\cdots s_{a_1}$ be
 a reduced expression. By Lemma \ref{lem:pinu},
 the terms $T_{a_t,s_{a_{t-1}}\cdot\cdots s_{a_1}\tuple\nu}$ appearing in
 $$T_{\pi_{\tuple\nu},\tuple\nu}=T_{a_r,s_{a_{r-1}}\cdot\cdots s_{a_1}\tuple\nu}
\circ\cdots\circ T_{a_2, s_{a_1}\tuple\nu}\circ T_{a_1,\tuple \nu}$$
 are all of the form $F^{\tuple\nu_{[a_{t}+2,m]}}H F^{\tuple\nu_{[1,a_{t}-1]}}.$
Thus the lemma follows from Propositions \ref{prop:12rel-SO} and \ref{prop:12relations} (i)-(l).
\end{proof}

\begin{lemma} \label{lem:T_nu-T}
Let $\tuple\nu=(\nu_1, \ldots, \nu_m)\in \mathbb{J}_{\tuple m}$.
If $\nu_a=\nu_{a+1}$, then
 $$T_{\pi_{\tuple \nu},\tuple\nu}T_{a,\tuple\nu}=T_{\pi_{\tuple \nu}(a),\tuple\nu_0}T_{\pi_{\tuple \nu},\tuple\nu}.$$
 Otherwise,
$T_{a,s_a\cdot\tuple\nu}=T^{-1}_{\pi_{\tuple \nu}(a),\tuple\nu_0}T_{\pi_{s_a\cdot\tuple \nu},s_a\cdot\tuple\nu}.$
\end{lemma}
\begin{proof}
 In the symmetric group $\mathfrak{S}_m$, we have
$$\pi_{\tuple\nu} s_a = \begin{cases}
s_{\pi_{\tuple\nu}(a)} \pi_{\tuple\nu} & \text{if } \nu_a = \nu_{a+1},
\\
\pi_{s_a \cdot \tuple\nu} & \text{if } \nu_a \neq \nu_{a+1}
\end{cases}$$
by \cite[Lemma 6.12]{ROS17}.
Since
both $\pi_{\tuple\nu} s_a$ and $s_{\pi_{\tuple\nu}(a)} \pi_{\tuple\nu}$
are reduced expressions when $\nu_a = \nu_{a+1}$, the lemma follows from Lemma \ref{lem:braid}.
\end{proof}

Now, the following result is clear.

\begin{proposition} \label{prop:isom}
The natural transformation $T_{\pi_{\tuple\nu},\tuple\nu}:F^{\tuple\nu}\to F^{\tuple\nu_0}$
with ${\tuple \nu}\in \mathbb{J}_{\tuple m}$
 is an isomorphism of functors, so that
 all $F^{\tuple\nu}$ with ${\tuple \nu}\in \mathbb{J}_{\tuple m}$ are isomorphic.

Moreover, the map $\Phi_{\tuple \nu}:\End(F^{\tuple\nu_0})\to \End(F^{\tuple \nu})$ given by
$$\Phi_{\tuple\nu}(Z)=T^{-1}_{\pi_{{\tuple\nu}},\tuple \nu} Z T_{\pi_{\tuple\nu},\tuple\nu}$$
 is an isomorphism satisfying
 \begin{itemize}[leftmargin=8mm]
  \item $\Phi_{\tuple\nu}(\mathrm{X}_{a,\tuple\nu_0})=\mathrm{X}_{\pi_{\tuple\nu}^{-1}\cdot a,\tuple\nu}$
   for all $1\leqs a\leqs m$, and
   \item $\Phi_{\tuple\nu}(\mathrm{T}_{b,\tuple\nu_0})=\mathrm{T}_{\pi_{\tuple\nu}^{-1}\cdot b,\tuple\nu}$
   for all $b\in \{1,\ldots, m-1\}\backslash \{m_+\}.$ \qed
  \end{itemize}
\end{proposition}

We now label the rows and the columns of the elements of
 $\mathrm{Mat}_{|\mathbb{J}_{\tuple m}|} (F^{\tuple \nu_0})$ by
 $({\tuple\nu}', {\tuple\nu}) \in ({\mathbb{J}_{\tuple m}})^2$, and
 write $E_{{\tuple\nu}', {\tuple\nu}}$
  for the elementary matrix with $1$ at position $({\tuple\nu}', {\tuple\nu})$ and $0$ elsewhere.
   For $({\tuple\nu}', {\tuple\nu}) \in ({\mathbb{J}_{\tuple m}})^2$, we have the  isomorphism
\[
\Hom(F^{\tuple\nu},F^{\tuple\nu'}) \simeq \End(F^{\tuple\nu_0}) E_{{\tuple\nu}', {\tuple\nu}}.
\]
Indeed, if we define

\begin{equation}\label{equation:definition_philambda_psilambda}
\begin{gathered}
\Phi_{{\tuple\nu}', {\tuple\nu}} : \End(F^{\tuple\nu_0}) E_{{\tuple\nu}', {\tuple\nu}} \to \Hom(F^{\tuple\nu},F^{\tuple\nu'}),
\\
\Psi_{{\tuple\nu}', {\tuple\nu}} : \Hom(F^{\tuple\nu},F^{\tuple\nu'}) \to  \End(F^{\tuple\nu_0}) E_{{\tuple\nu}', {\tuple\nu}}
\end{gathered}
\end{equation}
by
$$
\begin{array}{cccl}
 \Phi_{{\tuple\nu}', {\tuple\nu}}(Y E_{{\tuple\nu}', {\tuple\nu}})
 &\coloneqq& T^{-1}_{\pi_{{\tuple\nu}'},\tuple \nu'}Y T_{\pi_{\tuple\nu},\tuple\nu} &\text{for } Y \in \End(F^{\tuple\nu_0}) \\
\Psi_{{\tuple\nu}', {\tuple\nu}}(Z) &\coloneqq&
 (T_{\pi_{{\tuple\nu}'},\tuple \nu'} Z T^{-1}_{\pi_{\tuple\nu},\tuple\nu}) E_{{\tuple\nu}', {\tuple\nu}}& \text{for }
 Z \in \Hom(F^{\tuple\nu},F^{\tuple\nu'}),
\end{array}
$$
 then these two maps $\Phi_{\tuple\nu', \tuple\nu}$ and $\Psi_{\tuple\nu', \tuple\nu}$ are inverse isomorphisms.
 \smallskip

Finally, we  set
\begin{equation}\label{functorconstruction}
\begin{gathered}
\Phi_{\tuple m} \coloneqq \bigoplus_{({\tuple\nu}', {\tuple\nu}) \in ({\mathbb{J}_{\tuple m}})^2} \Phi_{\tuple\nu', \tuple\nu} : \mathrm{Mat}_{|\mathbb{J}_{\tuple m}|}(\End(F^{\tuple\nu_0})) \to \End(\bigoplus_{\tuple\nu\in \mathbb{J}_{\tuple m}}F^{\tuple \nu}),
\\
\Psi_{\tuple m}\coloneqq \bigoplus_{({\tuple\nu}', {\tuple\nu}) \in ({\mathbb{J}_{\tuple m}})^2} \Psi_{\tuple\nu', \tuple\nu} : \End(\bigoplus_{\tuple\nu\in \mathbb{J}_{\tuple m}}F^{\tuple \nu}) \to \mathrm{Mat}_{|\mathbb{J}_{\tuple m}|}(\End(F^{\tuple\nu_0})).
\end{gathered}
\end{equation}
By the above properties of $\Phi_{{\tuple\nu}', {\tuple\nu}}$ and $\Psi_{{\tuple\nu}, {\tuple\nu}'}$,
we know that $\Phi_{\tuple m}$ and $\Psi_{\tuple m}$ are inverse  isomorphisms,
yielding Theorem \ref{thm:End-F+F-}.

\subsection{Howlett-Lehrer theory}\label{sec:End-HC-cusp.}
\subsubsection{Bases of ${\rm End}_{G}(R_{L_J}^G(M))$} \label{sssec:bases-R(M)}

Let $\bfG$ be a connected reductive algebraic group and $F$ be a Frobenius map of $\bfG$ with
the  finite reductive group $G:=\bfG^F$.
Let $\bfB$ be an $F$-stable Borel subgroup of $\bfG$ and
$\bfT$ be an $F$-stable maximal torus of $\bfG$ with $\bfT \subset\bfB$.
Denote by $\Phi,\Phi^+$ and $\Delta$ the sets of roots,
positive roots and simple roots of $\bfG$ determined by $\bfT$ and
$\bfB$, respectively.
For any $F$-stable subset $J\subset \Delta$, let $P_J=U_JL_J$ be a standard parabolic subgroup of $G$,
where $U_J$ is the unipotent radical of $P_J$ and $L_J$ is a standard Levi subgroup of $P_J$.
Let $\lambda$ be an irreducible cuspidal character of $L_J$, $M$ be
 a left $K L_J$-module affording $\lambda$, and $\rho$ be the corresponding representation.
We shall describe the structure of  the endomorphism
algebra ${\rm End}_{G}(R_{L_J}^G(M))$
from \cite{HL80} or \cite[Chapter 10]{Car}, where
 $R_{L_J}^G$ is Harish-Chandra induction as in \S \ref{subsec:HCseries}.

\smallskip

The module $R_{L_J}^G(M)$ is isomorphic to the module obtained from
 the vector space $\mathfrak{F}(\rho)$ of maps $f: G\rightarrow M$ with
\[f(px)=\rho (p) f(x) \quad\text{ for all }p \in P_J   \text{ and } x \in  G,\]
on which the action of $G$ is defined by
\begin{equation*}\label{G_action}
  (g\cdot f)(x)=f(xg) \quad\text{ for all $f\in \mathfrak{F}(\rho)$ and
  $g, x\in  G$}.
  \end{equation*}
Hence we may identify $\mathfrak{F}(\rho)$ and $R_{L_J}^G(M)$.
We shall mention a $K$-linear basis of $\End_{KG}(R_{L_J}^G(M))$
 with the identification. To describe it,
let $W_G(L_J):=N(L_J)/{L_J}$ be the \emph{relative Weyl
group} of $L_J$ in $G$, and  $W(\la):=N(L_J)_\la/{L_J}$ be the \emph{relative Weyl
group} of the cuspidal pair $(L_J,\lambda)$ in $G$, where $N(L_J)_\la$ is the inertial group of $\lambda$ in $N(L_J)$.
For each $w\in W(\lambda)$,  let ${\mathrm{B}_{w,\rho}\in\End_{K G}(\mathfrak{F}(\rho))}$ be defined by
\[ (\mathrm{B}_{w,\rho}(f)) (x)= \frac 1{|U_J|}\widetilde{\rho}(\dot{w})\sum_{u\in U_J} f(\dot{w}^{-1}u x )\quad
  \text{ for all }f\in \mathfrak{F}(\rho)\text{ and } x \in  G,\]
where $\widetilde{\rho}$ is a fixed extension of $\rho$ to $N(L_J)_{\lambda}$
and $\dot{w}$ is a (fixed) pre-image of $w$ in $N(L_J)_{\lambda}$.
By \cite[Theorem 3.9]{HL80},  $\mathrm{B}_{w,\rho}$ is independent of the choice of $\dot{w}$
and the set $$\{\mathrm{B}_{w,\rho}|w\in W(\lambda)\}$$ forms a $K$-linear basis of $\End_{KG}(R_{L_J}^G(M))$.

\smallskip

Based on the above basis,  $\End_{KG}(R_{L_J}^G(M))$ has a so-called $\mathrm{T}_{w}$-basis in terms of
the structure of $W(\la)$, see \cite[Section~2 and~4]{HL80}.
In fact, we have $W(\lambda)=R(\lambda)\rtimes C(\lambda)$, where the groups
$R(\lambda)$ and $C(\lambda)$ are defined as follows.
Let $\Phi_J\subseteq\Phi$ be the set of roots corresponding to $J$, and
let
$$\hat\Omega:=\{\alpha\in\Phi\setminus\Phi_J\mid
   w(J\cup\{\alpha\})\subseteq\Delta\text{ for some }w\in W\}.$$
For $\alpha\in\hat\Omega$ we set $v(\alpha):=(w_0)_J w_0^\alpha$, where
$(w_0)_J$ and $w_0^\alpha$ are the longest elements in the Weyl groups $W_J$ and
$\langle W_J,s_\alpha\rangle$, respectively.
We define $$\Omega:=\{\alpha\in\hat\Omega\mid v(\alpha)^2=1\}.$$
Let $p_{\alpha,\lambda}\geqs 1$ be
the larger ratio between the degrees of the two different constituents of
$\R_{L_J}^{L_\alpha}(\lambda)$, where $\alpha\in \Omega$ and
$L_\alpha$ is the standard Levi
subgroup of $G$ corresponding to the set $J\cup\{\alpha\}$ of simple roots
(so that $L_J$ is a standard Levi subgroup of $L_\alpha$).
The set
$$\Phi_\lambda:=\{\alpha\in\Omega\mid s_\alpha\in W(\lambda),\,p_{\alpha,\lambda}\neq 1\}$$
is a root system with the set $\Delta_\la$ of simple roots contained in $\Phi_\lambda\cap\Phi^+$.
Now $R(\lambda)$ is exactly the Weyl group $\langle s_\alpha\mid \alpha\in\Phi_\lambda\rangle$
 and $C(\lambda)$ is the stabilizer of $\Delta_\lambda$ in $W(\lambda)$, see
\cite[Proposition~10.6.3]{Car}.

 \smallskip

For $w\in W(\lambda)$, we set $\text{ind}(w):=|U_0\cap (U_0)^{w_0w}|$, where $U_0$ is the
unipotent radical of the Borel subgroup $B$ and $w_0$ is the longest element in $W$. Also,
for $\alpha\in\Delta_\lambda$ 
we define
$\epsilon_{\alpha,\lambda}\in \{\pm1\}$ such that
\begin{equation*}   \label{def_ep}
\mathrm{B}_{s_\alpha,\rho}^2=\frac{1}{\text{ind}(s_\alpha)}\,\id + \epsilon_{\alpha,\lambda}
  \frac{p_{\alpha,\lambda}-1}{\sqrt{\text{ind}(s_\alpha)p_{\alpha,\lambda} }}\mathrm{B}_{s_\alpha,\rho}
\end{equation*}
(see \cite[Proposition~10.7.9]{Car}). 
Now, by \cite[Propoisition~10.8.2]{Car}, the following is well-defined:
\begin{itemize}
  \item $\mathrm{T}_{w,\rho}:=\sqrt{\text{ind}(w)}\,\mathrm{B}_{w,\rho}$ for
$w\in C(\lambda)$,
  \item $\mathrm{T}_{s_\alpha,\rho}:=
  \epsilon_{\alpha,\lambda}\,\sqrt{\text{ind}(s_\alpha)p_{\alpha,\lambda}}\,
  \mathrm{B}_{s_\alpha}$ for $\alpha\in\Delta_\lambda$,
\item $\mathrm{T}_{w,\rho}:=\mathrm{T}_{s_1,\rho}\cdots \mathrm{T}_{s_r,\rho}$ for
$w\in R(\lambda)$ with a reduced expression $w=s_1\cdots s_r$, where
$s_i:=s_{\alpha_i}$ and $\alpha_i\in\Delta_\lambda$,
  \item $\mathrm{T}_{w,\rho}:=\mathrm{T}_{w_1,\rho}\mathrm{T}_{w_2,\rho}$ for
  $w=w_1w_2\in W(\lambda)$ with $w_1\in C(\lambda)$ and $w_2\in R(\la)$.
\end{itemize}
Moreover, the set $\{\mathrm{T}_{w}:=\mathrm{T}_{w,\rho}\mid w\in W(\lambda)\}$ also forms a $K$-linear basis of
the endomorphism algebra $\End_{KG}(R_{L_J}^G(M))$ of $R_{L_J}^G(M)$.


\subsubsection{Structure of Howlett-Lehrer algebras}\label{sssec:str-HL algebra}
Instead of the endomorphism algebra $\End_{KG}(R_{L}^G(M))$, it is a larger algebra connecting with our representation datum
constructed in \S \ref{sub:explicit-repdatumSO}. To introduce it,
for $w\in W$ and $wJ\subset \Delta$ let $\mathrm{B}_{\dot{w},\rho}$ be
the homomorphism from $\mathfrak{F}(\rho)$ to $\mathfrak{F}(w\rho)$ defined by
$$(\mathrm{B}_{\dot{w},\rho}f)(x)=\frac{1}{|U_{wJ}|}\sum\limits_{u\in U_{wJ}}f(\dot{w}^{-1}ux).$$
Analogous to $\End_{KG}(R_{L_J}^{G}(M))$,  the endomorphism algebra
$\End_{KG}(R_{L_{wJ}}^{G}(wM))$ has a basis $\{\mathrm{T}_{v,w\rho}|v\in W(w\lambda)\}$.

\smallskip

Here we define $$D_{\lambda}:=\{w\in W|wJ\subset \Delta, w\Delta_{\lambda}\subset \Phi^+ \}.$$
By \cite[Lemma 3.12 ]{HL83}, if $w\in W$ and $wJ\subset \Delta$,
 then the coset $wW(\lambda)$ contains an element of $D_\la.$
The following result ensures the existence of a canonical bijection $\phi\to \zeta_\phi$ from $\Irr(W(\la))$
  to the set of irreducible constituents of $R_{L_J}^G(M)$.
\begin{lemma}\cite{HL83}\label{lem:sum}
Let $w\in D_\lambda.$ The map $$\tau_w:\End_{KG}(R_{L_J}^{G}(M))\to \End_{KG}(R_{L_{wJ}}^{G}(wM))$$ given by $\tau_w(T)=\mathrm{B}_{\dot{w},\rho}T\mathrm{B}_{\dot{w},\rho}^{-1}$
is an isomorphism satisfying $\tau_w(\mathrm{T}_{v,\rho})=\mathrm{T}_{wvw^{-1},w\rho}$ for all $v\in W(\lambda).$\qed
\end{lemma}

It is the endomorphism algebra $\End_{KG}(\bigoplus_{w\in D_\la}R_{L_{wJ}}^G(wM))$, which contains
the algebra $\End_{KG}(R_{L}^G(M))$ as a direct summand,
 essentially considered by Howlett-Lechrer to prove Comparison Theorem.
We call it the Howlett-Lehrer algebra in this paper.
Using the isomorphisms in Lemma \ref{lem:sum}, we see that the Howlett-Lehrer algebra has the structure of a matrix algebra.
\begin{theorem}\cite{HL83}\label{thm:HL-algebra}
 We have an algebra isomorphism
 $$\End_{KG}(\bigoplus_{w\in D_\la}R_{L_{wJ}}^G(wM))\cong \mathrm{Mat}_{|D_\la|}(\End_{KG}(R_{L_J}^G(M))).$$
\end{theorem}

In \S \ref{ssub:connectionHL}, we shall give a connection of
the Howlett-Lehrer algebra with our previous representation datum.

\subsection{Ramified Hecke algebras}\label{sec:Ram-algebra}

Throughout this section, we let $\widetilde{G}_{n}=\O_{2n+1}(q)$ and $G_n=\SO_{2n+1}(q)$ with $q$ odd,
and let $(E,F,X,T;E',F',X',T';H,H')$ be the representation datum on $\scrQU_R$
constructed in \S \ref{sub:explicit-repdatumSO}.

\subsubsection{Basis of $\End(F^{\tuple \nu}(E_{t_+,t_-}))$}\label{subsec:ramiHecke}
Here we first recall Lusztig's description \cite{L77} of
endomorphism algebras of Harish-Chandra induction over quadratic unipotent cuspidal
modules, using the language of Howlett-Lehrer theory in \S \ref{sec:End-HC-cusp.}.
Then we apply it to reformulate an algebra homomorphism
obtained from our representation datum in \S \ref{sub:explicit-repdatumSO}.
 \smallskip

Let $\lambda:=\chi_{\Theta_{t_+},\Theta_{t_-}}\times1^{m_+}\times\zeta^{m_-}$
be a quadratic unipotent  cuspidal character of $L_{r,1^m}$ as in \S\ref{subsec:quchar},
 $M:= E_{t_+,t_-}\otimes K_1^{m_+}\otimes K_\zeta^{m_-}$
 be the $KL_{r,1^m}$-module affording $\lambda$, and $\rho$ be the corresponding representation.
 We write $L_J:=L_{r,1^m}$  with $J=\{s_1,\dots, s_r\}$.
The relative Weyl group $W_G(L_J)$ of $L_J$ is
 the Weyl group of type $B_m$. We define $v_{1}:=\dot{t}_{r+1}$ and $v_{i}:=\dot{s}_{r+i} \,\,\text{for}\, \,2\leqs i\leqs m$,
where $\dot{t}_{k}$ and $\dot{s}_k$ are as defined in \S\ref{sub:lift-refl} (with $\text{ind}(\dot{t}_k)=2k-1$ and $\text{ind}(\dot{s}_k)=1$).
The relative Weyl group $W(\lambda)$ of the cuspidal pair $(L_J,\lambda)$ is
isomorphic to the Weyl group $$W_{m_+}\times W_{m_-}$$ of type
 $B_{m_+}\times B_{m_-}$, where
 $W_{m_+}$ is generated by images of $v_{i}$ in $W(\lambda)$ for $1\leqs i\leqs m_+$
 and $W_{m_-}$ is generated by images of $u_1:=\dot{t}_{r+m_++1}$ and $u_j:=\dot{s}_{r+m_++j}$ in $W(\lambda)$ for $2\leqs j\leqs m_-.$

 \smallskip

 We shall choose and fix an extension of $\rho$ to its stabilizer in $G_n$.
Recall that  $L_J=L_{r,1^m}=G_r \times(\bbF_q^{\times})_1 \times\cdots\times(\bbF_q^{\times})_m$. Let $$\widetilde{L}_J=\widetilde{L}_{r,1^m}={\widetilde{G}}_{r}\times(\bbF^{\times}_q)_1
\times\cdots\times(\bbF^{\times}_q)_m$$ be the corresponding Levi subgroup of $\widetilde{G}_n.$
In addition, let
 $$H_J={\widetilde{G}}_{r}\times((\bbF^{\times}_q)_1.H_1)
\times\cdots\times((\bbF^{\times}_q)_m.H_m)\leqs \widetilde{G}_n,$$
where for each $i$,
$H_i$ is the subgroup
of $\widetilde{G}_n$ generated by
\begin{gather*}
\left(\begin{array}{ccccc}
 \id_{m-k} &               &          &              &\\
                &          &          & 1 &\\
                &          & \id_{\widetilde{G}_{r+k-1}} &               &\\
                & 1 &          &          &\\
                &          &        &      & \id_{m-k}\\
\end{array}\right).
\end{gather*}
It is clear that the group $H_J$
is isomorphic to $\widetilde{G}_{r}\times(\bbF_q^{\times}.2)^m$ and
 $$N_{\widetilde{G}_n}(\widetilde{L}_{J})\cong H_J\rtimes\mathfrak{S}_{m}\cong\widetilde{G}_r \times((\bbF_q^{\times}.2)\wr
\frakS_{m}).$$
Let $\widetilde{\lambda}=\chi\times {1}^{m_+}\times {\zeta}^{m_-}\in \Irr(\widetilde{L}_{J})$,
where $\chi$ is the character of the inflation module $\mathrm{Inf}_{G_r}^{\widetilde{G}_r}(E_{t_+,t_-})$.
Then
 $$N_{\widetilde{G}_n}(\widetilde{L}_J)_{\widetilde{\lambda}}\cong \widetilde{G}_r\times((\bbF_q^{\times}.2)\wr
\frakS_{m_+})\times((\bbF_q^{\times}.2)\wr
\frakS_{m_-})$$ and
$\widehat{\lambda}:=\chi\times 1_{(\bbF_q^{\times}.2)\wr
\frakS_{m_+}}\times sp|_{(\bbF_q^{\times}.2)\wr
\frakS_{m_-}}$ is an extension of $\widetilde{\lambda}$ to  $N_{\widetilde{G}_n}(\widetilde{L}_J)_{\widetilde{\lambda}}$.
Since  $\widetilde{\lambda}$ is an extension of $\lambda$
and $N_{G_n}(L_J)_{\lambda}=N_{\widetilde{G}_n}(\widetilde{L}_J)_{\widetilde{\lambda}}\cap G_n,$
we conclude that the restriction of $\widehat{\lambda}$ to
$N_{G_n}(L_J)_{\lambda}$ is an extension of $\lambda$. We fix it and write
$\widetilde{\rho}$ for its corresponding representation.

\smallskip

 As in \S \ref{sec:End-HC-cusp.}, with the above $\widetilde{\rho}$, 
 the endomorphism algebra
 $\End_{K G_n}(R^{G_n}_{L_J}(M))$
 has a $K$-linear basis $$\{\mathrm{B}_{w}:=\mathrm{B}_{{w},\rho}\mid w\in W(\lambda)\}.$$

\smallskip

To describe a $\mathrm{T}_w$-basis for the endomorphism algebra
 $\End_{K G_n}(R^{G_n}_{L_J}(M))$ as in \S \ref{sec:End-HC-cusp.},
 we analyze the structure of $W(\lambda)$. In fact, since $J=\{s_1,\dots, s_r\}$
 is the only subset of $\Delta$ of the same type, it follows that
  $J$ is self dual in $I$ for all $J\subset I\subset \Delta$,
  i.e., $(w_0)_I(w_0)_J(J)=-J.$
  By \cite[Lemma 10.10.1]{Car}, the set $\{\alpha|\alpha\in \Delta-J\}$ forms a simple system of $\Omega$
  and $W(\lambda)=R(\lambda).$ In addition,
  $\Delta_\lambda=\{v_i,u_j|1\leqs i\leqs m_+,1\leqs j\leqs m_-\}.$

\begin{proposition}[Lusztig \cite{L77}] \label{prop:stru-end-quadr}
The set $\{\mathrm{T}_{w}\mid w\in W(\lambda)\}$
  forms a $K$-linear basis of
 the endomorphism algebra $\End_{K G}(\R_{L_J}^{G_{n}}(M))$,
 satisfying the following relations:
\begin{equation*}
\begin{aligned}
&\mathrm{T}_{v_{1}}^2=(q^{2t_++1}-1)\mathrm{T}_{v_{1}}+q^{2t_++1}, &
&\mathrm{T}_{v_i}^2=(q-1)\mathrm{T}_{v_i} +q \,\,\text{for}\,\, \   2\leqs i\leqs m_+,\\
&\mathrm{T}_{u_{1}}^2=(q^{2t_-+1}-1)\mathrm{T}_{u_{1}}+q^{2t_-+1},  &
&\mathrm{T}_{u_i}^2=(q-1)\mathrm{T}_{u_i} +q \,\,\text{for}\,\, \   2\leqs i\leqs m_-,\\
&\mathrm{T}_x\mathrm{T}_w=\mathrm{T}_{xw}  \,\, \text{if } \,\,  l(xw)=l(x)+l(w),  &
&\mathrm{T}_{v_i}\mathrm{T}_{u_j}=\mathrm{T}_{u_j}\mathrm{T}_{v_i} \,\,\text{for any}\,\, \   i\neq j,\\
\end{aligned}
\end{equation*}
where  $l$ is the length function for $W_{m_+}\times W_{m_-}.$\qed
\end{proposition}

 We would like to point out that the computations of radio $p_{\alpha,\lambda}$ in Proposition \ref{prop:stru-end-quadr}
  can be reduced to the unipotent cases by Lusztig's Jordan decomposition
  (see \cite[Proposition~7.9]{L77}).  Then we can make use of the following diagrams (see  \cite[p. 464]{Car}):
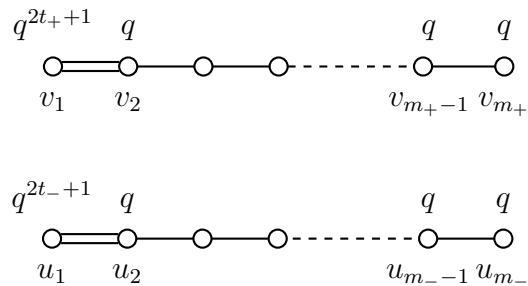
 \begin{align*}
\begin{split}
\begin{tikzpicture}[scale=.4]
\draw[thick] (0 cm,0) circle (.3cm);
\node [below] at (0 cm,-.5cm) {$v_{1}$};
\node [above] at (0 cm,0.5cm) {$q^{2t_++1}$};
\draw[thick] (0.3 cm,-0.15cm) -- +(1.9 cm,0);
\draw[thick] (0.3 cm,0.15cm) -- +(1.9 cm,0);
\draw[thick] (2.5 cm,0) circle (.3cm);
\node [below] at (2.5 cm,-.5cm) {$v_{2}$};
\node [above] at (2.5 cm,0.5cm) {$q$};
\draw[thick] (2.8 cm,0) -- +(1.9 cm,0);
\draw[thick] (5 cm,0) circle (.3cm);
\draw[thick] (5.3 cm,0) -- +(1.9 cm,0);
\draw[thick] (7.5 cm,0) circle (.3cm);
\draw[dashed,thick] (7.8 cm,0) -- +(4 cm,0);
\draw[thick] (12.3 cm,0) circle (.3cm);
\node [below] at (12.5 cm,-.5cm) {$v_{m_+-1}$};
\node [above] at (12.5 cm,0.5cm) {$q$};
\draw[thick] (12.6 cm,0) -- +(2.1 cm,0);
\draw[thick] (15 cm,0) circle (.3cm);
\node [below] at (15 cm,-.5cm) {$v_{m_+}$};
\node [above] at (15 cm,0.5cm) {$q$};
\end{tikzpicture}
\end{split}
\end{align*}
\begin{align*}
\begin{split}
\begin{tikzpicture}[scale=.4]
\draw[thick] (0 cm,0) circle (.3cm);
\node [below] at (0 cm,-.5cm) {$u_{1}$};
\node [above] at (0 cm,0.5cm) {$q^{2t_-+1}$};
\draw[thick] (0.3 cm,-0.15cm) -- +(1.9 cm,0);
\draw[thick] (0.3 cm,0.15cm) -- +(1.9 cm,0);
\draw[thick] (2.5 cm,0) circle (.3cm);
\node [below] at (2.5 cm,-.5cm) {$u_{2}$};
\node [above] at (2.5 cm,0.5cm) {$q$};
\draw[thick] (2.8 cm,0) -- +(1.9 cm,0);
\draw[thick] (5 cm,0) circle (.3cm);
\draw[thick] (5.3 cm,0) -- +(1.9 cm,0);
\draw[thick] (7.5 cm,0) circle (.3cm);
\draw[dashed,thick] (8 cm,0) -- +(4.2 cm,0);
\draw[thick] (12.5 cm,0) circle (.3cm);
\node [below] at (12.5 cm,-.5cm) {$u_{m_--1}$};
\node [above] at (12.5 cm,0.5cm) {$q$};
\draw[thick] (12.8 cm,0) -- +(1.9 cm,0);
\draw[thick] (15 cm,0) circle (.3cm);
\node [below] at (15 cm,-.5cm) {$u_{m_-}$};
\node [above] at (15 cm,0.5cm) {$q$};
\end{tikzpicture}
\end{split}
\end{align*}
where the simple roots are labeled by the simple reflections, and the $p_{\alpha,\lambda}$'s are over the corresponding simple roots.

\smallskip
Now, we denote and fix $\tuple m=(m_+,m_-)\comp_2 m$, and abbreviate
$\tuple\nu_0^{\tuple m}$ to $\tuple\nu_0$. Then $F^{\tuple\nu_0}=(F')^{m_-}F^{m_+}.$
Remark \ref{rem:scalar} (b) says that
both $(E,F,X,T)$ and $(E',F',X',T')$ are representation data of type $A$, which lead to the following
well-defined $R$-algebra homomorphisms by \S \ref{sec:datumA} (\ref{equ:heckealgebrahom}):
$$
\begin{array}{rclrcl}
 \phi_{F^{m_+}} \, :\, \bfH_{K,m_+}^q  &\rightarrow&  \End(F^{m_+})^\op & \phi_{(F')^{m_-}} : \bfH_{K,m_-}^q &\rightarrow&  \End((F')^{m_-})^\op\\
  X_k  &\mapsto & {F^{m_+-k}} X {F^{k-1}},    &                 X_{k'}  &\mapsto&  {(F')^{m_--k'}} X' {(F')^{k'-1}},\\
   T_l  & \mapsto &  {F^{m_+-l-1}}T{F^{l-1}}; &                 T_{l'}  &\mapsto&  {(F')^{m_--l-1'}}T'{(F')^{l'-1}},
\end{array}
$$
where $1\leqs k\leqs m_+$, $2\leqs l\leqs m_+$, $1\leqs k'\leqs m_-$ and $2\leqs l'\leqs m_-.$
%
So,  attached to the representation datum $(E,F,X,T;E',F',X',T';H,H')$, we have the following $R$-algebra homomorphisms:
\begin{equation}\label{equ:datahom}
\begin{aligned}
\phi_{F^{\tuple\nu_0}}:\quad&\bfH_{K,m_+}^{q}\otimes\bfH_{K,m_-}^{q}\to\End(F^{\tuple\nu_0})^\op\\
X_k\otimes 1& \mapsto X_{k,\tuple\nu_0},\ \ \ 1\otimes X_{k'} \mapsto X_{m_++k',\tuple\nu_0},\\
T_l\otimes1&\mapsto T_{l,\tuple\nu_0}, \ \ \ \ \ 1 \otimes T_{l'}\  \mapsto T_{m_++l',\tuple\nu_0}.
 \end{aligned}
 \end{equation}
Evaluating at $E_{t_+,t_-}$, the images in (\ref{equ:datahom}) are respectively represented by right multiplication by
$$e'_{n,r+m_+}e_{r+m_+,r}\iota_{n,r+k}(\dot{t}_{r+k})e'_{n,r+m_+}e_{r+m_+,r},$$
 $$\zeta(2)\zeta(-1)^{k'-1}e'_{n,r+m_+}e_{r+m_+,r}\iota_{n,r+m_++k'}(\dot{t}_{r+m_++k'})
 e'_{n,r+m_+}e_{r+m_+,r},$$ $$e'_{n,r+m_+}e_{r+m_+,r}\dot{s}_{r+l}e'_{n,r+m_+}e_{r+m_+,r},\quad\text{and}\quad
 e'_{n,r+m_+}e_{r+m_+,r}\dot{s}_{r+m_++l'}
 e'_{n,r+m_+}e_{r+m_+,r}.$$
%
Moreover, we have $B_{w,\rho}=\widetilde{\rho}(\dot{w})B_{\dot{w},\rho}$ for any lift $\dot{w}$ of $w\in W(\lambda).$
By easy computation, we have

\begin{lemma} Let $\iota_{n,m}$ be the embedding map $G_m\hookrightarrow G_n$ defined in \S \ref{subsec:Levi}.
Then
\begin{itemize}
  \item[$(i)$] $\widetilde{\rho}(\iota_{n,r+k}(\dot{t}_{r+k}))=1$ for all $1\leqs k\leqs m_+$,
  \item[$(ii)$] $\widetilde{\rho}(\iota_{n,r+m_++k'}(\dot{t}_{r+m_++k'}))=\zeta(-1)^{k'-1}$ for all $1\leqs k'\leqs m_-,$ and
  \item[$(iii)$] $\widetilde{\rho}(\dot{s}_{r+i}))=1$ for all $2\leqs i\leqs m_+$ and $m_++2\leqs i\leqs m.$
\end{itemize}
\end{lemma}

According to Howlett-Lehrer theory, we identify $(F')^{m_-}F^{m_+}(E_{t_+,t_-})$ with
$\mathfrak{F}(\rho)=R_{L_J}^{G_{n}}(M)$.
Therefore, by definition, the following endomorphisms
 of  $F^{\tuple\nu_0}(E_{t_+,t_-})$ and of  $\mathfrak{F}(\rho)$ are correspondingly
 identified:
 $$
 \begin{array}{lccl}
    X_{k,\tuple\nu_0}(E_{t_+,t_-}) & \leftrightarrow  & q^{r+k-1}\mathrm{B}_{t_{r+k}} \\
   T_{l,\tuple\nu_0}(E_{t_+,t_-}) & \leftrightarrow  & q\mathrm{B}_{s_{r+l}} \\
   X_{m_++k',\tuple\nu_0}(E_{t_+,t_-}) &  \leftrightarrow  & \zeta(2)q^{r+m_++k'-1}\mathrm{B}_{t_{r+m_++k'}} \\
   T_{m_++l',\tuple\nu_0}(E_{t_+,t_-}) & \leftrightarrow & q\mathrm{B}_{s_{r+m_++l'}}
 \end{array}
$$


 \smallskip

Hence, evaluating at the module $E_{t_+,t_-}$,
 the map in \eqref{equ:datahom} can be reformulated as follows:
\begin{equation} \label{map:relations}
\begin{aligned}
  \phi_{F^{\tuple\nu_0}}(E_{t_+,t_-}):\bfH_{K,m_+}^{q}\otimes\bfH_{K,m_-}^{q} &\to
 \End(F^{\tuple\nu_0}(E_{t_+,t_-}))^\op,\\
   X_1\otimes1 &\mapsto \epsilon_{v_1}q^{-t_+-1}\mathrm{T}_{v_1}\\
   T_l\otimes1  &\mapsto \mathrm{T}_{v_{l+1}},\\
   1\otimes X_{1} &\mapsto \epsilon_{u_1}q^{-t_--1}\mathrm{T}_{u_{1}},\\
   1\otimes T_{l'} &\mapsto \mathrm{T}_{u_{l'+1}}.
\end{aligned}
\end{equation}

\begin{remark}  In (\ref{map:relations}), we have $\epsilon_{v_i}=1$ for $i>2$ and $\epsilon_{u_j}=1$ for $j>2$ by Proposition
\ref{prop:12relations} $(b)$ and $(e)$.
Both signs  $\epsilon_{v_1}$ and $\epsilon_{u_1}$ are unknown at present, but will become clear later.
\end{remark}

\subsubsection{Structure of $\End(\bigoplus_{\tuple\nu\in \mathbb{J}_{\tuple m}}F^{\tuple \nu}(E_{t_+,t_-}))$}\label{ssub:connectionHL}
We can now describe the connection of $\End(\bigoplus_{\tuple\nu\in \mathbb{J}_{\tuple m}}F^{\tuple \nu}(E_{t_+,t_-}))$
with the Howlett-Lehrer algebra introduced in \S \ref{sec:End-HC-cusp.}.

\smallskip

 Since $J=\{s_1,\dots, s_r\}$ is the only subset of $\Delta$ of the same type,
 it follows that for $w\in W$ with $w(J)\subset \Delta$,
 $w(J)$ must be equal to $J$, i.e., $w\in W_G(L_J).$
Therefore, in the $\SO_{2n+1}(q)$ case,
 $D_{\lambda}=\{w\in W|w(J)\subset \Delta, w(\Delta_{\lambda})\subset \Phi^+ \}$
 is exactly the set of  minimal length elements in the right cosets
 of $W(\la)=W_{m_+}\times W_{m_-}$ in $W_G(L_J)=W_m$.
 However, the latter can be identified
 with the set of  minimal length elements in the right cosets of $\mathfrak{S}_{m_+}\times \mathfrak{S}_{m_-}$ in $\mathfrak{S}_m$,
yielding a canonical one-to-one correspondence $\tuple\nu\mapsto\pi_{\tuple \nu}$
 between $\mathbb{J}_{\tuple m}$ and $D_{\la}$.

\smallskip

For $\tuple\nu\in \mathbb{J}^m$, we write
 $K^{\tuple \nu}=K^{\nu_1}\otimes K^{\nu_2}\otimes\cdots\otimes K^{\nu_m}$,
  where $K^+:=K_1$ and $K^-:=K_\zeta,$ so that
 $\pi_{\tuple\nu}^{-1}(E_{t_+,t_-}\otimes K_1^{m_+}\otimes K_\zeta^{m_-})=E_{t_+,t_-}\otimes K^{\tuple\nu}.$
 Now each cuspidal pair $(L_{r,1^m},E)$
 conjugate to $(L_{r,1^m},E_{t_+,t_-}\otimes K_1^{m_+}\otimes K_\zeta^{m_-})$
 is of the form
 $E_{t_+,t_-}\otimes K^{\tuple \nu}$ for some $\tuple\nu\in \mathbb{J}_{\tuple m}.$
Identifying $F^{\tuple\nu}(E_{t_+,t_-})$ with
 $\mathfrak{F}(\pi_{\tuple\nu}^{-1}\rho)=R_{L_{r,1^m}}^{G_{n}}(E_{t_+,t_-}\otimes K^{\tuple \nu})$,
 we see that the following endomorphisms of $F^{\tuple\nu}(E_{t_+,t_-})$ and of $\mathfrak{F}(\pi_{\tuple\nu}^{-1}\rho)$
 are correspondingly coincident:
 $$\begin{array}{cccl}
    {X_{a,\tuple\nu}}(E_{t_+,t_-}) & \leftrightarrow  & \varepsilon q^{r+a-1}\mathrm{B}_{t_{r+a},\pi_{\tuple\nu}^{-1}\rho} \\
      T_{a,\tuple \nu}(E_{t_+,t_-}) & \leftrightarrow  & q\mathrm{B}_{\dot{s}_{r+a},\pi_{\tuple\nu}^{-1}\rho}
   \end{array}
   $$
 where
 $\varepsilon=1$ if $\nu_a=+$ and $\varepsilon=\zeta(2)$ if $\nu_a=-$,
and $\mathrm{B}_{\dot{w},\rho}$
is as defined in \S \ref{sssec:str-HL algebra}.
Moreover, if $\tuple\nu\in \mathbb{J}_{\tuple m}$,
 then the isomorphism $\Phi_{\tuple\nu}$ as in Proposition \ref{prop:isom} maps to
 $\tau_{\pi_{\tuple \nu}}$ as in Lemma \ref{lem:sum}.

\begin{theorem}\label{thm:connectwith-HL}
We have an algebra isomorphism
 $$\End_{KG}(\bigoplus_{\tuple\nu\in \mathbb{J}_{\tuple m}}F^{\tuple \nu}(E_{t_+,t_-}))\cong \End_{KG}(\bigoplus_{w\in D_\la}R_{L_{wJ}}^G(wM)).$$\qed
 \end{theorem}

%
%
%
%
%
%
%
%
%

\subsubsection{Algebraic structure of ramified Hecke algebras}
\label{ssub:eigen-XX'}
Let $M$ be the $KL_J$-module as in \S \ref{subsec:ramiHecke}.
Recall that the ramified Hecke algebra $\scrH(KG_n, M)$
is defined to be the
 $\End_{K G_n}(R^{G_n}_{L_J}(M))^\op,$  which is canonically isomorphic to
$\End_{KG_n}(F^{\tuple\nu_0}(E_{t_+,t_-}))^\op$.
We shall explicitly determine its algebraic structure
in Theorem \ref{thm:HL-BC}, starting with observations about the eigenvalues of $X$ on $F$ and of $X'$ on $F'$.

\smallskip
\begin{lemma}\label{Lem:eigenvalues} Let $N$ be a $KG_n$-module  and $s,s' >0$.
Then the eigenvalues of $X'1_{F^s}(N)$ on $F'F^s(N)$ are equal to those of $X'(N)$ on $F'(N).$
Similarly, the eigenvalues of $X1_{(F')^{s'}}(N)$ on $F(F')^{s'}(N)$
are equal to those of $X(N)$ on $F(N).$
\end{lemma}
 \begin{proof} 
We only show the former part of the lemma and leave the latter part for the reader.
We define $$\Delta_s=1_{F^{(s-1)}}H'\circ1_{F^{(s-2)}}H'1_{F}
\circ\cdots\circ 1_{F^2}H'1_{F^{(s-3)}}
\circ 1_{F}H'1_{F^{(s-2)}}\circ H'1_{F^{(s-1)}}.$$
Then $\Delta_s$ is an isomorphism between $F'F^s$ and $F^s F'$ and  its inverse is
$$\Delta_s^{-1}=
H1_{F^{(s-1)}}\circ 1_{F}H1_{F^{(s-2)}}
\circ 1_{F^2}H1_{F^{(s-3)}}\circ \cdots \circ 1_{F^{(s-2)}}H 1_{F}\circ1_{F^{(s-1)}}H.$$
Then $\Delta_s=q^{-s}T_{\pi_{\tuple \nu},\tuple \nu}$
 for $\tuple m=(s,1)\comp_2 s+1$ and $\tuple \nu=\{-,\,+,\dots,\,+\}.$
By  Lemma \ref{Lem:change},
we have the following diagram
$$\xymatrix{F'F^s\ar[r]_{\Delta_s}^{\sim}\ar[d]^{X'1_{F^s}} &F^sF' \ar[d]^{1_{F^s}X'}\\
 F'F^s\ar[r]_{\Delta_s}^{\sim}& F^sF'
}$$
i.e., the isomorphism $\Delta_s:F'F^s\to  F^sF'$
 intertwines the endomorphisms $X'1_{F^s}$ and $1_{F^s}X'.$
 Therefore, evaluating at $N$,
 the eigenvalues of $X'1_{F^s}(N)$ on $F'F^s(N)$
 are the same as those of $1_{F^s}X'(N)$ on $F^sF'(N)$, and hence of $X'(N)$ on $F'(N)$.
\end{proof}

\begin{lemma}\label{Lem:sym-spinor}
Let $E_{t_+,t_-}$ be a quadratic unipotent cuspidal module of $G_r$.
Then
the eigenvalues of $X'(E_{t_-,t_+})$ on $F'(E_{t_-,t_+})$ are the same as those of $X(E_{t_+,t_-})$ on $F(E_{t_+,t_-})$ and vice versa.
\end{lemma}
\begin{proof}
By Proposition \ref{prop:spin-SO},
we have the following commutative diagram:
$$\xymatrix{F'_{r+1,r} \mathrm{Spin}_r\ar[r]_{\Phi_r}^{\sim}\ar[d]^{X'1_{\mathrm{Spin}}} &\mathrm{Spin}_{r+1}F_{r+1,r} \ar[d]^{1_{\mathrm{Spin}} X}\\F'_{r+1,r}
\mathrm{Spin}_r\ar[r]_{\Phi_r}^{\sim}& \mathrm{Spin}_{r+1}F_{r+1,r}
}$$
Since $\mathrm{Spin}_r(E_{t_+,t_-})\cong E_{t_-,t_+}$, it follows that
the eigenvalues of $X'(E_{t_-,t_+})$ on $F'(E_{t_-,t_+})$ are the same as those of $X(E_{t_+,t_-})$ on $F(E_{t_+,t_-})$.
\end{proof}

\begin{lemma}\label{Lem:eigen}
The eigenvalues of $X_{1,0}=X(E_{0,0})$ on $F(E_{0,0})$ and of $X'_{1,0}=X'(E_{0,0})$ on $F' (E_{0,0})$
 are both  $1$ and $(-q)^{-1}$.
\end{lemma}
\begin{proof}
Let $B$ be the standard Borel subgroup of $G_1=\SO_3(q)$,
and let $U$ be the unipotent radical of $B$.
We have
$u_\alpha(2v^{-1})\dot{t}_1u_\alpha(v)\dot{t}_1u_\alpha(2v^{-1})={\rm diag}(2v^{-2},1,v^{2}/2)\dot{t}_1,$
where $$ \dot{t}_1=\left(\begin{array}{ccc} & & -1 \\ & -1 & \\ -1\\ \end{array}\right)\quad \text{and}\quad u_\alpha(v)=\begin{pmatrix}1&v&-v^2/2\\ &1&-v\\&&1\end{pmatrix} $$
for any $v\in \bbF_q^{\times}.$
So $e_U\dot{t}_1u_\alpha(v)\dot{t}_1e_U=e_U{\rm diag}(2v^{-2},1,v^{2}/2)\dot{t}_1e_U$
if $v\in \bbF_q^{\times}$ and $e_U\dot{t}_1u_\alpha(v)\dot{t}_1e_U=
e_U$ if $v=0$.
Hence $$e_U\dot{t}_1e_U\dot{t}_1e_U=\frac{1}{q}e_U+\frac{1}{q}\sum\limits_{v\in \bbF_q^{\times}}{\rm diag}(2v^{-2},1,v^{2}/2)e_U\dot{t}_1e_U.$$

Now, since
$X_{1,0}=e_{1,0}\dot{t}_1e_{1,0}=e_B\dot{t}_1e_B$ and $X_{1,0}'=e'_{1,0}\dot{t}_1e'_{1,0}=\zeta(2)e_\zeta e_U\dot{t}_1e_{\zeta}e_U$ by \S \ref{sub:explicit-repdatumSO},
it follows that
$$
  (X_{1,0})^2=e_1 e_U\dot{t}_{1}e_U\dot{t}_{1}e_{U}=\frac{1}{q}e_1 e_U+\frac{1}{q}\sum\limits_{v\in \bbF_q^{\times}}e_1 {\rm diag}(2v^{-2},1,v^{2}/2)e_U\dot{t}_1e_U\\
$$
and $$(X'_{1,0})^2= e_\zeta e_U\dot{t}_{1}e_U\dot{t}_{1}e_{U}
  =\frac{1}{q}e_\zeta e_U+\frac{1}{q}\sum\limits_{v\in \bbF_q^{\times}}e_\zeta {\rm diag}(2v^{-2},1,v^{2}/2)e_U\dot{t}_1e_U.$$
However, since $e_1 {\rm diag}(2v^{-2},1,v^{2}/2)=e_1$ and $e_\zeta {\rm diag}(2v^{-2},1,v^{2}/2)
=\zeta(2)e_\zeta$ for all $v\in \bbF_q^{\times},$ we have
$$
  (X_{1,0})^2=\frac{1}{q}e_B+\frac{q-1}{q}e_1e_U\dot{t}_1e_U
  =\frac{q-1}{q}X_{1,0}+\frac{1}{q}\id,
$$
and
$$
  (X'_{1,0})^2=\frac{1}{q}e_\zeta e_U+\zeta(2)\frac{q-1}{q}e_\zeta e_U\dot{t}_1e_U
  =\frac{q-1}{q}X'_{1,0}+\frac{1}{q}\id.
$$
\end{proof}
\begin{remark}
The fact that the eigenvalues of $X_{1,0}=X(E_{0,0})$ on $F(E_{0,0})$ and of $X'_{1,0}=X'(E_{0,0})$ on $F' (E_{0,0})$
are same can be deduced from Lemma \ref{Lem:sym-spinor}.
\end{remark}

\begin{theorem}\label{thm:HL-BC}
Let $E_{t_+,t_-}$ be a  quadratic unipotent cuspidal module of $G_r$ with $r=r_++r_-$
and $r_\pm=t_\pm(t_\pm+1)$. Assume that $n=m+r$ and $(m_+,m_-)\comp_2\ m$.
Then the map  $\phi_{F^{\tuple\nu_0}}(E_{t_+,t_-})$ of (\ref{map:relations}) factors through
a $K$-algebra isomorphism $$\bfH^{q\,;\,\tuple\xi_{t_+}}_{K,m_+}\otimes\bfH^{q\,;\,\tuple\xi_{t_-}}_{K,m_-}
\mathop{\longrightarrow}\limits^\sim\scrH(KG_n, E_{t_+,t_-}\otimes K_1^{m_+}\otimes K_\zeta^{m_-}),$$
where $\tuple\xi_{t_+}=((-q)^{t_+},(-q)^{-1-t_+})~\mbox{and}~\tuple\xi_{t_-}=((-q)^{t_-},(-q)^{-1-t_-})$.
\end{theorem}

\begin{proof}
The theorem will be proved by Howlett-Lehrer theory \cite{HL80} and Lemma \ref{Lem:eigenvalues},
if the eigenvalues of $X(E_{t_+,t_-})$ on $F(E_{t_+,t_-})$
 and $X'(E_{t_+,t_-})$ on $F'(E_{t_+,t_-})$ are turned out to be those in $\tuple\xi_{t_\pm}$. Write $X:=X(E_{t_+,t_-})$ and  $X':=X'(E_{t_+,t_-})$ for simplicity.
 Proposition \ref{prop:stru-end-quadr} shows that
\begin{equation*} \label{eq:xsign}
(X-\epsilon_{t_+}\,(-q)^{-1-t_+})(X-\epsilon_{t_+}\,(-q)^{t_+})=0
\end{equation*}
\begin{equation*} \label{eq:xsign2}
(X'-\epsilon'_{t_-}\,(-q)^{-1-{t_-}})(X'-\epsilon'_{t_-}\,(-q)^{t_-})=0,
\end{equation*}
where $\epsilon_{t_+} ,\epsilon'_{t_-} \in\{\pm 1\}$.
Therefore, to prove the theorem, it suffices to show that
\begin{equation} \label{eq:nosign}
(X-(-q)^{-1-t_+})(X-(-q)^{t_+})=0,\\
\end{equation}
\begin{equation} \label{eq:nosign2}
(X'-(-q)^{-1-{t_-}})(X'-(-q)^{t_-})=0,
\end{equation}
i.e., it suffices to show that
$\epsilon_{t_+}=\epsilon'_{t_-}=1$  for all $t_{\pm} \geqs 0$.
This will be done by induction on $t_++t_-$ into three cases
$t_+ > t_-,t_+ < t_-$ or $t_+=t_-$,  using the theory of Brauer tree.

\smallskip

First, we notice from Lemma \ref{Lem:eigen} that
the eigenvalues of $X_{1,0}=X(E_{0,0})$ on $F (E_{0,0})$ and of $X'_{1,0}=X'(E_{0,0})$ on $F'(E_{0,0})$
 are both $1$ and $(-q)^{-1}$. In particular, we have $\epsilon_{t_+}=\epsilon'_{t_-}=1$ for $t_{\pm}= 0$.

\medskip

In the following, we write $\chi_{t_+,t_-}$ for the character afforded by the cuspidal module $E_{t_+,t_-}$.
For a symbol $(\Theta_+,\Theta_-)$, we shall generally write $\chi_{\Theta_+,\Theta_-}$ and $E_{\Theta_+,\Theta_-}$
for the corresponding quadratic unipotent character and module, respectively.
Also, for a character $\varphi$, we write $E_\varphi$ for the module affording $\varphi$.
In addition, we shall use the notation introduced in \S \ref{subsec:quchar}.

\medskip

\textbf{Case 1: ${t_+ > t_-}$}. In this case, we
use a similar argument to that for \cite[Theorem~6.5]{DVV2}.
The cuspidal character $\chi_{t_+,t_-}$ is attached to the symbol
$\tuple\Theta_{t_+,t_-}=\Theta_{t_+}\times\Theta_{t_-}$ with
$$ \Theta_{t_+} \, = \, \left\{ \begin{array}{ccccccc} t_+ & t_+-1 & t_+-2 & \ldots  & -t_+ & -t_+-1 & \ldots  \\
				 & & & & & -t_+-1 & \ldots  \\ \end{array} \right\}$$
and
$$ \Theta_{t_-} \, = \, \left\{ \begin{array}{ccccccc} t_- & t_--1 & t_--2 & \ldots  & -t_- & -t_--1 & \ldots  \\
				 & & & & & -t_--1 & \ldots  \\ \end{array} \right\}.$$
We shall choose a prime $\ell$ such that
$\Theta_{t_-}$ is a $d$-cocore and
$\Theta_{t_+}$ has only one $d$-cohook that can be removed.
For this aim, we indeed choose the prime
$\ell$ to be odd and such that the order of $q$ in $k^\times$ is $f:=4t_+$.
In particular, the order of $q^2$ is $d = 2t_+$, and so $\ell$ is a unitary prime.
We shall determine the values of $\epsilon_{t_+}$ and $\epsilon_{t_-}'$
by
comparing the eigenvalues of $X$ on $F(E_{t_+,t_-})$ modulo $\ell$,
with a splitting $\ell$-modular system $(K,\mathcal{O},k)$.

\smallskip

Now the cuspidal character $\chi_{t_+,t_-}$ belongs to an $\ell$-block $B_E$ with
cyclic defect groups (see \S \ref{subsec:Brauertree} (a) or \cite{FS90}).
Moreover, the unipotent label of the block $B_E$ is $(\Theta_{t_+-1},\Theta_{t_-})$.
 Let $s:=s_{r_+,r_-}$. Then $C_{G_r^*}(s)^*=G_{r_+}\times G_{r_-}$.
By \cite{FS90}, there exists a cyclic block $B_C=B_1\otimes B_2$ of $C_{G_r^*}(s)^*$,
with $B_1$ a cyclic block of $G_{r_+}$
and $B_2$ a defect zero block of $G_{r_-}$,
such that the Jordan decomposition $$\mathcal{L}_s:\mathcal{E}(G_r,(s))\to \mathcal{E}(C_{G_r^*}(s)^*,(1))$$
sends
 $\chi_{\Theta_+\times \Theta_-}$ to $\chi_{\Theta_+}\times \chi_{\Theta_-}$
  with $\chi_{\Theta_+}\in B_1$ and $\chi_{\Theta_-}\in B_2$ and induces a graph isomorphism  between the Brauer trees of $B_E$ and $B_C$.
Hence, by \S \ref{subsec:Brauertree} (a2),
 the Brauer tree of $B_E$ is
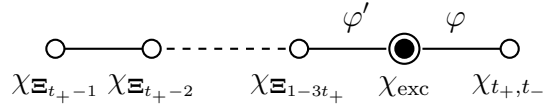
\begin{center}
\begin{tikzpicture}[scale=.4]
\draw[thick] (-1.7 cm,0) circle (.3cm);
\node [below] at (-1.7 cm,-.5cm) {$\chi_{\tuple\Xi_{t_+-1}}$};
\draw[thick] (-1.4 cm,0) -- +(2.6 cm,0);
\draw[thick] (1.5 cm,0) circle (.3cm);
\node [below] at (1.5 cm,-.5cm) {$\chi_{\tuple\Xi_{t_+-2}}$};
\draw[dashed,thick] (2 cm,0) -- +(4 cm,0);
\draw[thick] (6.4 cm,0) circle (.3cm);
\node [below] at (6.4 cm,-.5cm) {$\chi_{\tuple\Xi_{1-3t_+}}$};
\draw[thick] (6.7 cm,0) -- +(2.6 cm,0);
\draw[thick,fill=black] (9.9 cm,0) circle (.3cm);
\node [above] at (8.3 cm,.2cm) {$\varphi'$};
\draw[thick] (9.9 cm,0) circle (.5cm);
\node [below] at (9.9 cm,-.55cm) {$\chi_{\text{exc}}$};
\draw[thick] (10.5 cm,0) -- +(2.6 cm,0);
\node [above] at (11.6 cm,.2cm) {$\varphi$};
\draw[thick] (13.4 cm,0) circle (.3cm);
\node [below] at (13.4 cm,-.5cm) {$\chi_{t_+,t_-}$};
\end{tikzpicture}
\end{center}
where the symbol $\tuple\Xi_k=\Xi_{k,\,+}\times \Theta_{t_-}$ is the label of the quadratic unipotent character
$\chi_{\tuple\Xi_k}$ and for $k \in \{1-3t_+,\ldots,t_+-1\}$,
\begin{equation}\label{Xi}
\begin{aligned}
\Xi_{k,\,+}\, = \, \left\{ \begin{array}{cccccccc} t_+-1 & t_+-2 & \ldots  & \ldots  &  \ldots   &\widehat{k}& \ldots \\
				 & & k+2t_+ &  -t_+ & -t_+-1 & \ldots  & \ldots    \\ \end{array} \right\},
\end{aligned}
\end{equation}
obtained by adding the $d$-cohook $(k,k+d)$ to $\Theta_{t_+-1}$.
Here the notation $\widehat k$ means that the integer $k$ has been removed.

\smallskip

Let $\chi$ be an exceptional character of $B_E$.
Then the  isomorphism $\mathcal{L}_s$
sends $\chi$ to the character
  $\mathcal{L}_s(\chi):=\chi_{e}\times \chi_{\Theta_{t_-}}$,
  where $\chi_{e}$ is an exceptional character of $B_1$.
     Since the Jordan decomposition of characters
    commutes with Harish-Chandra induction,
    the irreducible constituents of $F(E_\chi)=R_{L_{r,1}}^{G_{r+1}}(E_\chi\otimes K_1)$
are one-to-one corresponding to the  irreducible constituents of
 $$R_{G_{r_+}\times G_{r_-}\times\bbF_q^{\times}}^{G_{r_++1}\times G_{r_-}}(E_{\mathcal{L}_s(\chi)})
 \cong (R_{G_{r_+}\times\bbF_q^{\times}}^{G_{r_++1}}
 (E_{\chi_{e}}\otimes K_1)\otimes E_{\Theta_{t_-}}=F(E_{\chi_{e}})\otimes  E_{\Theta_{t_-}}.$$
As shown in the proof of \cite[Theorem 6.5]{DVV2} (p. 59),
the character of $F(E_{\chi_{e}})$ has at most 2 irreducible constituents.
 Hence $F(E_\chi)$ has at most 2 irreducible constituents.

\smallskip

By adding a 1-hook to $\Theta_{t_+}$ in all possible ways, we conclude
that the module $F(E_{t_+,t_-})$ is the sum of the two quadratic unipotent modules
$E_{\tuple\Lambda}$ and $E_{\tuple\Lambda'}$ with
$\tuple\Lambda=\Lambda_+\times\Theta_{t_-}$ and $\tuple\Lambda'=\Lambda'_+\times\Theta_{t_-}$,
where
$$\begin{aligned}
\Lambda_+ = &\, \left\{\begin{array}{ccccccc} t_+ & t_+-1 & \ldots  & -t_+ & -t_+-1 & -t_+-2 &  \ldots  \\
				 & & & & -t_+& -t_+-2 & \ldots  \\ \end{array} \right\}~\mbox{and}\\
\Lambda_+'= &\, \left\{\begin{array}{ccccccc} t_++1 & t_+-1 & t_+-2 & \ldots   & -t-1 & \ldots  \\
				 & & & & -t_+-1 & \ldots  \\ \end{array} \right\}.
\end{aligned}$$
Let $B$ and $B'$ be the $\ell$-blocks containing $E_{\tuple\Lambda}$ and $E_{\tuple\Lambda'}$, respectively.
It is easy to see that the $d$-cocores of $\Lambda_+$ and $\Lambda_+'$ are
different. Hence $B$ and $B'$ are different, and so their
idempotents, say $b$ and $b'$, are orthogonal.
It follows that
$$F(E_{t_+,t_-})=bF(E_{t_+,t_-})\oplus b'F(E_{t_+,t_-})=E_{\tuple\Lambda}\oplus E_{\tuple\Lambda'}.$$
Moreover,
the eigenvalues of  $X(E_{t_+,t_-})$  on $E_{\tuple\Lambda}$ and  on
 $E_{\tuple\Lambda'}$  are different: one is $ \epsilon_{t_+}(-q)^{-1-t_+}$, and the other is
 $\epsilon_{t_+}(-q)^{t_+}.$

\smallskip

In the Brauer tree of $B_E$,
the irreducible modular character  $\varphi$ is the $\ell$-reduction of $\chi_{t_+,t_-}$.
Hence $F(E_\varphi)=bF(E_\varphi)\oplus b'F(E_\varphi)$,
 and the operator $X(E_\varphi)$ on   $F(E_\varphi)$ has eigenvalues
$\epsilon_{t_+}(-q)^{-1-t_+}$ and $\epsilon_{t_+}(-q)^{t_+}.$

\smallskip

Since $\varphi$ is an irreducible constituent of the $\ell$-reduction of $\chi$, it follows that
$bF(E_\varphi)$ and $b'F(E_\varphi)$ are constituents of $bF(E_\chi)$ and $b'F(E_\chi)$, respectively.
Hence both $bF(E_\chi)$ and $b'F(E_\chi)$ are non-zero, and so irreducible by
the facts that $F(E_\chi)=bF(E_\chi)\oplus b'F(E_\chi)$ and that $F(E_\chi)$ has at most 2 irreducible constituents.
Now the eigenvalues of
$X(E_\chi)$ on $F(E_\chi)$
must be modulo $\ell$ congruent to the eigenvalues of $X(E_\varphi)$ on $F(E_\varphi)$, which are equal to
$\epsilon_{t_+}(-q)^{-1-t_+}$ and $\epsilon_{t_+}(-q)^{t_+}$.
However, since $\varphi'$ is also an irreducible constituent of the $\ell$-reduction of
$\chi$,  one of $bF(E_{\varphi'})$ or $b'F(E_{\varphi'})$ must be non-zero,
and so $X(E_{\varphi'})$ must have  an eigenvalue modulo $\ell$
congruent to $\epsilon_{t_+}(-q)^{-1-t_+}$ or $\epsilon_{t_+}(-q)^{t_+}$.

\smallskip

To obtain $\epsilon_{t_+}=1$,
we continue to compute the eigenvalues of $X(E_{\tuple\Xi_{1-3t_+}})$ on $bF(E_{\tuple\Xi_{1-3t_+}})$
and $b'F(E_{\tuple\Xi_{1-3t_+}})$.
We have
$$F(E_{\tuple\Xi_{1-3t_+}})=E_{\tuple\Upsilon}\oplus E_{\tuple\Upsilon'}\oplus E_{\tuple\Upsilon''}$$
 where $\tuple\Upsilon=\Upsilon_+\times\Theta_{t_-}$,
$\tuple\Upsilon'=\Upsilon'_+\times\Theta_{t_-}$ and
$\tuple\Upsilon''=\Upsilon_+''\times\Theta_{t_-}$ with
\begin{align*}
\Upsilon_+\, &= \, \left\{ \begin{array}{ccccccc} t_+ & t_+-2 & t_+-3 & \ldots  & \ldots  & \widehat{1-3t_+} &  \ldots   \\
				 & & & -t_++1 &  -t_+ &  \ldots  & \ldots\\ \end{array} \right\},\\
\Upsilon'_+\, &= \, \left\{\begin{array}{cccccccc} t_+-1 & t_+-2 & \ldots  & \ldots &  \ldots  & \widehat{-3t_+} & \ldots \\
				 & & -t_++1 &  -t_+ & -t_+-1 & \ldots & \ldots  \\ \end{array} \right\}~\mbox{and}\\
\Upsilon''_+\, &= \, \left\{ \begin{array}{cccccccc} t_+-1 & t_+-2 & \ldots  & \ldots  &  \ldots  &\widehat{1-3t_+} & \ldots \\
				 & & -t_++2 &  -t_+ & -t_+-1 & \ldots  & \ldots \\ \end{array} \right\}.
\end{align*}
Since both $\Theta_{t_-}$ and $\Upsilon_+''$ are $d$-cocores,
we see that $E_{\tuple\Upsilon''}$ is projective.
Moreover, since the $d$-cocore of the symbol $\Upsilon_+$ is
\begin{align*}
\left\{ \begin{array}{ccccccc} t_+ & t_+-2 & t_+-3 & \ldots  & \ldots   \\
				 & & & -t_+ &  \ldots  \\ \end{array} \right\}
\end{align*}
and the $d$-cocore of the symbol $\Upsilon_+'$ is
\begin{align*}
\left\{\begin{array}{ccccccc} t_+-1 & t_+-2 & \ldots  & \ldots  &  \ldots   \\
				 & &  -t_++1 &  -t_+-1 &  \ldots  \\ \end{array} \right\},
\end{align*}
we conclude that the characters $\chi_{\tuple\Upsilon}$ and $\chi_{\tuple\Upsilon'}$
belong to the $\ell$-blocks $B$ and $B'$, respectively.

\smallskip

Suppose that $t_+ \geqs 2$.
Observe that the symbols $\Xi_{k,\,+}$ all have defect
$$|2t_+-3| = |2(t_+-2)+1|$$ and $E_{-1,s} = E_{0,s}$.
Hence the quadratic unipotent modules $E_{\tuple\Xi_k}$ all lie in the
Harish-Chandra series above $E_{t_+-2,t_-}\otimes  K_1^{4t_+-2}$.
Furthermore, the
bipartition $\mu_k$ such that $\Theta_{t_+-2}(\mu_k) = \Xi_{k,\,+}$ is $((1^{t_+-1-k}),(k+3t_+-1))$,
except when $t_+ =1$  in which case $\mu_k = ((k+2),(1^{-k}))$.

\smallskip

 By the inductive hypothesis,
 the map $\phi_{F^m}(E_{t_+-2,t_-})$
 yields a $K$-algebra isomorphism
 $$\bfH^{q,\,{\tuple\xi_{t_+-2}}}_{K,m}\simto\scrH(K G_{n'}, E_{t_+-2,t_-}\otimes K_1^{m}),$$
 where $n'=n-(4t_+-2)$.
Furthermore, the induced bijection
$$\Irr(KG_{n'}, E_{t_+-2,t_-}\otimes K_1^m)\mathop{\longleftrightarrow}\limits^{1:1}
\Irr(\bfH^{q,{\tuple\xi_{t_+-2}}}_{K,m})$$
maps the module $E_{\Theta_{t-2}(\tuple\mu)\times\Theta_{t_-}}$
to the module $S(\tuple\mu)_{K}^{q,\, \tuple \xi_{t_+-2}}$ for each 2-partition $\tuple\mu$ of $m$.
Under this parametrization, the module $E_{\tuple\Xi_{1-3t_+}}$ of $KG_r$
and the modules $E_{\tuple\Upsilon}$, $E_{\tuple\Upsilon'}$
and $E_{\tuple\Upsilon''}$ of $KG_n$ are respectively mapped to the modules
$$S(\tuple\lambda)_K^{q,\,\tuple\xi_{t_+-2}},\quad
S(\tuple\mu)_K^{q,\,{\tuple\xi_{t_+-2}}},\quad
S(\tuple\mu')_K^{q,\,{\tuple\xi_{t_+-2}}},\quad
S(\tuple\mu'')_K^{q,\,{\tuple\xi_{t_+-2}}}$$
labeled by the following 2-partitions
$$\tuple\lambda=((1^{4t_+-2}),\emptyset),\quad
\tuple\mu=((21^{4t_+-3}),\emptyset),\quad
\tuple\mu'=((1^{4t_+-1}),\emptyset),\quad
\tuple\mu''=((1^{4t_+-2}),(1)).$$
Note that the $({\tuple\xi_{t_+-2}}, q)$-shifted residue of the boxes
$Y(\tuple\mu)\setminus Y(\tuple\lambda)$ and $Y(\tuple\mu')\setminus Y(\tuple\lambda)$
are $(-1)^{t_+} q^{t_+-1}$ and $(-1)^{t_+} q^{-3t_+}$, respectively.
They are congruent to $(-q)^{-1-t_+}$ and $(-q)^{t_+}$ modulo $\ell$,
since $q^{2t_+}$ is congruent to $-1$ modulo $\ell$.
We conclude that the eigenvalues of
the operator $X(E_{\tuple\Xi_{1-3t_+}})$ on
$E_{\tuple\Upsilon}$ and $E_{\tuple\Upsilon'}$
are  congruent to $(-q)^{-1-t_+}$ and
$(-q)^{t_+}$ modulo $\ell$, respectively.
Now, at least one of these must be modulo $\ell$ congruent to
 the eigenvalue of $X(E_{\varphi'})$, since
  $\varphi'$ is an irreducible constituent of the $\ell$-reduction of
$E_{\tuple\Xi_{1-3t_+}}$.

\smallskip

By the previous argument,
the
eigenvalue of $X(E_{\varphi'})$ on $F(E_{\varphi'})$ is modulo $\ell$ congruent to
$\epsilon_{t_+} (-q)^{-1-t_+}$ or $\epsilon_{t_+} (-q)^{t_+}$.
However, the choice of $\ell$ implies that any pair of
$$(-q)^{-1-t_+},\ \ -(-q)^{-1-t_+},\ \ (-q)^{t_+},\ \ -(-q)^{t_+}$$
 are not congruent modulo $\ell$.  This forces
$\epsilon_{t_+} =1$.

\smallskip

We now suppose that $t_+ <2$. In this case, we have $t_+=1$ and $t_-=0$ since $t_+>t_-.$
Hence the quadratic unipotent characters are indeed unipotent, and so
we also have  $\epsilon_{t_+} =1$
 by the proof of \cite[Theorem 6.5]{DVV2}.

\smallskip
\smallskip

Next we prove that $\epsilon'_{t_-} =1$. We keep the previous notation. In particular,
the cuspidal module
$E_{t_+,t_-}$ belongs to the cyclic $\ell$-block $B_E$ and $\chi$ is an exceptional  character of $B_E$.
Since the constituents of
 $F'(E_\chi)=R_{L_{r,1}}^{G_{r+1}}(E_\chi\otimes K_\zeta)$ are one-to-one corresponding to the constituents of
$$R_{G_{r_+}\times G_{r_-}\times\bbF_q^{\times}}^{G_{r_+}\times G_{r_-+1}}(E_{\mathcal{L}_s(\chi)}\otimes K_\zeta)\cong
E_{\chi_{e}}\otimes (R_{G_{r_-}\times\bbF_q^{\times}}^{G_{r_-+1}}(E_{\Theta_{t_-}}\otimes K_\zeta))\cong
E_{\chi_{e}}\otimes F'(E_{\Theta_{t_-}}).$$
Now $F'(E_{\Theta_{t_-}})$ has at most 2 irreducible constituents,
it follows that $F'(E_\chi)$ has at most 2 irreducible constituents.
Hence $X'(E_\chi)$ has at most 2 eigenvalues on $F'(E_\chi)$, whose product is equal to $(-q)^{-1}.$ 

\smallskip

By adding a 1-hook to $\Theta_{t_-}$ in all possible ways, we
have that $F'(E_{t_+,t_-})$ is the sum of the quadratic unipotent characters
$\chi_{\tuple\Pi}$ and $\chi_{\tuple\Pi'}$ with
$\tuple\Pi={\Theta_{t_+}\times\Xi_-}$ and $\tuple\Pi'={\Theta_{t_+}\times\Xi'_-}$,
where
 $$\begin{aligned}
\Xi_- = &\, \left\{\begin{array}{ccccccc} t_- & t_--1 & \ldots  & -t_- & -t_--1 & -t_--2 &  \ldots  \\
				 & & & & -t_-& -t_--2 & \ldots  \\ \end{array} \right\}~\mbox{and}\\
\Xi_-'= &\, \left\{\begin{array}{ccccccc} t_-+1 & t_--1 & t_--2 & \ldots   & -t-1 & \ldots  \\
				 & & & & -t_--1 & \ldots  \\ \end{array} \right\}.
\end{aligned}$$
Notice that both $\Xi_-$ and $\Xi_-'$ are $d$-cocores since $2t_-+1<2t_+.$
Let $\widetilde{B}$ and $\widetilde{B}'$
 be the $\ell$-blocks of $K G_n$ containing $\chi_{\tuple\Pi}$ and $\chi_{\tuple\Pi'}$, respectively,
  with the corresponding idempotents $\widetilde{b}$ and $\widetilde{b}'$.
  Since $\Xi_-$ and $\Xi_-'$ are
different $d$-cocores, the idempotents $\widetilde{b}$ and $\widetilde{b}'$ are orthogonal.
Hence $\widetilde{b}F'(E_{t_+,t_-})=\chi_{\tuple\Pi}$
and $\widetilde{b}'F'(E_{t_+,t_-})=\chi_{\tuple\Pi'}$,
on which we may assume that
$X'(E_{t_+,t_-})$ has the eigenvalues
$\epsilon'_{t_-}(-q)^{-1-t_-}$ and  $\epsilon'_{t_-}(-q)^{t_-}$, respectively.

 \smallskip

Also, we have that $F'(E_\varphi)=\widetilde{b}F'(E_\varphi) \oplus \widetilde{b}'F'(E_\varphi)$ and that
the operator $X'(E_\varphi)$ on $F'(E_\varphi)$ has eigenvalues $\epsilon'_{t_-}(-q)^{-1-t_-}$ and  $\epsilon'_{t_-}(-q)^{t_-}$.

 \smallskip

 Since only one $d$-cohook can be removed from $\Theta_{t_+}$ and
 no $d$-cohook can be removed from $\Xi_-$ (resp. $\Xi'_-$),
 the $\ell$-block $\widetilde{B}$ (resp. $\widetilde{B}'$)
 has cyclic defect groups.
 Furthermore, since the $d$-cocore of $\Theta_{t_+}$ equals $\Theta_{t_+-1}$,
the Brauer trees of $\widetilde{B}$ and $\widetilde{B}'$ are respectively
\begin{center}
\begin{equation}\label{Pi}
\begin{tikzpicture}[scale=.4]
\draw[thick] (-1.7 cm,0) circle (.3cm);
\node [below] at (-1.7 cm,-.5cm) {$\chi_{\tuple\Lambda_{t_+-1}}$};
\draw[thick] (-1.4 cm,0) -- +(2.6 cm,0);
\draw[thick] (1.5 cm,0) circle (.3cm);
\node [below] at (1.5 cm,-.5cm) {$\chi_{\tuple\Lambda_{t_+-2}}$};
\draw[dashed,thick] (2 cm,0) -- +(4 cm,0);
\draw[thick] (6.4 cm,0) circle (.3cm);
\node [below] at (6.4 cm,-.5cm) {$\chi_{\tuple\Lambda_{1-3t_+}}$};
\draw[thick] (6.7 cm,0) -- +(2.6 cm,0);
\draw[thick,fill=black] (9.9 cm,0) circle (.3cm);
\draw[thick] (9.9 cm,0) circle (.5cm);
\node [below] at (9.9 cm,-.55cm) {$\varphi_{\text{exc}}$};
\draw[thick] (10.5 cm,0) -- +(2.6 cm,0);
\draw[thick] (13.4 cm,0) circle (.3cm);
\node [below] at (13.4 cm,-.5cm) {$\chi_{\tuple\Pi}$};
\end{tikzpicture}
\end{equation}
\end{center}
and
\begin{center}
\begin{equation}\label{Pi'}
\begin{tikzpicture}[scale=.4]
\draw[thick] (-1.7 cm,0) circle (.3cm);
\node [below] at (-1.7 cm,-.5cm) {$\chi_{\tuple\Lambda'_{t_+-1}}$};
\draw[thick] (-1.4 cm,0) -- +(2.6 cm,0);
\draw[thick] (1.5 cm,0) circle (.3cm);
\node [below] at (1.5 cm,-.5cm) {$\chi_{\tuple\Lambda'_{t_+-2}}$};
\draw[dashed,thick] (2 cm,0) -- +(4 cm,0);
\draw[thick] (6.4 cm,0) circle (.3cm);
\node [below] at (6.4 cm,-.5cm) {$\chi_{\tuple\Lambda'_{1-3t_+}}$};
\draw[thick] (6.7 cm,0) -- +(2.6 cm,0);
\draw[thick,fill=black] (9.9 cm,0) circle (.3cm);
\draw[thick] (9.9 cm,0) circle (.5cm);
\node [below] at (9.9 cm,-.55cm) {$\varphi'_{\text{exc}}$};
\draw[thick] (10.5 cm,0) -- +(2.6 cm,0);
\draw[thick] (13.4 cm,0) circle (.3cm);
\node [below] at (13.4 cm,-.5cm) {$\chi_{\tuple\Pi'}$};
\end{tikzpicture}
\end{equation}
\end{center}
Here  $\tuple\Lambda_k=\Xi_{k,\,+}\times\Xi_-$
and $\tuple\Lambda'_k=\Xi_{k,\,+}\times\Xi'_-,$
 where the $\Xi_{k,\,+}$ are defined as (\ref{Xi}).
 It turns out that $F'(E_{\tuple\Xi_k})$ is exactly equal to $E_{\tuple\Lambda_k}+ E_{\tuple\Lambda'_k},$ hence $E_{\tuple\Lambda_k}=\widetilde{b}F'(E_{\tuple\Xi_k})$ and $E_{\tuple\Lambda'_k}=\widetilde{b}'F'(E_{\tuple\Xi_k}).$
 Since $\varphi$ is an irreducible constituent of the $\ell$-reduction of $\chi$,
 it follows that   $\widetilde{b}F'(E_\chi)$ and $\widetilde{b}'F'(E_\chi)$ are non-zero.
 Hence both are irreducible and  $$F'(E_\chi)=\widetilde{b}F'(E_\chi)\oplus \widetilde{b}'F'(E_\chi).$$
 Clearly, the character of $\widetilde{b}F'(E_\chi)$ (resp. $\widetilde{b}'F'(E_\chi)$)
 is an exceptional character of $\widetilde{B}$ (resp. $\widetilde{B}'$).

  \smallskip

  We claim that both $\widetilde{b}F'(E_{\varphi'})$ and $\widetilde{b}'F'(E_{\varphi'})$ are non-zero and irreducible.
  In \mbox{fact}, since $\varphi'$ is an irreducible constituent of the $\ell$-reduction of
$\chi$, one of $\widetilde{b}F'(E_{\varphi'})$ or $\widetilde{b}'F'(E_{\varphi'})$ must be non-zero.
Without loss of generality, we may assume that
$\widetilde{b}F'(E_{\varphi'})\neq0$. Since
$E_{\tuple \Pi}$
and $\widetilde{b}F'(E_\chi)$  are in the same Brauer tree,
their $\ell$-reductions have only one irreducible constituent in common.
However, $\widetilde{b}F'(E_{\varphi'})$ appears in the $\ell$-reductions of $\widetilde{b}F'(E_\chi)$
and $\widetilde{b}F'(E_{t_+,t_-})=E_{\tuple \Pi}$. Hence $\widetilde{b}F'(E_{\varphi'})$ is irreducible.
However, the $\ell$-reduction of $F'(E_\chi)$ has two different irreducible constituents, one of which
is now $\widetilde{b}F'(E_{\varphi'})$. Since $\widetilde{b}'F'(E_{\varphi'})$ also appears in
the $\ell$-reduction of $F'(E_\chi)$,
we conclude that $\widetilde{b'}F'(E_{\varphi'})$ is nonzero and irreducible,
proving the claim.

\smallskip

We now consider the eigenvalues of
$X'(E_{\varphi'})$ on $\widetilde{b}F'(E_{\varphi'})$ and $\widetilde{b}'F'(E_{\varphi'})$.
Note that the eigenvalues of $X'(E_{t_+,t_-})$
 $\widetilde{b}F'(E_{t_+,t_-})=E_{\tuple\Pi}$ and $\widetilde{b}'F'(E_{t_+,t_-})=E_{\tuple\Pi'}$
consist of the set $\{\epsilon'_{t_-}(-q)^{-1-t_-}, \epsilon'_{t_-}(-q)^{t_-}\}$.
Moreover, the eigenvalues of
$X'(E_\chi)$ must be modulo $\ell$ congruent to the eigenvalues of $X'(E_\varphi)$ on $F'(E_\varphi)$, which are equal to
$\epsilon'_{t_-}(-q)^{-1-t_-}$
and $\epsilon'_{t_-}(-q)^{t_-}$.
 In addition,
since $\widetilde{b}F'(E_{\varphi'})$ (resp. $\widetilde{b}'F'(E_{\varphi'})$)
is an irreducible constituent of $\widetilde{b}F'(E_\chi)$ (resp. $\widetilde{b}'F'(E_\chi)$),
 the eigenvalue of
$X'(E_{\varphi'})$ on $\widetilde{b}F'(E_{\varphi'})$ (resp. $\widetilde{b}'F'(E_{\varphi'})$)
must be modulo $\ell$  congruent to the eigenvalue of $X'(E_\chi)$ on $\widetilde{b}F'(E_\chi)$
(resp. $\widetilde{b}'F'(E_\chi)$), which is equal to
$\epsilon'_{t_-}(-q)^{-1-t_-}$ or $\epsilon'_{t_-}(-q)^{t_-}$.

\smallskip

Finally, we compute
the eigenvalues of $X'(E_{\tuple\Xi_{1-3t_+}})$ on $\widetilde{b}F'(E_{\tuple\Xi_{1-3t_+}})$,
 by the inductive hypothesis on the eigenvalues of $X'(E_{t_+-2,t_-})$.
Indeed, the module $E_{\tuple\Xi_{1-3t_+}}$
belongs to the Harish-Chandra series above $E_{t_+-2,t_-}\otimes K_1^{u}$, where $u=4t_+-2$.
Also, the eigenvalues of $X'(E_{\tuple\Xi_{1-3t_+}})$ on $F'(E_{\tuple\Xi_{1-3t_+}})$
is a subset of eigenvalues of  $X'(F^{u}(E_{t_+-2,t_-}))$ on $F'F^{u}(E_{t_+-2,t_-})$.
Since
$$\xymatrix{F'F^{u}\ar[r]_{\Delta_{u}}^{\sim}\ar[d]^{X'1_{F^{u}}} &F^{u}F' \ar[d]^{1_{F^{u}}X'}\\
 F'F^{u}\ar[r]_{\Delta_{u}}^{\sim}& F^{u}F'.
}$$
is a commutative diagram by Lemma \ref{Lem:eigenvalues}, the isomorphism
$$\Delta_{u}=1_{F^{(u-1)}}H'\circ1_{F^{(u-2)}}H'1_{F}\circ\cdots\circ
1_{F^2}H'1_{F^{(u-3)}}1_{F}\circ H'1_{F^{(u-2)}}\circ H'1_{F^{(u-1)}}$$
 intertwines $X'1_{F^{u}}$ and $1_{F^{u}}X'$.
Therefore, the eigenvalues of $X'(F^{u}(E_{t_+-2,t_-}))$
are the same as those of $X'(E_{t_+-2,t_-})$ on $F'(E_{t_+-2,t_-})$,
 which are $(-q)^{-1-t_-}$ and $(-q)^{t_-}$ by induction.
 So the eigenvalues of $X'(E_{\tuple\Xi_{1-3t_+}})$ are also $(-q)^{t_-}$ and $(-q)^{-1-t_-}$.

 \smallskip

 Now, with the same argument as for $\epsilon_{t_+}=1$, we get $\epsilon'_{t_-}=1$.

\smallskip
\smallskip

\textbf{Case 2: ${t_+ < t_-}.$} This follows by Case 1 and Lemma \ref{Lem:sym-spinor},
or by a similar argument as for Case 1, interchanging $t_+$ and $t_-.$

\smallskip

\textbf{Case 3: $t_+=t_-=t$}.
Set $r=t_+(t_++1)+t_-(t_-+1)=2t(t+1)$ and $n=r+1$.
In this case, the cuspidal character $E_{t,t}$ is attached to
the symbol $\tuple\Theta=\Theta_{t}\times\Theta_{t}$ with
$$ \Theta_{t} \, = \, \left\{ \begin{array}{ccccccc} t & t-1 & t-2 & \ldots  & -t & -t-1 & \ldots  \\
				 & & & & & -t-1 & \ldots  \\ \end{array} \right\}.$$

\smallskip

We have
that $F'(E_{t,t})$ is the sum of the quadratic unipotent characters
$E_{\tuple\Pi}$ and $E_{\tuple\Pi'}$,
where $\tuple\Pi=\Theta_{t}\times \Xi_-$
and $\tuple\Pi'=\Theta_{t}\times \Xi'_-$ 
with
$$\begin{aligned}
\Xi_- = &\, \left\{\begin{array}{ccccccc} t & t-1 & \ldots  & -t & -t-1 & -t-2 &  \ldots  \\
				 & & & & -t& -t-2 & \ldots  \\ \end{array} \right\}~\mbox{and}\\
\Xi_-'= &\, \left\{\begin{array}{ccccccc} t+1 & t-1 & t-2 & \ldots   & -t-1 & \ldots  \\
				 & & & & -t-1 & \ldots  \\ \end{array} \right\}.
\end{aligned}$$
Also, we have that $F'(E_{\tuple\Pi})$
is the sum of the quadratic unipotent characters
$E_{\tuple\Upsilon}$,
$E_{\tuple\Upsilon'}$ and $E_{\tuple\Upsilon''}$,
where $\tuple\Upsilon=\Theta_{t}\times\Upsilon_{-}$,
 $\tuple\Upsilon'=\Theta_{t}\times\Upsilon'_{-}$ and
$\tuple\Upsilon''=\Theta_{t}\times\Upsilon''_{-}$ with
$$\begin{aligned}
\Upsilon_- = &\, \left\{\begin{array}{ccccccc} t+1 & t-1 & \ldots  & -t & -t-1 & -t-2 &  \ldots  \\
				 & & & & -t& -t-2 & \ldots  \\ \end{array} \right\},\\
\Upsilon_-'= &\, \left\{\begin{array}{cccccccc} t & t-1 & \ldots  & -t-1& -t-2 &-t-3& \ldots  \\
				 & & & -t& -t-1 &-t-3& \ldots  \\ \end{array} \right\}~\mbox{and}\\
\Upsilon''_-= &\, \left\{\begin{array}{cccccccc} t & t-1 & \ldots  & -t-1& -t-2 &-t-3& \ldots  \\
				 & & & 1-t& -t-2 &-t-3& \ldots  \\ \end{array} \right\}.
\end{aligned}$$

We notice that the eigenvalues of $X(E_{t,t})$ on $F(E_{t,t})$
are the same as those of $X((F')^{2}(E_{t,t}))$ on $(F')^{2}(E_{t,t})$.
  This is because the diagram
$$\xymatrix{F(F')^{2} \ar[r]_{\nabla_{2}}^{\sim}\ar[d]^{X1_{(F')^{2}}} &(F')^{2}F  \ar[d]^{1_{(F')^{2}}X}\\
 F(F')^{2} \ar[r]_{\nabla_{2}}^{\sim}& (F')^{2}F.
}$$
is commutative and
$\nabla_{2}=1_{F}H\circ H1_{F'}$
is an isomorphism between functors $F(F')^{2}$ and $(F')^{2}F$. Since $E_{\tuple \Upsilon}$ is a constituent $(F')^2(E_{t,t}),$ we will use the eigenvalues of
  $X(E_{\tuple\Upsilon})$ on $F(E_{\tuple\Upsilon})$ to deduce the eigenvalues of
  $X(E_{t,t})$ on $F(E_{t,t})$.

\smallskip

As in Case 1, we choose a prime $\ell$ such that
only one $d$-cohook can be removed from $\Upsilon_-$
and no $d$-cohook can be removed from $\Theta_{t}$.
This time we choose $\ell$ to be odd
and such that the order of $q$ in $k^\times$ is $f:=4t+4$.
Thus the order of $q^2$ is $d = 2t+2$, and  $\ell$ is also unitary.
As before, we first prove $\epsilon_{t}=1$.
\smallskip

The $\ell$-block containing  $E_{\tuple\Upsilon}$ has
cyclic defect groups.
Furthermore, since the $d$-cocore of $\Upsilon_-$ is
$$\Theta_{t-1}=\begin{aligned}
&\, \left\{\begin{array}{ccccccc} t-1 & t-2 & \ldots & -t+1  & -t & \ldots  \\
				 & & & & -t & \ldots  \\ \end{array} \right\},
\end{aligned}$$
  the Brauer tree of the $\ell$-block containing  $E_{\tuple\Upsilon}$  is
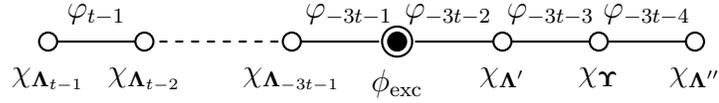
\begin{center}
\begin{tikzpicture}[scale=.4]
\draw[thick] (-1.7 cm,0) circle (.3cm);
\node [below] at (-1.7 cm,-.5cm) {$\chi_{\tuple\Lambda_{t-1}}$};
\draw[thick] (-1.4 cm,0) -- +(2.6 cm,0);
\node [above] at (-0.1 cm,.2cm) {$\varphi_{t-1}$};
\draw[thick] (1.5 cm,0) circle (.3cm);
\node [below] at (1.5 cm,-.5cm) {$\chi_{\tuple\Lambda_{t-2}}$};
\draw[dashed,thick] (2 cm,0) -- +(4 cm,0);
\draw[thick] (6.4 cm,0) circle (.3cm);
\node [below] at (6.4 cm,-.5cm) {$\chi_{\tuple\Lambda_{-3t-1}}$};
\draw[thick] (6.7 cm,0) -- +(2.6 cm,0);
\draw[thick,fill=black] (9.9 cm,0) circle (.3cm);
\node [above] at (8.3 cm,.2cm) {$\varphi_{-3t-1}$};
\draw[thick] (9.9 cm,0) circle (.5cm);
\node [below] at (9.9 cm,-.55cm) {$\phi_{\text{exc}}$};
\draw[thick] (10.5 cm,0) -- +(2.6 cm,0);
\node [above] at (11.6 cm,.2cm) {$\varphi_{-3t-2}$};
\draw[thick] (13.4 cm,0) circle (.3cm);
\node [below] at (13.4 cm,-.5cm) {$\chi_{\tuple\Lambda'}$};
\draw[thick] (13.7 cm,0) -- +(2.6 cm,0);
\node [above] at (15.0 cm,.2cm) {$\varphi_{-3t-3}$};
\draw[thick] (16.6 cm,0) circle (.3cm);
\node [below] at (16.6 cm,-.5cm) {$\chi_{\tuple\Upsilon}$};
\draw[thick] (16.9 cm,0) -- +(2.6 cm,0);
\node [above] at (18.2 cm,.2cm) {$\varphi_{-3t-4}$};
\draw[thick] (19.8 cm,0) circle (.3cm);
\node [below] at (19.8 cm,-.5cm) {$\chi_{\tuple\Lambda''}$};
\end{tikzpicture}
\end{center}
where $\tuple\Lambda'=\Theta_{t}\times\Lambda'_-$, $\tuple\Lambda''
=\Theta_{t}\times\Lambda''_-$ and, for each $k \in \{-3t-1,\ldots,t-1\}$,
the symbol $\tuple\Lambda_k=\Theta_{t}\times\Lambda_{k,-}$
of $E_{\tuple\Lambda_k}$ is obtained by adding the $d$-cohook $(k,k+d)$ to $\Theta_{t-1}$.
Explicitly, we have
$$\begin{aligned}
\Lambda_{k,-}\, =& \, \left\{ \begin{array}{cccccccc} t-1 & t-2 & \ldots   &  \ldots   & \ldots & \widehat{k} & \ldots \\
				 & & k+2t+2 &  -t & -t-1 & \ldots & \ldots \\ \end{array}\right\},\\
\Lambda'_- = &\, \left\{\begin{array}{cccccccc} t & t-1 & \ldots  & -t & -t-1 & -t-2 &-t-3 & \ldots  \\
				 & & & & -t& -t-1 &-t-3 &\ldots  \\ \end{array} \right\}~\mbox{and}\\
\Lambda''_- = &\, \left\{\begin{array}{ccccccc} t+2 & t-1 & \ldots  & -t & -t-1 & -t-2 &  \ldots  \\
				 & & & & -t-1& -t-2 & \ldots  \\ \end{array} \right\}.
\end{aligned}$$
Here the notation $\widehat k$ means that the integer $k$ has been removed.
Clearly, for each $k$, the symbol $\Lambda_{k,-}$ has defect
$|2t-3| = |2(t-2)+1|$.
Therefore, 
the quadratic unipotent characters $E_{\tuple\Lambda_k}$ all lie in the
Harish-Chandra series above $E_{t,t-2}\otimes K_\zeta^{2t+2}$ with the convention $E_{s,-1} = E_{s,0}$.

\smallskip

Now we compute $F(E_{\tuple\Lambda'}),$  $F(E_{\tuple\Lambda''}),$ $F(E_{\tuple\Upsilon})$ and $F(E_{\tuple\Lambda_k})$ for $k \in \{-3t-1,\ldots,t-1\}.$
%
By adding  a 1-hook to $\Theta_t$ in all possible ways, we obtain two symbols
$$\begin{aligned}
\Xi_+ = &\, \left\{\begin{array}{ccccccc} t & t-1 & \ldots  & -t & -t-1 & -t-2 &  \ldots  \\
				 & & & & -t& -t-2 & \ldots  \\ \end{array} \right\}~\mbox{and}\\
\Xi_+'= &\, \left\{\begin{array}{ccccccc} t+1 & t-1 & t-2 & \ldots   & -t-1 & \ldots  \\
				 & & & & -t-1 & \ldots  \\ \end{array} \right\}.
\end{aligned}$$
Hence
$$\begin{aligned}
 F(E_{\tuple\Lambda'})\, =&\, E_{\Xi_+\times\Lambda_-'}\oplus  E_{\Xi'_+\times\Lambda_-'},\\
F(E_{\tuple\Lambda''})\, =&\,  E_{\Xi_+\times\Lambda_-''}\oplus  E_{\Xi'_+\times\Lambda_-''},\\
F(E_{\tuple\Upsilon})\,  =&\,  E_{\Xi_+\times\Upsilon_-}\oplus  E_{\Xi'_+\times\Upsilon_-},\\
\end{aligned}$$
and for $k \in \{-3t-1,\ldots,t-1\}$, we have
$$\begin{aligned}
F(E_{\tuple\Lambda_k})\,  =&\,  E_{\Xi_+\times\Lambda_{k,-}}\oplus  E_{\Xi'_+\times\Lambda_{k,-}}.
\end{aligned}$$

\smallskip
We denote $E_{\tuple\Xi_k}:=E_{\Xi_+\times\Lambda_{k,-}}$ and $E_{\tuple\Xi'_k}:=E_{\Xi'_+\times\Lambda_{k,-}}$
Observe that the above constituents belong to two different blocks, say $B'$ and $B''$,
and that they are exactly distributed into two Brauer trees:

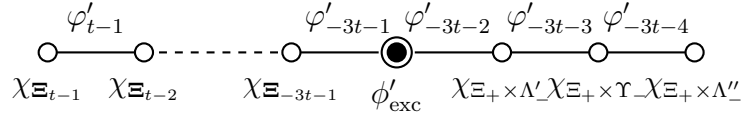
\begin{center}
\begin{tikzpicture}[scale=.4]
\draw[thick] (-1.7 cm,0) circle (.3cm);
\node [below] at (-1.7 cm,-.5cm) {$\chi_{\tuple\Xi_{t-1}}$};
\draw[thick] (-1.4 cm,0) -- +(2.6 cm,0);
\node [above] at (-0.1 cm,.2cm) {$\varphi'_{t-1}$};
\draw[thick] (1.5 cm,0) circle (.3cm);
\node [below] at (1.5 cm,-.5cm) {$\chi_{\tuple\Xi_{t-2}}$};
\draw[dashed,thick] (2 cm,0) -- +(4 cm,0);
\draw[thick] (6.4 cm,0) circle (.3cm);
\node [below] at (6.4 cm,-.5cm) {$\chi_{\tuple\Xi_{-3t-1}}$};
\draw[thick] (6.7 cm,0) -- +(2.6 cm,0);
\draw[thick,fill=black] (9.9 cm,0) circle (.3cm);
\node [above] at (8.3 cm,.2cm) {$\varphi'_{-3t-1}$};
\draw[thick] (9.9 cm,0) circle (.5cm);
\node [below] at (9.9 cm,-.55cm) {$\phi'_{\text{exc}}$};
\draw[thick] (10.5 cm,0) -- +(2.6 cm,0);
\node [above] at (11.6 cm,.2cm) {$\varphi'_{-3t-2}$};
\draw[thick] (13.4 cm,0) circle (.3cm);
\node [below] at (13.4 cm,-.5cm) {$\chi_{\Xi_+\times\Lambda_-'}$};
\draw[thick] (13.7 cm,0) -- +(2.6 cm,0);
\node [above] at (15.0 cm,.2cm) {$\varphi'_{-3t-3}$};
\draw[thick] (16.6 cm,0) circle (.3cm);
\node [below] at (16.6 cm,-.5cm) {$\chi_{\Xi_+\times\Upsilon_-}$};
\draw[thick] (16.9 cm,0) -- +(2.6 cm,0);
\node [above] at (18.2 cm,.2cm) {$\varphi'_{-3t-4}$};
\draw[thick] (19.8 cm,0) circle (.3cm);
\node [below] at (19.8 cm,-.5cm) {$\chi_{\Xi_+\times\Lambda_-''}$};
\end{tikzpicture}
\end{center}
and
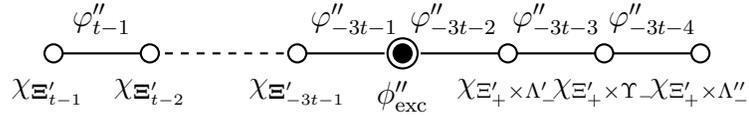
\begin{center}
\begin{tikzpicture}[scale=.4]
\draw[thick] (-1.7 cm,0) circle (.3cm);
\node [below] at (-1.9 cm,-.5cm) {$\chi_{\tuple\Xi'_{t-1}}$};
\draw[thick] (-1.4 cm,0) -- +(2.6 cm,0);
\node [above] at (-0.1 cm,.2cm) {$\varphi''_{t-1}$};
\draw[thick] (1.5 cm,0) circle (.3cm);
\node [below] at (1.5 cm,-.5cm) {$\chi_{\tuple\Xi'_{t-2}}$};
\draw[dashed,thick] (2 cm,0) -- +(4 cm,0);
\draw[thick] (6.4 cm,0) circle (.3cm);
\node [below] at (6.4 cm,-.5cm) {$\chi_{\tuple\Xi'_{-3t-1}}$};
\draw[thick] (6.7 cm,0) -- +(2.6 cm,0);
\draw[thick,fill=black] (9.9 cm,0) circle (.3cm);
\node [above] at (8.3 cm,.2cm) {};
\draw[thick] (9.9 cm,0) circle (.5cm);
\node [above] at (8.3 cm,.2cm) {$\varphi''_{-3t-1}$};
\node [below] at (9.9 cm,-.55cm) {$\phi''_{\text{exc}}$};
\draw[thick] (10.5 cm,0) -- +(2.6 cm,0);
\node [above] at (11.6 cm,.2cm) {$\varphi''_{-3t-2}$};
\node [above] at (11.6 cm,.2cm) {};
\draw[thick] (13.4 cm,0) circle (.3cm);
\node [below] at (13.4 cm,-.5cm) {$\chi_{\Xi'_+\times\Lambda_-'}$};
\draw[thick] (13.7 cm,0) -- +(2.6 cm,0);
\node [above] at (15.0 cm,.2cm) {$\varphi''_{-3t-3}$};
\draw[thick] (16.6 cm,0) circle (.3cm);
\node [below] at (16.6 cm,-.5cm) {$\chi_{\Xi'_+\times\Upsilon_-}$};
\draw[thick] (16.9 cm,0) -- +(2.6 cm,0);
\node [above] at (18.2 cm,.2cm) {$\varphi''_{-3t-4}$};
\draw[thick] (19.8 cm,0) circle (.3cm);
\node [below] at (19.8 cm,-.5cm) {$\chi_{\Xi'_+\times\Lambda_-''}$};
\end{tikzpicture}
\end{center}
Let $b'$ and $b''$ be the idempotents of $B'$ and $B''$ in $K G_{n+2}$, respectively.

\smallskip

We claim that for $i \in \{-3t-4,\ldots,t-1\}$,
$F(E_{\varphi_i})=E_{\varphi'_i}\oplus E_{\varphi''_i}$, where $E_{\varphi'_i}=b'F(E_{\varphi_i})$
and $E_{\varphi''_i}=b''F(E_{\varphi_i})$.
Indeed, the claim is true for $i \in \{-3t-4,t-1\}$
since $\varphi_{t-1}$ and $\varphi_{-3t-4}$
are the $\ell$-reductions of $\chi_{\Xi'_+\times\Lambda_{t-1}}$
and $\chi_{\tuple\Lambda''}$, respectively.
For
the remaining cases, we consider the following Brauer sub-tree:
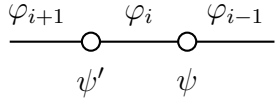
\begin{center}
\begin{tikzpicture}[scale=.4]
\draw[thick] (10.7 cm,0) -- +(2.4 cm,0);
\node [above] at (11.6 cm,.2cm) {$\varphi_{i+1}$};
\draw[thick] (13.4 cm,0) circle (.3cm);
\node [below] at (13.4 cm,-.5cm) {$\psi'$};
\draw[thick] (13.7 cm,0) -- +(2.6 cm,0);
\node [above] at (15.0 cm,.2cm) {$\varphi_i$};
\draw[thick] (16.6 cm,0) circle (.3cm);
\node [below] at (16.6 cm,-.5cm) {$\psi$};
\draw[thick] (16.9 cm,0) -- +(2.6 cm,0);
\node [above] at (18.2 cm,.2cm) {$\varphi_{i-1}$};
\end{tikzpicture}
\end{center}

\smallskip

If $\varphi_i$ is a common irreducible constituent of the $\ell$-reductions of $\psi$ and $\psi'$
(here we assume that $\psi'$ may be exceptional),
then $b'F(E_{\varphi_i})$ (resp. $b''F(E_{\varphi_i})$) is the common constituent of  the $\ell$-reductions of
$b'F(E_{\psi})$ and $b'F(E_{\psi'})$
 (resp. $b''F(E_{\psi})$ and $b''F(E_{\psi'})$).
On one hand, if $b'F(E_{\varphi_i})\neq0$, then $b'F(E_{\varphi_i})$ is irreducible and equal to $\varphi'_i$,
 since otherwise  the $\ell$-reductions of $b'F(E_{\psi})$ and $b'F(E_{\psi'})$ will have
at least two common irreducible constituents,
  contradicting the fact that $b'F(E_{\psi})$ and $b'F(E_{\psi'})$ are in the same Brauer tree.
  On the other hand, $b'F(E_{\varphi_i})$ can not be $0$, since otherwise the $\ell$-reduction of $b'F(E_{\psi})$
will have only one irreducible constituent, which is also a contradiction.
The same argument adapts to $b''F(E_{\varphi_i}).$ So the claim holds.

\smallskip

Now,
 the eigenvalue of  $X(E_{\tuple\Upsilon})$ on $b'F(E_{\tuple\Upsilon})$
 must be modulo $\ell$ congruent to that
  of $X(E_{\tuple \Lambda_{-3t-1}})$ on $b'F(E_{\tuple \Lambda_{-3t-1}})$ since they are in the same Brauer tree.
Also, the eigenvalue of  $X(E_{\tuple\Upsilon})$ on $b'F(E_{\tuple\Upsilon})$
 is an eigenvalue of $X((F')^{2}(E_{t,t}))$ on $F((F')^{2}(E_{t,t}))$
 since $E_{\tuple\Upsilon}$ is a constituent of $(F')^{2}(E_{t,t}).$
Hence
the eigenvalue of $X(E_{\tuple \Lambda_{-3t-1}})$ on $b'F(E_{\tuple \Lambda_{-3t-1}})$
must be modulo $\ell$ congruent to one of the eigenvalues of
$X(E_{t,t})$ on $F(E_{t,t})$, which are $\epsilon_{t}(-q)^{t}$
and $\epsilon_{t}(-q)^{-1-t}.$

\smallskip

Finally, in order to get that $\epsilon_{t}=1$,
we determine the eigenvalues of $X(E_{\tuple \Lambda_{-3t-1}})$ on $b'F(E_{\tuple \Lambda_{-3t-1}})$
and $b''F(E_{\tuple \Lambda_{-3t-1}})$ in terms of those of $X(E_{t,t-2})$.
Note that
$E_{\tuple \Lambda_{-3t-1}}$ belongs to
the Harish-Chandra series above $ E_{t,t-2}\otimes K_\zeta^{2t+2}$.
Also, the eigenvalues of
 $X(E_{\tuple \Lambda_{-3t-1}})$
 on $F(E_{\tuple \Lambda_{-3t-1}})$
are eigenvalues of
  $X((F')^{2t+2}(E_{t,t-2}))$ on $F(F')^{2t+2}(E_{t,t-2})$.
  This is because
$$\xymatrix{F(F')^{2t+2}\ar[r]_{\nabla_{2t+2}}^{\sim}\ar[d]^{X1_{(F')^{2t+2}}} &(F')^{2t+2}F \ar[d]^{1_{(F')^{2t+2}}X}\\
 F(F')^{2t+2}\ar[r]_{\nabla_{2t+2}}^{\sim}& (F')^{2t+2}F.
}$$
is a commutative diagram and
$$\nabla_{2t+2}=1_{F^{(2t+1)}}H\circ1_{F^{(2t)}}H1_{F}
\circ\cdots\circ 1_{F^2}H1_{(F')^{(2t-1)}}1_{F}
\circ H1_{(F')^{(2t)}}\circ H1_{(F')^{(2t+1)}}$$
is an isomorphism between functors $F(F')^{2t+2}$ and $(F')^{2t+2}F$.
 Therefore, we know that
  the eigenvalues of $X((F')^{2t+2}(E_{t,t-2}))$ on $F((F')^{2t+2}(E_{t,t-2}))$
 are the same as those of $X(E_{t,t-2})$ on $F(E_{t,t-2}).$
By the inductive hypothesis on the eigenvalues of $X(E_{t,t-2})$  on $F(E_{t,t-2})$,
we conclude that the eigenvalues of
$X(E_{\tuple \Lambda_{-3t-1}})$  on $F(E_{\tuple \Lambda_{-3t-1}})$ are $(-q)^{t}$ and $(-q)^{-1-t}$.
Note that any two of $(-q)^{-1-t}, -(-q)^{-1-t}, (-q)^{t}$ and $-(-q)^{t}$ are not modulo $\ell$ congruent
since $q^{2t+3} \equiv -q$.  By the above argument, at least one of $(-q)^{t}$ and $(-q)^{-1-t}$ must be congruent to
$\epsilon_{t} (-q)^{-1-t}$ or $\epsilon_{t} (-q)^{t}$.
Thus we deduce that $\epsilon_{t} =1$.

\smallskip

To prove $\epsilon'_{t}=1$, we argue similarly as above, by considering the action of
$X'(E_{t,t})$ on $F'(E_{t,t})$. This finishes the proof.
\end{proof}

\subsection{Categorical action on $\scrQU_K$}\label{sec:QU}
In this section, we shall show that the representation datum on the category  $\scrQU_K$ leads to
a categorical action on it. 

\subsubsection{Quiver Hecke algebras of disconnected quiver}\label{subsec:disjointunion}
We will need the structure of quiver Hecke algebras of disconnected quivers,
which is determined by Rostam \cite{ROS17} (see also \cite{PR21}).

\smallskip

Let $\K$ be a (not necessarily finite) set with a partition $\K = \K_+\sqcup\K_-$.
For ${\tuple\nu}\in \mathbb{J}^m$, we define
 $$\K^{\tuple\nu}=\{\tuple k\in \K^m\mid  k_i \in \K_{{\nu}_i}~\text{for all}~1\leqs i\leqs m\}.$$
In particular,
 $\K^{{\tuple\nu}_0^{\tuple m}} \simeq \K_+^{m_+} \times  \K_-^{m_-}$ for  $\tuple m=(m_+,m_-) \comp_2 \,m$.

 \smallskip

Suppose that $\Gamma_+$ and $ \Gamma_-$ are the loop-free quivers with vertex sets $\K_+$ and $\K_-$, respectively.
Let $\Gamma$ be the quiver that is the disjoint union of $\Gamma_+$ and $ \Gamma_-$.
Then the Cartan datum associated to the
 quiver $\Gamma$ is a direct sum of the
 Cartan data associated to the  quivers $\Gamma_+$ and $\Gamma_-$,
  and the Kac-Moody Lie algebra corresponding to  the  quiver $\Gamma$
  is $\frakg_{\Gamma}=\frakg_{\Gamma_+}\oplus \frakg_{\Gamma_-}.$
As in Remark \ref{rmk:quiver}, let $Q$, $Q^+$ and $Q^-$ be the matrices associated with the quivers $\Gamma$, $\Gamma_+$ and $ \Gamma_-$,
respectively, so that $Q_{i, j}=Q^{+}_{i,j}$, $Q_{i', j'}=Q^{-}_{i',j'}$ and $Q_{k, k'} =Q_{k', k} =1$ for any $i,j, k\in \K_+ $ and $i',j',k'\in \K_{-}$.
Let $\bfH_{m}:=\bfH_{m}(Q)$ be as  in \textsection\ref{subsec:QHA}.
For ${\tuple\nu} \in \mathbb{J}_{\tuple m}$, we define an idempotent
$$e({\tuple\nu}) := \sum_{\tuple{k} \in \K^{\tuple\nu}} e(\tuple{k}),$$
where $e(\tuple{k})$ is as defined in \S \ref{def:KLRalg}.
Then $$e(\tuple m) \coloneqq \sum_{{\tuple\nu} \in \mathbb{J}_{\tuple m}} e({\tuple\nu})$$
is a central idempotent in $\bfH_{m}$ and the sum is a decomposition of orthogonal idempotents. Also, we have the following decomposition
of $\bfH_{m}$ into subalgebras:
$$\bfH_{m} = \bigoplus_{\tuple m \comp_2 \,m} e(\tuple m)\bfH_{m}.$$
\smallskip

The algebra $e(\tuple m)\bfH_{m}$ has the structure of a matrix algebra.
To describe it, let $\tuple m=(m_+,m_-)\comp_2 \,m$.
We write $e_{\tuple m} \coloneqq e({\tuple\nu}_0^{\tuple m})$ for short,
and note that the subalgebra $e_{\tuple m} \bfH_{m} e_{\tuple m}$
(with unit $e_{\tuple m}$) of $e({\tuple m})\bfH_{m}$
is isomorphic to the algebra
$$ \bfH_{m_+}(Q^+) \otimes \bfH_{m_-}(Q^-).$$
The isomorphism can be stated as follows. First, reindex the generators
$\tau_1, \dots, \tau_{m_- - 1}$ and $x_1, \dots, x_{m_-}$ of $\bfH_{m_-}(Q^-)$
by $\tau_{m_++1}, \dots, \tau_{m-1}$ and $x_{m_+ + 1},\dots, x_{m}$, respectively.
For $w = (w_+,  w_-) \in \mathfrak{S}_{m_+}\times\mathfrak{S}_{m_-}$ and
$\tuple{k} = (\tuple{k}^+, \tuple{k}^-) \in \K_+^{m_+} \times \K_-^{m_-},$
 we may set
$$\tau_w^{\otimes}\coloneqq \tau_{w_+} \otimes\tau_{w_-} \in  \bfH_{m_+}(Q^+) \otimes \bfH_{m_-}(Q^-)$$
and
\begin{equation*}
\label{equation:e_otimes}
e^{\otimes}(\tuple{k}) \coloneqq e(\tuple{k}^+)\otimes e(\tuple{k}^-) \in \bfH_{m_+}(Q^+) \otimes \bfH_{m_-}(Q^-).
\end{equation*}
Note from \cite[Lemma 6.16]{ROS17} that $\tau_{\pi_{{\tuple\nu}}} e({\tuple\nu}):=\tau_{a_1} \cdots \tau_{a_r} e({\tuple\nu})$
is well-defined for ${\tuple\nu} \in \mathbb{J}^m$ and a reduced expression $s_{a_1} \cdots s_{a_r}$ of $\pi_{{\tuple\nu}}$.
Then the map, which sends
\begin{itemize}[leftmargin=8mm]
\item the generators $\tau_a^{\otimes} \in  \bfH_{m_+}(Q^+) \otimes \bfH_{m_-}(Q^-)$ for $a \in \{1, \dots, m-1 \} \setminus \{m_+\}$ to $\tau_a e_{\tuple m} \in e_{\tuple m}\bfH_{m}e_{\tuple m}$,
\item the generators $x_b \in  \bfH_{m_+}(Q^+) \otimes \bfH_{m_-}(Q^-)$ for $b \in \{1, \dots, m\}$ to $x_b e_{\tuple m} \in e_{\tuple m}\bfH_{m}e_{\tuple m}$, and
\item the generators $e^{\otimes}(\tuple{k}) \in  \bfH_{m_+}(Q^+) \otimes \bfH_{m_-}(Q^-)$ for $\tuple{k} \in \K^{{\tuple\nu}_0^{\tuple m}}$ to $e(\tuple{k}) \in e_{\tuple m}\bfH_{m}e_{\tuple m},$
\end{itemize}
is an algebra isomorphism from $ \bfH_{m_+}(Q^+) \otimes \bfH_{m_-}(Q^-)$ to
 $e_{\tuple m} \bfH_{m} e_{\tuple m}$ (see \cite[Proposition 6.25]{ROS17}).

\smallskip

By the above isomorphism, we may
identify $ \bfH_{m_+}(Q^+) \otimes \bfH_{m_-}(Q^-)$ with $e_{\tuple m} \bfH_{m} e_{\tuple m}$, and
so
$$e(\tuple m)\bfH_{m}\cong \mathrm{Mat}_{|\mathbb{J}_{\tuple m}|} ( \bfH_{m_+}(Q^+) \otimes \bfH_{m_-}(Q^-)).$$
We now label the rows and the columns of the elements of
 $$\mathrm{Mat}_{|\mathbb{J}_{\tuple m}|} ( \bfH_{m_+}(Q^+) \otimes \bfH_{m_-}(Q^-))$$ by
 $({\tuple\nu}', {\tuple\nu}) \in ({\mathbb{J}_{\tuple m}})^2$,
  and write $E_{{\tuple\nu}', {\tuple\nu}}$
  for the elementary matrix with one $1$ at position $({\tuple\nu}', {\tuple\nu})$ and $0$ elsewhere.
Then for $({\tuple\nu}', {\tuple\nu}) \in ({\mathbb{J}_{\tuple m}})^2$, we have the following
$R$-module isomorphism
\[
e({\tuple\nu}') \bfH_{m} e({\tuple\nu}) \simeq(\bfH_{m_+}(Q^+) \otimes \bfH_{m_-}(Q^-)) E_{{\tuple\nu}', {\tuple\nu}}.
\]
Indeed, if we define

\begin{equation}\label{equation:definition_philambda_psilambda2}
\begin{gathered}
\Phi_{{\tuple\nu}', {\tuple\nu}} :(\bfH_{m_+}(Q^+) \otimes \bfH_{m_-}(Q^-)) E_{{\tuple\nu}', {\tuple\nu}} \to e({\tuple\nu}') \bfH_{m} e({\tuple\nu}),
\\
\Psi_{{\tuple\nu}', {\tuple\nu}} : e({\tuple\nu}') \bfH_{m} e({\tuple\nu}) \to (\bfH_{m_+}(Q^+) \otimes \bfH_{m_-}(Q^-)) E_{{\tuple\nu}', {\tuple\nu}}
\end{gathered}
\end{equation}
by
$$
\begin{array}{cccl}
\Phi_{{\tuple\nu}', {\tuple\nu}}(v E_{{\tuple\nu}', {\tuple\nu}}) &\coloneqq & \tau_{\pi_{{\tuple\nu}'}^{-1}} v \tau_{\pi_{\tuple\nu}}
 &\mbox{for~} v \in\bfH_{m_+}(Q^+) \otimes \bfH_{m_-}(Q^-), \\
\Psi_{{\tuple\nu}', {\tuple\nu}}(w) &\coloneqq & (\tau_{\pi_{{\tuple\nu}'}} w \tau_{\pi_{\tuple\nu}^{-1}}) E_{{\tuple\nu}', {\tuple\nu}}
&\mbox{for~} w \in e({\tuple\nu}')\bfH_{n} e({\tuple\nu}),
\end{array}
$$
then these two maps $\Phi_{\tuple\nu', \tuple \nu}$ and $\Psi_{\tuple\nu', \tuple\nu}$ are inverse isomorphisms.
We now set
\begin{equation}\label{quiverconstruction}
\begin{gathered}
\Phi_{\tuple m} \coloneqq \bigoplus_{({\tuple\nu}', {\tuple\nu}) \in ({\mathbb{J}_{\tuple m}})^2} \Phi_{\tuple \nu', \tuple \nu} : \mathrm{Mat}_{|\mathbb{J}_{\tuple m}|} (\bfH_{m_+}(Q^+) \otimes \bfH_{m_-}(Q^-)) \to e(\tuple m) \bfH_m,
\\
\Psi_{\tuple m} \coloneqq \bigoplus_{({\tuple\nu}', {\tuple\nu}) \in ({\mathbb{J}_{\tuple m}})^2} \Psi_{\tuple \nu', \tuple \nu} : e(\tuple m) \bfH_m \to \mathrm{Mat}_{|\mathbb{J}_{\tuple m}|}(\bfH_{m_+}(Q^+) \otimes \bfH_{m_-}(Q^-)).
\end{gathered}
\end{equation}
By the above properties of $\Phi_{{\tuple\nu}', {\tuple\nu}}$ and $\Psi_{{\tuple\nu}, {\tuple\nu}'}$,
we know that $\Phi_{\tuple m}$ and $\Psi_{\tuple m}$ are inverse $R$-module
isomorphisms.
Therefore, the following is clear. 
\begin{theorem}\cite{ROS17}
\label{theorem:disjoint_guiver}
We have an algebra isomorphism:
\[
\bfH_{m} \simeq \bigoplus_{\tuple m= (m_+,m_-) \comp_2\ m} \Mat_{|\mathbb{J}_{\tuple m}|}(\bfH_{m_+}(Q^+) \otimes \bfH_{m_-}(Q^-)).
\]
\end{theorem}
We define
$\bfH^{\Lambda_+}_{m_+}(Q^+) \otimes \bfH^{\Lambda_-}_{m_-}(Q^-)$ to be
the tensor product of two cyclotomic quiver Hecke algebras $\bfH^{\Lambda_+}_{m_+}(Q^+)$ and $\bfH^{\Lambda_-}_{m_-}(Q^-)$,
where $\Lambda_+ \in \X_{\K_+}^{+}$
and $\Lambda_- \in \X_{\K_-}^{+}$ are integral dominant weights of $\frakg_{\Gamma_+}$ and $\frakg_{\Gamma_-}$, respectively.

\begin{theorem}\cite{ROS17}
\label{theorem:disjoint_guiver_cyclotomic}
The isomorphism in Theorem~\ref{theorem:disjoint_guiver} factors through the cyclotomic quotients, in other words, we have
$$
\bfH_m^{\Lambda}(Q) \simeq \bigoplus_{\tuple m= (m_+,m_-) \comp_2\ m} \mathrm{Mat}_{|\mathbb{J}_{\tuple m}|} (\bfH^{\Lambda_+}_{m_+}(Q^+) \otimes \bfH^{\Lambda_-}_{m_-}(Q^-)),$$
where $\Lambda= \Lambda^+\times\{0\}+\{0\}\times \Lambda^-\in  \X_{\K}^{+}$ is an integral dominant weight of $\frakg_{\Gamma}$.
\end{theorem}

\subsubsection{$\mathfrak{A}(\frakg_{\infty})$-representation on $\scrQU_K$}\label{subsec:QUK}
According to the representation datum on $\scrQU_K$
 constructed in \S \ref{sub:explicit-repdatumSO},
we know from \cite[Lemma 4.7]{Go} and the proof of Theorem \ref{thm:HL-BC} that
 the generalized eigenvalues of both $X$ on $F$ and $X'$ on $F'$ are $q^{\,\bbZ} \sqcup -q^{\,\bbZ}$.
\begin{definition}\label{def:infty}
Let $\scrI_\infty$ and $\scrI'_\infty$ both denote the subset $q^{\,\bbZ} \sqcup -q^{\,\bbZ}$ of $K^\times.$
Denote by $\K_\infty(q)=\scrI_\infty(q)\sqcup \scrI'_\infty(q)$
the quiver that is the disjoint union of the quivers $\scrI_\infty(q)$ and  $\scrI'_\infty(q)$.
We define $\frakg_\infty$ to be the (derived) Kac-Moody algebra associated to the quiver $\K_\infty(q)$,
so that
$\frakg_\infty=\fraks\frakl_{I_\infty}\oplus\fraks\frakl_{I'_\infty}.$ \qed
\end{definition}

\smallskip

For brevity, we will write
$$I_\infty=I_\infty(q), I'_\infty=I'_\infty(q)~\mbox{and}~\K_\infty=\K_\infty(q),$$ and
write
$$
\begin{array}{rcl}
 \{\Lambda_i\times\{0\}\}_{i\in I_\infty} & \sqcup & \{\{0\}\times\Lambda_{i'}\}_{i'\in I_\infty'}, \\
  \{\alpha_i\times\{0\}\}_{i\in I_\infty} & \sqcup & \{\{0\}\times\alpha_{i'}\}_{i'\in I_\infty'}, \\
 \{\alpha_i^\vee\times\{0\}\}_{i\in I_\infty} & \sqcup & \{\{0\}\times\alpha_{i'}^\vee\}_{i'\in I_\infty'}
\end{array}
$$
for
the fundamental weights, simple roots and simple coroots of $\frakg_\infty$,
respectively. If there is no risk of confusion, we shall
write $\Lambda_i$ instead of $\Lambda_i\times\{0\}$,
and $\Lambda'_{i'}$ instead of $\{0\}\times\Lambda_{i'}.$
Then the weight lattice of $\frakg_\infty$ is
$$\P_{\infty} =\big( \bigoplus_{i\in I_\infty} \bbZ \Lambda_i\big)\oplus\big(\bigoplus_{i'\in I_\infty'} \bbZ \Lambda'_{i'}\big),$$
and there is a Lie algebra isomorphism $$(\fraks\frakl_\bbZ)^{\oplus 4} \simto \frakg_{\infty}$$
such that
$$\begin{array}{rclcrcl}
  (\alpha^\vee_k,0,0,0) & \mapsto & \alpha^\vee_{q^{k}}\times\{0\}, & \quad & (0,\alpha^\vee_k,0,0) &  \mapsto & \alpha^\vee_{-q^{k}}\times\{0\}, \\
  (0,0,\alpha^\vee_k,0) & \mapsto &\{0\}\times\alpha^\vee_{q^{k}}, & \quad & (0,0,0,\alpha^\vee_k) & \mapsto &\{0\} \times\alpha^\vee_{-q^{k}}.
\end{array}
$$

Suppose that
$\big(\{E_s\}_{s\in \K_\infty},\{F_s\}_{s\in \K_\infty},\{x_s\}_{s\in \K_\infty},\{\tau_{s,t}\}_{s,t\in \K_\infty}\big)$
(or equivalently, $(E,F,$ $X,T;E',F',X',T';H,H')$) is the
$\mathfrak{A}(\frakg_{\infty})$-representation datum
 on $\scrQU_K$ constructed in \S \ref{sub:explicit-repdatumSO}.
%
As before, let $E_{t_+,t_-}$ be a quadratic unipotent cuspidal module of $G_r$ with $r=r_++r_-$
and $r_\pm=t_\pm(t_\pm+1)$. Assume that $n=m+r$.
If we write $\bfH_m(Q)$  for the quiver Hecke algebra associated with the quiver $\K_\infty$, then
the map
\begin{equation}
\begin{aligned}
\phi_m\, :\bfH_m(Q)&\rightarrow (\bigoplus_{\tuple k,\tuple k'\in \K^m} \Hom(F_{k_m}\cdots
F_{k_1},F_{k'_m}\cdots F_{k'_1}))^\op\cong\End\big(\bigoplus\limits_{\tuple\nu \in \mathbb{J}^m}F^{\tuple\nu}\big)^\op\\
e(\tuple k)&\mapsto 1_{F_{k_m}\cdots F_{k_1}} \\
e( \tuple\nu)&\mapsto 1_{F^{\nu_m}\cdots F^{\nu_1}} \\
x_{a,\tuple k}&\mapsto F_{k_m}\cdots F_{k_{a+1}}x_{k_a}F_{k_{a-1}}\cdots F_{k_1}\\
\tau_{a,\tuple k}&\mapsto F_{k_m}\cdots F_{k_{a+2}}\tau_{k_{a+1},k_a}
F_{k_{a-1}}\cdots F_{k_1}
\end{aligned}
\end{equation}
is an algebra homomorphism.
Evaluating of $\phi_m$ at $E_{t_+,t_-}$, we get the following algebra homomorphism
$$\phi_m(E_{t_+,t_-})\, :\bfH_m(Q)\rightarrow \End_{KG}\big(\bigoplus\limits_{\tuple k \in \K^m}F_{\tuple k}(E_{t_+,t_-})\big)^\op\cong \End_{KG}\big(\bigoplus\limits_{ \tuple\nu \in \mathbb{J}^m}F^{\tuple\nu}(E_{t_+,t_-})\big)^\op$$
 with restrictions
$$\begin{array}{rccl}
 \phi_m(E_{t_+,t_-})\, : & e(\tuple k)\bfH_m(Q)e(\tuple k) & \rightarrow  &  \End_{KG}(F_{\tuple k}(E_{t_+,t_-}))^\op,\\
 \phi_m(E_{t_+,t_-})\, : & e(\tuple\nu)\bfH_m(Q)e(\tuple\nu) & \rightarrow  &  \End_{KG}(F^{\tuple\nu}(E_{t_+,t_-}))^\op, \,\text{and}
 \\
 \phi_m(E_{t_+,t_-})\, : & e(\tuple m)\bfH_m(Q) & \rightarrow  &
                 \End_{KG}\big(\bigoplus\limits_{\tuple\nu \in \mathbb{J}_{\tuple m}}F^{\tuple\nu}(E_{t_+,t_-})\big)^\op.
\end{array}
$$
In particular, when we fix $\tuple m=(m_+,m_-)\comp_2 m$,
the restriction of $\phi_m(E_{t_+,t_-})$ to $e_{\tuple m}\bfH_m(Q)e_{\tuple m}$
is exactly the algebra homomorphism $\phi_{F^{\tuple\nu^{\tuple m}_0}}(E_{t_+,t_-})$ as  defined
 in  \S \ref{subsec:ramiHecke}.

\begin{proposition}\label{keylemma}
The element $\tau_{\pi_{\tuple\nu}}e(\tuple\nu)$ constructed in \S \ref{subsec:disjointunion} under the above homomorphism is exactly the natural transformation $T_{\pi_{\tuple\nu},\tuple\nu}.$ So we have the following commutative diagram:
 $$\xymatrix{e(\tuple m)\bfH_m(Q)\ar[d]^{\phi_m(E_{t_+,t_-})}\ar[r]_{\Psi_{\tuple m}\,\,\,\,\,\,\,\,\,\,\,\,\,\,\,}
 &\Mat_{|\mathbb{J}_{\tuple m}|}(\bfH_{m_+}(Q^+)\otimes \bfH_{m_-}(Q^-)) \ar[d]^{\mathrm{Mat}(\phi_{F^{\tuple \nu_0^{\tuple m}}})}\\
  \End\big(\bigoplus\limits_{\tuple\nu \in \mathbb{J}_{\tuple m}}F^{\tuple\nu}(E_{t_+,t_-})\big)^\op\ar[r]_{\Psi_{\tuple m}\,\,\,\,\,\,\,\,\,\,\,\,\,\,\,} &  \Mat_{|\mathbb{J}_{\tuple m}|}(\End_{KG}(F^{\tuple \nu_0^{\tuple m}}(E_{t_+,t_-})))^\op.
}$$
\end{proposition}
\begin{proof}
The first assertion immediately follows by their definition and Lemma \ref{lem:pinu},
and the second follows by the first  assertion and
the construction of $\Psi_{\tuple m}$ in \eqref{functorconstruction} and \eqref{quiverconstruction}.
\end{proof}

The following result, which is one of our main theorems of this section,
shows that the direct sum of certain Howlett-Lehrer algebras is indeed a cyclotomic quiver Hecke algebra.

\begin{theorem}\label{thm:F^nu}
 The above homomorphism $\phi_m(E_{t_+,t_-})$ factors through an isomorphism
$$\bfH^{\Lambda_{t_+,t_-}}_m(Q)\cong \End_{KG}\big(\bigoplus\limits_{\tuple\nu \in \mathbb{J}^m}F^{\tuple\nu}(E_{t_+,t_-})\big)^\op$$
where $\Lambda_{t_+,t_-}=\Lambda_{(\tuple\xi_{t_+},\tuple\xi_{t_-})}=
\Lambda_{(-q)^{t_+}}+\Lambda_{(-q)^{-1-t_+}}+\Lambda'_{{(-q)^{t_-}}}+
\Lambda'_{{(-q)^{-1-t_-}}}.$
\end{theorem}
\begin{proof} Let $\tuple m$ and $\tuple m'$ be different 2-composition of $m$.
 If $\tuple\nu\in \mathbb{J}_{\tuple m}$ and  $\tuple\nu'\in \mathbb{J}_{\tuple m'}$,
 then $F^{\tuple\nu}(E_{t_+,t_-})$ and $F^{\tuple\nu'}(E_{t_+,t_-})$ are in different Lusztig series,
 and so $F^{\tuple\nu}(E_{t_+,t_-})$
 and $F^{\tuple\nu'}(E_{t_+,t_-})$
 have no common constituents. In particular,
 $$\Hom_{KG}(F^{\tuple\nu}(E_{t_+,t_-}),F^{\tuple\nu'}(E_{t_+,t_-}))=0.$$
Hence we have
  $$\End _{KG}(\bigoplus_{\tuple\nu \in \mathbb{J}^m}F^{\tuple\nu}(E_{t_+,t_-}))
  \cong \bigoplus\limits_{\tuple m\comp_2\,\, m }\big(\End_{KG} \big(\bigoplus_{\tuple\nu \in \mathbb{J}_{\tuple m}}F^{\tuple\nu}(E_{t_+,t_-})\big)\big).$$
By Proposition \ref{keylemma}, in order to prove the theorem,
it suffices to show that $\phi_m(E_{t_+,t_-})$ restricts to an isomorphism
  $$\bfH^{\Lambda_{t_+}}_{m_+}(Q^+)\otimes\bfH^{\Lambda_{t_-}}_{m_-}(Q^-)\xrightarrow{\sim} \End_{KG}((F')^{m_-}F^{m_+}(E_{t_+,t_-}))^\op.$$
However, by Brundan-Kleshchev-Rouquier equivalence (see \S \ref{subsec:BKR}),
the left side is isomorphic to
$\bfH^{q\,;\,\tuple\xi_{t_+}}_{K,m_+}
\otimes\bfH^{q\,;\,\tuple\xi_{t_-}}_{K,m_-}$, which is isomorphic to the right side by Theorem \ref{thm:HL-BC}.
\end{proof}

\smallskip

Recall that the minimal categorical representation $$\scrL(\Lambda):=\bigoplus_{\beta\in Q^+}\bfH_{\beta}^\Lambda(Q)\mod$$
 provides an action of $\frakA(\frakg_{\infty}).$
 Now we prove that $\scrQU_K$ is isomorphic to a direct sum of minimal categorical representations.

 \smallskip

Let $(KG_n,E_{t_+,t_-})\mod$ be the Serre subcategory of $\scrQU_K$ generated by the modules
$F^{\tuple\nu}(E_{t_+,t_-})$  for all $\tuple \nu\in \mathbb{J}^m$.
It can actually be generated by
 the modules $(F')^{m_-}F^{m_+}(E_{t_+,t_-})$ satisfying $m_++m_-=m$, since
 for a given 2-composition  $\tuple m=(m_+,m_-)\comp_2 m$ of $m$,
Proposition \ref{prop:isom} shows that all functors $F^{\tuple\nu}$ with $\tuple \nu\in \mathbb{J}_{\tuple m}$
 are isomorphic to $(F')^{m_-}F^{m_+}$. In particular,
$$\Irr((KG_n,E_{t_+,t_-})\mod)=\bigsqcup_{(m_+,m_-)\comp_2 m}\Irr(KG_n,E_{t_+,t_-}\otimes K_1^{m_+}\otimes K_\zeta^{m_-})$$
 We define $$\scrQU_{K,t_+,t_-}:=\bigoplus_{n\geqslant 0}(KG_n,E_{t_+,t_-})\mod,$$
and so
$$\scrQU_K=\bigoplus_{{t_+,t_-}\geqslant 0}\scrQU_{K,t_+,t_-}.$$
If $t_-=0$, then $\scrQU_{K,t_+,t_-}$ becomes the category $\scrU_{K,t_+}$ defined in \cite{DVV2}.

\smallskip


 The following result implies Theorem B.

\begin{theorem}\label{thm:ginfinityB}
 Let $t_+, t_- \geqslant 0$ and $\tuple\xi_{t_+}$ and $\tuple\xi_{t_-}$ be as in \eqref{parameter}.\begin{itemize}[leftmargin=8mm]
  \item[$\mathrm{(a)}$]
  The  functors $F_s,E_s$ for $s\in\scrI_{\infty}\sqcup\scrI'_\infty$ yield a representation of
  $\frakA(\frakg_{\infty})$ on $\scrQU_{K,t_+,t_-}$ which is isomorphic to $\scrL(\Lambda_{t_+,t_-})_\infty$.
  \item[$\mathrm{(b)}$]
  The map
  $$|\tuple\mu_+,\tuple\xi_{t_+}\rangle_\infty
  \otimes |\tuple\mu_-,\tuple\xi_{t_-}\rangle_\infty\mapsto [E_{\Theta_{t_+}(\tuple\mu_+),
  \Theta_{t_-}(\tuple\mu_-)}]$$ gives an
  $\frakg_\infty$-module isomorphism $$
  \bfF({\tuple\xi_{t_+}})_{\infty}\otimes
  \bfF({\tuple\xi_{t_-}})_{\infty} \simto [\scrQU_{K,t_+,t_-}].$$
  \end{itemize}
\end{theorem}

\begin{proof} We first mention that the
$\frakg_\infty$-module $\bfF(\tuple\xi_{t_+})_\infty \otimes \bfF(\tuple\xi_{t_-})_\infty$ in (b)
is well-defined, since the pair $(\tuple\xi_{t_+},\tuple\xi_{t_-})$ belongs to $(\scrI_\infty)^2\times(\scrI'_\infty)^2$.

\smallskip

Note that $\scrQU_{K,t_+,t_-}$
is stable by the functors $F$ and $F'$. Also, by the Mackey formula and \cite[Proposition~2.2]{GHJ}
$\scrQU_{K,t_+,t_-}$ is  stable by their adjoint functors $E$ and $E'$. Hence
the representation datum
$(E,F,X,T;E',F',X',T';H,H')$ on $\scrQU_{K}$ constructed in \S \ref{sub:explicit-repdatumSO}
 restricts to a representation datum on $\scrQU_{K,t_+,t_-}$.
Write $\bfH^{\Lambda_{t_+,t_-}}_{m}:=\bfH^{\Lambda_{t_+,t_-}}_{m}(Q)$ and the functor
$$\frakE_{m}:=(\bigoplus\limits _{\tuple\nu\in  \mathbb{J}^m}F^{\tuple\nu}(E_{t_+,t_-}))
\otimes_{\bfH^{\Lambda_{t_+,t_-}}_{m}}-.$$
We have an equivalence of semisimple abelian $K$-categories
$$\frakE:=\bigoplus\limits _{m\in\bbZ}\frakE_{m}\, :\,
\scrL(\Lambda_{{t_+},{t_-}})_\infty \simto \scrQU_{K,t_+,t_-}.$$
We shall show that it is actually an isomorphism of $\frakg_{\infty}$-representations.

\smallskip

We first claim that there are isomorphisms of functors
$$\Phi^{\pm}:\frakE F^{\pm}\xrightarrow{\sim} F^{\pm}\frakE
~\mbox{and}~ \Psi^{\pm}:\frakE E^{\pm}\xrightarrow{\sim} E^{\pm}\frakE$$
between
$\scrL(\Lambda_{t_+,t_-})_{\infty}$ and $\scrQU_{K,t_+,t_-}$.
In fact,
the functor
  $$F^{\pm}\frakE_{m}\ :\,\bfH_m^{\Lambda_{t_+,t_-}}\mod\longrightarrow(KG_{n+1},E_{t_+,t_-})\mod$$
is obtained by tensoring
with the $(KG_{n+1},\bfH_m^{\Lambda_{t_+,t_-}})$-bimodule
$F^{\pm}\big(\bigoplus\limits_{\tuple\nu\in  \mathbb{J}^m}F^{\tuple\nu}(E_{t_+,t_-})\big).$
On the other hand, the functor
  $$\frakE_{m}F^{\pm}\ :\,\bfH_m^{\Lambda_{t_+,t_-}}\mod\longrightarrow(KG_{n+1},E_{t_+,t_-})\mod$$
is obtained by tensoring with the $(KG_{n+1},\bfH^{\Lambda_{t_+,t_-}}_{m})$-bimodule
$$\big(\bigoplus_{\tuple\nu\in  \mathbb{J}^{m+1}}F^{\tuple\nu}(E_{t_+,t_-})\big)\otimes_{\bfH^{\Lambda_{t_+,t_-}}_{m+1}}\bfH^{\Lambda_{t_+,t_-}}_{m+1}(\sum_{\tuple\nu\in \mathbb{J}^m}e(\tuple \nu,\pm))\otimes_{\bfH^{\Lambda_{t_+,t_-}}_{m}}\bfH^{\Lambda_{t_+,t_-}}_{m}$$
which is isomorphic to $F^{\pm}\big(\bigoplus\limits_{\tuple\nu\in  \mathbb{J}^m}F^{\tuple\nu}(E_{t_+,t_-})\big)$ since
 the idempotent $\sum _{\tuple\nu\in \mathbb{J}^m}e(\tuple\nu,\pm)$ kills the summand $F^{\mp}\big(\bigoplus\limits_{\tuple\nu\in  \mathbb{J}^m}F^{\tuple\nu}(E_{t_+,t_-})\big)$ in $\big(\bigoplus_{\tuple\nu\in  \mathbb{J}^{m+1}}F^{\tuple\nu}(E_{t_+,t_-})\big),$
 where $$e(\tuple\nu,\,+):=\sum\limits_{\tuple k\in \K^{\tuple\nu},i\in I} e(\tuple k, i)~\mbox{and}~
e(\tuple\nu,-):=\sum\limits_{\tuple k\in \K^{\tuple\nu},i'\in I'} e(\tuple k, i').$$
  More precisely,
  the left actions of $KG_{n+1}$
are same in both cases,
  while the
right action of $\bfH^{\Lambda_{t_+,t_-}}_{m}$
comes from the right action of
$\bfH^{\Lambda_{t_+,t_-}}_m$ on $\big(\bigoplus\limits_{\tuple\nu\in  \mathbb{J}^m}F^{\tuple\nu}(E_{t_+,t_-})\big)$
 and the functoriality
of $F$ in the first case,
and from the right action of
 $\bfH^{\Lambda_{t_+,t_-}}_{m+1}$ on $\big(\bigoplus\limits_{\tuple\nu\in  \mathbb{J}^{m+1}}F^{\tuple\nu}(E_{t_+,t_-})\big)$
 and the
inclusion  $\bfH^{\Lambda_{t_+,t_-}}_m \subset\bfH^{\Lambda_{t_+,t_-}}_{m+1}$ in the second case.
We denote the above $(KG_{n+1},\bfH^{\Lambda_{t_+,t_-}}_{m})$-bimodule isomorphism by $\Phi^{\pm},$
so we get isomorphisms of the functors
$\Phi^{\pm}:\frakE F^{\pm}\xrightarrow{\sim} F^{\pm}\frakE.$
Now the isomorphisms $\Psi^{\pm}:\frakE E^{\pm}\xrightarrow{\sim} E^{\pm}\frakE$ follow by adjunction.

\smallskip

Next, we prove that
\begin{itemize}[leftmargin=8mm]
  \item[(i)] the isomorphisms  $\Phi^{\pm}:\frakE F^{\pm}\xrightarrow{\sim} F^{\pm}\frakE$
   intertwine the endomorphisms $\frakE X^{\pm}$ and $X^{\pm}\frakE,$ and
  \item[(ii)] the isomorphisms  $ \Phi^{\alpha,\beta}:=F^{\alpha} \Phi^{\beta}\circ\Phi^{\alpha} F^{\beta}:\frakE F^{\alpha}F^{\beta}\xrightarrow{\sim} F^{\alpha}F^{\beta}\frakE$ and
      $ \Phi^{\beta,\alpha}:=F^{\beta} \Phi^{\alpha}\circ\Phi^{\beta} F^{\alpha}:\frakE F^{\beta}F^{\alpha}\xrightarrow{\sim} F^{\beta}F^{\alpha}\frakE$ satisfy the commutative diagram:
      $$\xymatrix{\frakE F^{\beta}F^{\alpha}\ar[d]^{\frakE T^{ \alpha\beta}}
\ar[r]_{\sim}^{\Phi^{\beta,\alpha}} & F^{\beta}F^{\alpha}\frakE \ar[d]^{T^{\alpha\beta}\frakE}\\
 \frakE F^{\alpha}F^{\beta}\ar[r]_{\sim}^{\Phi^{\alpha,\beta}} &F^{\alpha}F^{\beta}\frakE}$$
\end{itemize}
for any $\alpha, \beta\in \{\pm\}.$ Here we write $X^{+}:=X,$ $X^{-}=X',$ $T^{+,\,+}=T,$ $T^{-,-}=T',$ $T^{+,-}=H'$ and $T^{-,\,+}=H$.
 Indeed, let $z\in\End(\bigoplus\limits_{\tuple w\in \mathbb{J}^{d}}F^{\tuple w})$ for some $d$.
 Let $M\in \bfH^{\Lambda_{t_+,t_-}}_m\mod$, so that $\bigoplus\limits_{\tuple w\in \mathbb{J}^{d}}F^{\tuple w}(M)\in \bfH^{\Lambda_{t_+,t_-}}_{m+d}\mod$.
The action of $z$ on
$\bigoplus\limits_{\tuple w\in \mathbb{J}^{d}}F^{\tuple w}\frakE_{m}(M)$ is represented
by the action of $\phi_{m}(z)\otimes 1$ on
$$\bigoplus\limits_{\tuple w\in \mathbb{J}^{d}}F^{\tuple w}(\bigoplus\limits_{\tuple\nu\in  \mathbb{J}^m}F^{\tuple \nu}(E_{t_+,t_-}))\otimes_{\bfH^{\Lambda_{t_+,t_-}}_{m}}M$$
which is equal to
$$\bigoplus\limits_{\tuple w\in
\mathbb{J}^{d},\tuple\nu\in \mathbb{J}^m}F^{\tuple w}F^{\tuple\nu}(E_{t_+,t_-}\otimes_{\bfH^{\Lambda_{t_+,t_-}}_{m+1}}
\bfH^{\Lambda_{t_+,t_-}}_{m+1})\otimes_{\bfH^{\Lambda_{t_+,t_-}}_{m}}M$$
or
$$\bigoplus\limits_{\tuple u\in \mathbb{J}^{d+m}}F^{\tuple u}(E_{t_+,t_-})\otimes_{\bfH^{\Lambda_{t_+,t_-}}_{d+m}}\big(\bigoplus\limits_{\tuple w\in \mathbb{J}^{d}}F^{\tuple w}(M)\big).$$
 The action of $z$ on
$\frakE_{m+w}\big(\bigoplus\limits_{\tuple w\in \mathbb{J}^{d}}F^{\tuple w}(M)\big)$ is represented
by the action of $1\otimes z$ on $\bigoplus\limits_{\tuple w\in \mathbb{J}^{d}}F^{\tuple w}(M)$ in
$F^{d+m}(E_{t_+,t_-})\otimes_{\bfH^{\Lambda_{t_+,t_-}}_{d+m}}F^d(M).$
They obviously coincide.
Hence (i) and (ii) follow
by taking $z=X^{\pm}\in \End(F^{\pm})$
and  $z=T^{\alpha\beta}\in \Hom(F^{\alpha}F^{\beta},F^{\beta}F^{\alpha})$ for $\alpha,\beta\in\{\pm\}$.

\smallskip
Recall that we have the decompositions
$$F=\bigoplus_{s\in I_{\infty}} F_s,\ \ \ \
F'=\bigoplus_{s'\in I'_{\infty}} F'_{s'}.$$
Restricting to the direct summand $F_s,$
we get  the isomorphism $\Phi_{s}:\frakE F_{s}\xrightarrow{\sim} F_{s}\frakE.$
Moreover, by the equivalence of  representation data $$(E,F,X,T;E',F',X',T';H,H')$$  and $$\big(\{E_s\}_{s\in \K_\infty},\{F_s\}_{s\in\K_\infty},\{x_s\}_{s\in \K_\infty},\{\tau_{s,t}\}_{s,t\in \K_\infty}\big),$$
 we conclude that
 \begin{itemize}[leftmargin=8mm]
  \item[(i)] the isomorphism  $\Phi_s:\frakE F_{s}\xrightarrow{\sim} F_{s}\frakE$
   intertwines the endomorphisms $\frakE x_{s}$ and $x_{s}\frakE,$ and
  \item[(ii)] the isomorphisms  $\Phi_{t,s}:= F_t \Phi_s\circ \Phi_t F_s:
\frakE F_{t}F_{s}\xrightarrow{\sim} F_{t}F_{s}\frakE$ and $ \Phi_{s,t}:=F_s \Phi_t\circ\Phi_s F_t:\frakE F_{s}F_{t}\xrightarrow{\sim} F_{s}F_{t}\frakE$ satisfy the commutative diagram:
      $$\xymatrix{\frakE F_{t}F_{s}\ar[d]^{\frakE \tau_{st}}
\ar[r]_{\sim}^{\Phi_{t,s}} & F_{t}F_{s}\frakE \ar[d]^{\tau_{st}\frakE}\\
 \frakE F_{s}F_{t}\ar[r]_{\sim}^{\Phi_{s,t}} &F_{s}F_{t}\frakE.}$$
\end{itemize}
Thus we have proved that $\frakE$ is an isomorphism of $\frakA(\frakg_{\infty})$-representations, as wanted.

\smallskip

Finally, we equip $\scrQU_{K,t_+,t_-}$ with the
$\frakA(\frakg_{\infty})$-representation which is transferred from the $\frakg_\infty$-representation on
$\scrL(\Lambda_{{\tuple\xi_t}})_{\infty}$ via the equivalence $\frakE$.
We deduce that $\frakE$ induces on the Grothendieck groups a $\frakg_{\infty}$-module isomorphism
$\bfL(\Lambda_{t_+,t_-})_\infty = [\scrL(\Lambda_{t_+,t_-})_{\infty}] \simto [\scrQU_{K,t_+,t_-}]$.
Thus
the theorem
 follows from Theorem \ref{thm:HL-BC} and the $\frakg_{\infty}$-module isomorphism
$\bfF({{\tuple\xi_{t_+}}})_{\infty}\bigotimes\bfF({{\tuple\xi_{t_-}}})_{\infty}=\bfL(\Lambda_{{t_+,t_-}})_{\infty}$.
\end{proof}

\begin{remark}\label{rmk:Morinotcatrep}
(1) We remark here that Morita equivalences of categories may not preserve categorical representations
in general. To see this, we define the category $$\scrL(\Lambda_{\tuple\xi_{t_+}})\otimes
\scrL(\Lambda_{\tuple\xi_{t_-}}):=\bigoplus_{m_+\in\bbN,m_-\in \bbN}
\bfH^{q;\,\tuple\xi_{t_+} }_{K,m_+}\otimes
\bfH^{q;\,\tuple\xi_{t_-} }_{K,m_-}\mod.$$
By \cite{DVV2}, $\scrL(\Lambda_{\tuple\xi_{t_+}})\otimes
\scrL(\Lambda_{\tuple\xi_{t_-}})$ is Morita equivalent to $\scrU_{K,t_+}\otimes\scrU_{K,t_-}$.
Through the functor
$\frakC=\bigoplus\limits_{m_+,m_-\geqs 0}\Hom_{KG_{r+m_++m_-}}((F')^{m_-}F^{m_+}(E_{t_+,t_-}),-)$,
  we know that  $\scrQU_{K,t_+,t_-}$ and $\scrL(\Lambda_{{t_+},{t_-}})_{\infty}$
are equivalent  to the category
 $\scrL(\Lambda_{\tuple \xi_{t_+}})\otimes \scrL(\Lambda_{\tuple\xi_{t_-}})$.
 So $\scrQU_K$ and $\scrU_K\otimes\scrU_K$ are Morita equivalent.
 However,  the latter does not have an
 $\frak{A}(\frak{g}_\infty)$-representation structure.
 \smallskip

 (2) In the modular situation, it is conjectured by Brou\'{e} that $\scrQU_k$ and $\scrU_k\otimes\scrU_k$ are also Morita equivalent.

\end{remark}

\subsection{Categorical action on $\scrQU_k$\label{sec:uk-typeBC}}
Parallel to \cite[\S 6.5]{DVV2},  we consider quadratic unipotent representations in positive characteristic.

\smallskip

\subsubsection{The $\mathfrak{A}(\frakg'_{2d})$-representation on $\scrQU_k$\label{sec:QU2d}}

We have assumed that $\ell\nmid q$ and both $\ell$ and $q$ are odd.
Also, recall that
$d$ (resp. $f$) is the order of $q^2$ (resp. $q$) modulo $\ell$.

\begin{definition} Let $\K_{2d}$ be the quiver obtained from $\K_\infty$ by specialization $\mathcal{O} \to k$.
We define $\frakg'_{2d}$ to be the derived Kac-Moody algebra associated to the quiver $\K_{2d}$.\qed
\end{definition}

\smallskip

 (1) If $f$ is odd then $f=d$, and $-1$ cannot be expressed as a power of $q$ in $k$.
In this case, the quiver $\K_{2d}$ can be decomposed as follows:
$$\K_{2d}=\scrI_{d}\,\sqcup\, \scrI'_{d}=(I_{d,1}\,\sqcup\, I_{d,2})\,\sqcup\, (I'_{d,1}\,\sqcup\, I'_{d,2}),$$
where
$I_{d,1}=q^\bbZ$, $I_{d,2}=-q^\bbZ$,
$I'_{d,1}=q^\bbZ$ and
$I'_{d,2}=-q^{\bbZ}$ are all cyclic quivers
 of size $d$
 and $\scrI_{d}=I_{d,1}\,\sqcup\, I_{d,2}$ and $ \scrI'_{d}=I'_{d,1}\,\sqcup\, I'_{d,2}$.
This yields a Lie algebra isomorphism
$$
({\widetilde{\fraks\frakl}_d}^{\oplus 2})^{\oplus 2}\simeq\fraks\frakl_{\scrI_{d}}\oplus\,\fraks\frakl_{\scrI_{d}'}=\frakg'_{2d}$$
such that
$$\begin{array}{rclcrcl}
  (\alpha^\vee_k,0,0,0) & \mapsto & \alpha^\vee_{q^{k}}\times\{0\}, & \quad\quad & (0,\alpha^\vee_k,0,0) &  \mapsto & \alpha^\vee_{-q^{k}}\times\{0\}, \\
  (0,0,\alpha^\vee_k,0) & \mapsto & \{0\}\times\alpha^\vee_{q^{k}}, & \quad\quad & (0,0,0,\alpha^\vee_k) & \mapsto &\{0\}\times \alpha^\vee_{-q^{k}}.
\end{array}
$$

\smallskip

(2) If $f$ is even, then $f=2d$ and $q^d = -1$. Hence
$$\K_{2d} = \scrI_{2d}\sqcup \scrI'_{2d}$$
is a union of 2 cyclic quiver of size $2d$ and we have an isomorphism
$${\widetilde{\fraks\frakl}_f}^{\oplus 2}\simeq\fraks\frakl_{I_{2d}}\oplus\fraks\frakl_{I'_{2d}}\simeq\frakg'_{2d}$$
such that $(\alpha^\vee_k,0)\mapsto\alpha^\vee_{q^{k}}\times\{0\}$ and $(0,\alpha^\vee_{k'})\mapsto\{0\}\times\alpha^\vee_{q^{k}}$.

\smallskip

The specialization from $\mathcal{O}\subset K$ to $k$ yields a morphism of quivers
$\sp\,:\,\K_\infty\to \K_{2d}$ and a morphism of abelian groups $\P_\infty\to\P_{2d}$
such that $\Lambda_i\mapsto\Lambda_{\sp(i)}$ and $\Lambda_{i'}\mapsto\Lambda_{\sp(i')}$.
The infinite sums
$$E_i=\bigoplus_{\sp(j)=i}E_j,\quad E'_{i'}=\bigoplus_{\sp(j')=i'}E'_{j'}, \quad
F_i=\bigoplus_{\sp(j)=i}F_j, \quad
 F'_{i'}=\bigoplus_{\sp(j')=i'}F'_{j'}$$
give the well-defined operators
of $\frakg'_{2d}$ on $\bfF({\tuple\xi_{t_+}})_{\infty}\otimes\bfF({\tuple\xi_{t_-}})_{\infty}$.
This yields a representation of $\frakg'_{2d}$ on $\bfF({\tuple\xi_{t_+}})_{\infty}\otimes\bfF({\tuple\xi_{t_-}})_{\infty} $ such that
the linear map

$$\begin{array}{rrcl}
  \sp\,:\, & \Res_{\frakg'_{2d}}^{\frakg_\infty}(\bfF({\tuple\xi_{t_+}})_{\infty}
  \otimes\bfF({\tuple\xi_{t_-}})_{\infty}) &  \to
   & \bfF({\tuple\xi_{t_+}})_{2d}\otimes\bfF({\tuple\xi_{t_-}})_{2d} \\
    & |\tuple\mu_+,\tuple\xi_{t_+}\rangle_\infty\otimes |\tuple\mu_-,\tuple\xi_{t_-}\rangle_\infty & \mapsto & |\tuple\mu_+,\tuple\xi_{t_+}
    \rangle_{2d}\otimes |\tuple\mu_-,\tuple\xi_{t_-}\rangle_{2d}
\end{array}
$$\\
is a $\frakg'_{2d}$-equivariant isomorphism.

\smallskip

Under the map $d_\scrQU:[\scrQU_K]\to[\scrQU_k]$ and the isomorphism
$$\bigoplus_{t_+,t_-\in\bbN}(\bfF({\tuple\xi_{t_+}})_{\infty}\otimes
\bfF({\tuple\xi_{t_-}})_{\infty})\simto [\scrQU_{K}]$$
in Theorem \ref{thm:ginfinityB}, the map $\sp$ endows $[\scrQU_k]$ with
a representation of $\frakg'_{2d}$ which is compatible with the representation associated with
the representation datum constructed in \S \ref{sub:explicit-repdatumSO}. More precisely, we have

\begin{proposition}\label{prop:lreductionB}
For each $i \in \scrI_{2d}$ and $i' \in \scrI'_{2d}$, let $k E_i$ and $k F_i$
be the generalized $i$-eigenspace of $X$ on $k E$ and $k F$, respectively,
and let $k E'_{i'}$ and $k F'_{i'}$
be the generalized $i'$-eigenspace of $X'$ on $k E'$ and $k F'$, respectively. Then
  \begin{itemize}[leftmargin=8mm]
    \item[$\mathrm{(a)}$] $[k E_i],$ $[k E'_{i'}],$ $[k F_i],$ $ [k F'_{i'}]$ endow $[\scrQU_k]$ with a structure
    of $\frakg'_{2d}$-module,
    \item[$\mathrm{(b)}$]the decomposition map $d_\scrQU$
     yields a $\frakg'_{2d}$-module isomorphism
     $$ \mathrm{Res}_{\frakg'_{2d}}^{\frakg_\infty} \, [\scrQU_K]\simto [\scrQU_k],$$ and
      \item[$\mathrm{(c)}$] the map
     $|\tuple\mu_+,\tuple\xi_{t_+}\rangle_{2d}\otimes |\tuple\mu_-,\tuple\xi_{t_-}\rangle_{2d} \mapsto d_\scrQU([E_{\Theta_{t_+}(\tuple\mu_+),\Theta_{t_-}(\tuple\mu_-)}])$ yields a
$\frakg'_{2d}$-module isomorphism
$$\bigoplus_{t_+,t_-\in\bbN}(\bfF({\tuple\xi_{t_+}})_{2d}\otimes\bfF({\tuple\xi_{t_-}})_{2d})  \simto [\scrQU_k].$$

  \end{itemize}
\end{proposition}

\begin{proof} Since $\ell$ is odd, it follows from Theorem \ref{basicset} that the decomposition map $d_\scrQU$
is a vector space isomorphism. Hence the proposition holds by Theorem \ref{thm:ginfinityB}.
\end{proof}

\smallskip

\begin{theorem}\label{thm:B} For odd $\ell$ and $q$ with $\ell \nmid q(q^2-1)$,
the representation datum  on
$\scrQU_k$ constructed in \S \ref{sub:explicit-repdatumSO}
yields an $\frakA(\frakg'_{2d})$-representation such that
the decomposition map $d_\scrQU\,:\,[\scrQU_K]\to[\scrQU_k]$ intertwines the representations
of  $\frakg_{\infty}$ and $\frakg'_{2d}$.
There is a $\frakg'_{2d}$-module isomorphism
$$\bigoplus_{t_+,t_-\in\bbN}(\bfF({\tuple\xi_{t_+}})_{2d}\otimes\bfF({\tuple\xi_{t_-}})_{2d})  \simto [\scrQU_k]$$ sending
$|\tuple\mu_+,\tuple\xi_{t_+}\rangle_{2d}\otimes |\tuple\mu_-,\tuple\xi_{t_-}\rangle_{2d}$ to $ d_\scrQU([E_{\Theta_{t_+}(\tuple\mu_+),\Theta_{t_-}(\tuple\mu_-)}])$.
In particular, the classes of the simple quadratic unipotent modules in $[\scrQU_k]$
are weight vectors for the $\frakg'_{2d}$-action.

\end{theorem}
\begin{proof} For $f$ odd (resp. even),
 the quadratic unipotent modules
$E_{\Theta_{t_+}(\tuple\mu_+),\Theta_{t_-}(\tuple\mu_-)}$
 and $E_{\Theta_{s_+}(\tuple\nu_+),\Theta_{s_-}(\tuple\nu_-)}$ of
 $G_n$ lie in the same $\ell$-block  of  $G_n$  if and only if
 their corresponding symbols $\Theta_{t_\pm}(\tuple\mu_\pm)$ and $\Theta_{s_\pm}(\tuple\nu_\pm)$ have the same $d$-core (resp. $d$-cocore).
 If this is the case, then by \cite[Lemma 6.10]{DVV2},
 $|\tuple\mu_+,\tuple\xi_{t_+}\rangle_{2d}\otimes |\tuple\mu_-,\tuple\xi_{t_-}\rangle_{2d}$ and $|\tuple\nu_+,\tuple\xi_{s_+}\rangle_{2d}\otimes |\tuple\nu_-,\tuple\xi_{s_-}\rangle_{2d}$ have  the same weight. This means that
weight spaces are sums of blocks. Hence
 the theorem follows by Theorem \ref{def:2kacmoodyrep} and Proposition \ref{prop:lreductionB}.
\end{proof}

\subsubsection{$\mathfrak{A}(\frakg_{2d})$-representation on $\scrQU_k$ in the linear prime case}\label{subsec:linear}
 When $\ell$ is linear, we have that $f$ is odd and $f=d$. In this case,
the Kac-Moody algebra $\widehat{ \frakg_{2d}}$ associated with the quiver $\K_{2d}$
is isomorphic to $(\widehat{\fraks\frakl}_d)^{\oplus 4}$.
The action of $\frakg_{2d}'$ on $\scrQU_k$ can be naturally
extended to an action of an algebra $\frakg_{2d}$ lying between $\frakg_{2d}'$
and $\widehat{\frakg_{2d}}$. 

\smallskip

Let $\widehat \X_{2d}$ and $\widehat \X_{2d}^\vee$ be the lattices corresponding
to $\widehat\frakg_{2d}$. Since $f$ is odd, $\K_{2d}$ is the disjoint union
of 4 cyclic quivers. We choose $\alpha_1\times\{0\}$  $\alpha_{1'}\times\{0\}$,  $\{0\}\times\alpha_{-1}$and $\{0\}\times\alpha_{-1'}$ to be the affine roots
attached to these quivers. Then we have
$$\widehat \X_{2d} = \P_{2d} \oplus
 \bbZ \delta_1 \oplus \bbZ\delta_1'\oplus \bbZ\delta_2\oplus
 \bbZ\delta_2',\quad \widehat \X_{2d}^\vee = \Q_{2d}^\vee
\oplus \bbZ \partial_1 \oplus \bbZ \partial_1'\oplus \bbZ
\partial_2 \oplus \bbZ \partial_2',$$
where
$$\begin{array}{rclrclrclrcl}
    \delta_1 &=& \sum \alpha_{q^{j}}\times\{0\}, & \delta_1' &=&\sum \{0\}\times \alpha_{q^{j}},\\
    \delta_2 &=&\sum \alpha_{-q^{j}}\times\{0\}, & \delta_2'&=& \sum\{0\}\times \alpha_{-q^{j}}, \\
    \partial_1 &=& \Lambda_{1}^\vee\times\{0\}, & \partial_1' &=& \{0\}\times\Lambda_{1}^\vee,\\
    \partial_2 &=& \Lambda_{-1}^\vee\times\{0\}, & \partial_2' &=& \{0\}\times\Lambda_{-1}^\vee.
  \end{array}
$$
We set
$$\partial=\partial_1+\partial_2, \partial'=\partial_1'+\partial_2', \delta=(\delta_1+\delta_2)/2 \mbox{~and~}
\delta'=(\delta_1'+\delta_2')/2,$$
and define $$\frakg_{2d}:=\frakg'_{2d}\oplus\bbC\partial\oplus\bbC\partial'$$
so that it can be viewed as the Kac-Moody algebra associated with the lattices
$$\X_{2d} := \P_{2d} \oplus \bbZ\delta\oplus \bbZ\delta' \simeq \widehat \X_{2d} / \{(\delta_1-\delta_2)\oplus(\delta_1'-\delta_2')\},\quad
\X_{2d}^\vee :=  \Q_{2d}^\vee \oplus \bbZ \partial\oplus\bbZ \partial'.$$
Clearly, the pairing $\widehat\X_{2d}^\vee\times\widehat\X_{2d}\longrightarrow\bbZ$ induces
a perfect pairing $\X_{2d}^\vee\times\X_{2d}\longrightarrow\bbZ$.

\smallskip

For $t_{+},t_- \in \mathbb{N}$, the Fock space $\bfF(\tuple\xi_{t_+})_{2d}\otimes\bfF(\tuple\xi_{t_-})_{2d}$ has a tensor product
decomposition into level~$1$ Fock spaces as a representation of $\widehat{\frakg}_{2d}$:
$$\bfF(\tuple\xi_{t_+})_{2d}\otimes\bfF(\tuple\xi_{t_-})_{2d} \simeq \bfF((-q)^{t_+})_d \otimes \bfF((-q)^{-1-{t_+}})_d\otimes \bfF((-q)^{t_-})_d\otimes \bfF((-q)^{-1-{t_-}})_d.$$
Out of the charged Fock spaces $(\bfF((-q)^{t_\pm})_d,t_\pm)$ and $(\bfF((-q)^{-1-t_\pm})_d,-1-t_\pm)$ and
the isomorphism $\widetilde \frakg_{2d} \simeq (\widehat{\fraks\frakl}_d)^{\oplus 2}$
(which depends on the parities of $t_{\pm}$) we can therefore equip $\bfF(\tuple\xi_{t_+})_{2d}\otimes\bfF(\tuple\xi_{t_-})_{2d}$ with an
action of $\widehat \frakg_{2d}$ which in turn restricts to an action of $\frakg_{2d}$.

\vspace{2ex}

Recall that two quadratic unipotent characters labeled by $\Theta_+\times\Theta_-$ and $\Theta_+'\times\Theta_-'$,
respectively, are in the same $\ell$-block if and
only if the  symbols $\Theta_+$ and $\Theta_+'$  have the same $d$-core and the  symbols $\Theta_-$ and $\Theta_-'$  have the same $d$-core. In particular
the quadratic unipotent characters of a given isolated  $\ell$-block have the same $1$-core, however,
this means the quadratic unipotent characters of a given isolated  $\ell$-block all lie in
 the same set, say $\Irr(\scrQU_{k,t_+,t_-})$.
 Consequently, for each $t_+,t_-\in\bbN$
we can form the category $\scrQU_{k,t_+,t_-}$ such that
$$d_\scrQU\,:\,[\scrQU_{K,t_+,t_-}]\to[\scrQU_{k,t_+,t_-}]$$ is an isomorphism, yielding
$$ \scrQU_k = \bigoplus_{t_+,t_- \in \bbN} \scrQU_{k,t_+,t_-} \quad \text{ with } \quad
[\scrQU_{k,t_+,t_-}] \simeq \bfF(\tuple\xi_{t_+})_{2d}\otimes \bfF(\tuple\xi_{t_-})_{2d}.$$
Using the action of $\frakg_{2d}$ on $ \bfF({\tuple\xi_t})_{2d}$ defined above we
equip each $[\scrQU_{k,t_+,t_-}]$ with a structure of $\frakg_{2d}$-module
which extends the structure of $\frakg_{2d}'$-module defined in \S \ref{sec:QU2d}.

\smallskip

Let $W_{2d}$ be the Weyl group of $\frakg_{2d}$.
 Notice that the Weyl groups of  $\frakg_{2d}$ and $\widehat{\frakg}_{2d}$ are same
 and that $\X_{2d}$ is a quotient of $\widehat{\X}_{2d}$.
Similar to \cite[\S 6.5.2]{DVV2}, we get the following
theorem.

\begin{theorem}\label{thm:Linearprime}
Let $\ell$ and $q$ be odd, and $\ell \nmid q(q^2-1)$.
Assume that $f$ is odd. For each pair $t_+,t_-\in \bbN$,
the representation datum $$\big(\{E_s\}_{s\in \K_{2d}},\{F_s\}_{s\in \K_{2d}},\{x_s\}_{s\in \K_{2d}},\{\tau_{s,t}\}_{s,t\in \K_{2d}}\big)$$
 defines a representation of
$\frakA(\frakg_{2d})$ on $\scrQU_{k,t_+,t_-}$
which categorifies the $\frakg_{2d}$-module
$\bfF(\tuple\xi_{t_+})_{2d}\otimes \bfF(\tuple\xi_{t_-})_{2d}$,
i.e, $$\scrQU_{k,t_+,t_-}=\bigoplus_{\mu\in\X_{2d}}\scrQU_{k, t_+, t_-, \mu}$$ and $[\scrQU_{k,t_+,t_-,\mu}]=(\bfF(\tuple\xi_{t_+})_{2d}\otimes \bfF(\tuple\xi_{t_-})_{2d})_\mu$.
Moreover, if ${\scrQU}_{k,t_+,t_-,\mu}\neq 0$,
then ${\scrQU}_{k,t_+,t_-,\mu}$ is exactly an
isolated $\ell$-block  of
 $k G_m$ for some $m\in\bbN.$ In addition,
 two isolated $\ell$-blocks are in the same $W_{2d}$-orbit if and only if they have the same degree vectors.
\end{theorem}
\begin{proof}
We notice that the argument mainly involves computations in the Grothen-dieck group.
Hence the theorem follows by the Jordan decomposition and a similar argument as for
\cite[Theorem 6.12]{DVV2}.
\end{proof}

\subsubsection{$\mathfrak{A}(\frakg_{2d})$-representation on $\scrQU_k$ in the unitary prime case}\label{subsec:unitary}
Here we extend the action of $\frakA(\frakg_{2d}')$ to $\frakA(\frakg_{2d})$
on $\scrQU_k$, assuming that $f$ is even and so $f=2d$. Under the assumption,
the quiver $\K_{2d}=I_{2d}\sqcup I'_{2d}$
is a union of 2 cyclic quivers of size $2d$ and we have an isomorphism
$$\widetilde{\fraks\frakl}_f^{\oplus 2}\simto\frakg'_{2d}$$  sending
$(\alpha^\vee_k,0)$ to $\alpha^\vee_{q^{k}}\times\{0\}$,
and $(0,\alpha^\vee_{k})$ to $\{0\}\times\alpha^\vee_{q^{k}}$.
Let $\frakg_{2d}$ be the Kac-Moody algebra associated with the quiver $\K_{2d}$, so that
$$\widehat{\fraks\frakl}_f^{\oplus 2}\simeq\frakg_{2d}.$$ We
set $\X_{2d} = \P_{2d}\, \oplus\, \bbZ \delta/2 \oplus\, \bbZ \delta'/2$.

\begin{theorem} \label{thm:unitary}
Let  $\ell$ and $q$ be odd, and $\ell \nmid q(q^2-1)$.
Assume that $f$ is even. Then the representation datum on $\scrQU_k$
gives a $\frakg_{2d}$-representation which extends the action of $\frakg'_{2d}$,
 i.e.,
 $$\scrQU_{k}=\bigoplus_{\mu\in\X_{2d}}\scrQU_{k, \mu}$$ and $[\scrQU_{k,\mu}]=(\bfF(\tuple\xi_{t_+})_{2d}\otimes \bfF(\tuple\xi_{t_-})_{2d})_\mu$.
Moreover,
if ${\scrQU}_{k,\mu}\neq 0$,
then ${\scrQU}_{k,\mu}$ is exactly an
isolated $\ell$-block  of  $k G_m$ for some $m\in\bbN.$
\end{theorem}
\begin{proof} It follows by a similar argument as for \cite[Theorem 6.18]{DVV2}.
\end{proof}

\section{Brou\'{e} abelian defect conjecture for $\SO_{2n+1}(q)$ at linear primes}\label{cha:broue}

Now we prove Brou\'{e} abelian defect conjecture for $\SO_{2n+1}(q)$ at linear primes
with $q$ odd. 
Our proof makes use of the reduction theorem of
Bonnaf\'{e}-Dat-Rouquier \cite[Theorem 7.7]{BDR17} (or Theorem \ref{th:introequiv}),
so that we can indeed focus on investigating isolated  blocks of $\SO_{2n+1}(q)$
with abelian defect groups.

\smallskip

Throughout this section,
we always let $\widetilde{G}=\widetilde{G}_n=\O_{2n+1}(q)$ and
$G=G_n=\SO_{2n+1}(q)$.
In addition, we assume that $\ell$ is a linear prime with respect to $q$
and $d$ is the multiplicative order of $q^2$ mod $\ell$ as before.
Till \S \ref{sec:broconj},
we continue with our assumption that both $q$  and $\ell$ are odd.

\smallskip

\subsection{Isolated blocks with abelian defect groups}\label{sub:iso-blocksabelian}
We first fix an isolated $\ell$-block $b:=b_{\Delta_+,\Delta_-,w_+,w_-}$
 of $G$  with an abelian defect group (i.e., with $w_+<\ell $ and $w_-<\ell$) as in \S \ref{subsec:SOqublock}.
 Also, we fix the following notation: $w=w_++w_-$, $m=m_++m_-$ and $n=n_++n_-$, where $m_{\pm}=\mathrm{rank}(\Delta_\pm)$
 and $n_\pm=m_\pm+dw_\pm$.

\smallskip

Let $L=G_{m}\times \GL_d(q)_1\times\dots\times \GL_d(q)_w
$, which embeds into $G_n$ as in \S \ref{subsec:Levi}.
Here the subscripts $1,\ldots,w$ are used to denote a fixed order of the factors
and to distinguish them from each other.
We shall need
  a sequence $$L=L_0<L_{1}<\dots <L_{w-1}< L_{w}=G$$ of Levi subgroups of $G$, where
$$L_i=L_{m+di,(d^{w-i})}\cong G_{m+di}\times \GL_d(q)_{i+1}\times\dots\times \GL_d(q)_w
$$ with the similar embedding to that of $L$ into $G$.

\smallskip

According to \S \ref{subsec:SOqublock}, the block $b$ has a defect group $P$ which is isomorphic to $P_1\times\dots\times P_w$,
where $P_j\in {\rm Syl}_\ell(\GL_d(q)_j)$
is isomorphic to a (fixed) Sylow $\ell$-subgroup $P_0$ of $\GL_d(q)$ for each $1\leqs j\leqs w$.
We write $\ell^a$ for the maximal power of $\ell$ dividing $q^d-1$, so that $|P_0|=\ell^a$.

\subsubsection{Intermediate subgroups}\label{ssub:subgroups}
We shall investigate the Brauer correspondent of $b$ in  a  subgroup $N$
containing $N_G(P)$. To do this,
  we 
%
%
%
let $\widetilde{L}={\widetilde{G}}_{m}\times\GL_d(q)_1\times\cdots\times\GL_d(q)_w$ and $H={\widetilde{G}}_{m}\times(\GL_d(q)_1.T_1)\times\cdots\times(\GL_d(q)_w.T_w)$ be two subgroups of
$\widetilde{G}_n,$ where for each $i$,
$T_i$ is the subgroup
of $G$ generated by
\begin{gather*}
\left(\begin{array}{ccccc}
 \id_{d(w-i)} &               &          &              &\\
                &          &          & \id_d &\\
                &          & \id_{G_{m+d(i-1)}} &               &\\
                & \id_d &          &          &\\
                &          &        &      & \id_{d(w-i)}\\
\end{array}\right).
\end{gather*}
Naturally, the group $H$
is isomorphic to $\widetilde{G}_{m}\times(\GL_d(q).2)^w$.

\smallskip

Let $\frakS_{w}$ be the subgroup of permutation matrices of
$G$ whose conjugate action on $L$ permutes
the $w$ factors $\GL_d(q)_1,\ldots,\GL_d(q)_{w}$ of $L$.
Let $\frakS_{w_+}$ be the subgroup of $\frakS_w$ only permuting
the former $w_+$ factors $\GL_d(q)_1,\ldots,\GL_d(q)_{w_+}$
of $L$, and $\frakS_{w_-}$ be
 the subgroup of $\frakS_w$ only permuting
the latter $w_-$  factors $\GL_d(q)_{w_++1},\ldots,\GL_d(q)_{w}$ of $L$.
Since $\frakS_{w}\cap H=\{\id_{G_n}\}$
and $\frakS_{w}$ normalizes $H$,
we have
$$\widetilde{N}:=H.\frakS_w\cong\widetilde{G}_{m}\times((\GL_d(q).2)\wr
\frakS_w)$$ and
$$\widetilde{M}:=H.(\frakS_{w_+}\times \frakS_{w_-})\cong\widetilde{G}_{m}\times((\GL_d(q).2)\wr
\frakS_{w_+})\times(\GL_d(q).2)\wr
\frakS_{w_-}).$$
Now, we set $N=\widetilde{N}\cap G$ and $M=\widetilde{M}\cap G$, so that $|{N}|=2^ww!|L|$ and $|M|=2^ww_+!w_-!|L|$.

\smallskip

It is clear that $C_G(P)\leqs L$ and $N_G(P)\leqs N$.


\subsubsection{Blocks of $L_i$ and $N$}\label{ssub:blockofN}

The blocks of $L$ are easier to describe since they
can be expressed into the tensor products of blocks of factors of $L$.
If there is no confusion, sometimes we will use the natation for blocks to simultaneously
denote their corresponding block idempotents, and vice versa.

\smallskip

Recall that $L=G_{m}\times \GL_d(q)_1\times\dots\times \GL_d(q)_w$. We
write $\id_d$ for the identity matrix of $\GL_d(q)^*$($\cong \GL_d(q)$)
 and $-\id_d$ for the  matrix $\diag(-1,\dots,-1)$ of $\GL_d(q)^*$($\cong \GL_d(q)$).
 Let
\begin{itemize}[leftmargin=8mm]
  \item $a_i^+$ be the  block idempotent of $\mathcal{O}{\rm GL}_d(q)_i$ with block label $(\id_d,\emptyset)$ (i.e.,
   $a_i^+$ be the principle block idempotent of $\mathcal{O}{\rm GL}_d(q)_i$), and
  \item $a_i^-$ be the block idempotent of $\mathcal{O}{\rm GL}_d(q)_i$ with block label $(-\id_d,\emptyset)$,
\end{itemize}
where $1\leqs i\leqs w$.

\smallskip

For a sign vector ${\tuple \nu} =
(\nu_1,\dots,\nu_w)$, we define
\begin{itemize}[leftmargin=8mm]
 \item $a^{{\tuple \nu}}=a_1^{\nu_1}\otimes \cdots\otimes a_w^{\nu_w}$
so that $a^{{\tuple \nu}}$ is a block idempotent  of $\GL_d(q)_1\times\dots\times \GL_d(q)_w,$
\item $a^{{\tuple \nu}_{[i,j]}}=a_{i}^{\nu_{i}}\otimes a_{i+1}^{\nu_{i+1}}\otimes \cdots\otimes a_{j}^{\nu_{j}}$
  for $1\leqs i\leqs j\leqs w$ so that $a^{{\tuple \nu}_{[i,j]}}$
  is a block idempotent of $\mathcal{O}(\GL_d(q)_{i}\times\GL_d(q)_{i+1}\times \cdots\times  \GL_d(q)_{j}).$
\end{itemize}
Correspondingly, we define block idempotents of $L$ and $L_i\cong G_{m+di}\times \GL_d(q)_{i+1}\times\dots\times \GL_d(q)_w
$ respectively as follows:
$$b^{{\tuple \nu}}_{0,0}:
=b_{\Delta_+,\Delta_-,0,0}\otimes a^{{\tuple \nu}}$$
and
$$b^{{\tuple \nu}_{[i+1,w]}}_{i_+,i_-}:=b_{\Delta_+,\Delta_-,i_+,i_-}\otimes a^{{\tuple \nu}_{[i+1,w]}},$$
 where $0\leqs i\leqs w$, $i=i_++i_-$
 and
$b_{\Delta_+,\Delta_-,i_+,i_-}$ is the block idempotent of  $G_{m+di}$
with block label $(\Delta_+\times\Delta_-,i_+, i_-)$
 as in \S \ref{subsec:SOqublock}.

\smallskip

Let ${\tuple w}=(w_+,w_-)$. As in \S \ref{subsec:signvectors},
we let $\mathbb{J}_{\tuple w}:=\mathbb{J}_{w_+,w_-}$ be the set of
sign vectors ${\tuple \nu} = (\nu_1, \dots, \nu_w) \in \mathbb{J}^w$ with $w_+$ copies of the $+$ sign
and $w_-$ copies of the $-$ sign.
Recall that $b_{\Delta_+\times\Delta_-,0,0}$ is
a block of $\mathcal{O}G_{m}$ of defect zero.

 In the following, we shall denote and fix $${\tuple \nu}_0=(+,\ldots,\,+,-,\ldots,-)\in \mathbb{J}_{\tuple w},$$ and
 $$  f_0 := b^{{\tuple \nu}_0}_{0,0}= b_{\Delta_+\times\Delta_-,0,0}\otimes a^{{\tuple \nu}_0}
$$
 which is a block idempotent of
  $\mathcal{O}L$.
It is easy to see that $M$ is
exactly the inertial group of the block $\mathcal{O}Lf_0$ of $L$ in $N.$

\smallskip

Now, we define an idempotent $f$ of  $\mathcal{O}L$ to be
 $$
   f :=  \sum_{\tuple\nu\in \mathbb{J}_{\tuple w}} b^{{\tuple \nu}}_{0,0}= \sum_{\tuple\nu\in \mathbb{J}_{\tuple w}} b_{\Delta_+\times\Delta_-,0,0}\otimes a^{{\tuple \nu}}.
$$
  Observe that the above is indeed a primitive idempotent decomposition of $f$
 including $f_0$.
 We are going to show that $f$ is a block idempotent of $\mathcal{O}N$.

\begin{lemma}\label{lem:defectgroup} Let ${\tuple \nu}\in \mathbb{J}^w$  and $1\leqs i\leqs w$ with $i=i_++i_-$.
Then $P$ is a defect group of $\mathcal{O}L_ib_{i_+,i_-}^{{\tuple \nu}_{[i+1,w]}}$.
\end{lemma}

\begin{proof} Since $a^{\pm}_j$ is labeled by $(\pm \id_d,\emptyset)$,
the Sylow $\ell$-subgroup $P_j$ of $\GL_d(q)_j$
is a defect group of
 $\mathcal{O}{\rm GL}_d(q)_ja^{\pm}_j$ for $i+1\leqs j\leqs w$.
According to \S \ref{subsec:SOqublock}, the group
$(P_{1}\times\dots\times P_{i})$ is a defect group of
$\mathcal{O}G_{m+di}b_{\Delta_+,\Delta_-,i_+,i_-}$.
Hence the lemma follows by the definition of $b_{i_+,i_-}^{{\tuple \nu}_{[i+1,w]}}$.
\end{proof}


\begin{lemma}$M$ stabilizes $f_0$ and as an $\mathcal{O}(M\times L)$-module,
$\mathcal{O}Mf_0$ is indecomposable. In particular,
$\mathcal{O}Mf_0$ is a block of $M$ with vertex $\Delta(P)$.
\end{lemma}
\begin{proof} $M$ clearly stabilizes $f_0$. By lemma \ref{lem:defectgroup}, $\mathcal{O}Lf_0$ has vertex $\Delta(P)$. Since $C_G(P)\leqs L$,
the conjugate of $\Delta(P)$ by an element of $M\times L$ outside $L\times L$ is
never conjugate to $\Delta(P)$ in $L\times L$. Consequently, the stabilizer of
$\mathcal{O}Lf_0$ in $M\times L$ is exactly $L\times L$.
So $\mathcal{O}Mf_0=Ind_{L\times L}^{M\times L}(\mathcal{O}Lf_0)$ is
indecomposable as an $\mathcal{O}(M\times L)$-module. Its vertex is clearly
contained in $\Delta (P)$ and its restriction to $L\times L$ has a summand with
vertex $\Delta(P)$. So $\mathcal{O}Mf_0$ has vertex $\Delta(P)$ as an
$\mathcal{O}(M\times L)$-module. Similarly, $\mathcal{O}Mf_0$ is also
indecomposable as an $\mathcal{O}(M\times M)$-module and has vertex $\Delta(P).$
\end{proof}

\begin{proposition}\label{prop:defectgroup2}
 We have $f=Tr_{M}^{N}(f_0)$ and that
$\mathcal{O}Nf$ is indecomposable as an $\mathcal{O}(N\times N)$-module
and has vertex $\Delta(P)$. In particular,
$\mathcal{O}Nf$ is a block of $N$.
\end{proposition}
\begin{proof} The first conclusion of
the proposition immediately follows by the definition of $f$.
Recall that $M$ is exactly the inertial group of the block $\mathcal{O}Lf_0$
in $N$. Hence $\mathcal{O}Nf$
is the Clifford correspondent of $\mathcal{O}Mf_0$, and thus
the other conclusions are clear.
\end{proof}

\begin{remark} If $w_-=0$, then the block $b$ is unipotent and $N$ itself
is the inertial group of the block $\mathcal{O}Lf_0$ of $L$ in $N.$
This case has been studied by Livesey \cite{Li12}.
\end{remark}

\subsubsection{Brauer correspondence between $\mathcal{O}Gb$ and $\mathcal{O}Nf$} \label{ssub:BauerCor-b&f}
In this section we show that $\mathcal{O}Nf$ is the Brauer correspondent
 of $\mathcal{O}Gb$ in $N$, starting with a result of Fong and Srinivasan \cite[\S 12]{FS89}.

\begin{lemma}\label{FSlemma}
Let $V$ be the underlying orthogonal space of $G$, and
$b_{s,\kappa}$ be an $\ell$-block of $G$ with label $(s,\kappa)$ and defect group $D$.
Write $V_0=C_V(D)$ and $V_+=[V,D]$
so that $V=V_0\oplus V_+$ and correspondingly $C_G(D)=C_0\times C_+$,
where $C_0=\SO(V_0)$.
Then  $Br_D^G(b_{s,\kappa})$
is a sum of all blocks of $\mathcal{O}C_G(D)$ of the form $b_{s_0,\kappa}\otimes b_{s_+,\emptyset}$,
 where $s$ and $s_0\times s_+$ are conjugate in $G$ and
  $b_{s_0,\kappa}$ is a  block of $C_0$ of defect zero
  and $b_{s_+,\emptyset}$ is a block of $C_+$ with defect group $Z(D)$.
\end{lemma}
\begin{proof} Let $b_0$ be a block of $C_0$ and
$b_+$ be a block of $C_+$ such that $b_0\otimes b_+$ is the block of $C_G(D)=C_0\times C_+$
 which is a Brauer correspondent of $b_{s,\kappa}.$
 The result is clear if $D=1$, in which case $C_0=G$, $C_+=1$ and $(s,\kappa)$
 labels the unique character in $b_{s,\kappa}$. So we may assume that $D\neq 1$.

 \smallskip

 As in  \cite[(12)]{FS89},
    let $z\in Z(D)$ be such that $z^{\ell}=1$ and $[V,z]=V_+.$
 Then $Q=C_G(z)$ is 
 a Levi subgroup of $G$, and
we have $Q=Q_0\times Q_+$,
where $Q_0=\SO(V_0)$ and $Q_+\cong \GL_m(\epsilon q^d)$ for some $m.$
In particular, $C_0=Q_0$ and $C_+\leqs Q_+.$

 \smallskip

Since $z$ is central,
 there exists a unique block $b_z$ of $Q$ such that
$$(1,b_{s,\kappa})\unlhd(\langle z\rangle,b_z)\unlhd(D,b_0\otimes b_+).$$
Let $b_z=b_{z,0}\otimes b_{z,\,+}$
where $b_{z,0}$ and $b_{z,\,+}$ are blocks of $Q_0$ and $Q_+$,
 respectively. Then $b_{z,0}=b_{s_0,\mu}\in\mathcal{E}_{\ell}(Q_0, (s_0))$
and $b_{z,\,+}\in \mathcal{E}_{\ell}(Q_+, (s_+))$
 for some $s_0\in Q^*_0$, $s_+\in Q_+^*$ and $\mu\in\Psi(s_0)$.
As is shown at the beginning of \cite[(12)]{FS82},
 $b_{z,0}$ has defect zero and $b_{z,\,+}$ has defect group isomorphic to
 $D$. It follows that $s_0$ is an $\ell'$-element and $\mu\in \mathcal{R}(s_0)$.
 By \cite[Theorem 3.2]{BM},
we conclude that $s_0\times s_+$ and $s$ are conjugate in $G^*$, and so we
may suppose that $s =s_0 \times s_+$.

 \smallskip

Now $C_G(D) =C_Q(D)$, and we may view $(D, b_+)$ as a Brauer
pair of $Q_+$. Then $(1, b_{z,\,+})\unlhd (D, b_+)$ holds in $Q_+$, and
$(D, b_+)$ has a Brou\'{e} labeling $(D, t, \emptyset)$ as in \cite[(3.1)]{Bro85}, where $t\in Q^*$.
Here, the third component of the label is empty since $Q_+\cong \GL(m, \epsilon q^d)$ and $D$
acts fixed-point freely on the underlying space of $Q_+$.
Again by \cite[Theorem 3.2]{BM}, $t$ and $s_+$ are conjugate in $Q_+^*$,
 and we may suppose $t =s_+$.
Hence $(D, b_+)$ has a Brou\'{e} labeling $(D, s_+, \emptyset)$.
Now, by the proof of \cite[(12A)]{FS82}, we know that $\mu=\kappa$
and so $b_0\otimes b_+=b_{s_0,\kappa}\otimes b_{s_+,\emptyset}$.

\smallskip
Conversely, for each block of the form $b_{s_0,\kappa}\otimes b_{s_+,\emptyset}$ as in the conclusion
of the lemma, it is clear that $(b_{s_0,\kappa}\otimes b_{s_+,\emptyset})^G$ is defined.
By the argument above,
we see that $b_{s,\kappa}$ is its Brauer correspondent in $G$.
Thus the lemma holds.
\end{proof}

Notice that $C_G(P)=G_m\times C_{\GL_d(q)_{1}}(P_{1})\times \cdots\times  C_{\GL_d(q)_{w}}(P_{w})$.
For $0\leqs j\leqs w$, let
\begin{itemize}[leftmargin=8mm]
  \item $c^+_j$ be the principal block idempotent of $C_{\GL_d(q)_j}(P_j)$ (with block label $(\id_d,\emptyset)$)
  so that  $c^{+}_j$ is the  Brauer correspondent to $a^{+}_{j}$, and
  \item  $c^-_{j}$ be the unique  block idempotent of $C_{\GL_d(q)_{j}}(P_{j})$
  (with block label $(-\id_d,\emptyset)$) that is the  Brauer correspondent to $a^-_{j}$.
 \end{itemize}
Let ${\tuple \nu} = (\nu_1, \dots, \nu_w) \in \mathbb{J}^w$.
We similarly denote
$$c^{{\tuple \nu}}=c_{1}^{\nu_{1}}\otimes \cdots\otimes c_{w}^{\nu_{w}}, \ \ \ \
c^{{\tuple \nu}_{[i,j]}}=c_{i}^{\nu_{i}}\otimes c_{i+1}^{\nu_{i+1}}\otimes \cdots\otimes c_{j}^{\nu_{j}}$$
so that $c^{{\tuple \nu}_{[i,j]}}$ is a block idempotent of
$\mathcal{O}(C_{\GL_d(q)_{i}}(P_{i})\times \cdots\times  C_{\GL_d(q)_{j}}(P_{j}))$,
where $1\leqs i\leqs j\leqs w.$  We write
$$e_0=b_{\Delta_+\times\Delta_-,0,0}\otimes c^{{\tuple \nu}_0}~~\,\mbox{and}~~\,
e=\sum_{\tuple\nu\in \mathbb{J}_{\tuple w}} b_{\Delta_+\times\Delta_-,0,0}\otimes c^{{\tuple\nu}}.$$
Then $e_0$ is a block of $C_G(P)$ and $e$ is a sum of blocks of $C_G(P)$.

\smallskip

\begin{lemma}\label{lem:OGbOBrhom}
We have $Br_{P}^{L}(f_0)=e_0$ and $Br_{P}^{G}(b)=e.$
\end{lemma}

\begin{proof} Clearly, the first equality holds since
$Br_{P_{j}}^{\GL_d(q)_{j}}(a^{{\nu}_j}_j)=c^{{\nu}_j}_j$ for each $1\leqslant j \leqslant w$.

For the second equality, we recall that the block $b=b_{\Delta_+,\Delta_-,w_+,w_-}$ has unipotent label $\Delta_+\times\Delta_-$
and semisimple label $s_{n_+,n_-}$ which has $2n_+$ eigenvalues 1 and $2n_-$ eigenvalues $-1$.
So, by Lemma \ref{FSlemma},
$Br_{P}^{G}(b)$ is a sum of all blocks
of the form $b_{s_0,\kappa}\otimes b_{s_+,\emptyset}$, where
 $s_0\times s_+$ is conjugate to $s_{n_+,n_-}$,
  $b_{s_0,\kappa}=b_{s_0,\Delta_+\times\Delta_-}$ is a  block of $C_0$ of defect zero
  and $b_{s_+,\emptyset}$ has defect group $P.$
  Since $m_{\pm}=\mathrm{rank}(\Delta_\pm)$,
  it follows that the semisimple label $s_0$ of $b_{s_0,\Delta_+\times\Delta_-}$
 is conjugate to $s_{m_+,m_-}$, and so $b_{s_0,\kappa}=b_{s_{m_+,m_-},\Delta_+\times\Delta_-}$.

 \smallskip

Note that $C_{\GL_d(q)_{j}}(P_{j})$ is indeed a Coxeter torus of $\GL_d(q)_{j}$ for each $1\leqslant j \leqslant w$.
So the semisimple element $s_+$ has the form
  $(\pm\id_d)_1\times (\pm\id_d)_2 \times\dots \times (\pm\id_d)_w$.
  However, since $n_\pm=m_\pm+dw_\pm,$ we conclude that
  $s_+$ is conjugate to $(\id_d)^{w_+}\times(-\id_d)^{w_-}$.
By definition, the blocks of the factors $C_{\GL_d(q)_{j}}(P_{j})$ with
 block idempotents $c_j^{\pm}$ are exactly those with
  label $((\pm\id_d)_j,\emptyset)$.
   Thus  $Br_{P}^{G}(b)=e$.
  \end{proof}

\begin{lemma}\label{lem:omn}  Suppose that ${\tuple \nu}\in \mathbb{J}^w$ with
$i_++i_-=i$.
Then  $$
Br_P^{G}(b_{i_+,i_-}^{{\tuple \nu_{[i+1,w]}}})=Br_P^{L_i}(b_{i_+,i_-}^{{\tuple \nu_{[i+1,w]}}})=
\sum_{\tuple\omega\in \mathbb{J}_{i_+,i_-}} b_{\Delta_+\times\Delta_-,0,0}\otimes c^{{\tuple\omega}}\otimes c^{{\tuple \nu}_{[i+1,w]}}.$$

\end{lemma}

\begin{proof} Clearly, we have  $C_G(P)<L_i$, so $Br_P^{G}(b_{i_+,i_-}^{{\tuple \nu_{[i+1,w]}}})=Br_P^{L_i}(b_{i_+,i_-}^{{\tuple \nu_{[i+1,w]}}})$. It follows that
$$Br_P^{L_i}(b_{i_+,i_-}^{{\tuple \nu}_{[i+1,w]}})= Br_{P_{1}\times\dots\times
P_{i}}^{G_{m+di}}(b_{\Delta_+,\Delta_-,i_+,i_-})\otimes \bigotimes_{j=i+1}^w Br_{P_{j}}^{\GL_d(q)_{j}}(a^{{\nu}_j}_j)
.$$
A similar argument as for Lemma \ref{lem:OGbOBrhom} shows that
 $$Br_{P_{1}\times\dots\times
P_{i}}^{G_{m+di}}(b_{\Delta_+,\Delta_-,i_+,i_-})=\sum_{\tuple\omega\in \mathbb{J}_{i_+,i_-}} b_{\Delta_+\times\Delta_-,0,0}\otimes c^{{\tuple\omega}}.$$
Hence the lemma holds.
\end{proof}

\begin{proposition}\label{prop:OGbONf}
The blocks $\mathcal{O}Gb$ and $\mathcal{O}Nf$  both have defect group $P$ and
are
Brauer correspondents.
\end{proposition}

\begin{proof}
The former part of the proposition follows from the fact that $N_G(P)\leqs N$ and
Proposition \ref{prop:defectgroup2}. For the latter part of the proposition,
we compute by Lemma \ref{lem:omn} that
$$Br_P^N(f)= Br_P^{L}(f)=\sum_{\tuple\nu\in  \mathbb{J}_{\tuple w}}b_{\Delta_+,\Delta_-,0,0}\otimes c^{{\tuple \nu}}=Br_P^G(b).$$
Hence, by Lemma \ref{lem:OGbOBrhom}, the block $\mathcal{O}Nf$ is the Brauer correspondent
of $\mathcal{O}Gb$ in $N$.
\end{proof}

%

\subsubsection{Two Morita equivalences}\label{ssub:moritas}
Here we give  Morita equivalences of the block $\mathcal{O}Nf$ and
of the block $\mathcal{O}N_G(P)e$  for later use.

\begin{lemma}\label{lem:Moritaf&a}
The blocks
$\mathcal{O}Nf$ and $\mathcal{O}((\GL_d(q).2\wr \frakS_{w_+})\times(\GL_d(q).2\wr \frakS_{w_-}))a^{\tuple \nu_0}$
are Morita equivalent.
\end{lemma}

\begin{proof} It is clear that $\mathcal{O}Nf$ and $\mathcal{O}Mf_0$
are Morita equivalent since they are Clifford correspondents.

\smallskip

Let $\theta$ be the unique character of $\mathcal{O}G_{m}b_{\Delta_+,\Delta_-,0,0}$.
Then $\theta$ is $\widetilde{G}_{m}$-invariant
since $\widetilde{G}_{m}\cong G_{m}\times \{\pm \id\}.$
 Hence $\theta$ extends to $\widetilde{G}_{m}$, and so
$\mathcal{O}\widetilde{G}_{m}b_{\Delta_+,\Delta_-,0,0}$ is the direct sum of 2 blocks, say
$\mathcal{O}\widetilde{G}_{m}b'_{\Delta_+,\Delta_-,0,0}$ and $\mathcal{O}\widetilde{G}_{m}b''_{\Delta_+,\Delta_-,0,0}$,
both of which are Morita equivalent to
$\mathcal{O}G_{m}b_{\Delta_+,\Delta_-,0,0}$ by \cite[Theorem 9.18]{CE04}.

\smallskip

Since $\widetilde{M}\cong M\times\{\pm\id\},$
$\mathcal{O}Mf_0$ is Morita equivalent to the block
of $\mathcal{O}\widetilde{M}$
 with block idempotent $b'_{\Delta_+,\Delta_-,0,0}\otimes a^{\tuple \nu_0}.$
However, the latter block is clearly Morita equivalent to
the  block of $\mathcal{O}((\GL_d(q).2\wr \frakS_{w_+})\times(\GL_d(q).2\wr \frakS_{w_-}))$
with block idempotent $a^{\tuple \nu_0}$ since $b'_{\Delta_+,\Delta_-,0,0}$ is a block of defect zero.
Now the lemma holds by the transitivity of Morita equivalences.
\end{proof}


Notice that $N_G(P)\cong (\widetilde{G}_m\times N_{\GL_d(q).2}(P_0)\wr \frakS_{w})\cap N$
and  $e=\sum_{\tuple\nu\in \mathbb{J}_{\tuple w}} b_{\Delta_+\times\Delta_-,0,0}\otimes c^{{\tuple\nu}}$
is the Brauer correspondent of $b$ in $N_G(P)$,
and the inertial group of $e_0$
in $N_G(P)$ is $(\widetilde{G}_m\times N_{\GL_d(q).2}(P_0)\wr \frakS_{w_+}\times N_{\GL_d(q).2}(P_0)\wr \frakS_{w_-})\cap N$.

\begin{lemma}\label{lem:Moritah&c}
The blocks $\mathcal{O}N_G(P)e$ and
$\mathcal{O}((N_1\wr \frakS_{w_+})\times(N_1\wr \frakS_{w_-}))c^{\tuple \nu_0}$
are Morita equivalent, where $N_1=N_{\GL_d(q).2}(P_0)$.
\end{lemma}

\begin{proof} Clearly, the blocks $e$ and $e_0$ are Clifford correspondents.
Hence the result follows by an analogous argument as for Lemma \ref{lem:Moritaf&a}.
\end{proof}

\subsection{Isolated RoCK blocks} \label{subsec:RouBlock}

In \cite{Ro98}, Rouquier defined an important kind of $\ell$-cores,
referred to as Rouquier cores.
 Chuang and Kessar \cite{chukes2002} proved Brou\'{e}'s conjecture
 for blocks of symmetric groups with a Rouquier core.
Nowadays such blocks are called RoCK blocks, named after Rouquier, Chuang and
Kessar.
Their analogs have been defined in the case of unipotent blocks of finite classical groups,
see Miyachi \cite{Mi01}, Turner \cite{turner2002} and Livesey \cite{Li12}.

\smallskip

In the case of isolated blocks,
we shall define and investigate isolated RoCK blocks  of $\SO_{2n+1}(q)$
at linear primes.
It will turn out that those with
abelian defect groups are derived equivalent to
their corresponding local blocks (see Theorem C or Theorem \ref{thm:Iso-RoCK}).

\subsubsection{Isolated RoCK blocks}\label{sec:rouquierblock}
We start with the definition of a Rouquier core.

\begin{definition} \label{def:rouquiercore}
A $d$-core partition $\lambda$ (resp. $d$-core symbol $\Delta)$
is said to be a \emph{Rouquier $d$-core}  with respect to $\omega$
if it has a $d$-abacus presentation (resp. $2d$-abacus presentation),
on which for each $i=1,...,d-1$ (resp. for each $i=1,...,d-1,d+1,...,2d-1$),
there are at least $\omega-1$ more beads on the  $i$-th runner than on the $(i-1)$-th runner.
\end{definition}

\begin{definition}\label{def:isoRoqB}
An isolated $\ell$-block $b_{\Delta_+,\Delta_-,w_+,w_-}$  of $\SO_{2n+1}(q)$
is called an \emph{isolated RoCK block} of $\SO_{2n+1}(q)$
if the symbol $\Delta_+$ is a Rouquier $d$-core  with respect to $w_+$
and the symbol $\Delta_-$ is a Rouquier $d$-core  with respect to $w_-$.
\end{definition}

Now we are going to prove Theorem C, which is restated here.

\begin{theorem} \label{thm:Iso-RoCK}
 Let $G=\SO_{2n+1}(q)$, where $q$ is odd.
Assume that $\ell$ is an odd linear prime with respect to $q$, and
that $b$ is an isolated RoCK $\ell$-block of $G$ with an abelian defect group $P$.
Then $b$ is derived equivalent to its Brauer correspondent.
\end{theorem}

\noindent\textbf{SETUP.} In order to prove Theorem \ref{thm:Iso-RoCK}, from now on
we will work with the following setup:
Let $b:=b_{\Delta_+,\Delta_-,w_+,w_-}$ be
a (fixed) isolated RoCK block of $G$ which has  abelian defect groups.
Then all results in \S \ref{sub:iso-blocksabelian} are applicable to $b$.
We keep the notation there, especially $m_{\pm}=\mathrm{rank}(\Delta_\pm)$,
$n_\pm=m_\pm+dw_\pm$, $w=w_++w_-$, $m=m_++m_-$ and $n=n_++n_-$,
 and fix
a $2d$-abacus representation of $\Delta_\pm$ as in Definition \ref{def:rouquiercore}.

\smallskip
\subsubsection{$\mathcal{O}$-split monomorphisms of algebras} \label{ssub:mono-algebra}

By Proposition \ref{prop:OGbONf} and \cite[Ch. 14, Theorem2]{alperi1986}), we see
that the $\mathcal{O}(G\times G)$-module $\mathcal{O}Gb$ and the
$\mathcal{O}(N\times N)$-module $\mathcal{O}Nf$ both have vertex $\Delta(P)$ and
are Green correspondents. Let $X$ be the Green correspondent of $\mathcal{O}Gb$
in $G\times N$,
so that $X$ is the unique indecomposable summand of $Res_{G\times
N}^{G\times G}(\mathcal{O}Gb)$ with vertex $\Delta (P)$.
Noting that $\mathcal{O}Nf$ is
 the unique indecomposable summand of $Res_{N\times N}^{G\times N}(X)$ with
vertex $\Delta (P)$, we have $Xf\neq0$, so $Xf=X$ and
$X$ is an $(\mathcal{O}Gb$,$\mathcal{O}Nf)$-bimodule.

\smallskip

We now define  functors from $\mathcal{O}L$-mod to $\mathcal{O}G$-mod by
recursively doing Harish-Chandra induction and taking block components.
Let ${\tuple \nu}\in \mathbb{J}_{\tuple w}$, and let
$$b_i^{{\tuple \nu}}= b_{\Delta_+,\Delta_-,\sum_{j=1}^{i}\sigma_+(\nu_{j}),
\sum_{j=1}^{i}\sigma_-(\nu_{j})}\otimes a_{i+1}^{\nu_{i+1}}\otimes
\cdots\otimes a_{w}^{\nu_w}$$
be a
block idempotent of $\mathcal{O}L_i$, where
$$\sigma_+( \nu_j) =\begin{cases}1&\text{if ${\nu_j}=+$}\\
0&\text{if $\nu_j=-$}
\end{cases}~\mbox{and}~\sigma_-(\nu_j) =\begin{cases}0&\text{if $\nu_j=+$}\\
1&\text{if $\nu_j=-$}
\end{cases}$$  are functions  defined on $I=\{+,-\}$.
We shall shortly write $b_i$ for $b_i^{{\tuple \nu}_0}$. In particular, by definition, we have  $b_w^{{\tuple \nu}}=b$ and $b_0^{{\tuple \nu}_0}=f_0.$

\smallskip

For $0\leqs i\leqs w-1$, let $U_i$ be the subgroup of
$L_{i+1}$  presented by
\begin{gather*}
\left(\begin{array}{ccccc}
 \id_{d(w-i-1)} &               &          &               &\\
                & \id_d   & *   & * &\\
                &          & \id_{G_{m+di}} & *              &\\
                &          &          & \id_d &\\
                &          &        &      & \id_{d(w-i-1)}\\
\end{array}\right).
\end{gather*}
Let $U=U_0U_1\cdots U_{w-1}$ so that $U$ is the unipotent radical of the
parabolic subgroup of $G$ containing $L$ as a Levi complement.
As before, we write $e_{U_i}= \frac{1}{|U_{i}|} \sum_{u\in U_i}u$. Note
that $|U_i|$ is a power of $q^d$, equal to 1 mod $\ell$.
Then $e_{U}= e_{U_0}e_{U_1}\cdots\e_{U_{w-1}}$.
Let $$Z^{{\tuple \nu}}=\ _GZ^{{\tuple \nu}}_L: =\mathcal{O}Gb^{{\tuple \nu}}_we_{U_{w-1}}b^{{\tuple \nu}}_{w-1}\cdots e_{U_{0}}b^{{\tuple \nu}}_{0}=\mathcal{O}Ge_{U}b^{{\tuple \nu}}_{w}b^{{\tuple \nu}}_{w-1}\cdots b^{{\tuple \nu}}_0,$$
which is an $(\mathcal{O}Gb$,$\mathcal{O}Lf)$-bimodule.
Observe that $b^{{\tuple \nu}}_i$ and $e_{U_i}$
are all idempotents in $(\mathcal{O}G)^L$, which commute with each
other.
Hence their product is an idempotent in $(\mathcal{O}G)^L$,
and so $Z^{{\tuple \nu}}$ is a direct summand of $_G\mathcal{O}Gb_L$.
The functor
$Z^{{\tuple \nu}}\otimes_{\mathcal{O}L}-$ from $\mathcal{O}L$-mod to $\mathcal{O}G$-mod is exactly
$R^{L_{w},b^{{\tuple \nu}}_w}_{L_{w-1},b^{{\tuple \nu}}_{w-1}}\cdots R^{L_1,b^{\tuple\nu}_1}_{L_0,b^{{\tuple \nu}}_0}$,
where  $R^{L_j,b^{\tuple\nu}_j}_{L_{j-1},b^{{\tuple \nu}}_{j-1}}$ is as defined in \S\ref{subsec:HCseries}.

\smallskip

Since $\frakS_w$ transitively acts on $\mathbb{J}_{\tuple w}$ in a natural way,
we see that $b_i^{{\tuple \nu}}$ and $b_i^{{\tuple \nu}_0}$  are conjugate
for each $1\leqs i\leqs w$ and any ${\tuple \nu}\in \mathbb{J}_{\tuple w}$.
Hence
all $Z^{{\tuple \nu}}$ are isomorphic as $\mathcal{O}G$-modules
since Harish-Chandra induction is independent of choices of unipotent radicals.

Finally, we set
 $$Y:=\bigoplus_{\tuple{ \nu}\in  \mathbb{J}_{\tuple w}}Z^{{\tuple \nu}}$$
 and write  $Z:= Z^{{\tuple \nu}_0}$ for brevity.
Notice that  all of idempotents $b^{{\tuple \nu}}_{w}b^{{\tuple \nu}}_{w-1}\cdots b^{{\tuple \nu}}_0$ for $\tuple{ \nu}\in  \mathbb{J}_{\tuple w}$
are orthogonal to each
other. Hence $Y=\mathcal{O}G(\sum_{\tuple{ \nu}\in  \mathbb{J}_{\tuple w}}e_Ub^{{\tuple \nu}}_{w}b^{{\tuple \nu}}_{w-1}\cdots b^{{\tuple \nu}}_0),$ and
so $Y$ is also a direct summand of $_G\mathcal{O}Gb_L$.

\begin{proposition} \label{prop:splitmono}
There is a sequence of $\mathcal{O}$-split monomorphisms of algebras
\begin{equation}\label{equation:monos}
\mathcal{O}Nf\hookrightarrow \operatorname{End}_{\mathcal{O}G}(X)\hookrightarrow \operatorname{End}_{\mathcal{O}G}(Y).
\end{equation}
Also, the left $\mathcal{O}Gb$-module $X$ is a progenerator for $\mathcal{O}Gb$.
\end{proposition}

\begin{proof} As mentioned above, $_G\mathcal{O}Gb_G$ and $_GX_N$ are Green correspondents.
Hence $_G\mathcal{O}Gb_G$ is isomorphic to a direct summand of $Ind_{G\times
N}^{G\times G}(_GX_N)$.
It follows that $_G\mathcal{O}Gb$ is a direct summand of $[G:N]$ copies of $_GX$, and
so $_GX$ is a progenerator for
$\mathcal{O}Gb$.

\smallskip

Multiplying $X$ to the
right, we obtain an $\mathcal{O}$-split homomorphism of algebras
$\mathcal{O}Nf\rightarrow \operatorname{End}_{\mathcal{O}G}(X)$
which is monomorphic since $_N\mathcal{O}Nf_N$ is a direct summand of $Res_{N\times
N}^{G\times N}(_GX_N)$.

\smallskip

Since
$(G\times L)\unlhd (G\times N)$, it follows that $Res_{G\times L}^{G\times N}(_GX_N)$
has an indecomposable summand $M$
such that $Res_{G\times L}^{G\times
N}(_GX_N)$ is a direct sum of conjugates of $M$ in $G\times N$.
It is possible to choose  a set of coset representatives of
$(G\times L)$ in $(G\times N)$ that
all normalize $\Delta (P)$.
So $Res_{G\times L}^{G\times N}(_GX_N)$ is the sum
of indecomposable modules all with vertex $\Delta (P)$.

\smallskip

Now, using the Brauer homomorphism, we have
$$
\begin{array}{clr}
Y(\Delta (P))&=(\mathcal{O}G\sum_{\tuple{\tuple \nu}\in {\mathbb{J}_{\tuple w}}}b^{{\tuple \nu}}_we_{U_{w-1}}b^{{\tuple \nu}}_{w-1}\cdots e_{U_{0}}b^{{\tuple \nu}}_{0})(\Delta (P)) \\ \\
&=kC_G(P)\sum_{\tuple{\tuple \nu}\in {\mathbb{J}_{\tuple w}}}Br_{P}^G(b^{{\tuple \nu}}_we_{U_{w-1}}b^{{\tuple \nu}}_{w-1}\cdots e_{U_{0}}b^{{\tuple \nu}}_{0})\\ \\
&=kC_G(P)\sum_{\tuple{\tuple \nu}\in {\mathbb{J}_{\tuple w}}}
Br_{P}^G(b^{{\tuple \nu}}_0)Br_{P}^{G}(b^{{\tuple \nu}}_1)\dots Br_{P}^{G}(b^{{\tuple \nu}}_w) & (\text{since\,} Br_{P}^G(e_{U_i})=1) \\ \\
&=kC_G(P)\sum_{\tuple{\tuple \nu}\in {\mathbb{J}_{\tuple w}}}Br_{P}^{G}(b^{{\tuple \nu}}_w)  & (\text{Lemma } \ref{lem:omn})        \\ \\
&=kC_G(P)\sum_{\tuple{\tuple \nu}\in {\mathbb{J}_{\tuple w}}} b_{\Delta_+,\Delta_-,0,0}\otimes c^{{\tuple \nu}} & (\text{Lemma }\ref{lem:omn})  \\ \\
&=kC_G(P)\sum_{\tuple{\tuple \nu}\in {\mathbb{J}_{\tuple w}}}Br_{P}^G(b)\\ \\
&=\mathcal{O}Gb(\Delta (P)).
\end{array}
$$

Hence  $_GY_L$ has  all direct summands of $_G\mathcal{O}Gb_L$ with vertex
containing $\Delta (P)$. As shown above, we conclude that
$_GX_L$ is a direct summand of $_GY_L$. Thus there is
 an $\mathcal{O}$-split monomorphism
 $\operatorname{End}_{\mathcal{O}G}(X)\hookrightarrow \operatorname{End}_{\mathcal{O}G}(Y)$.
 \end{proof}

 \smallskip

\subsubsection{The $K$-dimension of $\operatorname{End}_{KG}(K\otimes_\mathcal{O}Z)$}
\label{ssub:K-dimensionZ}

Here we compute the $K$-dimension of $\operatorname{End}_{KG}(K\otimes_\mathcal{O}Z)$,
using a result of Chuang-Kessar \cite[Lemma 4 (2)]{chukes2002}.

\begin{lemma}\label{lem:aba}
Let $\sigma$ be a Rouquier $d$-core partition  with respect to $\omega$.
Let $\lambda=(\lambda_1,\lambda_2,\dots)$ be a partition with $d$-core $\sigma$
and weight $v\leqs \omega$, and $\mu=(\mu_1,\mu_2,\ldots )$ be a partition with $d$-core $\sigma$ and weight $v-1$.
 Suppose that $\mu_i\leq\lambda_i$
for all $i$. Then $\mu$ is obtained by
removing a $d$-hook from $\lambda$. If this removal occurs on the $\alpha$th
runner then the complement of the Young diagram of $\mu$ in that of $\lambda$ is
the Young diagram of the hook partition $(\alpha+1,1^{(d-\alpha-1)})$.\qed
\end{lemma}

\begin{proposition}\label{prop:dimZ}
The $K$-dimension of $\operatorname{End}_{KG}(K\otimes_\mathcal{O}Z)$ is
$2^ww_{+}!w_-!dim_K(KLf_0)$.
\end{proposition}

\begin{proof}
We shall prove the proposition on the level of character, following the proof
of \cite[Proposition 9.1.4]{Li12}.
That is, if the character of $K\otimes_\mathcal{O}Z$ is $\sum_{\chi\in \Irr(G)}m_\chi \chi$,
then the dimension to compute is
exactly $\sum_{\chi\in \Irr(G)}m_\chi^2$.

\smallskip
Observe that
$K\otimes_\mathcal{O}Z\cong(K\otimes_\mathcal{O}Z)\otimes_{KLf_0}KLf_0$. We write
$R_Z(\varphi)$ for its character, where $R_Z$
 denotes $R^{L_w,b_w}_{L_{w-1},b_{w-1}}\cdots R^{L_1,b_1}_{L_0,b_0}
$ and
 $\varphi$ is the character of $KLf_0$ as a character of $L$.
We shall first
find the multiplicity of each irreducible constituent of $R_Z(\varphi)$
appeared in $R_Z(\varphi)$, and then
sum up all of their squares to obtain the desired dimension.
Recall that  $f_0= b_{\Delta_+\times\Delta_-,0,0}\otimes a^{{\tuple \nu}_0}$.
The irreducible characters in the block $KLf_0$ are of the form
$$\chi:=\chi_{\Delta_+,\Delta_-}\otimes\chi_{s_1,{\lambda_1}}\otimes\cdots\otimes\chi_{s_{w_+},{\lambda_{w_+}}}\otimes
\chi_{s'_{w_++1},{\lambda'_{w_++1}}}\otimes\cdots
\otimes\chi_{s'_w,{\lambda'_w}}
$$
where $\chi_{\Delta_+,\Delta_-}:=\chi_{s_{m_+,m_-},\Delta_+\times\Delta_-}$ and
for $1\leqs i\leqs w_+$ either
\begin{itemize}
  \item $s_i$ is the identity matrix $\id_d$ of $\GL_d(q)$ and $\lambda_i$ is a $d$-hook partition, or
  \item $s_i$ is a non-trivial $\ell$-element of $\GL_d(q)$
   and $\lambda_i$ is the partition $(1)$,
\end{itemize}
for $1\leqs j\leqs w_-$ either
\begin{itemize}
  \item  $s'_{w_++j}=-\id_d$ and $\lambda'_{w_++j}$ is an $d$-hook partition, or
  \item  $s'_{w_++j}$ is the product of $-\id_d$ and a non-trivial $\ell$-element of
  $\GL_d(q)$ and $\lambda_i$ is the partition $(1)$.
\end{itemize}

\smallskip

Now we make the following notation:
\begin{itemize}[leftmargin=8mm]
  \item  The matrices $s_i$s are distributed into different conjugacy classes
   with \begin{itemize}
          \item $r_0$ ones equal to $\id_d$,
          \item $\alpha_1$ ones conjugate to $t_1$, $\beta_1$ ones conjugate to $t_{1}^{-1}$,
          \item $\alpha_2$ ones conjugate to $t_2$, $\beta_2$ ones conjugate to $t_{2}^{-1}$, etc.
        \end{itemize}
    \item The partitions $\lambda_i$ for $s_i=\id_d$ are distributed into $d$ classes with
         \begin{itemize}
           \item $l_0$ ones equal to $(1^d)$,
           \item $l_1$ ones equal to $(2,1^{d-2})$, \ldots,
           \item $l_{d-1}$ ones equal to $(d)$, where $\sum_il_i=r_0$.
         \end{itemize}
    \item  The matrices  $s_i'$s are distributed into different conjugacy classes
   with
   \begin{itemize}
     \item $r_0'$ ones equal to $-\id_d$,
     \item $\alpha_1'$ ones conjugate to $t_1'$, $\beta_1'$ ones conjugate to
     $(t_{1}')^{-1}$,
     \item $\alpha_2'$ ones conjugate to $t_2'$, $\beta_2'$ ones conjugate to
     $(t_{2}')^{-1}$, etc.
   \end{itemize}
  \item The partitions $\lambda_i'$s for $s_i=-\id_d$ are distributed into $d$ classes with
     \begin{itemize}
       \item  $l_0'$ ones equal to $(1^d)$,
       \item   $l_1'$ ones equal to $(2,1^{d-2})$,  \ldots,
       \item  $l_{d-1}'$ ones equal to $(d)$, where $\sum_il_i'=r_0'$.
     \end{itemize}
\end{itemize}

Write $r_i=\alpha_i+\beta_i$ and $r'_i=\alpha'_i+\beta'_i$ so that $\sum_ir_i=w_+$ and $\sum_ir'_i=w_-$.

\smallskip

To compute $R_Z(\chi)$,
we shall take a step-by-step strategy by recursively using
 the Littlewood-Richardson coefficients and
the branching rules. Each step will be essentially from
$G_{m+di}\times \GL_d(q)$ to $G_{m+d(i+1)}$
with the characters of other factors (i.e., copies of $\GL_d(q)$) kept untouched.
 This is similar to the way as in \cite[\S 7.2]{Li12} for unipotent characters,
depending on Lemma \ref{lem:aba}.
 To do this, we define a block $b_{\Delta_+,\Delta_-,i}$ of $G_{m+di}$ by $b_{\Delta_+,\Delta_-,i}:=b_{\Delta_+,\Delta_-,i,0}$ for $0\leqs i\leqs w_+$ and
$b_{\Delta_+,\Delta_-,i}:=b_{\Delta_+,\Delta_-,w_+,i-w_+}$ for $w_+ < i\leqs w.$
For a quadratic unipotent character $\chi_{\Lambda_+,\Lambda_-}$
of
$G_{m+di}$ such that $\Lambda_\pm$ has $d$-core $\Delta_\pm$,
Lemma \ref{lem:aba} shows that for $0\leqs i\leqs w_+-1$
$$\begin{array}{l}
  R_{G_{m+di}\times\GL_d(q)_{i+1},b_{\Delta_+,\Delta_-,i}\otimes a^+_{i+1}
}^{G_{m+d(i+1)},b_{\Delta_+,\Delta_-,i+1}}(\chi_{\Lambda_+,\Lambda_-}\otimes\chi_{1,(\alpha+1,1^{(d-\alpha-1)})}
) \\
=\chi_{\Theta_+',\Lambda_-}+\chi_{\Theta_+'',\Lambda_-},
\end{array}
$$
 where
$\Theta_+'$ and $\Theta_+''$ are symbols
obtained from $\Lambda_+$ by sliding a bead $1$ place down  $\alpha$th and
$(\alpha+d)$th runner, respectively.
Similarly, for $w_+ \leqs i\leqs w-1,$
$$\begin{array}{l}
R_{G_{m+di}\times\GL_d(q)_{i+1},b_{\Delta_+,\Delta_-,i}\otimes a^-_{i+1}
}^{G_{m+d(i+1)},b_{\Delta_+,\Delta_-,i+1}}(\chi_{\Lambda_+,\Lambda_-}
\otimes\chi_{-1,(\alpha+1,1^{(d-\alpha-1)})}) \\
=\chi_{\Lambda_+,\Theta_-'}+\chi_{\Lambda_+,\Theta_-''},
\end{array}
$$
 where
$\Theta_-'$ and $\Theta_-''$ are bipartitions
obtained from $\Lambda_+$ by sliding a bead $1$ place down  $\alpha$th and
$(\alpha+d)$th runner, respectively.

\smallskip

Suppose that we slide single beads $j$ times down the $l$th runner of a core after $w$ steps,
in which the bottom bead has been moved down $\sigma_1^l$ times,
the second bottom bead has been moved down $\sigma_2^l$
times, etc, with $\sigma_1^l\geqs \sigma_2^l\geqs\cdots$ and
$\sum_i\sigma_i^l=j$.
The number of ways of sliding single beads from the beginning to the end
is equal to the number of ways of writing the numbers
$1,\ldots,j$ in the Young diagram of the partition $\sigma^l:=(\sigma_1^l,\sigma_2^l,\ldots)$
such that numbers increase across rows and down columns.
This is exactly the same as the degree of the
character $\zeta^{\sigma^l}$ of the symmetric group $\mathfrak{S}_j$ corresponding to
the partition $\sigma^l$.
Hence, we have

\begin{align*}
R_Z(\chi)=&
\sum_{\widehat{\Upsilon}}{l_0\choose|\sigma^0|}\dim\zeta^{\sigma^0}\dim\zeta^{\tau^0}
\cdots{l_{d-1}\choose|\sigma^{d-1}|}\dim\zeta^{\sigma^{d-1}}\dim\zeta^{\tau^{d-1}}\\
&{l'_0\choose|\sigma'^0|}\dim\zeta^{\sigma'^0}\dim\zeta^{\tau'^0}
\cdots{l'_{d-1}\choose|\sigma'^{d-1}|}\dim\zeta^{\sigma'^{d-1}}\dim\zeta^{\tau'^{d-1}} \\
&
\dim\zeta^{\nu^1}\dim\zeta^{\nu'^1}\dim\zeta^{\nu^2}\dim\zeta^{\nu'^2}\cdots\\
&\chi(s_{m_+,m_-}\times s_1\times \overline{s_1}\times\cdots\times s_{w_+}\times
\overline{s_{w_+}}\\
&\times s'_{w_++1}\times
\overline{s'_{w_++1}}\times\cdots\times s'_{w}\times
\overline{s'_{w}},\mu)
\end{align*}
where
$\widehat{\Upsilon}$ represents
$$
\begin{array}{cc}
 l_0+\cdots+l_{d-1}=r_0 & |\sigma^i|+|\tau^i|=l_i, |\nu^i|=r_i \\
l'_0+\cdots+l'_{d-1}=r'_0 & |\sigma'^i|+|\tau'^i|=l'_i, |\nu'^i|=r'_i,
\end{array}
$$
and
each part of $\mu$ is as follows:
\begin{itemize}[leftmargin=8mm]
  \item  $\mu_{\Gamma_i}=\nu^i$ with $\Gamma_i$ the minimal polynomial of $t_i\times t_i^{-1}$,
   \item  $\mu_{\Gamma'_i}=\nu'^i$ with $\Gamma'_i$ the minimal
polynomial of $t'_i\times {t'_i}^{-1}$,
  \item  $\mu_{X-1}$ is the symbol whose $d$-core is $\Delta_+$ and whose
$d$-quotients are $[\sigma^0,\ldots,\sigma^{d-1}]$ and $[\tau^0,\ldots,\tau^{d-1}]$
with respect to the $2d$-abacus representation of $\Delta_+$, and
  \item  $\mu_{X+1}$ is the symbol whose $d$-core is $\Delta_-$ whose
$d$-quotients are $[\sigma'^0,\ldots,\sigma'^{d-1}]$ and
$[\tau'^0,\ldots,\tau'^{d-1}]$
with respect to the $2d$-abacus representation of $\Delta_-$.
\end{itemize}

\smallskip

To express $R_Z(\varphi)$,
we notice that we get the same character of $G$ from $R_Z$
after permutating the $\lambda_i$s of $\chi$.
 The number of those permutations is$$\frac{w_+!}{l_0!l_1!\cdots
l_{d-1}!\alpha_1!\beta_1!\alpha_2!\beta_2!\cdots}\frac{w_-!}{l'_0!l'_1!\cdots
l'_{d-1}!\alpha'_1!\beta'_1!\alpha'_2!\beta'_2!\cdots}.$$
Since  $\varphi=
\sum_{\chi\in {\rm Irr}(\mathcal{O}Lf_0)}\chi(1)\chi,$
 we have
\begin{align*}
R_Z(\varphi)=\sum_{\Upsilon^+}&\frac{w_+!}{l_0!l_1!\cdots
l_{d-1}!\alpha_1!\beta_1!\alpha_2!\beta_2!\cdots}\frac{w_-!}{l'_0!l'_1!\cdots
l'_{d-1}!\alpha'_1!\beta'_1!\alpha'_2!\beta'_2!\cdots}\\
&\dim(\chi_{\Delta_+,\Delta_-}\otimes\chi_{s_1,{\lambda_1}}\otimes\cdots\otimes\chi_{s_{w_+},{\lambda_{w_+}}}
\otimes
\chi_{s'_{w_++1},{\lambda_{w_++1}}}\otimes\cdots
\otimes\chi_{s'_w,{\lambda_w}})\\
&{l_0\choose|\sigma^0|}\dim\zeta^{\sigma^0}\dim\zeta^{\tau^0}\cdots{l_{d-1}
\choose|\sigma^{d-1}|}\dim\zeta^{\sigma^{d-1}}\dim\zeta^{\tau^{d-1}}\\
&{l'_0\choose|\sigma'^0|}\dim\zeta^{\sigma'^0}\dim\zeta^{\tau'^0}\cdots{l'_{d-1}
\choose|\sigma'^{d-1}|}\dim\zeta^{\sigma'^{d-1}}\dim\zeta^{\tau'^{d-1}}\\
&\dim\zeta^{\nu^1}\dim\zeta^{\nu'^1}\dim\zeta^{\nu^2}\dim\zeta^{\nu'^2}\cdots\\
&\chi(s_{m_+,m_-}\times s_1\times \overline{s_1}\times\cdots\times s_{w_+}\times
\overline{s_{w_+}}\\ &\times s'_{w_++1}\times
\overline{s'_{w_++1}}\times\cdots\times s'_{w}\times
\overline{s'_{w}},\mu),
\end{align*}
where $\Upsilon^+=\Upsilon \cup \Upsilon'$ with
 $\Upsilon'=\{\alpha_i+\beta_i=r_i, \alpha'_i+\beta'_i=r_i'\}$ and
 $$\Upsilon=
\left\{\begin{array}{ll}
l_0+\cdots+l_{d-1}+r_1+r_2+\cdots=w_+, & |\sigma^i|+|\tau^i|=l_i~\mbox{and}\\
l'_0+\cdots+l'_{d-1}+r'_1+r'_2+\cdots=w_-, & |\sigma'^i|+|\tau'^i|=l'_i.
\end{array}\right.
$$

Splitting $\Upsilon^+$ into $\Upsilon$ and $\Upsilon'$, we obtain

\begin{align*}
R_Z(\varphi)=\sum_{\Upsilon}&\frac{w_+!}{l_0!l_1!\cdots l_{d-1}!r_1!r_2!\cdots}\frac{w_-!}{l'_0!l_1!\cdots l'_{d-1}!r'_1!r'_2!\cdots}\\
&{l_0\choose|\sigma^0|}\dim\zeta^{\sigma^0}\dim\zeta^{\tau^0}\cdots{l_{d-1}
\choose|\sigma^{d-1}|}\dim\zeta^{\sigma^{d-1}}\dim\zeta^{\tau^{d-1}}\\
&{l'_0\choose|\sigma'^0|}\dim\zeta^{\sigma'^0}\dim\zeta^{\tau'^0}\cdots{l'_{d-1}
\choose|\sigma'^{d-1}|}\dim\zeta^{\sigma'^{d-1}}\dim\zeta^{\tau'^{d-1}}\\
&\dim\zeta^{\nu^1}\dim\zeta^{\nu'^1}\dim\zeta^{\nu^2}\dim\zeta^{\nu'^2}\cdots\\
&[\sum_{\Upsilon'}{r_1\choose\alpha_1}{r'_1\choose\alpha'_1}{r_2\choose\alpha_2}{r'_2\choose\alpha'_2}\cdots
\dim(\chi_{\Delta_+,\Delta_-}\otimes\\
&\chi_{s_1,{\lambda_1}}\otimes\cdots\otimes\chi_{s_{w_+},{\lambda_{w_+}}}\otimes
\chi_{s'_{w_++1},{\lambda_{w_++1}}}\otimes\cdots
\otimes\chi_{s'_w,{\lambda_w}})]\\
&\chi(s_{m_+,m_-}\times s_1\times \overline{s_1}\times\cdots\times s_{w_+}\times
\overline{s_{w_+}}\\ &\times s'_{w_++1}\times
\overline{s'_{w_++1}}\times\cdots\times s'_{w}\times
\overline{s'_{w}},\mu).
\end{align*}
So
the dimension ${\rm Dim}$ of $\operatorname{End}_{KG}((K\otimes_\mathcal{O}Z)\otimes_KKLf)$ over $K$ is
\begin{align*}
\sum_{\Upsilon}&(\frac{w_+!}{l_0!l_1!\cdots l_{d-1}!r_1!r_2!\cdots}\frac{w_-!}{l'_0!l_1!\cdots l'_{d-1}!r'_1!r'_2!\cdots})^2\\
&{l_0\choose|\sigma^0|}^2(\dim\zeta^{\sigma^0})^2(\dim\zeta^{\tau^0})^2\cdots{l_{d-1}
\choose|\sigma^{d-1}|}^2(\dim\zeta^{\sigma^{d-1}})^2(\dim\zeta^{\tau^{d-1}})^2\\
&{l'_0\choose|\sigma'^0|}^2(\dim\zeta^{\sigma'^0})^2(\dim\zeta^{\tau'^0})^2\cdots{l'_{d-1}
\choose|\sigma'^{d-1}|}^2(\dim\zeta^{\sigma'^{d-1}})^2(\dim\zeta^{\tau'^{d-1}})^2\\
&(\dim\zeta^{\nu^1})^2(\dim\zeta^{\nu'^1})^2(\dim\zeta^{\nu^2})^2(\dim\zeta^{\nu'^2})^2\cdots\\
&[\sum_{\Upsilon'}{r_1\choose\alpha_1}{r'_1\choose\alpha'_1}{r_2\choose\alpha_2}{r'_2\choose\alpha'_2}\cdots
\dim(\chi_{\Delta_+,\Delta_-}\otimes\\
&\chi_{s_1,{\lambda_1}}\otimes\cdots\otimes\chi_{s_{w_+},{\lambda_{w_+}}}\otimes
\chi_{s'_{w_++1},{\lambda_{w_++1}}}\otimes\cdots
\otimes\chi_{s'_w,{\lambda_w}})]^2,
\end{align*}
which is equal to
\begin{align*}
\sum_{\Upsilon}&
(\frac{w_+!}{l_0!l_1!\cdots
l_{d-1}!r_1!r_2!\cdots})^2{l_0\choose\sigma_0}^2\cdots{l_{d-1}\choose\sigma_{d-1}}
^2\\&
(\frac{w_-!}{l'_0!l'_1!\cdots
l'_{d-1}!r'_1!r'_2!\cdots})^2{l'_0\choose\sigma'_0}^2\cdots{l'_{d-1}\choose\sigma'_{d-1}}
^2\\
&[\sum_{\sigma^i\vdash\sigma_i,\tau^i\vdash\tau_i}(\dim\zeta^{\sigma^0}
)^2(\dim\zeta^{\tau^0})^2\cdots(\dim\zeta^{\sigma^{d-1}})^2(\dim\zeta^{\tau^{d-1}}
)^2]\\
&[\sum_{\sigma'^i\vdash\sigma'_i,\tau'^i\vdash\tau'_i}(\dim\zeta^{\sigma'^0}
)^2(\dim\zeta^{\tau'^0})^2\cdots(\dim\zeta^{\sigma'^{d-1}})^2(\dim\zeta^{\tau'^{d-1}}
)^2]\\
&(\dim\zeta^{\nu^1})^2(\dim\zeta^{\nu'^1})^2(\dim\zeta^{\nu^2})^2(\dim\zeta^{\nu'^2})^2\cdots\\
&[\sum_{\Upsilon'}{r_1\choose\alpha_1}{r'_1\choose\alpha'_1}{r_2\choose\alpha_2}{r'_2\choose\alpha'_2}\cdots
\dim(\chi_{s_{m_+,m_-},{\Delta_+\times\Delta_-}}\otimes\\
&\chi_{s_1,{\lambda_1}}\otimes\cdots\chi_{s_{w_+},{\lambda_{w_+}}}\otimes
\chi_{s'_{w_++1},{\lambda_{w_++1}}}\otimes\cdots
\otimes\chi_{s'_w,{\lambda_w}})]^2.
\end{align*}
It follows from the fact $\sum_{\sigma\vdash h}\dim(\zeta^\sigma)^2=h!$ that
\begin{align*}
{\rm Dim}=\sum_{\Upsilon}&(\frac{w_+!}{l_0!l_1!\cdots
l_{d-1}!r_1!r_2!\cdots})^2{l_0\choose\sigma_0}^2\cdots{l_{d-1}\choose\sigma_{d-1}}
^2\sigma_0!\tau_0!\cdots\sigma_{d-1}!\tau_{d-1}!r_1!r_2!\cdots\\
&(\frac{w_-!}{l'_0!l'_1!\cdots
l'_{d-1}!r'_1!r'_2!\cdots})^2{l'_0\choose\sigma'_0}^2\cdots{l'_{d-1}\choose\sigma'_{d-1}}
^2\sigma'_0!\tau'_0!\cdots\sigma'_{d-1}!\tau'_{d-1}!r'_1!r'_2!\cdots\\
&[\sum_{\Upsilon'}{r_1\choose\alpha_1}{r'_1\choose\alpha'_1}{r_2\choose\alpha_2}{r'_2\choose\alpha'_2}\cdots
\dim(\chi_{\Delta_+,\Delta_-})\\
&\dim(\chi_{s_1,{\lambda_1}}) \cdots\dim(\chi_{s_{w_+},{\lambda_{w_+}}})\dim(\chi_{s'_{w_++1},{\lambda_{w_++1}}})\cdots
\dim(\chi_{s'_w,{\lambda_w}})]^2.
\end{align*}

By \cite[Lemma 7.3.1]{Li12}, we know that $\dim(\chi_{t,(1)})=\dim(\chi_{t^{-1},(1)})$,
so for fixed $(r_1,r_2,\ldots)$,
$$\dim(\chi_{\Delta_+,\Delta_-})\dim(\chi_{s_1,\lambda_1})\cdots \dim(\chi_{s_w,\lambda_w})$$
keeps constant when $\alpha_i$ varies. Combining this and the fact that $\sum_{i=0}^r {r\choose i}=2^r$,
we get

\begin{align*}
{\rm Dim}=w_+!w_-!\sum_{\Upsilon}&\frac{w_+!}{l_0!l_1!\cdots
l_{d-1}!r_1!r_2!\cdots}{l_0\choose\sigma_0}\cdots{l_{d-1}\choose\sigma_{d-1}}\\
&\frac{w_-!}{l'_0!l'_1!\cdots
l'_{d-1}!r'_1!r'_2!\cdots}{l'_0\choose\sigma'_0}\cdots{l'_{d-1}\choose\sigma'_{d-1}}\\
&2^{\sum
r_i}2^{\sum
r'_i}[\sum_{\Upsilon'}{r_1\choose\alpha_1}{r'_1\choose\alpha'_1}{r_2\choose\alpha_2}{r'_2\choose\alpha'_2}\cdots
\\
&\dim(\chi_{\Delta_+,\Delta_-}\otimes\chi_{s_1,{\lambda_1}}\otimes\cdots\chi_{s_{w_+},{\lambda_{w_+}}}\otimes\\
&\chi_{s'_{w_++1},{\lambda_{w_++1}}}\otimes\cdots
\otimes\chi_{s'_w,{\lambda_w}})^2],
\end{align*}
and so
\begin{align*}
{\rm Dim}=&w_+!w_-!\sum_{\Upsilon_0}\frac{w_+!}{l_0!l_1!\cdots l_{d-1}!r_1!r_2!\cdots}\frac{w_-!}{l'_0!l'_1!\cdots l'_{d-1}!r'_1!r'_2!\cdots}\\
&2^{\sum r_i}2^{\sum
l_i}2^{\sum r'_i}2^{\sum
l'_i}[\sum_{\Upsilon'}{r_1\choose\alpha_1}{r'_1\choose\alpha'_1}{r_2\choose\alpha_2}{r'_2\choose\alpha'_2}\cdots
\\
&\dim(\chi_{\Delta_+,\Delta_-}\otimes\chi_{s_1,{\lambda_1}}\otimes\cdots\chi_{s_{w_+},{\lambda_{w_+}}}\otimes\\
&\chi_{s'_{w_++1},{\lambda_{w_++1}}}\otimes\cdots
\otimes\chi_{s'_w,{\lambda_w}})^2],
\end{align*}
where $\Upsilon_0=\{l_0+\cdots+r_1+\cdots=w_+, l'_0+\cdots+r_1'+\cdots=w_-\}$.
If we write
$$ \Upsilon_1=\left\{\begin{array}{l}
  l_0+\cdots+l_{d-1}+\alpha_1+\beta_1+\alpha_2+\beta_2\cdots=w_+,~\mbox{and} \\
  l'_0+\cdots+l'_{d-1}+\alpha'_1+\beta'_1+\alpha'_2+\beta'_2\cdots=w_-
\end{array}\right.
$$
then
 $$
\begin{array}{rcl}
{\rm Dim}&=&2^ww_+!w_-!\sum_{\Upsilon_1}\frac{w_+!}{l_0!l_1!\cdots
l_{d-1}!\alpha_1!\beta_1!\alpha_2!\beta_2!\cdots}\frac{w_-!}{l'_0!l'_1!\cdots
l'_{d-1}!\alpha'_1!\beta'_1!\alpha'_2!\beta'_2!\cdots}\\
&&\dim(\chi_{\Delta_+\times\Delta_-}\otimes\chi_{s_1,{\lambda_1}}\otimes
\cdots\chi_{s_{w_+},{\lambda_{w_+}}}\otimes\\
&&\chi_{s'_{w_++1},{\lambda_{w_++1}}}\otimes\cdots
\otimes\chi_{s'_w,{\lambda_w}})^2\\
&=&2^ww_+!w_-!\dim_K(KLf_0).
\end{array}$$
\end{proof}

\subsubsection{Morita equivalence between $\mathcal{O}Gb$ and $\mathcal{O}Nf$}

\begin{theorem}\label{thm:MoritaONf} The block
$\mathcal{O}Nf$ is Morita equivalent to $\mathcal{O}Gb$.
\end{theorem}
\begin{proof}
We first prove that  $\mathcal{O}Nf$ and $\operatorname{End}_{\mathcal{O}G}(Y)$ have the same
$\mathcal{O}$-rank. It suffices to show that
$K\otimes_\mathcal{O}\mathcal{O}Nf$ and $K\otimes_\mathcal{O}\operatorname{End}_{\mathcal{O}G}(Y)$
have the same $K$-dimension.

Since  $Z^{\tuple \nu}$ are isomorphic as $\mathcal{O}G$-module for all $\tuple \nu\in \mathbb{J}_{\tuple w}$,
 we have that $\operatorname{End}_{\mathcal{O}G}(Y)$
 is a matrix algebra of $\operatorname{End}_{\mathcal{O}G}(Z)$ of
 size $\frac{w!}{w_+!w_-!}.$
 Thus $$\dim_K\operatorname{End}_{KG}(Y)=(\frac{w!}{w_+!w_-!})^2
 \dim_K\operatorname{End}_{KG}(Z)$$
 which is equal to
 $(\frac{w!}{w_+!w_-!})^2\dim_K(KMf_0)$ by Proposition \ref{prop:dimZ}.
 However, this is exactly $\dim_K(KNf)$ by the fact that $|N:M|=\frac{w!}{w_+!w_-!}$ and
 Clifford's theory of blocks \cite[\S 5.2]{Nagao}.

 \smallskip

Now all monomorphisms in (\ref{equation:monos}) 
become isomorphisms. Since $X$ is a progenerator
as a left $\mathcal{O}Gb$-module by  Proposition \ref{prop:splitmono},
it follows that $_{\mathcal{O}Gb}X_{\mathcal{O}Nf}$ induces
a Morita equivalence between $\mathcal{O}Gb$ and $\mathcal{O}Nf$.
\end{proof}

\smallskip

We can now  prove Theorem \ref{thm:Iso-RoCK}.

\begin{proof}[Proof of Theorem \ref{thm:Iso-RoCK}]
It is known that Brou\'{e}'s abelian defect group conjecture is true for blocks
with cyclic defect groups. Hence the block $a^+$ (resp. $a^-$) of $\mathcal{O}(\GL_d(q).2)$ is derived
equivalent to the block $c^+$ (resp. $c^-$) of $\mathcal{O}N_{\GL_d(q).2}(P_0)$.
By \cite[Theorem 4.3(b)]{marcus1994},
the  block $a^{\tuple\nu_0}$ of
$\mathcal{O}((\GL_d(q).2\wr \frakS_{w_+})\times(\GL_d(q).2\wr \frakS_{w_-}))$
is derived equivalent to the  block $c^{\tuple\nu_0}$ of
$\mathcal{O}((N_{\GL_d(q).2}(P_0)\wr \frakS_{w_+})\times (N_{\GL_d(q).2}(P_0)\wr \frakS_{w_-}))$.
Now the result follows by Theorem \ref{thm:MoritaONf}, and
Lemmas \ref{lem:Moritaf&a} and \ref{lem:Moritah&c}.
\end{proof}

\subsection{Brou\'{e}'s conjecture}\label{sec:broconj}

At the end, we prove Theorem A, which can be reformulated as follows.

\begin{theorem}\label{thm:derivedBC}
 Let $G=\SO_{2n+1}(q)$, where $q$ is odd. Assume that $\ell$ is a linear prime with respect to $q$.
Let $B$ be an $\ell$-block of $G$ over $R$ where $R=k$ or $\mathcal{O}$, and $P$ be a defect group of $B$.
If $P$ is abelian, then $B$ is derived equivalent to its Brauer correspondent in $N_{G}(P)$.
\end{theorem}

\begin{proof} We first deal with the case where $R=k$.  By the reduction theorem of
Bonnaf\'{e}-Dat-Rouquier \cite[Theorem 7.7]{BDR17}
(or Theorem \ref{th:introequiv}),
there is a Morita equivalence between the block $B$ and an isolated  block of a Levi subgroup $L$ of $G$ preserving the local structure,
 where $L$ is of the form $$\SO_{2m+1}(q)\times \prod_{i}\GL_{n_i}(q^{m_i})\times \prod_{j}U_{n_j}(q^{m_j}).$$
 Since $\ell$ is linear (i.e., $f$ is odd),
 the order of $q^{m_j}$ mod $\ell$ is odd and
 thus $-q^{m_j}$ mod $\ell $ is even for any $m_j\in \bbN.$
 It follows that $\ell$ is also a linear prime for all $U_{n_j}(q^{m_j})$s.
 Assume that the theorem holds for isolated block of $\SO_{2m+1}$ for all $m$.
 Then the theorem immediately holds since
 Brou\'{e}'s abelian defect group conjecture has been proved to be true for
  $\GL_n(q)$ \cite{CR} and for $U_n(q)$ in the case of linear primes \cite{DVV2}.

  \smallskip

So it remains to prove that the theorem holds for isolated blocks of $G$.
When $\ell=2$,
 all quadratic unipotent characters are in $\mathcal E_{2}(G, (1))$.
By \cite[Theorem 21.14]{CE04}, all of them belong to
 the principal $2$-block of $G$, which in this case
 the defect groups are abelian if and only if $n=0,$ i.e., $G=G_0=1.$
 So there is nothing to prove.
Hence we may assume that $\ell\neq 2.$
 If $\ell$
divides $q-1,$
then every  isolated $\ell$-block of $G$
is an isolated RoCK block by definition, in which case Theorem \ref{thm:Iso-RoCK} applies.
So we may assume that $\ell\nmid q-1$. By Theorem \ref{thm:Linearprime},
the isolated blocks of $G$ in $\scrQU_{k,t_+,t_-}$ with the same degree vector $\tuple w$
form a single orbit under the action of the affine Weyl group $W_{2d}$.
Observe that there is always an isolated RoCK block in each orbit.
Let $B'$ be such a block in the orbit of $B.$
 An application of Theorem \ref{thm:Iso-RoCK} shows that
$B'$ is derived equivalent to its Brauer correspondent.
By Lemma \ref{lem:Moritah&c}, the Brauer correspondents of $B$
and $B'$ are Morita equivalent to the same block of some group.
Now the theorem follows since
 the derive equivalence of $B$ and $B'$
is guaranteed by
a result of Chuang and Rouquier \cite[Theorem~6.4]{CR} saying that
isolated blocks of $G$ lying in the same orbit under the action of
$W_{2d}$ are derived equivalent.

\smallskip

Now we consider the case where $R=\mathcal{O}.$
Notice that the above derived equivalence
can be indeed chosen to be a splendid Rickard equivalence over $k$
by a similar argument as for \cite[Theorem 7.6 or 7.20]{CR}.
Hence the lifting theorem of splendid Rickard equivalence \cite[Theorem 5.2]{Ri96}
yields a splendid Rickard equivalence over $\mathcal{O},$ which of course induces a
 derived equivalence between $B$ and its correspondence.
\end{proof}

\subsection*{Acknowledgements}
We are deeply grateful to Xin Huang, Gunter Malle, Lizhong Wang and Wolfgang Willems
for inspiring discussions, and to Zhicheng Feng, Conghui Li and Zhenye Li
for improving an earlier version of this paper. In particular, the second author would like to
express his deepest thanks to Gunter Malle for his warm-hearted help during
visiting him at TU Kaiserslautern supported by the Alexander von Humboldt Foundation.

\end{document}